\documentclass[12pt]{article}
\usepackage{amsmath,amsthm,amssymb}
\usepackage[mathscr]{euscript}
\usepackage{verbatim}
\usepackage{dsfont,bm}
\usepackage{color}
\usepackage{mathtools}
\usepackage{graphicx}
\usepackage{enumerate}
\usepackage{esint} 
\usepackage{mathabx}
\allowdisplaybreaks[4]

\usepackage{xcolor}
\definecolor{darkgreen}{rgb}{0,0.75,0}
\definecolor{darkred}{rgb}{0.75,0,0}
\definecolor{darkmagenta}{rgb}{0.5,0,0.5}
\usepackage[colorlinks,citecolor=darkgreen,linkcolor=darkred,urlcolor=darkmagenta]{hyperref}
\newtheorem{theorem}{Theorem}[section]
\newtheorem{cor}[theorem]{Corollary}
\newtheorem{lem}[theorem]{Lemma}
\newtheorem{prop}[theorem]{Proposition}

\theoremstyle{definition}
\newtheorem{definition}[theorem]{Definition}
\newtheorem{assumption}[theorem]{Assumption}
\newtheorem{remark}[theorem]{Remark}
\newtheorem{example}[theorem]{Example}

\newtheorem{notation}[theorem]{Notation}

\numberwithin{equation}{section}

\newcommand{\arxiv}[1]{{\tt \href{https://arxiv.org/abs/#1}{arXiv:#1}}}
\newcommand{\doi}[1]{{\tt \href{https://doi.org/#1}{DOI:#1}}}
\newcommand{\set}[1]{\left\{ #1 \right\}}

\newcommand{\abs}[1]{{\left\lvert #1\right\rvert}}
\newcommand\norm[1]{\left\lVert #1\right\rVert} %norm
\newcommand{\one}{\mathds{1}} %indicator
\newcommand{\loc}[0]{\operatorname{loc}}
\newcommand{\diag}[0]{\operatorname{diag}}
\newcommand{\od}[0]{\operatorname{od}}
\newcommand{\offdiag}[1]{#1^{2}_{\od}}
\newcommand{\offdiagp}[1]{(#1)^{2}_{\od}}
\DeclareMathOperator*{\esssup}{ess\,sup}
\DeclareMathOperator*{\essinf}{ess\,inf}
\DeclareMathOperator*{\osc}{osc}

\newcommand{\divergence}[0]{\operatorname{div}}
\newcommand{\pv}[0]{\mathop{\mathrm{p.v.}}}

\newcommand{\contfunc}{C}
\newcommand{\ambient}{\mathcal{X}} %Ambient metric space
\newcommand{\refmeas}{m} %Measure on ambient space
\newcommand{\unifdom}{U} %Uniform domain
\newcommand{\form}{\mathcal{E}}
\newcommand{\domain}{\mathcal{F}}
\newcommand{\bdrymeas}{\mu} %Measure on \partial U (Either harmonic measure or limit of harmonic measure)
\newcommand{\bdrypcaf}{A^{(\bdrymeas)}}%PCAF corresponding to \mu

\newcommand{\goodmeas}{\nu} %capacity good measure
\newcommand{\goodpcaf}{A^{(\goodmeas)}} %pcaf corresponding to the capacity good measure

\newcommand{\scdiff}{\Psi}
\newcommand{\scjump}{\Phi} 
\newcommand\hmeas[2]{\omega^{#1}_{#2}}  %Harmonic measure on \partial U
\newcommand\gren[1]{g_{#1}}%Green function on U for the diffusion on ambient space:
\newcommand\grenref[1]{g^{\operatorname{ref}}_{#1}}%Green function on U for the diffusion on ambient space:
\newcommand{\harmprofile}{h}
\newcommand{\hprof}[1]{\harmprofile^{\unifdom}_{#1}} %Harmonic profile of an unbounded domain U
\newcommand{\formref}{\form^{\on{ref}}}
\newcommand{\domainref}{\domain(\unifdom)}  
\newcommand{\trform}{\widecheck{\form}}
\newcommand{\trdomain}{\widecheck{\domain}}
\newcommand{\trformc}{\widecheck{\Gamma}_{c}}

\newcommand{\diffref}{X^{\operatorname{ref}}}
\newcommand{\lawref}{\mathbb{P}^{\operatorname{ref}}} 
\newcommand{\expref}{\mathbb{E}^{\operatorname{ref}}} 
\newcommand{\hkref}{p^{\on{ref}}}

\newcommand{\naimker}{\Theta}
\newcommand{\martinker}{K}
\newcommand{\diff}{X} %Diffusion on the ambient space
\newcommand{\formtr}{\widecheck{\form}^{\operatorname{ref}}}
\newcommand{\formtrsl}{\widecheck{\form}^{\operatorname{ref},(c)}}
\newcommand{\domaintr}{\widecheck{\domain}(\unifdom)}
\newcommand{\diffreftr}{\widecheck{\diff}^{\on{ref}}}
\newcommand{\hkreftr}{\widecheck{p}^{\on{ref}}}

\newcommand{\jumpker}{j}

\newcommand{\dimeuc}{N}
\newcommand{\Borel}{\mathscr{B}}
\newcommand{\cemetery}{\partial}
\newcommand{\oneptcpt}[1]{#1_{\cemetery}}
\newcommand{\oneptcptp}[1]{(#1)_{\cemetery}}
\newcommand{\events}{\mathscr{M}}
\newcommand{\minaugfilt}{\mathscr{F}}

\newcommand{\lawdiff}{\mathbb{P}}
\newcommand{\expdiff}{\mathbb{E}}
\newcommand{\shiftdiff}{\theta}
\newcommand{\difftr}{\widecheck{\diff}} 
\newcommand{\formrefgen}[1]{\form^{\on{ref},#1}}

\newcommand{\jumpmeas}{J}
\newcommand{\jumpmeastr}{\widecheck{\jumpmeas}}
\newcommand{\jumpkertr}{\widecheck{\jumpker}}
\newcommand{\scjthres}[1]{M_{#1}}
\newcommand{\cdcthres}[1]{M_{#1}}
\newcommand{\Capa}{\operatorname{Cap}}

  \def\sL {{\mathcal L}}

\def\bD {{\mathbb D}}

 \def\bN {{\mathbb N}} 
  \def\bR {{\mathbb R}}

\def\ol{\overline}
\def\dint{\int\kern-.6em\int}

\newcommand\restr[2]{{% we make the whole thing an ordinary symbol
		\left.\kern-\nulldelimiterspace % automatically resize the bar with \right
		#1 % the function
		\vphantom{\big|} % pretend it's a little taller at normal size
		\right|_{#2} % this is the delimiter
}} %restriction of a function
\def\diam{{\mathop{{\rm diam }}}}
\def\dist{{\mathop {{\rm dist}}}}
\def\div{{\mathop {{\rm div\, }}}}

\def\supp{\mathop{{\rm supp}}\nolimits}

\newcommand{\on}[1]{\operatorname{ #1}}

\def\wt{\widetilde}
\def\wh{\widehat}

\definecolor{dgreen}{rgb}{0, 0.6, 0.1}
\definecolor{dblue}{rgb}{0, 0.0, 0.6}
\definecolor{vdblue}{rgb}{0,.08, 0.45}
\definecolor{dred}{rgb}{0.7, 0.0, 0.0}
\definecolor{vdblue}{rgb}{0,.08, 0.45}

\definecolor{purple}{rgb}{0.6, 0.0, 0.6}
\definecolor{mytext}{rgb}{0.1, 0.1, 0.1}

\begin{document}
	
	\font\titlefont=cmbx14 scaled\magstep1
	\title{\titlefont Heat kernel estimates for boundary traces of reflected diffusions on uniform domains\footnote{This work was supported by the Research Institute for Mathematical Sciences, an International Joint Usage/Research Center located in Kyoto University.}}    
	\author{Naotaka Kajino\footnote{Research partially supported by JSPS KAKENHI Grant Numbers JP22H01128, JP23K22399.} \ and Mathav Murugan\footnote{Research partially supported by NSERC and the Canada research chairs program.}}
	\date{November 12, 2026}
	\maketitle
	\vspace{-0.5cm}
%	\ver
	
	\begin{abstract}
		We study the boundary trace processes of reflected diffusions on uniform domains.
		We obtain stable-like heat kernel estimates for such a boundary trace process when the diffusion on the underlying ambient space satisfies sub-Gaussian heat kernel estimates.
		Our arguments rely on new results of independent interest such as sharp two-sided estimates and the volume doubling property of the harmonic measure, the existence of a continuous extension of the Na\"im kernel to the topological boundary,
		and the Doob--Na\"im formula identifying the Dirichlet form of the boundary trace process
		as the pure-jump Dirichlet form whose jump kernel with respect to the harmonic measure is exactly (the continuous extension of) the Na\"im kernel.
		
		\vskip.2cm
		
		
		\vskip.2cm
		

	\end{abstract}
	
	\tableofcontents
	
	\section{Introduction} \label{sec:intro}
	
	The goal of this work is to study the  boundary trace of reflected diffusions on $\overline{\unifdom}$, where $\unifdom$ is a  `nice' domain.
	Given a reflected diffusion process on $\overline{\unifdom}$, the boundary trace process on $\partial \unifdom$ is obtained by removing the path of the reflection diffusion in the interior $\unifdom$ in a certain sense.
	The resulting boundary trace process is a jump process on $\partial \unifdom$.
	From an analytic viewpoint,  the generator of the boundary trace process  can be viewed as a non-local (integro-differential) operator on the boundary $\partial \unifdom$ associated to a local (differential) operator that is the generator of the corresponding diffusion process.
	For reflected Brownian motion on smooth domains, this non-local operator on the boundary is essentially the classical   Dirichlet-to-Neumann map.
	
	Although we are motivated by probabilistic considerations related to the boundary trace process mentioned above, the induced non-local operator on the boundary is also widely studied in the context of electrical impedance tomography and Calder\'on's inverse problem \cite{Uhl}.
	In a different direction, free boundary regularity for the obstacle problem was obtained for the fractional Laplacian by using the fact that it arises as an induced boundary operator corresponding to a degenerate elliptic (diffusion) operator \cite{CSS}.
	More generally, on the basis of such a correspondence between local (diffusion) operators on a domain and non-local operators on its boundary, properties of non-local operators can be understood by using  better knowledge of the corresponding local operators  \cite[\textsection 5]{CS}.
	
	A classical example of a trace process is the Cauchy process (rotationally symmetric $1$-stable process) on $\bR^\dimeuc$ that arises as the boundary trace process of the reflected Brownian motion on the upper half space $\bR^{\dimeuc} \times [0,\infty)$.
	The analytic version of this probabilistic fact is that the Dirichlet-to-Neumann map on the boundary of the $(\dimeuc+1)$-dimensional upper half-space is the square-root of the Laplacian on $\bR^{\dimeuc}$, the generator of the Cauchy process on $\bR^{\dimeuc}$.
	Given this classical example, the following natural guiding question motivates this work:
	\begin{quote}
		\emph{``Does the boundary operator behave like a fractional Laplace operator for more general diffusions (elliptic operators) and domains?''}
	\end{quote}
	Our main results answer this question affirmatively by obtaining quantitative versions of the following statement for a large class of reflected diffusions  and domains:
	\begin{quote}
		\emph{``The boundary trace process behaves like a  rotationally  symmetric stable  process. Equivalently, the induced non-local operator on the boundary behaves like a fractional Laplace operator." }
	\end{quote}
	
	In this work, we quantify the above statement on the boundary trace process in various ways by considering stable-like estimates of its jump kernel (equivalently, the integral kernel of the induced non-local operator on  the boundary), of its mean exit times from balls, and of its transition probability density (equivalently, the heat kernel associated to the non-local operator on  the boundary). 
	The significance of our results is that the boundary trace process shares many desirable properties of rotationally symmetric stable processes on $\bR^{\dimeuc}$ (or equivalently, the fractional Laplace operator) such as elliptic and parabolic Harnack inequalities.
	%These estimates are new even for reflected Brownian motion on  Lipschitz domains in $\bR^n$. 
	%Symmetric $\alpha$-stable processes and fractional Laplace operators enjoy many desirable properties. 
	%For instance, a new consequence of our results is that the Dirichlet-to-Neumann operator on the boundary of any   Lipschitz domain in $\mathbb{R}^{\dimeuc}, \dimeuc \ge 2$  satisfies the scale-invariant elliptic Harnack inequality. 
	%
	%In this paper, we give sharp two-sided estimates on the heat kernel of the boundary
	%trace process of a reflected symmetric diffusion on a uniform domain in a general state
	%space, assuming the scale-invariant elliptic Harnack inequality.
	We note that stable-like heat kernel estimates for jump processes have been extensively studied
	for the past two decades; see, e.g., \cite{BL,BGK-jump,CK03, CK08, CKW,GHL14, GHH-lb, GHH-mv, Mal, MS19}.
	Our heat kernel estimates for the boundary trace are new even for reflected Brownian motion on
	Lipschitz domains in $\bR^{\dimeuc}$ and for reflected diffusions on the upper half-space
	generated by uniformly elliptic divergence-form operators. Our results are applicable
	also to diffusions on nice fractals such as the Brownian motion on the standard
	Sierpi\'{n}ski carpet if we take as the domain $\unifdom$, e.g., the complement
	of the bottom line segment or that of the boundary of the unit square.
	
	More precisely, this paper is aimed at establishing the following results
	\eqref{it:result-hmeas-intro}, \eqref{it:result-dnformula-intro} and \eqref{it:result-shk-trace-intro}
	for a reflected diffusion on a uniform domain satisfying the capacity density
	condition (a natural condition guaranteeing that its boundary is thick enough everywhere
	in every scale), in the general setting of a strongly local regular symmetric Dirichlet space
	equipped with a complete metric and satisfying the volume doubling property and
	sub-Gaussian heat kernel estimates:
	\begin{enumerate}[\rm(i)]\setlength{\itemsep}{0pt}\vspace{-5pt}
		\item\label{it:result-hmeas-intro} Two-sided estimates on the harmonic measure and the associated elliptic measure
		at infinity that are sharp up to multiplicative constants (Theorem \ref{t:hmeas} and Proposition \ref{p:emeas}).
		\item\label{it:result-dnformula-intro} The identification of the Dirichlet form of the boundary trace process
		as the bilinear form given by the Doob--Na\"im formula, which in particular shows
		that the boundary trace process is a pure-jump process (Theorem \ref{t:dnformula}).
		Equivalently, this is an expression for the non-local operator on the boundary associated with a local (diffusion) operator on the domain.
		\item\label{it:result-shk-trace-intro} Two-sided heat kernel estimates for the boundary trace process
		that are similar to those for the rotationally symmetric stable processes on the Euclidean space
		(Theorem \ref{thm:shk-trace}).
	\end{enumerate}
	
	Let us first start in Subsection \ref{ssec:overview} with an overview of the most
	relevant results available in the literature and a summary of our main results. Then
	in Subsection \ref{ssec:results} we give the precise statements of our main results,
	introducing key notions needed for this purpose but referring to the main text
	for the technical details underlying their definitions.
	
	\subsection{Overview} \label{ssec:overview}
	
	A classical theorem of Spitzer \cite{Spi} (see also \cite{Mol}) implies that the boundary trace process of the reflected Brownian motion on the $(\dimeuc+1)$-dimensional upper half-space $\bR^{\dimeuc}\times[0,\infty)$ is the $\dimeuc$-dimensional Cauchy process.
	Molchanov and Ostrowski \cite{MO} discovered that one can realize every rotationally symmetric stable process on $\bR^{\dimeuc}$ as the trace process on the boundary of a reflected diffusion on the $(\dimeuc+1)$-dimensional upper half-space.
	This was later revisited in a celebrated work \cite{CS} by Caffarelli and Silvestre
	to analyze the fractional Laplace operator and is now known as the \emph{Caffarelli--Silvestre extension}.
	They demonstrated in \cite[\textsection 5]{CS} that properties of non-local operators could be understood by using corresponding properties of the associated local operators.
	The local and non-local operators in \cite{CS} are the generators of the diffusion in the upper half-space and its boundary trace process in \cite{MO}, respectively.
	Our present work is aimed at extending this idea to understand the behavior of the boundary trace process (a jump process) by using that of the associated diffusion process. 
	
	Let us examine the results of Molchanov--Ostrowski \cite{MO} and Caffarelli--Silvestre \cite{CS} in further detail
	to provide context. For $\alpha \in (0,2)$, we recall that the rotationally symmetric $\alpha$-stable process
	is generated by the \emph{fractional Laplace operator} $(-\Delta)^{\alpha/2}$ on $\bR^{\dimeuc}$, 
	\begin{equation*}
		(-\Delta)^{\alpha/2}f(x) \coloneqq c_{\dimeuc,\alpha} \cdot \pv \int_{\bR^{\dimeuc}} \frac{f(x)-f(y)}{\abs{x-y}^{\dimeuc+\alpha}}\,dy,
	\end{equation*}
	where $ c_{\dimeuc,\alpha} \in (0,\infty)$ is a normalizing constant.
	Writing $\bR^{\dimeuc+1}= \{(x,y) \mid x\in \bR^{\dimeuc},\,y \in \bR\}$ as
	$\bR^{\dimeuc} \times \bR$, we consider the Dirichlet form
	\begin{equation*}
		\form(u,u) \coloneqq \int_{\bR^{\dimeuc}} \int_{0}^{\infty} \abs{\nabla u}^2(x,y) \abs{y}^{1-\alpha}\,dy \,dx
	\end{equation*}
	on $L^2(\bR^{\dimeuc}\times[0,\infty),\abs{y}^{1-\alpha}\,dy \,dx )$.
	The corresponding diffusion is generated by the degenerate elliptic operator
	\begin{equation} \label{e:cs-generator}
		L_\alpha u \coloneqq y^{\alpha-1}\divergence\bigl(y^{1-\alpha}\nabla u\bigr)= \Delta_x u + \frac{1-\alpha}{y} \partial_{y}u + \partial_{y}^{2}u;
	\end{equation}
	here the terminology \emph{degenerate elliptic} refers to the degeneracy of the weight function
	$y^{1-\alpha}$ in the expression $L_\alpha u = y^{\alpha-1}\divergence\bigl(y^{1-\alpha}\nabla u\bigr)$
	near the boundary $\bR^{\dimeuc}\times\{0\}$.
	Gaussian heat kernel estimates for the diffusion generated by such a
	degenerate elliptic operator follow from results of \cite{FKS,Gri91,Sal}.
	To compute the Dirichlet form of the trace process on the boundary,
	we consider the Dirichlet boundary value problem
	\begin{equation} \label{e:csmo}
		L_\alpha u=0 \mbox{ on $\bR^{\dimeuc}\times (0,\infty)$}, \quad u(x,0)= f(x),
	\end{equation}
	where $f \colon \bR^{\dimeuc} \to \bR$ is a prescribed boundary value in a suitable function space. Then by \cite[\textsection 3.2]{CS}, the Dirichlet energy of the solution $u$ to \eqref{e:csmo} can be expressed in terms of the boundary data $f$ as 
	\begin{equation} \label{e:cse1}
	c_{\dimeuc,\alpha}	\int_{\bR^{\dimeuc}} \int_{0}^{\infty} \abs{\nabla u}^2(x,y) \abs{y}^{1-\alpha}\,dy \,dx
		= \int_{\bR^{\dimeuc}} f(\xi)(-\Delta)^{\alpha/2}f(\xi) \,d\xi.
	\end{equation}
	The equality \eqref{e:cse1} implies that the boundary trace process of the reflected diffusion
	generated by $L_\alpha$ is the rotationally symmetric $\alpha$-stable process.
	We refer to \cite{Kwa} for a recent result in this direction
	characterizing the class of L\'{e}vy processes on $\bR$ arising as the boundary traces
	of translation-invariant diffusions on $\bR\times[0,R)$ for some $R\in(0,\infty]$.
	
	An earlier example of an expression analogous to \eqref{e:cse1} that relates a local
	operator on a domain to a non-local operator on its boundary is the \emph{Douglas formula}
	due to J.~Douglas \cite{Dou}, which states that the harmonic function $u$ on the unit disk
	$\bD \coloneqq \{x \in \bR^2 \mid \abs{x} < 1 \}$ with boundary value regarded as a function
	$f\colon [0,2\pi) \to \bR$ has Dirichlet energy given by
	\begin{equation} \label{e:douglas}
		\int_{\bD} \abs{\nabla u}^2(x)\,dx
		= \frac{1}{8 \pi} \int_{0}^{2\pi}\int_{0}^{2\pi} \frac{(f(\eta)-f(\xi))^2}{\sin^2((\eta-\xi)/2)}\,d\eta\,d\xi.
	\end{equation}
	The right-hand side of \eqref{e:douglas} can be viewed as the Dirichlet form of the
	boundary trace process corresponding to the reflected Brownian motion on the unit disk.
	This result was later extended to any finitely connected bounded domain $D$ in
	$\bR^2$ with smooth boundary $\partial D$ by Osborn \cite{Osb}.
	He proved there that, if $u$ is a harmonic function on such $D$
	with boundary value $f \colon \partial D \to \bR$, then
	%its Dirichlet energy is given by
	\begin{equation} \label{e:cse2}
		\int_{D} \abs{\nabla u}^2(x)\,dx
		= \frac{1}{2} \int_{\partial D} \int_{\partial D} (f(\eta)-f(\xi))^2 \frac{\partial^2 \gren{D}(\xi,\eta)}{\partial \vec{n}_{\xi} \partial \vec{n}_{\eta}}\, d\sigma(\xi)\,d\sigma(\eta),
	\end{equation}
	where $\sigma$ is the surface measure on $\partial D$,
	$\gren{D}(\cdot,\cdot)$ is the Green function on $D$ and
	$\frac{\partial}{\partial \vec{n}_{\xi}}, \frac{\partial}{\partial \vec{n}_{\eta}}$
	denote the inward-pointing normal derivatives in $\xi,\eta$, respectively.
	More generally, now \eqref{e:cse2} is known to hold for any bounded domain $D$
	with $\contfunc^{3}$-boundary in $\bR^{\dimeuc}$ with $\dimeuc \geq 2$ as stated, e.g., in
	\cite[(5.8.4)]{CF} and was also extended by M.~Fukushima \cite[\textsection 2]{Fuk} to uniformly
	elliptic divergence-form operators with $\contfunc^{3}$-coefficients on such domains.
	
	Soon after \cite{Osb}, J.~Doob \cite{Doo} found a remarkable extension of \eqref{e:douglas} and \eqref{e:cse2}
	to domains that are not necessarily smooth. He stated the result under an abstract potential
	theoretic setting of (locally Euclidean) Green spaces in the sense of Brelot and Choquet \cite{BrCh}, in which
	the boundary values of the harmonic functions are prescribed on the \emph{Martin boundary} $\partial_{M} D$ of the domain $D$.
	To describe Doob's result, we recall the \textbf{Na\"im kernel} $\naimker^{D}_{x_0}(\cdot,\cdot)$ defined by
	\begin{equation} \label{e:Naim-kernel}
		\naimker^{D}_{x_0}(\xi,\eta) = \lim_{x \to \xi} \lim_{y \to \eta} \frac{\gren{D}(x,y)}{\gren{D}(x_0,x)\gren{D}(x_0,y)}
		\qquad \mbox{for $\xi,\eta \in \partial_{M} D$, $\xi \neq \eta$,}
	\end{equation}
	where the limits are with respect to the \emph{fine topology}, $x_0 \in D$ is an arbitrary
	base point, and $\gren{D}(\cdot,\cdot)$ is the Green function on $D$ as before.
	The existence of the above limits in the setting of Green spaces follows from the fundamental work
	\cite{Nai} by L.~Na\"im. Then it was shown in \cite[Theorem 9.2]{Doo} that the \textbf{Doob--Na\"im formula}
	\begin{equation} \label{e:dnf}
		\int_{D} \abs{\nabla u}^2(x)\,dx
		= \frac{1}{2} \int_{\partial_{M} D} \int_{\partial_{M} D} (f(\xi)-f(\eta))^2 \naimker^{D}_{x_0}(\xi,\eta)\, d\hmeas{D}{x_0}(\xi)\, d\hmeas{D}{x_0}(\eta)
	\end{equation}
	holds if $u$ is a harmonic function on the domain $D$ with fine boundary value
	$f \colon \partial_{M} D \to \bR$, where $\hmeas{D}{x_0}$ denotes the \textbf{harmonic measure},
	i.e., the probability distribution of the position of the first hitting to $\partial_{M} D$,
	of the Brownian motion on $D$ started at $x_0$.
	%M.~Fukushima \cite{Fuk} gave an alternative construction of the Na\"im kernel for Green spaces.
	There is a version of the Doob--Na\"im formula for transient symmetric
	Markov chains on countable state spaces due to M.~Silverstein \cite[Theorem 3.5]{Sil};
	see also \cite[Theorem 6.4]{BGPW} for a simple proof of it for nearest-neighbor random walks on trees.
	The equality \eqref{e:cse1} from \cite{CS} mentioned above can be considered as an extension of \eqref{e:dnf}
	to the case of the reflected diffusion generated by $L_{\alpha}$ as in \eqref{e:cs-generator};
	see also Example \ref{x:mo-cs} for this connection.
	
	Our principal concern in this paper is to establish nice two-sided heat kernel estimates
	for jump-type Dirichlet forms which arise from the Dirichlet forms of symmetric diffusions
	in the same way the right-hand side of \eqref{e:dnf} does from the Dirichlet form
	$\int_{D} \abs{\nabla(\cdot)}^2\,dx$ of the reflected Brownian motion on $D$.
	We consider a general symmetric diffusion with general locally compact state space,
	or more precisely, an $\refmeas$-symmetric diffusion associated with a strongly local
	regular symmetric Dirichlet form $(\form, \domain)$ on $L^2(\ambient, \refmeas)$
	where $(\ambient,d)$ is a metric space which contains at least two elements and whose every
	bounded closed set is compact and $\refmeas$ is a Radon measure on $\ambient$ with full support.
	We call $(\ambient, d, \refmeas, \form, \domain)$ a \textbf{metric measure Dirichlet space},
	or an \textbf{MMD space} for short. We refer to \cite{FOT,CF} for the theory of regular symmetric Dirichlet forms.
	
	The only essential a priori requirement on the MMD space for the purpose of this paper is
	that it satisfy the metric doubling property and the elliptic Harnack inequality.
	The MMD space $(\ambient, d, \refmeas, \form, \domain)$ is said to satisfy the
	\hypertarget{md-intro}{\textbf{metric doubling property} (MD)} if there exists $N \in \bN$
	such that any open ball $B$ in $(\ambient,d)$ can be covered by $N$ balls with radii half of
	that of $B$, and to satisfy the \textbf{(scale-invariant) elliptic Harnack inequality} (EHI) if
	\begin{equation} \tag*{$\on{EHI}$} \label{eq:EHI-intro}
		\esssup_{B(x,\delta r)} h \le C \essinf_{B(x,\delta r)} h
	\end{equation}
	for any open ball $B(x,r)$ in $(\ambient,d)$ and any non-negative
	harmonic function $h$ on $B(x,r)$ for some $C \in (1,\infty)$ and $\delta \in (0,1)$.
	The metric doubling property is the weakest possible requirement to guarantee decent
	behavior of the geometry of $(\ambient,d)$ in relation to heat kernel estimates, and
	is easily seen to follow from the well-known \textbf{volume doubling property} (VD) of
	$\refmeas$ (or $(\ambient, d, \refmeas)$), i.e., the existence of some $C \in (1,\infty)$ with
	\begin{equation} \tag*{$\on{VD}$} \label{eq:VD-intro}
		\refmeas(B(x,2r))\leq C \refmeas(B(x,r)) \quad \mbox{for all $x \in \ambient$ and all $r \in (0,\infty)$.}
	\end{equation}%
	It is also reasonable to assume \ref{eq:EHI-intro} because, in view of
	\eqref{e:Naim-kernel} and \eqref{e:dnf} above, we should need to have good control
	on the quantitative behavior of the Green function $\gren{D}(x_{0},\cdot)$ and
	the harmonic measure $\hmeas{D}{x_0}$, which are indeed non-negative harmonic
	functions on $D \setminus \{x_{0}\}$ and in $x_{0} \in D$, respectively.
	As established in \cite{Stu96,BGK,GHL15}, \ref{eq:EHI-intro} is implied by the
	conjunction of \ref{eq:VD-intro} and  \textbf{Gaussian} or \textbf{sub-Gaussian heat kernel estimates}.
	Our setting therefore includes diffusions with Gaussian heat kernel estimates as considered
	in \cite{Gri91,Sal,Stu96} such as Brownian motion on the Euclidean space or Riemannian manifolds
	with non-negative Ricci curvature, diffusions generated by uniformly elliptic divergence-form
	operators on $\bR^{\dimeuc}$ \cite{Mos} or degenerate elliptic operators \cite{FKS},
	diffusions on connected nilpotent Lie groups associated with left-invariant
	Riemannian metrics or with sub-Laplacians of the form $\Delta=\sum_{i=1}^{k} X_{i}^{2}$
	for a family $\{X_{i}\}_{i=1}^{k}$ of left-invariant vector fields satisfying
	H\"ormander's condition \cite{VSC}, and weighted Euclidean spaces and Riemannian
	manifolds \cite{GrS,Gri09}. Another significant class of examples arise from diffusions
	on fractals such as the Sierpi\'nski gasket, the Sierpi\'nski carpet and their variants
	\cite{Bar98,BB89, BB92, BB99, BP, BH, FHK, Kum}, where Gaussian heat kernel estimates
	are no longer true but sub-Gaussian ones do hold.
	
	As mentioned above, our goal is to prove heat kernel estimates for the Dirichlet forms obtained from
	general symmetric diffusions through the counterpart of the Doob--Na\"im formula \eqref{e:dnf}.
	A first observation to be made toward this aim is that such a jump-type Dirichlet form
	should be viewed as a quadratic form corresponding to a self-adjoint non-local operator
	with respect to a reference measure $\bdrymeas$ that is mutually absolutely continuous
	with respect to the harmonic measure $\hmeas{D}{x_0}$.
	Due to \eqref{e:cse2}, in the case of the reflected Brownian motion on a bounded domain
	with smooth boundary, this reference measure $\bdrymeas$ is usually taken to be the
	surface measure on the boundary, and then the generator of the boundary trace process is an
	integro-differential operator and can be identified as the Dirichlet-to-Neumann
	(or voltage-to-current) map as shown in \cite[Section 4]{Hsu}.
	In general, however, even for uniformly elliptic operators on smooth domains,
	the harmonic measure might differ significantly from
	the surface measure, and in fact can be singular as proved in \cite{CFK,MM}.
	It is worth mentioning that our results on the stable-like heat kernel estimates
	for boundary trace processes apply also to situations where the harmonic
	measure is singular with respect to the surface measure (see Example \ref{x:ba-qc}).
	
	By virtue of the nice characterizations of heat kernel estimates for jump-type Dirichlet
	forms established in \cite{CKW, GHH-lb, GHH-mv}, the proof of heat kernel estimates
	for boundary trace processes is reduced to verifying a set of quantitative bounds on
	the reference measure $\bdrymeas$ and on the jump kernel $\jumpkertr_{\bdrymeas}(\xi,\eta)$
	with respect to $d\bdrymeas(\xi)\,d\bdrymeas(\eta)$ of our jump-type Dirichlet form.
	Those quantitative bounds include the volume doubling property \ref{eq:VD-intro} of $\bdrymeas$
	and matching two-sided estimates on the jump kernel $\jumpkertr_{\bdrymeas}(\xi,\eta)$,
	which, for our jump-type Dirichlet form analogous to the right-hand side of \eqref{e:dnf}, is given by
	\begin{equation} \label{e:jump-kernel-intro}
		\jumpkertr_{\bdrymeas}(\xi,\eta) \coloneqq \naimker^{D}_{x_0}(\xi,\eta)\frac{d\hmeas{D}{x_0}}{d\bdrymeas}(\xi)\frac{d\hmeas{D}{x_0}}{d\bdrymeas}(\eta).
	\end{equation}
	Since the Na\"im kernel $\naimker^{D}_{x_0}(\cdot,\cdot)$ is defined in terms of
	a ratio of the Green function $\gren{D}$ by \eqref{e:Naim-kernel}, a natural
	choice of the setting for trying to prove such bounds on $\bdrymeas$ and
	$\jumpkertr_{\bdrymeas}$ would be a domain $D$ for which we could expect both
	the volume doubling property \ref{eq:VD-intro} of the harmonic measure
	$\hmeas{D}{x_0}$ and some good control on the boundary behavior of the Green
	function $\gren{D}$. Arguably the most general class of domains $D$ known
	in the literature to satisfy these requirements is that of \textbf{uniform domains}
	satisfying the \textbf{capacity density condition} (CDC).
	Indeed, this class of domains in $\bR^{\dimeuc}$ was shown by Aikawa and Hirata
	\cite{AH} to satisfy nice two-sided bounds on the harmonic measure which imply
	its volume doubling property, and by Aikawa \cite{Aik} to satisfy the
	\textbf{(scale-invariant) boundary Harnack principle} (BHP), which is a well-established
	analogue of \ref{eq:EHI-intro} for the ratios of positive harmonic functions with zero
	Dirichlet boundary condition along domain boundary. BHP for uniform domains is in fact
	available in our setting of an MMD space with \hyperlink{md-intro}{MD} and \ref{EHI}
	as proved in \cite{Lie,BM19,Che} and allows us to extend the Na\"im kernel
	$\naimker^{D}_{x_0}(\cdot,\cdot)$ continuously to the domain boundary,
	and as one of our main results we extend the two-sided
	bounds on the harmonic measure as in \cite{AH} to any uniform domain satisfying CDC
	in any MMD space with \hyperlink{md-intro}{MD} and \ref{EHI}.
	
	Uniform domains were introduced independently by Martio and Sarvas \cite{MS} and
	Jones \cite{Jon81}. This class includes Lipschitz domains, and more generally
	non-tangentially accesible (NTA) domains introduced by Jerison and Kenig \cite{JK}.
	We note that, due to the similarity in the definitions, uniform domains are also
	referred to as \emph{one-sided NTA domains} in, e.g., \cite{AHMT, HMM}.
	Uniform domains are relevant in various contexts such as extension property \cite{BS, Jon81, HeK},
	Gromov hyperbolicity \cite{BHK}, boundary Harnack principle \cite{Aik,GS}, geometric
	function theory \cite{MS,GeHa,Geh}, and heat kernel estimates \cite{GS, CKKW, Lie22, Mur24}.
	One reason for the importance of uniform domains is their close connection to Gromov
	hyperbolic spaces \cite{BHK}. Another reason is their abundance;
	in fact, by \cite[Theorem 1.1]{Raj} every bounded domain is arbitrarily close
	to a uniform domain in a large class of metric spaces. 
	
	The NTA domains introduced in \cite{JK} are examples of uniform domains satisfying CDC.
	CDC guarantees that every boundary point is regular for the associated diffusion
	and can be viewed as a stronger version of Wiener's test of regularity.
	Uniform domains satisfying CDC provide a fruitful setting to study various aspects of the harmonic
	measure \cite{Anc,AH,AHMT,AHMT2,CDMT}. For Brownian motion on the Euclidean space,
	CDC for a domain $D$ is formulated as the following estimate:
	\begin{equation} \label{e:CDC-intro}
		\Capa_{B(\xi,2r)}(B(\xi,r)) \lesssim \Capa_{B(\xi,2r)}( B(\xi,r) \setminus D) \quad \mbox{for all $\xi \in \partial D$, $0< r \lesssim \diam(D)$,}
	\end{equation}
	where $\Capa_{B(\xi,2r)}(K) $ denotes the capacity between the sets $K$ and $B(\xi,2r)^c$.
	The fact that uniform domains with CDC \eqref{e:CDC-intro} satisfy good properties of the
	harmonic measure was recognized by Aikawa and Hirata \cite{AH}. As we will see later,
	estimates on the harmonic measure play an important role in our work.
	
	We are thus led naturally to the setting of a uniform domain $\unifdom$ satisfying CDC
	in an MMD space $(\ambient, d, \refmeas, \form, \domain)$ with \hyperlink{md-intro}{MD}
	and \ref{eq:EHI-intro}. We could state our main results under this setting, but for the
	sake of simplicity of their statements and proofs, in most parts of this paper we will
	assume that $(\ambient, d, \refmeas, \form, \domain)$ satisfies \ref{eq:VD-intro} and
	sub-Gaussian heat kernel estimates instead of \hyperlink{md-intro}{MD} and \ref{eq:EHI-intro}.
	There is essentially no loss of generality in assuming so, because it was proved in
	\cite[Theorem 7.9]{BCM} (see also \cite[Theorem 5.4]{BCM} and \cite[Theorem 4.5]{KM23}) that \hyperlink{md-intro}{MD} and \ref{eq:EHI-intro} hold if and
	only if there exist a metric $\theta$ on $\ambient$ quasisymmetric to $d$ and
	an $\form$-smooth Radon measure $\nu$ on $\ambient$ with full $\form$-quasi-support
	such that the time-changed MMD space $(\ambient, \theta, \nu, \form^{\nu}, \domain^{\nu})$ satisfies
	\ref{eq:VD-intro} and sub-Gaussian heat kernel estimates. Here the quasisymmetry of $\theta$
	to $d$ means that every annulus in $\theta$ is comparable to one in $d$ in a uniform fashion,
	and the assumed properties of $\nu$ guarantees that an MMD space $(\ambient, \theta, \nu, \form^{\nu}, \domain^{\nu})$
	over $(\ambient, \theta, \nu)$ can be uniquely defined in such a way that
	$\domain^{\nu}\cap\contfunc_{\mathrm{c}}(\ambient)=\domain\cap\contfunc_{\mathrm{c}}(\ambient)$
	and $\form^{\nu}(u,u)=\form(u,u)$ for any $u\in\domain\cap\contfunc_{\mathrm{c}}(\ambient)$.
	In particular, $(\ambient, \theta, \nu, \form^{\nu}, \domain^{\nu})$ shares the same
	harmonic functions, Green functions, harmonic measures, and boundary trace processes
	as $(\ambient, d, \refmeas, \form, \domain)$, and hence studying these objects for
	$(\ambient, d, \refmeas, \form, \domain)$ is equivalent to doing so for
	$(\ambient, \theta, \nu, \form^{\nu}, \domain^{\nu})$. Similarly, for the notion of
	uniform domain, we adopt a particular formulation of it due to \cite{Mur24} which is
	stable under the change of the metric to one quasisymmetric to the original one.
	Therefore by considering $(\ambient, \theta, \nu, \form^{\nu}, \domain^{\nu})$
	instead of $(\ambient, d, \refmeas, \form, \domain)$, we may assume without loss
	of generality that our MMD space $(\ambient, d, \refmeas, \form, \domain)$ satisfies
	\ref{eq:VD-intro} and sub-Gaussian heat kernel estimates.
	
	The last missing piece for our study of boundary traces of reflected diffusions to make sense
	in this setting is the \emph{existence} of a nice reflected diffusion on $\unifdom$.
	This existence in the present generality
	has recently been proved by the second-named author in \cite{Mur24}. More specifically, it was proved that the
	\textbf{reflected Dirichlet space} $(\overline{\unifdom},d,\refmeas|_{\overline{\unifdom}},\formref,\domainref)$
	on $\unifdom$ defined in a standard way is an MMD space satisfying sub-Gaussian
	heat kernel estimates of the same form as $(\ambient, d, \refmeas, \form, \domain)$.
	Since $(\overline{\unifdom},d,\refmeas|_{\overline{\unifdom}})$ is easily seen to satisfy
	\ref{eq:VD-intro} as well, the problems of the validity of the Doob--Na\"im formula
	analogous to \eqref{e:dnf} and of obtaining heat kernel estimates for the resulting jump-type
	Dirichlet form now make perfect sense, and can be studied on the basis of the nice properties of
	the reflected Dirichlet space $(\overline{\unifdom},d,\refmeas|_{\overline{\unifdom}},\formref,\domainref)$
	proved in \cite{Mur24}. It is precisely in this setting that we prove our main results
	\eqref{it:result-hmeas-intro}, \eqref{it:result-dnformula-intro} and \eqref{it:result-shk-trace-intro}
	summarized before the beginning of this Subsection \ref{ssec:overview}.
	
	\begin{remark} \label{rmk:differences-CaoChan-intro}
	In fact, very similar results have been obtained independently in a recent preprint
	\cite{CC26} by Cao and Chen. Their result on the jump kernel of the boundary trace
	Dirichlet form gives only two-sided estimates on it and not its exact identification
	as the Na\"im kernel $\naimker^{\unifdom}_{x_0}(\cdot,\cdot)$. On the other hand,
	their framework is slightly more general than ours, mainly in that they assume the
	validity of CDC \eqref{e:CDC-intro} only for $0< r \lesssim \diam(\partial \unifdom)$
	instead of $0< r \lesssim \diam(\unifdom)$, and they have also proved that this weaker version
	of CDC is equivalent to \ref{eq:VD-intro} of $(\partial \unifdom,d,\hmeas{\unifdom}{x_0})$;
	see \cite[Theorem 6.1 and Remark 6.2]{CC26}.
	It is possible to extend most of our results in Sections \ref{sec:harm-ell-meas} and \ref{sec:boundary-trace}
	under this weaker version of CDC with very few changes to the statements and the proofs,
	but we refrain from doing so in the main text of this paper for the sake of better presentation
	of our results, and just briefly explain the necessary changes to the statements
	and the proofs under this more general setting in Subsection \ref{ss:wCDC}.
	We also provide a more detailed comparison of our results with
	those of \cite{CC26} in Remark \ref{r:compare}.
	\end{remark}
	
	To illustrate the generality of our result on stable-like heat kernel estimates for boundary
	trace processes, we list a few examples of diffusions and domains to which the result applies.
	The most classical ones among such are reflected Brownian motion on Lipschitz and
	more generally NTA domains in $\bR^{\dimeuc}$, which in particular include
	domains with fractal boundaries such as the von Koch snowflake domain.
	More generally, reflected Brownian motion could be replaced with a reflected diffusion
	generated by a uniformly elliptic divergence-form operator as in \cite{Mos} or
	degenerate elliptic operators corresponding to $A_2$-weights as in \cite{FKS}. Another
	class of examples is given by NTA domains in the Heisenberg group equipped
	with the Carnot--Carth\'eodory distance and the diffusion generated by the corresponding
	left-invariant sub-Laplacian satisfying the H\"ormander condition treated in \cite{VSC}
	as mentioned before. Specific examples of NTA domains in this setting are given in \cite{CG,CGN, Gre}.
	Our result on the stable-like heat kernel estimates for boundary trace processes applies
	also to the Brownian motion on the Sierpi\'nski carpet constructed in \cite{BB89}, which
	can be seen from \cite[Theorem 2.9]{Mur24} to be identified as the reflected diffusion
	on the complement of the outer square boundary and on that of the bottom line. These
	sets are indeed uniform domains in the Sierpi\'nski carpet (see \cite[Proposition 4.4]{Lie22}
	and \cite[Proposition 2.4]{CQ}) and easily seen to satisfy the capacity density condition
	with respect to the diffusion, and for them our other results on the two-sided estimates
	of the harmonic measure and on the identification of the boundary trace Dirichlet form
	through the Doob--Na\"im formula are also certainly new.
	
	We conclude this subsection with a description of the relation of our result on the
	Doob--Na\"im formula to the well-established theory of characterizing the trace Dirichlet
	forms of regular symmetric Dirichlet forms in terms of the \emph{Feller measures},
	developed in \cite{FHY,CFY} and \cite[Sections 5.4--5.7]{CF} by following an old idea
	of M.~Fukushima in \cite{Fuk}. The definition of the Feller measure appearing in this
	theory is entirely different from that of the Na\"im kernel as presented in
	\eqref{e:Naim-kernel}. Even though Fukushima \cite{Fuk} proved that they give rise to
	the same jump-type Dirichlet form in the (locally) Euclidean setting, this coincidence
	is not at all obvious from the definitions of the Feller measure and the Na\"im kernel,
	and it is not clear to us to what generality it could be extended. Our proof of the
	Doob--Na\"im formula is in fact completely independent of the theory of the Feller
	measures in \cite{Fuk,FHY,CFY,CF}. It is based on direct calculations of the jump,
	killing and strongly local parts of the trace form, and our argument for the jump
	part  is much simpler than Doob's in \cite{Doo} thanks to the volume doubling property
	\ref{eq:VD-intro} of the harmonic measure $\hmeas{D}{x_0}$ and the continuity of
	the Na\"im kernel $\naimker^{D}_{x_0}(\cdot,\cdot)$ up to the domain boundary
	(\emph{in the usual topology} rather than in the fine topology as considered by Na\"im \cite{Nai}).

	\subsection{Summary of the setting and statement of the main results} \label{ssec:results}
	
	As mentioned in Subsection \ref{ssec:overview}, in most parts of this paper we consider a metric space
	$(\ambient, d)$ which contains at least two elements and whose every bounded closed set is compact,
	a Radon measure $\refmeas$ on $\ambient$ with full support, and a strongly local
	regular symmetric Dirichlet form $(\form, \domain)$ on $L^2(\ambient, \refmeas)$.
	We call $(\ambient, d, \refmeas, \form, \domain)$ a \textbf{metric measure Dirichlet space},
	or an \textbf{MMD space} for short, set $B(x,r) \coloneqq \{y\in \ambient \mid d(x,y)<r\}$,
	$\diam(A) \coloneqq \sup_{x,y\in A}d(x,y)$ ($\sup\emptyset \coloneqq 0$) and
	$\dist(x,A) \coloneqq \inf_{y\in A}d(x,y)$ ($\inf\emptyset \coloneqq \infty$)
	for $x\in\ambient$, $r\in(0,\infty)$ and $A\subset \ambient$, and
	write $\overline{A}$ and $\partial A$ for the closure and boundary,
	respectively, of $A\subset\ambient$ in $\ambient$.
	The strongly continuous contraction semigroup on $L^2(\ambient,\refmeas)$
	associated with $(\form, \domain)$ is denoted by $(T_t)_{t > 0}$.
	We refer to the first and second paragraphs of Subsection \ref{ss:DF} below
	for a brief summary of the definitions adopted here from the theory of
	regular symmetric Dirichlet forms presented in \cite{FOT,CF}. We collect in
	Section \ref{sec:preliminaries} plenty of other relevant definitions and results from
	the potential theory and heat kernel estimates for regular symmetric Dirichlet forms.
	
	To keep the presentation of the main results simple, throughout this subsection we assume
	that $(\ambient, d, \refmeas, \form, \domain)$ satisfies the volume doubling property
	\ref{eq:VD-intro} (Definition \ref{dfn:volume-doubling}) and the heat kernel estimates
	\hyperlink{hke}{$\on{HKE(\scdiff)}$} (Definition \ref{d:HKE})
	\begin{equation}\label{e:hke-intro}
		\frac{c_{3}}{\refmeas\bigl(B(x,\scdiff^{-1}(t))\bigr)} \one_{[0,\delta]}\biggl(\frac{d(x,y)}{\scdiff^{-1}(t)}\biggr)
		\leq p_t(x,y) \leq \frac{c_{1}}{\refmeas\bigl(B(x,\scdiff^{-1}(t))\bigr)} \exp \biggl( -c_{2} t \wt{\scdiff}\biggl( \frac{d(x,y)}{t} \biggr) \biggr)
	\end{equation}
	for $\refmeas$-a.e.\ $x,y \in \ambient$ for each $t \in (0,\infty)$ for some $c_{1},c_{2},c_{3},\delta \in (0,\infty)$.
	Here $\{ p_t \}_{t>0}$ denotes the \emph{heat kernel} of $(\ambient, \refmeas, \form, \domain)$,
	i.e., a family of Borel measurable functions $p_{t} \colon \ambient\times\ambient \to [0,\infty]$
	such that $p_{t}$ is an integral kernel for $T_{t}$ with respect to $\refmeas$ for each $t \in (0,\infty)$,
	$\scdiff$ is a \emph{scale function}, i.e., a homeomorphism from $[0,\infty)$ to itself satisfying
	\eqref{e:reg} for any $r,R \in (0,\infty)$ with $r \leq R$ for some
	$\beta_{1},\beta_{2},C\in(1,\infty)$ with $\beta_{1} \leq \beta_{2}$, and
	$\wt{\scdiff}\colon[0,\infty)\to[0,\infty)$ is defined by \eqref{e:defPhi};
	for example, if $\beta\in(1,\infty)$ and $\scdiff(r)=r^{\beta}$ for any $r\in[0,\infty)$,
	then $\widetilde{\scdiff}(s)=\beta^{-\frac{\beta}{\beta-1}}(\beta-1)s^{\frac{\beta}{\beta-1}}$
	for any $s\in[0,\infty)$. Our main results summarized in \eqref{it:result-hmeas-intro},
	\eqref{it:result-dnformula-intro} and \eqref{it:result-shk-trace-intro} above indeed
	require $(\ambient, d, \refmeas, \form, \domain)$ to satisfy \ref{eq:VD-intro} and
	\hyperlink{hke}{$\on{HKE(\scdiff)}$} as part of their assumptions.
	In this setting, as stated in Proposition \ref{p:feller}, $\ambient$ is connected,
	$(\ambient, \refmeas, \form, \domain)$ is \emph{irreducible}
	(i.e., $\refmeas(A)\refmeas(\ambient \setminus A)=0$ for any Borel subset $A$ of $\ambient$
	that is \emph{$\form$-invariant}, i.e., satisfies $T_{t}(\one_{A}f)=0$ $\refmeas$-a.e.\ on $\ambient \setminus A$
	for any $f\in L^{2}(\ambient,\refmeas)$ and any $t\in(0,\infty)$),
	a (unique) continuous version $p=p_{t}(x,y)\colon(0,\infty)\times\ambient\times\ambient\to[0,\infty)$
	of the heat kernel of $(\ambient, \refmeas, \form, \domain)$ exists, and the following holds:
	the Markovian transition function $P_{t}(x,dy) \coloneqq p_{t}(x,y)\,\refmeas(dy)$ is conservative
	(i.e., $P_{t}(x,\ambient)=1$ for any $(t,x)\in(0,\infty)\times\ambient$) and has the \emph{Feller}
	and \emph{strong Feller properties}, so that there exists a conservative diffusion process
	$\diff = (\Omega, \events, \{\diff_{t}\}_{t\in[0,\infty]},\{\lawdiff_{x}\}_{x \in \ambient})$
	on $\ambient$ such that $\lawdiff_{x}( \diff_{t} \in dy ) = p_{t}(x,y)\,\refmeas(dy)$
	for \emph{any} $(t,x)\in(0,\infty)\times\ambient$.
	
	Let us recall several basic notions from the theory of regular symmetric Dirichlet forms.
	Let $\domain_e$ denote the \emph{extended Dirichlet space} of $(\ambient, \refmeas, \form, \domain)$
	(Definition \ref{d:ExtDiriSp}), i.e., the linear space of $m$-equivalence classes of
	$\refmeas$-a.e.\ pointwise limits $f$ of sequences $\{f_{n}\}_{n \in \mathbb{N}} \subset \domain$
	with $\lim_{k\wedge l\to\infty}\form(f_{k}-f_{l},f_{k}-f_{l})=0$, so that setting
	$\form(f,f) \coloneqq \lim_{n\to\infty}\form(f_{n},f_{n})\in\mathbb{R}$ gives a canonical extension of
	$\form$ to $\domain_e \times \domain_e$ and $\domain=\domain_{e}\cap L^{2}(\ambient,\refmeas)$.
	The Dirichlet form $(\form,\domain)$ is said to be \emph{transient} if $\{f \in \domain_e \mid \form(f,f)=0\}=\{0\}$,
	in which case $(\domain_e,\form)$ is a Hilbert space (see \cite[Theorem 1.6.2]{FOT}).
	For each $f\in\domain_{e}$, let $\wt{f}$ denote any \emph{$\form$-quasi-continuous} $\refmeas$-version of $f$,
	which exists by \cite[Theorem 2.1.7]{FOT} and is unique \emph{$\form$-q.e.}\ (i.e., up to sets of capacity zero)
	by \cite[Lemma 2.1.4]{FOT}; see \cite[Section 2.1]{FOT} and \cite[Sections 1.2, 1.3 and 2.3]{CF} for the
	definition and basic properties of $\form$-quasi-continuous functions with respect to a regular symmetric Dirichlet form.
	
	Let $D$ be a non-empty open subset of $\ambient$, and let $\refmeas|_D$ denote the restriction of
	$\refmeas$ to the Borel $\sigma$-algebra of $D$. The \emph{part process} of $\diff = \{ \diff_{t} \}_{t \ge 0}$
	killed upon exiting $D$ is denoted by $\diff^{D} = \{ \diff^{D}_{t} \}_{t \ge 0}$ (Definition \ref{d:part-form-process}).
	It is an $\refmeas|_D$-symmetric diffusion process on $D$, its Dirichlet form $(\form^D,\domain^0(D))$
	is a strongly local regular symmetric Dirichlet form on $L^{2}(D,\refmeas|_D)$ and identified as
	the \emph{part Dirichlet form} of $(\form,\domain)$ on $D$ given by
	\begin{equation}\label{eq:partDF-intro}
		\domain^0(D)=\{f\in\domain\mid\textrm{$\wt{f}=0$ $\form$-q.e.\ on $\ambient\setminus D$}\}
		\qquad\textrm{and}\qquad
		\form^D=\form|_{\domain^0(D)\times\domain^0(D)},
	\end{equation}
	and the extended Dirichlet space $\domain^{0}(D)_{e}$ of
	$(D,\refmeas|_D,\form^D,\domain^0(D))$ is identified similarly as
	\begin{equation}\label{eq:partDF-extended-intro}
		\domain^{0}(D)_{e}=\{f\in\domain_e\mid\textrm{$\wt{f}=0$ $\form$-q.e.\ on $\ambient\setminus D$}\}.
	\end{equation}
	As stated in Proposition \ref{p:feller}-\eqref{it:HKE-part-CHK-sFeller}, it follows from \ref{eq:VD-intro}
	and \hyperlink{hke}{$\on{HKE(\scdiff)}$} for $(\ambient, d, \refmeas, \form, \domain)$
	and the Feller and strong Feller properties of $\diff = \{ \diff_{t} \}_{t \ge 0}$
	that $\diff^{D} = \{ \diff^{D}_{t} \}_{t \ge 0}$ has the strong Feller property and
	a continuous heat kernel $p^D=p^D_t(x,y)\colon D\times D\to[0,\infty)$ and satisfies
	$\lawdiff_{x}( \diff^{D}_{t} \in dy ) = p^{D}_{t}(x,y)\,\refmeas|_D(dy)$ for \emph{any} $(t,x)\in(0,\infty)\times D$.
	Furthermore if in addition $(\form^D,\domain^0(D))$ is transient, then the \emph{($0$-order) capacity}
	$\Capa_D(A)$ of $A \subset D$ in $D$ is defined by \eqref{e:defCapD}, and the \textbf{Green function}
	$\gren{D}\colon D\times D\to[0,\infty]$ of $(\form,\domain)$ on $D$ is defined by
	\begin{equation} \label{e:green-dfn-intro}
		\gren{D}(x,y) \coloneqq \int_{0}^{\infty} p^D_t(x,y)\,dt
	\end{equation}
	and satisfies Proposition \ref{p:goodgreen}-\eqref{it:goodgreen-sym},\eqref{it:goodgreen-cont},\eqref{it:goodgreen-odf},\eqref{it:goodgreen-excessive},\eqref{it:goodgreen-harm},\eqref{it:goodgreen-max}
	with $\mathcal{N} =\emptyset$ by Lemma \ref{l:green}. Note that $(\form^D,\domain^0(D))$ is transient if $\diam(D)<\diam(\ambient)$,
	and in particular if $D=B(x,r)$ for some $(x,r) \in \ambient \times (0,\diam(\ambient)/2)$,
	by the irreducibility of $(\ambient,\refmeas,\form,\domain)$ from 
	Proposition \ref{p:feller}-\eqref{it:HKE-conn-irr-cons} and \cite[Proposition 2.1]{BCM}.
	
	In the rest of this subsection, we fix a \textbf{uniform domain} $\unifdom$ in $(\ambient,d)$ (Definition \ref{d:uniform}),
	i.e., a non-empty open subset $\unifdom$ of $\ambient$ with $\unifdom\not=\ambient$
	such that for some $c_{\unifdom}\in(0,1)$ and $C_{\unifdom}\in(1,\infty)$ the following holds:
	for every $x,y\in\unifdom$ there exists a continuous map $\gamma\colon[0,1]\to\unifdom$
	with $\gamma(0)=x$ and $\gamma(1)=y$ such that $\diam(\gamma([0,1]))\leq C_{\unifdom}d(x,y)$ and
	\begin{equation}\label{eq:uniform-intro}
		\delta_{\unifdom}(\gamma(t))
		\geq c_{\unifdom}
		\min\{d(x,\gamma(t)),d(y,\gamma(t))\}
		\quad\textrm{for any $t\in[0,1]$,}
	\end{equation}
	where
	\begin{equation}\label{eq:deltaU-intro}
	\delta_{\unifdom}(z)
	 \coloneqq \dist(z,\ambient\setminus\unifdom),
	\quad z\in\unifdom.
	\end{equation}
	This formulation of the notion of uniform domain is much less restrictive than that of
	\textbf{length uniform domain}, the usual one in the literature, which requires instead
	the last two inequalities with $\diam(\gamma([0,1])), d(x,\gamma(t)), d(y,\gamma(t))$
	replaced by the lengths of $\gamma,\gamma|_{[0,t]},\gamma|_{[t,1]}$ in $(\ambient,d)$, respectively.
	An advantage of the present formulation is that it is stable under the change of
	the metric to one quasisymmetric to $d$. An immediate but useful consequence of
	\eqref{eq:uniform-intro} is that for any $\xi \in \partial \unifdom$ and any
	$r \in (0,\diam(\unifdom)/4)$ there exists $\xi_r \in \unifdom$ such that
	\begin{equation}\label{eq:xir-intro}
		d(\xi,\xi_r)=r \quad \textrm{and} \quad \delta_\unifdom(\xi_r) > \frac{1}{2}c_{\unifdom} r
	\end{equation}
	(Lemma \ref{l:xir}). Throughout this paper, $\xi_{r}$ always denotes an arbitrary element of $\unifdom$
	satisfying \eqref{eq:xir-intro} for each given $(\xi,r) \in \partial \unifdom \times (0,\diam(\unifdom)/4)$.
	
	The most important feature of uniform domains is that they have been proved to satisfy the
	\hypertarget{bhp-intro}{\textbf{(scale-invariant) boundary Harnack principle} (BHP)} (Definition \ref{d:bhp} and Theorem \ref{t:bhp}).
	Namely, there exist $A_{0},A_{1},C_{1}\in(1,\infty)$ such that for any $\xi \in \partial \unifdom$,
	any $r\in(0,\diam(U)/A_{1})$ and any non-negative \emph{$\form$-harmonic functions $u,v$
		on $\unifdom\cap B(\xi,A_{0}r)$ with Dirichlet boundary condition relative to $\unifdom$}
	(Definitions \ref{dfn:harmonic} and \ref{d:dbdry}) such that $v>0$ $\refmeas$-a.e.\ on $\unifdom \cap B(\xi,r)$,
	\begin{equation}\label{eq:BHP-intro}
		\esssup_{x\in \unifdom \cap B(\xi,r)}\frac{u(x)}{v(x)}\leq C_{1}\essinf_{x\in \unifdom \cap B(\xi,r)}\frac{u(x)}{v(x)}.
	\end{equation}
	For length uniform domains in MMD spaces, \hyperlink{bhp-intro}{BHP} was proved first
	by Lierl \cite{Lie} under the assumption of \ref{eq:VD-intro} and \hyperlink{hke}{$\on{HKE(\scdiff)}$},
	and then by Barlow and the second-named author \cite{BM19} under the assumption of \ref{eq:EHI-intro}
	and some mild technical conditions. \hyperlink{bhp-intro}{BHP} for uniform domains in MMD spaces
	(Theorem \ref{t:bhp}) has been proved in a recent work \cite{Che} by Aobo Chen. 
	%under the weaker assumption that the metric doubling property \hyperlink{md-intro}{MD} and \ref{eq:EHI-intro} hold.
	
	As an important consequence of \hyperlink{bhp-intro}{BHP}, we have the local H\"{o}lder
	continuity of the ratios of $(0,\infty)$-valued $\form$-harmonic functions with Dirichlet boundary
	condition relative to $\unifdom$ (Lemma \ref{l:bhpholder}), which is an analogue of Moser's
	\ref{eq:EHI-intro}-based oscillation lemma \cite[\textsection 5]{Mos}. This fact leads to
	our first observation on the existence of a continuous extension of the Na\"im kernel to
	$\ol{\unifdom}$ stated in the following proposition.
	%(, which we will prove in Proposition \ref{p:naim} under the weaker assumption that \hyperlink{md-intro}{MD} and \ref{eq:EHI-intro} hold).
	We remark that, if the part Dirichlet form $(\form^{\unifdom},\domain^{0}(\unifdom))$ on $\unifdom$ is transient,
	then for each $x_{0} \in \unifdom$, the Green function $\gren{\unifdom}(x_{0},\cdot)\colon \unifdom \setminus \{x_{0}\} \to [0,\infty)$
	is continuous by Proposition \ref{p:goodgreen}-\eqref{it:goodgreen-cont}, $(0,\infty)$-valued by Lemma \ref{l:green}
	and the connectedness of $\unifdom$, and an $\form$-harmonic function on $\unifdom \setminus \{x_{0}\}$
	with Dirichlet boundary condition relative to $\unifdom$ by Proposition \ref{p:goodgreen}-\eqref{it:goodgreen-harm}
	and Lemma \ref{l:dbdy}, so that \hyperlink{bhp-intro}{BHP} is indeed applicable to $\gren{\unifdom}(x_{0},\cdot)$.
	For a set $A$, we define $A_{\diag} \coloneqq \{(x,x)\mid x\in A\}$ and
	$\offdiag{A} \coloneqq (A\times A)\setminus A_{\diag}$ (``od'' stands for ``off-diagonal'').

	\begin{prop}[Part of Proposition \ref{p:naim}]\label{prop:Naim-intro}
		Assume that the part Dirichlet form $(\form^{\unifdom},\domain^{0}(\unifdom))$ on $\unifdom$ is transient.
		Then for each $x_{0} \in \unifdom$, there exists a unique continuous function
		$\naimker^{\unifdom}_{x_{0}}\colon\offdiagp{\overline{\unifdom}\setminus\{x_{0}\}}\to(0,\infty)$,
		called the \emph{\textbf{Na\"im kernel} of $\unifdom$ with base point $x_{0}$}, such that
		\begin{equation} \label{eq:Naim-kernel-intro}
			\naimker^{\unifdom}_{x_{0}}(x,y)
			=\frac{\gren{\unifdom}(x,y)}{\gren{\unifdom}(x_0,x)\gren{\unifdom}(x_0,y)}
			\qquad\textrm{for any $(x,y) \in \offdiagp{\unifdom\setminus\{x_{0}\}}$.}
		\end{equation}
		Moreover, there exist $c_{0},C_{1}\in(0,\infty)$ such that for any $x_{0}\in\unifdom$ and any $(\xi,\eta)\in\offdiagp{\partial\unifdom}$,
		with $r \coloneqq r_{x_{0},\xi,\eta} \coloneqq c_{0}\min\{d(x_0,\eta),d(x_0,\xi),d(\eta,\xi)\}$,
		\begin{equation} \label{eq:Naim-kernel-estimates-intro}
			C_{1}^{-1} \frac{\gren{\unifdom}(\eta_{r},\xi_{r})}{\gren{\unifdom}(x_{0},\eta_{r}) \gren{\unifdom}(x_{0},\xi_{r})}
			\leq \naimker^{\unifdom}_{x_{0}}(\eta,\xi)
			\leq C_{1} \frac{\gren{\unifdom}(\eta_{r},\xi_{r})}{\gren{\unifdom}(x_{0},\eta_{r}) \gren{\unifdom}(x_{0},\xi_{r})}.
		\end{equation}
	\end{prop}
	
	In the rest of this subsection (and throughout Sections \ref{sec:harm-ell-meas} and \ref{sec:boundary-trace} below),
	we assume that the uniform domain $\unifdom$ satisfies the \textbf{capacity density condition} (CDC)
	(Definition \ref{d:cdc}), i.e., that there exist $A_{0} \in (8K,\infty)$ and $A_{1}, C \in (1,\infty)$
	such that for any $\xi \in \partial \unifdom$ and any $R \in (0,\diam(U)/A_{1})$,
	\begin{equation} \tag*{$\on{CDC}$} \label{eq:CDC-intro}
		\Capa_{B(\xi,A_{0} R)}(B(\xi,R)) \leq C \Capa_{B(\xi,A_{0} R)}(B(\xi,R)\setminus \unifdom).
	\end{equation}
	Here $K \in (1,\infty)$ is chosen so that $(\ambient,d)$ is \textbf{$K$-relatively ball connected}
	(Definition \ref{d:chain-rbc}-\eqref{it:rbc}); the existence of such $K$ follows from \ref{eq:VD-intro} and
	\hyperlink{hke}{$\on{HKE(\scdiff)}$} (see Remark \ref{r:ehi-hke} and Lemma \ref{l:chain}-\eqref{it:EHI-MD-RBC}),
	and allows us to apply \ref{eq:EHI-intro} in a nicely controlled manner and in particular to
	extend \ref{eq:CDC-intro} from one $A_{0} \in (8K,\infty)$ to \emph{any} $A_{0} \in (1,\infty)$
	(with different $A_{1},C$ for each $A_{0}$) (Lemma \ref{l:cdc1}-\eqref{it:cdc1-anyA0}).
	As already mentioned in Subsection \ref{ssec:overview}, \ref{eq:CDC-intro} is known to
	guarantee good quantitative behavior of the harmonic measure in the case of uniform domains
	in $\mathbb{R}^{\dimeuc}$ as proved by Aikawa and Hirata in \cite[Lemmas 3.5 and 3.6]{AH},
	which generalized earlier results by Dahlberg \cite[Lemma 1]{Dah} for Lipschitz domains and
	Jerison and Kenig \cite[Lemma 4.8]{JK} for NTA domains. As the first main theorem of this paper, we extend
	\cite[Lemmas 3.5 and 3.6]{AH} to our present general setting by proving the following theorem in Subsection \ref{ssec:hmeas}.
	Note that, since $\ambient$ is connected and $(\ambient,\refmeas,\form,\domain)$ is irreducible
	by Proposition \ref{p:feller}-\eqref{it:HKE-conn-irr-cons}, we have
	$\partial \unifdom \not= \emptyset$ by $\emptyset\not=\unifdom\not=\ambient$,
	$\ambient \setminus \unifdom$ has positive capacity with respect to $(\form,\domain)$
	by \ref{eq:CDC-intro} and \cite[Theorem 4.4.3-(ii)]{FOT}, and hence
	$(\form^{\unifdom},\domain^{0}(\unifdom))$ is transient by \cite[Proposition 2.1]{BCM}.
	
	\begin{theorem}[Theorem \ref{t:hmeas} and Corollary \ref{c:asdouble}]\label{thm:hmeas-intro}
		Define the \emph{\textbf{$\form$-harmonic measure} $\hmeas{\unifdom}{x_{0}}$ of $\unifdom$ with base point $x_{0}\in\unifdom$} (Definition \ref{d:hmeas})
		by $\hmeas{\unifdom}{x_{0}}(A) \coloneqq \lawdiff_{x_{0}}( \textrm{$\diff_{ \tau_{\unifdom} } \in A$, $\tau_{\unifdom}<\infty$} )$
		for each Borel subset $A$ of $\ambient$, where $\tau_{\unifdom} \coloneqq \inf\{t\in[0,\infty)\mid \diff_{t}\not\in\unifdom\}$ ($\inf\emptyset \coloneqq \infty$).
		Then there exist $C,A \in (1,\infty)$ such that for any $\xi \in \partial \unifdom$, any $x_{0} \in \unifdom$ and any $r \in (0,d(\xi,x_{0})/A)$,
		\begin{align} \label{eq:hmeas-estimate-intro}
			C^{-1} \gren{\unifdom}(x_{0}, \xi_{r}) \Capa_{B(\xi,2r)}(B(\xi,r))
			&\leq \hmeas{\unifdom}{x_{0}}(B(\xi,r) \cap \partial \unifdom)
			\leq C \gren{\unifdom}(x_{0}, \xi_{r}) \Capa_{B(\xi,2r)}(B(\xi,r)), \\
			\hmeas{\unifdom}{x_{0}}(B(\xi, r) \cap \partial \unifdom) &\leq C \hmeas{\unifdom}{x_{0}}(B(\xi,r/2) \cap \partial \unifdom).
			\label{eq:hmeas-doubling-intro}
		\end{align}
		In particular, the topological support $\supp_{\ambient}[\hmeas{\unifdom}{x_{0}}]$
		of $\hmeas{\unifdom}{x_{0}}$ in $\ambient$ is $\partial \unifdom$.
	\end{theorem}
	
	While our proof of the lower bound in \eqref{eq:hmeas-estimate-intro} follows the same line of reasoning as \cite{AH},
	for the upper bound in \eqref{eq:hmeas-estimate-intro} we give a new proof avoiding the delicate iteration argument
	(the so-called box argument) in \cite{AH}. Then \eqref{eq:hmeas-doubling-intro} follows by combining
	\eqref{eq:hmeas-estimate-intro} with Lemma \ref{l:hchain} (implied by \ref{eq:EHI-intro} and $\unifdom$
	being a uniform domain), Remark \ref{r:ehi-hke} and \cite[Lemma 5.23]{BCM}.
	
	Note that \eqref{eq:hmeas-doubling-intro} means the validity of the volume doubling property
	of $\hmeas{\unifdom}{x_{0}}$ only up to the scale of $\dist(x_{0},\partial \unifdom)$,
	which still gives \ref{eq:VD-intro} of $(\partial \unifdom,d,\hmeas{\unifdom}{x_{0}})$
	when $\unifdom$ is bounded (i.e., $\diam(\unifdom)<\infty$) but may \emph{not} when $\unifdom$
	is unbounded (i.e., $\diam(\unifdom)=\infty$). Since, as mentioned slightly before
	\eqref{e:jump-kernel-intro}, the general results on heat kernel estimates for
	jump-type Dirichlet forms in \cite{CKW, GHH-lb, GHH-mv} require the global version
	\ref{eq:VD-intro} of the volume doubling property of the reference measure, the
	$\form$-harmonic measure $\hmeas{\unifdom}{x_{0}}$ is not a good candidate for our choice of
	the reference measure for the boundary trace Dirichlet form when $\unifdom$ is unbounded.
	In fact, as stated in the next proposition and proved in Subsection \ref{ssec:elliptic-meas},
	in this case one can construct a canonical Radon measure on $\partial \unifdom$ which is
	mutually absolutely continuous with respect to $\hmeas{\unifdom}{x_{0}}$ and satisfies
	\ref{eq:VD-intro}, by utilizing \hyperlink{bhp-intro}{BHP} to take the limit of a suitably
	normalized version of $\hmeas{\unifdom}{x_{0}}$ as $x_{0}$ tends to infinity.
	The consideration of such a measure dates back to Kenig and Toro \cite[Corollary 3.2]{KT},
	who first studied it for NTA domains in $\mathbb{R}^{\dimeuc}$, and we call such a measure
	on an unbounded uniform domain the \textbf{$\form$-elliptic measure at infinity} of the domain,
	following \cite[Lemma 3.5]{BTZ}. Assuming that $\unifdom$ is unbounded, for each $x_{0}\in\unifdom$
	let $\hprof{x_{0}}$ denote the \textbf{$\form$-harmonic profile} of $\unifdom$
	with base point $x_{0}$, i.e., a $(0,\infty)$-valued continuous $\form$-harmonic function
	on $\unifdom$ with Dirichlet boundary condition relative to $\unifdom$ such that
	$\hprof{x_{0}}(x_{0})=1$, whose existence (Proposition \ref{p:hprofile})
	and uniqueness (Lemma \ref{l:uniqueprofile}) are well-known consequences of \hyperlink{bhp-intro}{BHP}.
	
	\begin{prop}[Part of Proposition \ref{p:emeas}] \label{prop:elliptic-meas-intro}
		Assume that $\unifdom$ is unbounded, and let $x_{0} \in \unifdom$. Then there exists
		a unique Radon measure $\nu^{\unifdom}_{x_0}$ on $\overline{\unifdom}$, called the
		\emph{\textbf{$\form$-elliptic measure at infinity} of $\unifdom$ with base point $x_{0}$},
		such that $\gren{\unifdom}(x_{0},x_{n})^{-1}\hmeas{\unifdom}{x_{n}}\big|_{\overline{\unifdom}}$
		converges in total variation on any compact subset of $\overline{\unifdom}$ to $\nu^{\unifdom}_{x_0}$ as $n\to\infty$
		for any $\{ x_{n} \}_{n\in\mathbb{N}}\subset \unifdom \setminus \{x_{0}\}$ with $\lim_{n\to\infty}d(x_{0},x_{n})=\infty$.
		Moreover, $\nu^{\unifdom}_{x_0}(\unifdom)=0$, $\nu^{\unifdom}_{y}=(\hprof{x_{0}}(y))^{-1}\nu^{\unifdom}_{x_0}$
		for any $y \in \unifdom$, and the following hold:
		\begin{enumerate}[\rm(a)]\setlength{\itemsep}{0pt}\vspace{-5pt}
			\item\label{it:elliptic-meas-density-intro} $\nu^\unifdom_{x_{0}}$ and
			$\hmeas{\unifdom}{x_{0}}\big|_{\overline{\unifdom}}$ are mutually absolutely continuous,
			a $(0,\infty)$-valued continuous version of the Radon--Nikodym derivative $d\nu^\unifdom_{x_{0}}/d \hmeas{\unifdom}{x_{0}}$
			on $\partial \unifdom$ exists, and there exist $C,A \in (1,\infty)$ independent of $x_{0}$ such that
			for any $\xi \in \partial U$, any $R \in (0,d(\xi,x_0)/A)$ and any $\eta \in B(\xi,R) \cap \partial U$,
			\begin{equation} \label{eq:elliptic-meas-density-bounds-intro}
				C^{-1} \frac{\hprof{x_{0}}(\xi_{R})}{\gren{\unifdom}(x_{0},\xi_{R})}
				\leq \frac{d\nu^\unifdom_{x_0}}{d\hmeas{\unifdom}{x_0}}(\eta)
				\leq C \frac{\hprof{x_{0}}(\xi_{R})}{\gren{\unifdom}(x_0,\xi_{R})}.
			\end{equation}
			\item\label{it:elliptic-meas-bounds-intro} There exists $C \in (0,\infty)$ independent
			of $x_{0}$ such that for any $\xi \in \partial \unifdom$ and any $R \in (0,\infty)$,
			\begin{equation} \label{eq:elliptic-meas-bounds-intro}
				C^{-1} \hprof{x_0}(\xi_{R}) \Capa_{B(\xi,2R)}(B(\xi,R))
				\leq \nu^\unifdom_{x_0}(B(\xi,R) \cap \partial \unifdom)
				\leq C \hprof{x_0}(\xi_{R}) \Capa_{B(\xi,2R)}(B(\xi,R)).
			\end{equation}
			In particular, $\supp_{\overline{\unifdom}}[\nu^\unifdom_{x_0}] = \partial \unifdom$
			and $(\partial \unifdom, d, \nu^\unifdom_{x_{0}})$ satisfies \ref{eq:VD-intro}.
		\end{enumerate}
	\end{prop}
	
	Lastly, we introduce the reflected Dirichlet form $(\formref,\domainref)$ on $\unifdom$
	and its trace Dirichlet form $(\formtr,\domaintr)$ to $\partial \unifdom$, and state our
	version of the Doob--Na\"im formula expressing $\formtr$ in terms of the Na\"im kernel
	$\naimker^{\unifdom}_{x_{0}}$ (Theorem \ref{t:dnformula}) and stable-like heat kernel estimates for
	$(\formtr,\domaintr)$ (Theorem \ref{thm:shk-trace}). First, we define the \textbf{reflected Dirichlet form}
	$(\formref,\domainref)$ of $(\form,\domain)$ on $\unifdom$ (Definition \ref{dfn:reflected-form}) by
	\begin{align} \label{eq:domainref-intro}
		\domainref & \coloneqq \biggl\{f \in \domain_{\loc}(\unifdom) \biggm| \int_{\unifdom} f^{2}\,dm + \int_{\unifdom} \one_{\unifdom} \, d\Gamma_{\unifdom}(f,f) < \infty \biggr\}, \\
		\formref(f,g) & \coloneqq \frac{1}{4}\biggl(\int_{\unifdom} \one_{\unifdom} \, d\Gamma_{\unifdom}(f+g,f+g) - \int_{\unifdom} \one_{\unifdom} \, d\Gamma_{\unifdom}(f-g,f-g)\biggr), \quad f,g \in \domainref,
		\label{eq:formref-intro}
	\end{align}
	where $\domain_{\loc}(\unifdom)$ denotes the \emph{space of functions on $\unifdom$ locally in $\domain$}
	(\eqref{e:Floc} in Definition \ref{d:domain-local}) and $\Gamma_{\unifdom}(f,f)$
	the \emph{$\form$-energy measure} of $f \in \domain_{\loc}(\unifdom)$
	(Definitions \ref{d:EnergyMeas} and \ref{d:domain-local}). Thanks to the assumption that $\unifdom$ is a
	uniform domain in $(\ambient,d)$, it turns out that $(\ol{U},d,\refmeas|_{\unifdom},\formref,\domainref)$
	is an MMD space satisfying \ref{eq:VD-intro} and \hyperlink{hke}{$\on{HKE(\scdiff)}$} 
	(\cite[Theorem 2.8]{Mur24}; Theorem \ref{thm:hkeunif}-\eqref{it:hkeunif}), and that the $1$-capacity (see \eqref{e:defCap1})
	of a subset $A$ of $\ol{\unifdom}$ with respect to $(\ol{U},\refmeas|_{\unifdom},\formref,\domainref)$
	is comparable to the $1$-capacity of $A$ with respect to $(\ambient, \refmeas, \form, \domain)$
	(\cite[Proposition 5.11-(i)]{Mur24}; Theorem \ref{thm:hkeunif}-\eqref{it:formref-cap-equiv}).
	Furthermore the sets $\bigl\{ \widetilde{u}^{\on{ref}} \bigm| u \in \domain(\unifdom) \bigr\}$
	and $\bigl\{ \widetilde{u}^{\on{ref}} \bigm| u \in \domain(\unifdom)_{e} \bigr\}$
	of $\formref$-quasi-continuous $\refmeas|_{\overline{\unifdom}}$-versions $\widetilde{u}^{\on{ref}}$
	of $u \in \domainref$ and of $u \in \domainref_{e}$ coincide with
	$\bigl\{ \widetilde{u}|_{\overline{\unifdom}} \bigm| u \in \domain \bigr\}$ and
	$\bigl\{ \widetilde{u}|_{\overline{\unifdom}} \bigm| u \in \domain_{e} \bigr\}$,
	respectively, with any two functions defined $\form$-q.e.\ on $\overline{\unifdom}$
	and equal $\form$-q.e.\ on $\overline{\unifdom}$ identified
	(\cite[Proposition 5.11-(iii)]{Mur24} and Theorem \ref{thm:hkeunif}-\eqref{it:formref-quasi-cont-equiv}).
	In particular, a \textbf{reflected diffusion} on $\unifdom$, a diffusion process
	$\diffref = ( \{ \diffref_{t} \}_{t \ge 0}, \{ \lawref_{x} \}_{x \in \ol{\unifdom}})$ on $\ol{\unifdom}$
	satisfying $\lawref_{x}( \diffref_{t} \in dy ) = \hkref_{t}(x,y)\,\refmeas|_{\unifdom}(dy)$
	for any $(t,x) \in (0,\infty) \times \ol{\unifdom}$ for the continuous heat kernel
	$\hkref=\hkref_{t}(x,y)$ of $(\ol{U},\refmeas|_{\unifdom},\formref,\domainref)$, exists by
	\ref{eq:VD-intro}, \hyperlink{hke}{$\on{HKE(\scdiff)}$} and Proposition \ref{p:feller}-\eqref{it:HKE-CHK},\eqref{it:HKE-Feller},
	and defines exactly the same harmonic measure $\hmeas{\unifdom}{x_{0}}$
	as the diffusion $\diff = ( \{ \diff_{t} \}_{t \ge 0}, \{ \lawdiff_{x} \}_{x \in \ambient} )$ on $\ambient$
	for any $x_{0} \in \unifdom$ by Lemma \ref{l:harmonicm}-\eqref{it:hmeas-quasi-support}.
	Moreover, with respect to $(\ol{U},\refmeas|_{\unifdom},\formref,\domainref)$,
	$\hmeas{\unifdom}{x_{0}}$ clearly charges no set of zero capacity (Lemma \ref{l:harmonicm}-\eqref{it:hmeas-smooth-support}),
	in particular $\partial \unifdom$ has positive capacity by $\hmeas{\unifdom}{x_{0}}(\partial \unifdom) > 0$ from \eqref{eq:hmeas-estimate-intro},
	$\partial \unifdom$ is an $\formref$-quasi-support of $\hmeas{\unifdom}{x_{0}}\big|_{\overline{\unifdom}}$ by \ref{eq:EHI-intro} and \cite[Exercise 4.6.1]{FOT}
	(Definition \ref{d:quasisupport} and Lemma \ref{l:harmonicm}-\eqref{it:hmeas-quasi-support}),
	and all these hold also for the $\form$-elliptic measure $\nu^{\unifdom}_{x_{0}}$ at infinity by
	Proposition \ref{prop:elliptic-meas-intro}-\eqref{it:elliptic-meas-density-intro} when $\unifdom$ is unbounded.
	Setting $x_{0} \coloneqq \wh{\xi}_{\diam(\unifdom)/5}$ and $\bdrymeas \coloneqq \hmeas{\unifdom}{x_{0}}$
	for arbitrarily chosen $\wh{\xi} \in \partial \unifdom$ when $\unifdom$ is bounded,
	and $\bdrymeas \coloneqq \nu^{\unifdom}_{x_{0}}$ for arbitrarily chosen $x_{0} \in \unifdom$
	when $\unifdom$ is unbounded, we can now apply the general theory of traces of regular
	symmetric Dirichlet forms in \cite[Corollary 5.2.10]{CF} and obtain a regular symmetric
	Dirichlet form $(\formtr,\domaintr)$ on $L^{2}(\partial \unifdom,\bdrymeas)$, called the
	\textbf{trace Dirichlet form} of $(\formref,\domainref)$ on $L^{2}(\partial \unifdom,\bdrymeas)$,
	defined by
	\begin{gather}\label{eq:trace-domain-intro}
		\widetilde{\domainref_{e}}|_{\partial \unifdom} \coloneqq \bigl\{ \widetilde{f}|_{\partial \unifdom} \bigm| f \in \domainref_{e} \bigr\}, \qquad
		\domaintr \coloneqq \widetilde{\domainref_{e}}|_{\partial \unifdom} \cap L^{2}(\partial \unifdom,\bdrymeas), \\
		\formtr(u,u) \coloneqq \formref(H^{\on{ref}}_{\partial \unifdom}u,H^{\on{ref}}_{\partial \unifdom}u),
		\quad u \in \widetilde{\domainref_{e}}|_{\partial \unifdom}.
		\label{eq:trace-form-intro}
	\end{gather}
	Here $\widetilde{f}$ denotes any $\formref$-quasi-continuous $\refmeas|_{\overline{\unifdom}}$-version of $f \in \domainref_{e}$,
	and any two functions equal $\formref$-q.e.\ on $\partial \unifdom$ are identified; since
	two $\formref$-quasi-continuous functions on $\ol{\unifdom}$ are equal $\formref$-q.e.\ on $\partial \unifdom$
	if and only if they are equal $\bdrymeas$-q.e.\ on $\partial \unifdom$ by $\partial \unifdom$
	being an $\formref$-quasi-support of $\bdrymeas$ and \cite[Theorem 3.3.5]{CF}, we can canonically consider
	$\widetilde{\domainref_{e}}|_{\partial \unifdom}$ as a linear space of $\mu$-equivalence classes of
	$\mathbb{R}$-valued Borel measurable functions on $\partial \unifdom$. Then for $u \in \widetilde{\domainref_{e}}|_{\partial \unifdom}$,
	$H^{\on{ref}}_{\partial \unifdom}u$ denotes the function defined $\formref$-q.e.\ on $\ol{\unifdom}$ by
	$(H^{\on{ref}}_{\partial \unifdom}u)(x) \coloneqq \expref_{x}\bigl[u(\diffref_{\tau_{\unifdom}})\one_{\{\tau_{\unifdom}<\infty\}}\bigr]$
	(Definition \ref{d:hmeas}), so that $H^{\on{ref}}_{\partial \unifdom}u\in\domainref_{e}$ by \cite[Theorem 3.4.8]{CF} and
	$\formtr(u,u)$ can be defined by \eqref{eq:trace-form-intro}, and any $u\in\widetilde{\domainref_{e}}|_{\partial \unifdom}$
	is $\formtr$-quasi-continuous on $\partial \unifdom$ by \cite[Theorem 5.2.6]{CF}.
	Also by \cite[Theorem 5.2.15]{CF}, the extended Dirichlet space $\domaintr_{e}$ of
	$(\partial \unifdom,\bdrymeas,\formtr,\domaintr)$ is identified as $\domaintr_{e}=\widetilde{\domainref_{e}}|_{\partial \unifdom}$
	and the canonical extension of $\formtr|_{\domaintr\times\domaintr}$ to $\domaintr_{e}$
	coincides with \eqref{eq:trace-form-intro}.
	
	Now we can state our version of the Doob--Na\"im formula, which expresses $\formtr$ in terms of the
	Na\"im kernel $\naimker^{\unifdom}_{x_{0}}$ introduced in Proposition \ref{prop:Naim-intro}, as follows.
	
	\begin{theorem}[Doob--Na\"im formula; Proposition \ref{prop:bdry-trace-pure-jump} and Theorem \ref{t:dnformula}] \label{thm:Doob-Naim-intro}
		For any $u \in \domaintr_{e}$,
		\begin{equation} \label{eq:Doob-Naim-intro}
			\formtr(u,u) =  \frac{1}{2} \int_{\offdiagp{\partial \unifdom}} (u(\xi)-u(\eta))^{2}\, \naimker^{\unifdom}_{x_{0}}(\xi,\eta) \,d\hmeas{\unifdom}{x_{0}}(\xi) \,d\hmeas{\unifdom}{x_{0}}(\eta).
		\end{equation}
		In particular, the trace Dirichlet form $(\formtr,\domaintr)$ on $L^{2}(\partial \unifdom,\bdrymeas)$ is of pure jump type.
	\end{theorem}
	
	Recall that $(\formtr,\domaintr)$ can be written as the sum of its strongly local,
	jump and killing parts by \cite[Theorem 4.5.2]{FOT} (see \eqref{e:Beurling-Deny}), so that
	Theorem \ref{thm:Doob-Naim-intro} can be rephrased as the identification of its jumping measure as
	$\naimker^{\unifdom}_{x_{0}}(\xi,\eta) \,d\hmeas{\unifdom}{x_{0}}(\xi) \,d\hmeas{\unifdom}{x_{0}}(\eta)$
	combined with the vanishing of its strongly local and killing parts. The latter claim is
	a simple consequence of the known characterization of these parts of trace Dirichlet forms
	in \cite[Theorems 5.6.2 and 5.6.3]{CF}, but we give alternative elementary arguments
	for each of these parts in Propositions \ref{prop:trace-slocal} and \ref{p:killing}, respectively.
	The former claim is the more interesting, and we prove it by an explicit evaluation of the jumping measure
	based on the continuity of the Na\"im kernel $\naimker^{\unifdom}_{x_{0}}$ from Proposition \ref{prop:Naim-intro}
	and the volume doubling property \eqref{eq:hmeas-doubling-intro} of the $\form$-harmonic measure $\hmeas{\unifdom}{x_{0}}$.
	
	We conclude this subsection with stating our third main theorem on stable-like heat kernel estimates
	for the trace Dirichlet form $(\formtr,\domaintr)$ in Theorem \ref{thm:shk-trace-intro} below.
	A key observation for its statement is the following fact implied by \ref{eq:EHI-intro},
	\hyperlink{bhp-intro}{BHP} and $\unifdom$ being a uniform domain in $(\ambient,d)$ (Lemma \ref{l:scale}):
	there exists $\scjump\colon \partial \unifdom \times  [0,\infty) \to [0,\infty)$ such that
	$\scjump(\xi,\cdot) \colon [0,\infty) \to [0,\infty)$ is a homeomorphism for any $\xi\in\partial \unifdom$,
	\eqref{e:Psireg1} holds for any $x,y \in \partial \unifdom$ and any $s,r\in(0,\infty)$ with $s\leq r\leq \diam(\unifdom)$
	(see Definition \ref{d:regscale}), and
	\begin{equation} \label{eq:scjump-intro}
		C_{2}^{-1} \wt{\scjump}(\xi,r) \leq \scjump(\xi,r) \leq C_{2} \wt{\scjump}(\xi,r) \quad \textrm{for any $\xi \in \partial \unifdom$ and any $r \in (0,\diam(\unifdom)/A_{1})$}
	\end{equation}
	for some $C_{2},A_{1}\in(1,\infty)$, where
	\begin{equation} \label{eq:pre-scjump-intro}
		\wt{\scjump}(\xi,r) \coloneqq 
		\begin{cases}
			\gren{\unifdom}(x_{0},\xi_r) & \textrm{if $\unifdom$ is bounded,} \\
			\hprof{x_{0}}(\xi_{r}) & \textrm{if $\unifdom$ is unbounded.}
		\end{cases}
	\end{equation}
	
	\begin{theorem}[Non-probabilistic part of Theorem \ref{thm:shk-trace}]\label{thm:shk-trace-intro}
		%Set $B_{\partial \unifdom}(\xi,r) \coloneqq B(\xi,r) \cap \partial \unifdom$ for $(\xi,r)\in \partial \unifdom \times (0,\infty)$, and
		Assume that $(\partial \unifdom,d)$ is \emph{uniformly perfect} (Definition \ref{d:UP}).
		%i.e., there exists $K_{0}\in(1,\infty)$ such that
		%$B_{\partial \unifdom}(\xi,r) \setminus B_{\partial \unifdom}(\xi,K_0^{-1}r) \not=\emptyset \emptyset$
		%for any $(\xi,r)\in(\partial \unifdom)\times(0,\infty)$ with $B_{\partial \unifdom}(\xi,r) \not= \partial \unifdom$.
		Then there exist $C_{1} \in (1,\infty)$ and a continuous heat kernel
		$\hkreftr=\hkreftr_{t}(\xi,\eta)\colon (0,\infty) \times \partial \unifdom \times \partial \unifdom \to [0,\infty)$
		of the trace Dirichlet form $(\formtr,\domaintr)$ of $(\formref,\domainref)$ on $L^{2}(\partial \unifdom,\bdrymeas)$
		such that for any $(t,\xi,\eta) \in (0,\infty) \times \partial \unifdom \times \partial \unifdom$,
		\begin{align}\label{eq:shk-trace-upper-intro}
			\hkreftr_{t}(\xi,\eta) &\leq C_{1} \biggl( \frac{1}{\bdrymeas\bigl(B(\xi,\scjump^{-1}(\xi,t)) \cap \partial \unifdom \bigr)} \wedge \frac{t}{\bdrymeas\bigl(B(\xi,d(\xi,\eta)) \cap \partial \unifdom \bigr) \scjump(\xi,d(\xi,\eta))} \biggr), \\
			\hkreftr_{t}(\xi,\eta) &\geq C_{1}^{-1} \biggl( \frac{1}{\bdrymeas\bigl(B(\xi,\scjump^{-1}(\xi,t)) \cap \partial \unifdom \bigr)} \wedge \frac{t}{\bdrymeas\bigl(B(\xi,d(\xi,\eta)) \cap \partial \unifdom \bigr) \scjump(\xi,d(\xi,\eta))} \biggr),
			\label{eq:shk-trace-lower-intro}
		\end{align}
		where $\scjump^{-1}(\xi,t) \coloneqq (\scjump(\xi,\cdot))^{-1}(t)$  and $B(\xi,0) \coloneqq \emptyset$.
		Moreover, $(\partial \unifdom,\bdrymeas,\formtr,\domaintr)$ is irreducible and conservative, and
		$\domaintr$ considered as a linear subspace of $L^{2}(\partial \unifdom,\bdrymeas)$ is identified as
		\begin{equation}\label{eq:shk-trace-intro-domain}
			\domaintr=\biggl\{u\in L^{2}(\partial \unifdom,\bdrymeas) \biggm|
			\int_{\offdiagp{\partial \unifdom}} (u(\xi)-u(\eta))^{2}\, \naimker^{\unifdom}_{x_{0}}(\xi,\eta) \,d\hmeas{\unifdom}{x_{0}}(\xi) \,d\hmeas{\unifdom}{x_{0}}(\eta)<\infty \biggr\}.
		\end{equation}
	\end{theorem}
	
	Recall that \eqref{eq:shk-trace-upper-intro} and \eqref{eq:shk-trace-lower-intro} with
	$\scjump(\xi,r)=r^{\alpha}$ for $\alpha\in(0,2)$ are the well-known form of the heat
	kernel estimates for the rotationally symmetric $\alpha$-stable process on $\mathbb{R}^{\dimeuc}$
	with $\bdrymeas$ and $d$ replaced by the Lebesgue measure and the Euclidean metric on $\mathbb{R}^{\dimeuc}$,
	respectively. The estimates \eqref{eq:shk-trace-upper-intro} and \eqref{eq:shk-trace-lower-intro}
	for the present case of the trace Dirichlet form $(\formtr,\domaintr)$ are of exactly the same form as
	these classical ones, except that the scaling relation between the space and time
	variables changes according to \eqref{eq:scjump-intro} and \eqref{eq:pre-scjump-intro},
	and for this reason we call \eqref{eq:shk-trace-upper-intro} and \eqref{eq:shk-trace-lower-intro}
	\textbf{stable-like heat kernel estimates}. Thanks to the recent characterization of stable-like heat
	kernel estimates obtained in \cite{CKW, GHH-lb, GHH-mv} and adapted for the present case
	in Theorem \ref{t:shkchar}, the proof of Theorem \ref{thm:shk-trace-intro}
	is reduced to verifying natural two-sided estimates \eqref{e:jphi} on the jump kernel
	$\jumpkertr_{\bdrymeas}(\xi,\eta) \coloneqq \naimker^{\unifdom}_{x_{0}}(\xi,\eta) \frac{d\hmeas{\unifdom}{x_{0}}}{d\bdrymeas}(\xi) \frac{d\hmeas{\unifdom}{x_{0}}}{d\bdrymeas}(\eta)$
	and an exit time lower estimate $\expref_{\xi}[ \tau_{B(\xi,r) \cap \partial \unifdom} ] \geq C^{-1} \scjump(\xi,r)$;
	here $\expref_{\xi}[\cdot]$ denotes the mean with respect to the Hunt process $\diffreftr=\{\diffreftr_{t}\}_{t\geq 0}$
	associated with $(\partial \unifdom,\bdrymeas,\formtr,\domaintr)$. The former estimates follow by combining
	\eqref{eq:Naim-kernel-estimates-intro}, \eqref{eq:hmeas-estimate-intro}, \eqref{eq:elliptic-meas-density-bounds-intro},
	\eqref{eq:elliptic-meas-bounds-intro} and \eqref{eq:scjump-intro}, whereas the latter can be proved
	by using \hyperlink{hke}{$\on{HKE(\scdiff)}$} for $(\ol{U},d,\refmeas|_{\unifdom},\formref,\domainref)$
	due to \cite[Theorem 2.8]{Mur24} and the fundamental feature of $(\formtr,\domaintr)$
	as a trace Dirichlet form of $(\formref,\domainref)$ that its Green function is precisely
	the restriction of the Green function of $(\formref,\domainref)$
	(Proposition \ref{prop:greeninv}-\eqref{it:greeninv}).
	As the probabilistic part of Theorem \ref{thm:shk-trace}, in the setting of
	Theorem \ref{thm:shk-trace-intro} we prove also that a version of the Hunt process
	$\diffreftr=\{\diffreftr_{t}\}_{t\geq 0}$ with continuous transition density
	$\hkreftr_{t}(\xi,\eta)$ for \emph{any} starting point $\xi\in\partial\unifdom$
	can be obtained as the time change of the reflected diffusion
	$\diffref = ( \{ \diffref_{t} \}_{t \ge 0}, \{ \lawref_{x} \}_{x \in \ol{\unifdom}})$
	by its positive continuous additive functional (PCAF) in the strict sense with Revuz measure $\bdrymeas$;
	see Subsection \ref{ss:capgood} and Theorem \ref{thm:shk-trace}-\eqref{it:shk-trace-Feller} for details.
	
	As mentioned in Remark \ref{rmk:differences-CaoChan-intro} above, the extensions
	of our results in Sections \ref{sec:harm-ell-meas} and \ref{sec:boundary-trace}
	to the case where ``$R \in (0,\diam(\unifdom)/A_1)$'' in \ref{eq:CDC-intro}
	is weakened to ``$R \in (0,\diam(\partial \unifdom)/A_1)$'' are described
	in some detail in Subsection \ref{ss:wCDC}. We then conclude this paper with
	brief discussions of several concrete examples in Subsection \ref{ssec:examples},
	illustrating in particular various possibilities of the quantitative behavior of
	the space-time scaling function $\scjump$, the reference measure $\bdrymeas$ and
	the jump kernel $\jumpkertr_\bdrymeas$ of the boundary trace Dirichlet form.
	
	\begin{notation} \label{ntn:intro}
		Throughout this paper, we use the following notation and conventions.
		\begin{enumerate}[\rm(a)]\setlength{\itemsep}{0pt}\vspace{-5pt}
			\item\label{it:set-inclusion} The symbols $\subset$ and $\supset$ for set inclusion \emph{allow} the case of the equality.
			\item\label{it:ineq-const} For $[0,\infty]$-valued quantities $A$ and $B$, we write $A \lesssim B$
			to mean that there exists an implicit constant $C \in (0,\infty)$
			depending on some unimportant parameters such that $A \leq CB$.
			We write $A \asymp B$ if $A \lesssim B$ and $B \lesssim A$.
			\item\label{it:natural-numbers} $\mathbb{N} \coloneqq \{n\in\mathbb{Z}\mid n>0\}$, i.e., $0\not\in\mathbb{N}$.
			\item\label{it:cardinality} The cardinality (the number of elements) of a set $A$ is denoted by $\#A\in\mathbb{N}\cup\{0,\infty\}$.
			\item\label{it:max-min} We set $\sup\emptyset \coloneqq 0$, $\inf\emptyset \coloneqq \infty$ and $0^{-1} \coloneqq \infty$.
			We set $a\vee b \coloneqq \max\{a,b\}$, $a\wedge b \coloneqq \min\{a,b\}$, $a^{+} \coloneqq a\vee 0$ and
			$a^{-} \coloneqq -(a\wedge 0)$ for $a,b\in[-\infty,\infty]$, and use the same notation
			also for $[-\infty,\infty]$-valued functions and equivalence classes of them.
			All numerical functions in this paper are assumed to be $[-\infty,\infty]$-valued.
			\item\label{it:Euclidean-norm} For $\dimeuc\in\mathbb{N}$, the Euclidean inner product and norm on $\mathbb{R}^{\dimeuc}$
			are denoted by $\langle\cdot,\cdot\rangle$ and $\abs{\cdot}$, respectively.
			\item\label{it:diag-offdiag} For a set $A$, we define $A_{\diag} \coloneqq \{(x,x)\mid x\in A\}$ and
			$\offdiag{A} \coloneqq (A\times A)\setminus A_{\diag}$ (``od'' stands for ``off-diagonal'').
			\item\label{it:indicator-sup-norm} Let $\mathcal{X}$ be a non-empty set. We define $\one_{A}=\one_{A}^{\mathcal{X}} \in \mathbb{R}^{\mathcal{X}}$ for $A \subset \mathcal{X}$ by
			$\one_{A}(x) \coloneqq \one_{A}^{\mathcal{X}}(x) \coloneqq \bigl\{\begin{smallmatrix} 1 &\textrm{if $x \in A$,} \\ 0 &\textrm{if $x \not\in A$,} \end{smallmatrix}$
			and set $\norm{u}_{\sup} \coloneqq \norm{u}_{\sup,\mathcal{X}} \coloneqq \sup_{x \in \mathcal{X}}\abs{u(x)}$ for $u \colon \mathcal{X} \to [-\infty,\infty]$
			and $\osc_{\mathcal{X}}u \coloneqq \sup_{x,y \in \mathcal{X}}\abs{u(x)-u(y)}$ for $u \colon \mathcal{X} \to \mathbb{R}$.
			We say that $u \colon \mathcal{X} \to [-\infty,\infty]$ is \emph{bounded} if $\norm{u}_{\sup} < \infty$.
			\item\label{it:measure-restr-AC} Let $(\mathcal{X},\mathscr{B})$ be a measurable space
			and let $\mu,\nu$ be measures on $(\mathcal{X},\mathscr{B})$.
			The $\mu$-completion of $\mathscr{B}$ is denoted by $\mathscr{B}^{\mu}$, and we set
			$\mathscr{B}^{*} \coloneqq \bigcap_{\textrm{$\lambda$:\ a $\sigma$-finite measure on $(\mathcal{X},\mathscr{B})$}}\mathscr{B}^{\lambda}$.
			We also set $\mathscr{B}|_{A} \coloneqq \{ B \cap A \mid B \in \mathscr{B} \}$
			and $m|_{A} \coloneqq m|_{\mathscr{B}|_{A}}$ for $A \in \mathscr{B}$,
			and we write $\nu \ll \mu$ to mean that $\nu$ is absolutely continuous with respect to $\mu$.
			When $\mu$ is $\sigma$-finite, the product measure space of $(\mathcal{X},\mathscr{B},\mu)$
			and itself is denoted by $(\mathcal{X} \times \mathcal{X},\mathscr{B} \otimes \mathscr{B},\mu \times \mu)$.
			\item\label{it:topology} Let $\mathcal{X}$ be a topological space. For $A \subset \mathcal{X}$, the closure and boundary of $A$ in $\mathcal{X}$ are denoted by $\overline{A}$ and $\partial A$, respectively,
			and we say that $A$ is \emph{relatively compact in $D \subset \mathcal{X}$}, and write $A \Subset D$, if and only if $A$ is included in a compact subset of $D$.
			We set $\contfunc(\mathcal{X}) \coloneqq \{ u \in \mathbb{R}^{\mathcal{X}} \mid \textrm{$u$ is continuous} \}$, $\supp_{\mathcal{X}}[u] \coloneqq \overline{\mathcal{X} \setminus u^{-1}(0)}$ for $u \in \contfunc(\mathcal{X})$, 
			$\contfunc_{\mathrm{c}}(\mathcal{X}) \coloneqq \{ u \in \contfunc(\mathcal{X}) \mid \textrm{$\supp_{\mathcal{X}}[u]$ is compact} \}$, and
			$\contfunc_{0}(\mathcal{X}) \coloneqq \{ u \in \contfunc(\mathcal{X}) \mid \textrm{$u^{-1}(\mathbb{R} \setminus (-\varepsilon,\varepsilon))$ is compact for any $\varepsilon \in (0,\infty)$}\}$.
			%	The one-point compactification of $\mathcal{X}$ is denoted by $\oneptcpt{\mathcal{X}}=\mathcal{X}\cup\{\cemetery_{\mathcal{X}}\}$, and
			The Borel $\sigma$-algebra of $\mathcal{X}$ is denoted by $\Borel(\mathcal{X})$,
			and we set $\Borel^{*}(\mathcal{X}) \coloneqq \Borel(\mathcal{X})^{*}$ and call
			$\Borel^{*}(\mathcal{X})$ the \emph{universal $\sigma$-algebra} of $\mathcal{X}$.
			\item\label{it:support-meas} Let $\mathcal{X}$ be a topological space having
			a countable open base, and let $m$ be a Borel measure on $\mathcal{X}$.
			The \emph{(topological) support} of $m$ in $\mathcal{X}$, that is, the smallest
			closed subset $F$ of $\mathcal{X}$ such that $m(\mathcal{X} \setminus F) = 0$,
			is denoted by $\supp_{\mathcal{X}}[m]$. For a $\Borel(\mathcal{X})^{m}$-measurable
			function $f \colon \mathcal{X} \to [-\infty,\infty]$ or an $m$-equivalence class $f$
			of such functions, we set $\supp_{m}[f] \coloneqq \supp_{\mathcal{X}}[\abs{f}\cdot m]$,
			where $\abs{f}\cdot m$ denotes the Borel measure on $\mathcal{X}$ defined by
			$(\abs{f}\cdot m)(A) \coloneqq \int_{A}\abs{f}\,dm$.
			\item\label{it:ball-diam-dist} Let $(\mathcal{X},d)$ be a metric space.
			We set $B(x,r) \coloneqq B_{d}(x,r) \coloneqq \{ y \in \mathcal{X} \mid d(x,y) < r \}$ and
			$S(x,r) \coloneqq S_{d}(x,r) \coloneqq \partial B(x,r)$ for $(x,r) \in \mathcal{X} \times (0,\infty)$
			and call each such $B(x,r)$ a \emph{ball} in $(\mathcal{X},d)$.
			We also set $\diam(A) \coloneqq \diam(A,d) \coloneqq \sup_{x,y\in A}d(x,y)$ for $A \subset \mathcal{X}$
			and $\dist(A,B) \coloneqq \dist_{d}(A,B) \coloneqq \inf_{(x,y)\in A\times B}d(x,y)$
			and $\dist(x,A) \coloneqq \dist_{d}(x,A) \coloneqq \dist(\{x\},A)$ for $A,B\subset \mathcal{X}$ and $x\in\mathcal{X}$.
			We say that a subset $A$ of $\mathcal{X}$ is \emph{bounded} if $\diam(A)<\infty$,
			and \emph{unbounded} if $\diam(A)=\infty$.
		\end{enumerate}
	\end{notation}

	\section{Preliminaries} \label{sec:preliminaries}
	
	In this section, we recall basic notions and results from metric geometry and the
	theory of regular symmetric Dirichlet forms, and prove some general results applied
	later in Section \ref{sec:boundary-trace} to the case of the boundary traces of reflected
	diffusions on uniform domains. Subsections \ref{ss:doubling} and \ref{ss:Uniform}
	concern purely metric-measure properties of the underlying state space,
	introducing the metric doubling and volume doubling properties and the definition
	and some basic features of uniform domains. Subsection \ref{ss:DF} summarizes some
	basics of the theory of regular symmetric Dirichlet forms and associated symmetric
	Hunt processes as presented in \cite{FOT,CF}. In Subsection \ref{ss:HKE}, we
	give the definition and some probabilistic consequences of sub-Gaussian
	heat kernel estimates for MMD spaces (strongly local regular Dirichlet spaces
	in which every bounded closed set is compact) and state the second-named author's
	result in \cite{Mur24} on the regularity and sub-Gaussian heat kernel estimates for
	reflected Dirichlet forms on uniform domains. Subsection \ref{ssec:harm-func-EHI}
	is devoted to formulating harmonic functions, the elliptic Harnack inequality and
	Dirichlet boundary condition relative to open sets, and presenting some related facts.
	In Subsection \ref{ss:trace}, we introduce trace Dirichlet forms and relevant notions,
	and give new elementary proofs of the identification of the strongly local part of
	trace Dirichlet forms as in \cite[Theorem 5.6.2]{CF} and of the vanishing of their
	killing part under a natural non-escape assumption, a simple consequence of
	\cite[Theorem 5.6.3]{CF}. In Subsection \ref{s:stablehke} we formulate the
	stable-like heat kernel estimates for pure-jump Dirichlet forms and state
	a nice characterization of them, following \cite{CKW, GHH-lb, GHH-mv}.
	Lastly, Subsection \ref{ss:capgood} presents a sufficient condition for a Borel measure $\goodmeas$
	on an MMD space to be a Radon measure corresponding to a positive continuous additive functional (PCAF)
	$\goodpcaf=\{\goodpcaf_{t}\}_{t\in[0,\infty)}$ \emph{in the strict sense} of the associated diffusion $\diff$
	(i.e., a PCAF of $\diff$ defined $\lawdiff_{x}$-a.s.\ for \emph{every} point $x$ of the state space)
	and for the support of $\goodmeas$ to coincide with the support of $\goodpcaf$.
	We also prove that, under the same condition on $\goodmeas$, the time-changed process
	$\difftr$ of $\diff$ by $\goodpcaf$ is a Hunt process on the support of $\goodmeas$
	sharing the same Green functions as $\diff$ under \emph{every} choice of the starting point
	(Proposition \ref{prop:greeninv}). We will see later in Section \ref{sec:boundary-trace}
	that our boundary trace processes are special cases of these general results.
	
	\subsection{Metric doubling and volume doubling properties} \label{ss:doubling}
	
	In much of this work, we will be in the setting of a metric doubling metric space equipped with a volume doubling measure.
	
	\begin{definition}[Metric doubling property (MD)]\label{dfn:metric-doubling}
		Let $(\ambient,d)$ be a metric space. The metric $d$, or the metric space $(\ambient,d)$,
		is said to be \textbf{(metric) doubling}, or to satisfy the \textbf{metric doubling property},
		abbreviated as \hypertarget{MD}{\textup{MD}}, if there exists $N \in \bN$ such that
		$B(x,R)$ is included in the union of some $N$ balls of radii $R/2$ in $(\ambient,d)$
		for any $(x,R) \in \ambient \times (0,\infty)$.
	\end{definition}
	
	Next, we recall the closely related volume doubling property on subsets of $\ambient$
	for Borel measures on $\ambient$. The pair $(\ambient,d,\refmeas)$ of a metric space
	$(\ambient,d)$ and a Borel measure $\refmeas$ on $\ambient$ is termed a \emph{metric measure space}.
	
	\begin{definition}[Volume doubling property (VD)] \label{dfn:volume-doubling}
		Let $(\ambient,d,\refmeas)$ be a metric measure space and let $V \subset \ambient$.
		The measure $\refmeas$, or the metric measure space $(\ambient,d,\refmeas)$,
		is said to be \textbf{(volume) doubling on $V$}, or to satisfy the
		\textbf{volume doubling property on $V$}, if there exists $D_0 \in [1,\infty)$ such that
		\[
		0 < m(B(x,2r) \cap V) \le D_0 m(B(x,r)\cap V) < \infty \quad
		\textrm{for all $x \in V$ and all $r>0$.}
		\]
		We say that $\refmeas$ or $(\ambient,d,\refmeas)$ is \textbf{(volume) doubling},
		or satisfies the \textbf{volume doubling property}, abbreviated as \hypertarget{VD}{\textup{VD}},
		if $(\ambient,d,\refmeas)$ is volume doubling on $\ambient$.
	\end{definition}
	
	The basic relationship between these notions is that if there exists a volume doubing measure on a metric space $(\ambient,d)$, then $(\ambient,d)$ is  metric doubling.
	Conversely, every complete, metric doubling metric space admits a volume doubling measure; see \cite[Chapter 13]{Hei}.
	By iterating the volume doubling condition, it is easy to see that for any metric measure space $(\ambient,d,\refmeas)$ satisfying \hyperlink{VD}{\textup{VD}},
	there exist $C \in (1,\infty)$ and $\beta \in (0,\infty)$ such that 
	\begin{equation}\label{e:vd}
		\frac{\refmeas(B(y,R))}{\refmeas(B(x,r))} \le C \biggl( \frac{d(x,y)+R}{r} \biggr)^\beta
		\quad \textrm{for all $x,y \in \ambient$ and all $0<r\le R$.}
	\end{equation}
	
	We further recall another closely related property known as the
	\emph{reverse volume doubling property} in the literature, to which
	the following definition is relevant.
	
	\begin{definition} \label{d:UP}
		We say that a metric space $(\ambient,d)$ is \emph{uniformly perfect} if there exists $K_0 \in (1,\infty)$ such that for all $x \in \ambient, r>0$ such that $B(x,r) \neq \ambient$, we have 
		\[
		B(x,r) \setminus B(x,K_0^{-1}r) \neq \emptyset.
		\]
	\end{definition}
	%	We record the following result for later use.
	\begin{lem}[{\cite[Exercise 13.1]{Hei}}] \label{l:rvd}
		Let $\refmeas$ be a volume doubling measure on a uniformly perfect metric space $(\ambient,d)$.
		Then the measure $\refmeas$ satisfies the following \emph{reverse volume doubling property}, abbreviated as
		\hypertarget{RVD}{\textup{RVD}}: there exist $C \in (1,\infty)$ and $\alpha \in (0,\infty)$
		such that for all $x \in \ambient$ and all $0<r\leq R<\diam(\ambient)$,
		\begin{equation} \label{e:rvd}
			\frac{\refmeas(B(x,R))}{\refmeas(B(x,r))} \ge C^{-1} \biggl( \frac{R}{r} \biggr)^\alpha.
		\end{equation}
	\end{lem}
	\subsection{Uniform domains} \label{ss:Uniform}
	Let $(\ambient,d)$ be a metric space and let $\unifdom \subset \ambient$ be an open set.
	A \emph{curve in $\unifdom$} is a continuous map $\gamma\colon[a,b] \to \unifdom$, and
	such $\gamma$ is said to be \emph{from $x$ to $y$} or to \emph{join $x$ and $y$}, where $x,y \in \unifdom$,
	if $\gamma(a)=x$ and $\gamma(b)=y$. We sometimes identify $\gamma$ with its image $\gamma([a,b])$,
	so that $\gamma \subset \unifdom$. The \emph{length} (in $(\ambient,d)$) of a curve $\gamma\colon [a,b] \to \ambient$ is defined as
	\begin{equation} \label{eq:length-dfn}
		\ell(\gamma) \coloneqq \sup \Biggl\{ \sum_{i=0}^{n-1} d(\gamma(t_i),\gamma(t_{i+1})) \Biggm| a \le t_0 < t_1 \ldots < t_n \le b \Biggr\}.
	\end{equation}
	We say that $(\ambient,d)$ is a length space if $d(x,y)$ is equal to
	the infimum of the lengths of curves in $\ambient$ from $x$ to $y$ for any $x,y\in\ambient$.
	
	\begin{definition}[Uniform domain] \label{d:uniform}
		Let $(\ambient,d)$ be a metric space, $\unifdom$ a non-empty open subset of $\ambient$ with $\unifdom \not= \ambient$,
		$c_{\unifdom} \in (0,1)$ and $C_{\unifdom} \in (1,\infty)$.
		Set $\delta_{\unifdom}(z) \coloneqq \dist(z,\ambient \setminus \unifdom)$ for $z \in \unifdom$.
		\begin{enumerate}[\rm(a)]\setlength{\itemsep}{0pt}\vspace{-5pt}
			\item We say that $\unifdom$ is a \textbf{length $(c_{\unifdom},C_{\unifdom})$-uniform domain} in $(\ambient,d)$ if for every pair of points $x,y \in \unifdom$, there exists a curve $\gamma$ in $\unifdom$ from $x$ to $y$ such that its length $\ell(\gamma) \le C_{\unifdom} d(x,y)$ and for all $z \in \gamma$, 
			\begin{equation} \label{eq:length-uniform-domain-dfn}
				\delta_{\unifdom}(z) \ge c_{\unifdom} \min\{ \ell(\gamma_{x,z}), \ell(\gamma_{z,y}) \},
			\end{equation}
			where $\gamma_{x,z}, \gamma_{z,y}$ denote the subcurves of $\gamma$ from $x$ to $z$ and from $z$ to $y$, respectively.
			Such a curve $\gamma$ is called a \emph{length $(c_{\unifdom},C_{\unifdom})$-uniform curve} in $\unifdom$.
			\item We say that $\unifdom$ is a \textbf{$(c_{\unifdom},C_{\unifdom})$-uniform domain} in $(\ambient,d)$ if for every pair of points $x,y \in \unifdom$, there exists a curve $\gamma$ in $\unifdom$ from $x$ to $y$ such that its diameter $\diam(\gamma) \le C_{\unifdom} d(x,y)$ and for all $z \in \gamma$, 
			\begin{equation} \label{eq:uniform-domain-dfn}
				\delta_{\unifdom}(z) \ge c_{\unifdom} \min\{ d(x,z), d(y,z) \}.
			\end{equation}
			Such a curve $\gamma$ is called a \emph{$(c_{\unifdom},C_{\unifdom})$-uniform curve} in $\unifdom$.
		\end{enumerate}
	\end{definition}
	There are different definitions of uniform domains in the literature \cite{Mar,Vai};
	note that our definition of length uniform domain is what is usually called uniform domain in the literature.
	The above definition of uniform domain was %has been
	introduced in \cite{Mur24} because of the advantage that this notion of uniform domain is preserved under quasisymmetric changes of the metric on the underlying space.
	Furthermore, this definition also allows us to consider metric spaces that do not have non-constant rectifiable curves. 
	
	The following is a variant of \cite[Lemma 3.20]{GS}.
	\begin{lem} \label{l:xir}
		Let $(\ambient,d)$ be a metric space, let $c_{\unifdom} \in (0,1)$,
		and let $\unifdom \subset \ambient$ be a $(c_{\unifdom},C_{\unifdom})$-uniform domain for some $C_{\unifdom} \in (1,\infty)$. Then for any $\xi \in \partial \unifdom$ and any $r \in (0,\diam(\unifdom)/4)$, there exists $\xi_r \in \unifdom$ such that 
		\begin{equation} \label{eq:xir}
			d(\xi,\xi_r)=r \quad \textrm{and} \quad \delta_\unifdom(\xi_r) > \frac{c_{\unifdom} r}{2}.
		\end{equation}
	\end{lem}
	\begin{proof}
		Since $r<\diam(\unifdom)/4$ we can choose a point $y \in \unifdom$ such that $d(\xi,y)>2r$,
		and by $\xi \in \partial \unifdom$ we can choose a point $x \in B(\xi,r/2) \cap \unifdom$.
		By considering a $(c_{\unifdom},C_{\unifdom})$-uniform curve $\gamma$ in $\unifdom$ from $x$ to $y$
		and the continuity of $d(\xi,\cdot)$ along $\gamma$, there exists $\xi_r \in \gamma$ such that $d(\xi,\xi_r)=r$, and then
		\begin{align*}
			\delta_{\unifdom}(\xi_r) &\geq  c_{\unifdom} \min\{ d(x,\xi_r), d(y,\xi_r) \} \\
			&\geq  c_{\unifdom} \min\{ d(\xi,\xi_r)-d(\xi,x), d(\xi,y) - d(\xi,\xi_r) \}> \frac{c_{\unifdom} r}{2}.
			\qedhere\end{align*}
	\end{proof}
	
	\begin{notation} \label{ntn:xir}
		Throughout this paper, given $(\ambient,d),c_{\unifdom},\unifdom$ as in Lemma \ref{l:xir},
		$\xi_{r}$ always denotes an arbitrary element of $\unifdom$ satisfying \eqref{eq:xir}
		for each $(\xi,r) \in \partial \unifdom \times (0,\diam(\unifdom)/4)$.
	\end{notation}

	We recall that the volume doubling property of measures is inherited by uniform domains.
	
	\begin{lem}[{\cite[Theorem 2.8]{BS}}, {\cite[Lemma 3.5]{Mur24}}] \label{l:doubling-unif}
		Let $(\ambient,d,\refmeas)$ be a metric measure space satisfying \hyperlink{VD}{\textup{VD}},
		and let $\unifdom$ be a uniform domain in $(\ambient,d)$. Then 
		\begin{equation} \label{eq:volume-doubling-unif-bdry-zero}
			\refmeas(\partial \unifdom)=0,
		\end{equation}
		and $(\unifdom,d,\refmeas|_{\unifdom})$ and $(\overline{\unifdom},d,\refmeas|_{\overline{\unifdom}})$
		satisfy \hyperlink{VD}{\textup{VD}}.
	\end{lem}

	\subsection{Regular Dirichlet space and symmetric Hunt process} \label{ss:DF}
	
	We now recall some basics of the theory of regular symmetric Dirichlet forms
	as presented in \cite{FOT,CF}. Throughout this subsection, we consider a locally compact
	separable metrizable topological space $\ambient$, a Radon measure $\refmeas$ on
	$\ambient$ with full support, i.e., a Borel measure $\refmeas$ on $\ambient$
	which is finite on any compact subset of $\ambient$ and strictly positive on any
	non-empty open subset of $\ambient$, and a \emph{symmetric Dirichlet form}
	$(\form,\domain)$ on $L^{2}(\ambient,\refmeas)$;
	that is, $\domain$ is a dense linear subspace of $L^{2}(\ambient,\refmeas)$, and
	$\form \colon \domain\times\domain \to \mathbb{R}$
	is a non-negative definite symmetric bilinear form which is \emph{closed}
	($\mathcal{F}$ is a Hilbert space under the inner product $\mathcal{E}_{1} \coloneqq \mathcal{E}+ \langle \cdot,\cdot \rangle_{L^{2}(\ambient,\refmeas)}$)
	and \emph{Markovian} ($f^{+}\wedge 1\in\mathcal{F}$ and $\mathcal{E}(f^{+}\wedge 1,f^{+}\wedge 1)\leq \mathcal{E}(f,f)$ for any $f\in\mathcal{F}$).
	We say that $(\mathcal{E},\mathcal{F})$ is \emph{regular} if
	$\mathcal{F} \cap \contfunc_{\mathrm{c}}(\ambient)$ is dense both in $(\mathcal{F},\mathcal{E}_{1})$
	and in $(\contfunc_{\mathrm{c}}(\ambient),\norm{\cdot}_{\mathrm{sup}})$, and that
	$(\mathcal{E},\mathcal{F})$ is called \emph{strongly local} if $\mathcal{E}(f,g)=0$
	for any $f,g\in\mathcal{F}$ with $\supp_{\refmeas}[f]$, $\supp_{\refmeas}[g]$ compact and
	$\supp_{\refmeas}[f-a\one_{\ambient}]\cap\supp_{\refmeas}[g]=\emptyset$ for some $a\in\mathbb{R}$;
	here $\contfunc_{\mathrm{c}}(\ambient)$ and $\supp_{\refmeas}[f]$ are as defined in
	Notation \ref{ntn:intro}-\eqref{it:topology},\eqref{it:support-meas}, and
	note that $\supp_{\refmeas}[f] = \overline{\ambient\setminus f^{-1}(0)}$
	if $f \colon \ambient \to [-\infty,\infty]$ is continuous.
	The quadruple $(\ambient,\refmeas,\form,\domain)$ is termed a
	\emph{regular Dirichlet space} if $(\mathcal{E},\mathcal{F})$ is regular, and
	a \emph{strongly local regular Dirichlet space} if $(\mathcal{E},\mathcal{F})$
	is regular and strongly local. In particular, if $(\ambient,\refmeas,\form,\domain)$
	is a regular Dirichlet space and $d$ is a metric on $\ambient$ compatible with the
	topology of $\ambient$ such that $B(x,r) \coloneqq \{ y \in \ambient \mid d(x,y) < r \}$
	is relatively compact in $\ambient$ for every $(x,r) \in \ambient \times (0,\infty)$,
	then the quintuple $(\ambient,d,\refmeas,\form,\domain)$ is termed
	a \textbf{not-necessarily-local metric measure Dirichlet space},
	or an \textbf{NLMMD space} in abbreviation. If $(\ambient,d,\refmeas,\form,\domain)$
	is an NLMMD space such that $\# \ambient \geq 2$ and $(\form,\domain)$ is strongly local,
	then $(\ambient,d,\refmeas,\form,\domain)$ is termed a
	\textbf{metric measure Dirichlet space}, or an \textbf{MMD space} in abbreviation.
	
	Associated with a symmetric Dirichlet form is a \emph{strongly continuous contraction semigroup} $(T_t)_{t > 0}$; that is, a family of symmetric bounded linear operators $T_t:L^2(\ambient,\refmeas) \to L^2(\ambient,\refmeas)$ such that
	\begin{equation*}
		T_{t+s}f=T_t(T_s f), \quad \norm{T_t f}_2 \le \norm{f}_{L^{2}(\ambient,\refmeas)}, \quad \lim_{t \downarrow 0} \norm{T_t f-f}_{L^{2}(\ambient,\refmeas)} =0, 
	\end{equation*}
	for all $t,s\in(0,\infty)$ and all $f \in L^2(\ambient,\refmeas)$. In this case, as stated in
	\cite[Lemma 1.3.4-(i)]{FOT} we can express $(\form,\domain)$ in terms of the semigroup as
	\begin{equation} \label{e:semigroup}
		\begin{split}
			\domain &= \biggl\{f \in L^{2}(\ambient,\refmeas) \biggm| \lim_{t \downarrow 0} \frac{1}{t}\langle f - T_{t} f, f \rangle_{L^{2}(\ambient,\refmeas)} < \infty \biggr\}, \\
			\form&(f,f) = \lim_{t \downarrow 0} \frac{1}{t}\langle f - T_{t} f, f \rangle_{L^{2}(\ambient,\refmeas)} \qquad \textrm{for all $f \in \domain$.}
		\end{split}
	\end{equation}
	As is well known, $T_t$ restricted to $L^2(\ambient,\refmeas) \cap L^\infty(\ambient,\refmeas)$ canonically
	extends to a positivity-preserving linear contraction on $L^\infty(\ambient,\refmeas)$ (see, e.g., \cite[pp.~6 and 7]{CF}).
	We say that $(\form,\domain)$ or $(\ambient,\refmeas,\form,\domain)$ is \emph{conservative}
	if $T_{t} \one_{\ambient} = \one_{\ambient}$ $\refmeas$-a.e.\ for any $t \in (0,\infty)$,
	and that $(\form,\domain)$ or $(\ambient,\refmeas,\form,\domain)$ is \emph{irreducible}
	if $\refmeas(A)\refmeas(\ambient \setminus A)=0$ for any $A \in \Borel(\ambient)$
	that is \emph{$\form$-invariant}, i.e., satisfies $T_{t}(\one_{A}f)=0$ $\refmeas$-a.e.\ on
	$\ambient \setminus A$ for any $f \in L^{2}(\ambient,\refmeas)$ and any $t\in(0,\infty)$.

	We next introduce a few notions relevant to the global behavior of $(T_t)_{t > 0}$.
	
	\begin{definition}[Extended Dirichlet space]\label{d:ExtDiriSp}
		We define the \emph{extended Dirichlet space} $\domain_e$ of $(\ambient,\refmeas,\form,\domain)$
		as the space of $\refmeas$-equivalence classes of functions $f\colon \ambient \to \bR$
		such that $\lim_{n\to \infty} f_{n}= f$ $\refmeas$-a.e.\ on $\ambient$ for some
		$\{ f_{n} \}_{n \in \mathbb{N}} \subset \domain$ with $\lim_{k\wedge l\to\infty}\form(f_{k}-f_{l},f_{k}-f_{l})=0$.
		Then the limit $\form(f,f) \coloneqq \lim_{n\to\infty}\form(f_{n},f_{n})\in\mathbb{R}$
		exists and is independent of a choice of such $\{ f_{n} \}_{n \in \mathbb{N}}$ for each
		$f \in \domain_{e}$, so that $\form$ is canonically extended to $\domain_e \times \domain_e$
		and satisfies $\lim_{n\to\infty}\form(f-f_{n},f-f_{n})=0$ for any such $(f_n)_{n \in \bN}$
		for each $f \in \domain_e$, and $\domain=\domain_{e}\cap L^{2}(\ambient,\refmeas)$; see \cite[Definition 1.1.4 and Theorem 1.1.5]{CF}.
	\end{definition}
	
	We say that $(\form, \domain)$ or $(\ambient,\refmeas,\form,\domain)$ is \emph{transient}
	if there exists $g \in L^{1}(\ambient, \refmeas) \cap L^{\infty}(\ambient, \refmeas)$
	that is strictly positive $\refmeas$-a.e.\ on $\ambient$ and satisfies
	\[
	\int_{\ambient} \abs{u(x)} g(x)\, \refmeas(dx) \leq \form(u, u)^{1/2} \quad \hbox{for every } u\in \domain.
	\]
	By \cite[Theorem 2.1.5-(i)]{CF}, $(\ambient,\refmeas,\form,\domain)$ is transient
	if and only if there exists $g \in L^{1}(\ambient, \refmeas) \cap L^{\infty}(\ambient, \refmeas)$
	that is strictly positive $\refmeas$-a.e.\ on $\ambient$ and satisfies 
	\begin{equation} \label{e:transient}
		\int_{\ambient} g Gg \,dm < \infty, \qquad \textrm{where $Gg \coloneqq \lim_{N\to\infty}\int_{0}^{N} T_t g\,dt$ $\refmeas$-a.e.}
	\end{equation}
	The transience of $(\ambient,\refmeas,\form,\domain)$ is equivalent to $\{f \in \domain_e \mid \form(f,f)=0\}=\{0\}$,
	in which case $(\domain_e, \form)$ is a Hilbert space (see \cite[Theorem 2.1.9]{CF}).
	On the other hand, we say that $(\form, \domain)$ or $(\ambient,\refmeas,\form,\domain)$
	is \emph{recurrent} if $Gf\in\{0,\infty\}$ $\refmeas$-a.e.\ on $\ambient$ for any
	$f\in L^{1}(\ambient,\refmeas)$ with $f\geq 0$ $\refmeas$-a.e.\ on $\ambient$,
	which is equivalent to the property that $\one_{\ambient}\in\domain_{e}$ and
	$\form(\one_{\ambient},\one_{\ambient})=0$ (see \cite[Theorem 2.1.8]{CF}).
	By \cite[Proposition 2.1.3-(iii)]{CF}, if $(\ambient,\refmeas,\form,\domain)$ is irreducible,
	then $(\ambient,\refmeas,\form,\domain)$ is either transient or recurrent.
	
	In the rest of this subsection, we assume that $(\ambient,\refmeas,\form,\domain)$
	is a regular Dirichlet space. As indispensable pieces of the theory of regular symmetric
	Dirichlet forms, we now recall some potential-theoretic notions from \cite[Section 2.1]{FOT}
	and \cite[Sections 1.2, 1.3 and 2.3]{CF}. First, we define the \emph{$1$-capacity}
	$\Capa_1(A)$ of $A \subset \ambient$ with respect to $(\ambient,\refmeas,\form,\domain)$ by
	\begin{equation} \label{e:defCap1}
		\Capa_1(A) \coloneqq \inf \bigl\{ \form_{1}(f,f) \bigm\vert \textrm{$f \in \domain$, $f \geq 1$ $\refmeas$-a.e.\ on a neighborhood of $A$} \bigr\},
	\end{equation}
	where $\form_{1} \coloneqq \form+\langle\cdot,\cdot\rangle_{L^{2}(\ambient,\refmeas)}$ as defined before.
	Note that $\Capa_1$ is countably subadditive by \cite[Lemma 2.1.2 and Theorem A.1.2]{FOT}.
	A subset $\mathcal{N}$ of $\ambient$ is said to be \emph{$\form$-polar} if $\Capa_{1}(\mathcal{N})=0$.
	For $A \subset \ambient$ and a statement $\mathcal{S}(x)$ on $x \in A$, we say that
	$\mathcal{S}$ holds \emph{$\form$-quasi-everywhere on $A$} (\emph{$\form$-q.e.\ on $A$} for short),
	or $\mathcal{S}(x)$ holds for \emph{$\form$-quasi-every} $x \in A$ (\emph{$\form$-q.e.}\ $x \in A$ for short),
	if $\mathcal{S}(x)$ holds for any $x \in A \setminus \mathcal{N}$ for some $\form$-polar $\mathcal{N} \subset \ambient$.
	When $A=\ambient$, we often write just ``\emph{$\form$-q.e.}''\ instead of ``$\form$-q.e.\ on $\ambient$''.
	A non-decreasing sequence $\{F_{k}\}_{k\in\mathbb{N}}$ of closed subsets of $\ambient$
	is called an \emph{$\form$-nest} if $\lim_{k\to\infty}\Capa_1(K \setminus F_{k})=0$
	for any compact subset $K$ of $\ambient$, or equivalently (see \cite[Theorem 1.3.14-(ii)]{CF}),
	if $\bigcup_{k\in\mathbb{N}} \domain_{F_{k}} $ is dense in $(\domain,\form_{1})$, where
	\begin{equation*}
		\domain_{F_{k}} \coloneqq \{f \in \domain \mid \textrm{$f=0$ $\refmeas$-a.e.\ on $\ambient \setminus F_{k}$} \}.
	\end{equation*}
	A function $f \colon D \setminus \mathcal{N} \to [-\infty,\infty]$, defined $\form$-q.e.\ on
	an open subset $D$ of $\ambient$ for some $\form$-polar $\mathcal{N} \subset \ambient$,
	is said to be \emph{$\form$-quasi-continuous on $D$} if there exists an $\form$-nest
	$\{F_{k}\}_{k\in\mathbb{N}}$ such that $F_{k} \cap \mathcal{N} = \emptyset$ and $f|_{D \cap F_{k}}$
	is an $\mathbb{R}$-valued continuous function on $D \cap F_{k}$ for any $k \in \mathbb{N}$
	(again, when $D=\ambient$, we often omit ``on $\ambient$''). For each $f \in \domain_e$,
	an $\form$-quasi-continuous $\refmeas$-version $\widetilde{f}$ of $f$ exists by \cite[Theorem 2.1.7]{FOT}
	(see also \cite[Theorem 1.3.14-(iii)]{CF}) and is unique $\form$-q.e.\ by \cite[Lemma 2.1.4]{FOT}.
	
	According to the fundamental theorem of M.\ Fukushima \cite[Theorem 7.2.1]{FOT},
	the assumption of the regularity of $(\ambient,\refmeas,\form,\domain)$ allows us to
	associate to $(\ambient,\refmeas,\form,\domain)$ an $\refmeas$-symmetric Hunt process on $\ambient$
	in the manner described below.
	
	Let $\diff = (\Omega, \events, \{\diff_{t}\}_{t\in[0,\infty]},\{\lawdiff_{x}\}_{x \in \oneptcpt{\ambient}})$
	be a \emph{Hunt process} on $\ambient$, i.e., a right-continuous strong Markov process on $(\oneptcpt{\ambient},\Borel(\oneptcpt{\ambient}))$
	which has the left limit $\diff_{t-}(\omega) \coloneqq \lim_{s\uparrow t}\diff_{s}(\omega)$ in $\oneptcpt{\ambient}$
	for any $(t,\omega)\in(0,\infty)\times\Omega$ and is quasi-left-continuous on $(0,\infty)$
	(see \cite[Definition A.1.23-(ii) and Theorem A.1.24]{CF}),
	where $\oneptcpt{\ambient} = \ambient \cup \{ \cemetery \}$ denotes
	the one-point compactification of $\ambient$. We always consider each function
	$f\colon \ambient \to [-\infty,\infty]$ as being defined also at $\cemetery$
	by setting $f(\cemetery) \coloneqq 0$. Let $\minaugfilt_{*}=\{\minaugfilt_{t}\}_{t \in [0,\infty]}$
	denote the minimum augmented admissible filtration of $\diff$ in $\Omega$ as defined in \cite[p.\ 397]{CF},
	so that $\minaugfilt_{*}$ is right-continuous, i.e.,
	$\minaugfilt_{t}=\bigcap_{s \in (t,\infty)}\minaugfilt_{s}$ for any $t \in [0,\infty)$ by \cite[Theorem A.1.18]{CF}.
	Let $\zeta$ denote the life time of $\diff$, i.e., a $[0,\infty]$-valued function on $\Omega$
	satisfying $\{ \diff_{t} = \cemetery \} = \{ \zeta \leq t \}$ for any $t \in [0,\infty]$,
	and for each $t \in [0,\infty]$ let $\shiftdiff_{t}$ denote the shift operator
	of $\diff$ by time $t$, i.e., a map $\shiftdiff_{t} \colon \Omega \to \Omega$
	satisfying $\diff_{s} \circ \shiftdiff_{t} = \diff_{s+t}$ for any $s \in [0,\infty]$;
	the existence of $\zeta$ and $\shiftdiff_{t}$ is part of the definition of $\diff$ being
	a Hunt process on $\ambient$. It then turns out (see, e.g., \cite[Exercise A.1.20-(i)]{CF})
	that the function $\oneptcpt{\ambient} \ni x \mapsto \lawdiff_{x}(A)$ is
	$\Borel^{*}(\oneptcpt{\ambient})$-measurable for any $A \in \minaugfilt_{\infty}$
	(recall Notation \ref{ntn:intro}-\eqref{it:measure-restr-AC},\eqref{it:topology} for $\Borel^{*}(\oneptcpt{\ambient})$),
	so that for each $\sigma$-finite Borel measure $\nu$ on $\oneptcpt{\ambient}$
	a $\sigma$-finite measure $\lawdiff_{\nu}$ on $\minaugfilt_{\infty}$
	is defined by $\lawdiff_{\nu}(A) \coloneqq \int_{\oneptcpt{\ambient}} \lawdiff_{x}(A)\,\nu(dx)$.
	For each $B \subset \oneptcpt{\ambient}$, we define
	$\sigma_{B},\dot{\sigma}_{B},\hat{\sigma}_{B} \colon \Omega \to [0,\infty]$ by
	\begin{equation} \label{eq:hitting-times}
		\begin{split}
			\sigma_{B}(\omega) & \coloneqq \inf\{ t \in (0,\infty) \mid \diff_{t}(\omega) \in B \}, \\
			\dot{\sigma}_{B}(\omega) & \coloneqq \inf\{ t \in [0,\infty) \mid \diff_{t}(\omega) \in B \}, \\
			\hat{\sigma}_{B}(\omega) & \coloneqq \inf\{ t \in (0,\infty) \mid \diff_{t-}(\omega) \in B \},
		\end{split}
	\end{equation}
	so that $\sigma_{B},\dot{\sigma}_{B},\hat{\sigma}_{B}$ are $\minaugfilt_{*}$-stopping times
	if $B \in \Borel(\oneptcpt{\ambient})$ by \cite[Theorem A.1.19 and Exercise A.1.26-(ii)]{CF}
	(see also \cite[Theorem A.2.3]{FOT}). A set $B \subset \oneptcpt{\ambient}$ is said to be
	\emph{$\diff$-nearly Borel measurable} if for any Borel probability measure
	$\nu$ on $\oneptcpt{\ambient}$ there exist $B_{1},B_{2} \in \Borel(\oneptcpt{\ambient})$
	such that $B_{1} \subset B \subset B_{2}$ and 
	\begin{equation} \label{eq:nearly-Borel}
		\lawdiff_{\nu}( \dot{\sigma}_{B_{2} \setminus B_{1} < \infty )
			=}\lawdiff_{\nu}( \textrm{$X_{t} \in B_{2} \setminus B_{1}$ for some $t \in [0,\infty)$} )=0.
	\end{equation}
	Then $\Borel^{\diff}(\oneptcpt{\ambient}) \coloneqq \{B \subset \ambient \mid \textrm{$B$ is $\diff$-nearly Borel measurable}\}$
	is a $\sigma$-algebra in $\oneptcpt{\ambient}$ included in $\Borel^{*}(\oneptcpt{\ambient})$,
	and $\sigma_{B},\dot{\sigma}_{B},\hat{\sigma}_{B}$ are easily seen to be
	$\minaugfilt_{*}$-stopping times for any $B \in \Borel^{\diff}(\oneptcpt{\ambient})$
	by the definition of $\minaugfilt_{*}$ and \cite[Theorem A.2.3]{FOT}.
	A $\Borel^{*}(\ambient)$-measurable function $u \colon \ambient \to [0,\infty]$
	is said to be \emph{$\diff$-excessive} if
	$[0,\infty) \ni t \mapsto \expdiff_{x}[u(\diff_{t})] \in [0,\infty]$ is non-increasing
	and $\lim_{t\downarrow 0}\expdiff_{x}[u(\diff_{t})]=u(x)$ for any $x \in \ambient$.
	
	We say that the Hunt process $X$ on $\ambient$ is \emph{$\refmeas$-symmetric}
	if its Markovian transition function $P_{t}(x,dy) \coloneqq \lawdiff_{x}(\diff_{t} \in dy)$,
	$(t,x) \in (0,\infty) \times \ambient$, is $\refmeas$-symmetric, i.e., if
	\begin{equation}\label{eq:Hunt-process-symmetry}
		\int_{\ambient}(P_{t}f)(x)g(x)\,\refmeas(dx) = \int_{\ambient}f(x)(P_{t}g)(x)\,\refmeas(dx)
	\end{equation}
	for any Borel measurable functions $f,g\colon \ambient \to [0,\infty]$ for each $t \in (0,\infty)$.
	In this case, an $\diff$-nearly Borel measurable subset $\mathcal{N}$ of $\ambient$
	is said to be \emph{properly exceptional} for $\diff$ if $\refmeas(\mathcal{N})=0$ and 
	\begin{equation}\label{eq:properly-exceptional}
		\lawdiff_{x}( \dot{\sigma}_{\mathcal{N}} \wedge \hat{\sigma}_{\mathcal{N}} = \infty )=1
		\quad \textrm{for any $x \in \ambient \setminus \mathcal{N}$.}
	\end{equation}
	For any such $\mathcal{N}$, we define the \emph{restriction}
	$\diff|_{\ambient \setminus \mathcal{N}}$ of $\diff$ to $\ambient \setminus \mathcal{N}$ by
	\begin{equation}\label{eq:Hunt-process-restriction}
		\Omega_{\ambient \setminus \mathcal{N}} \coloneqq \{ \dot{\sigma}_{\mathcal{N}} \wedge \hat{\sigma}_{\mathcal{N}} = \infty\}, \quad
		\diff|_{\ambient \setminus \mathcal{N}}
		 \coloneqq \bigl(\Omega_{\ambient \setminus \mathcal{N}}, \minaugfilt_{\infty}|_{\Omega_{\ambient \setminus \mathcal{N}}}, \bigl\{\diff_{t}|_{\Omega_{\ambient \setminus \mathcal{N}}}\bigr\}_{t\in[0,\infty]},\{\lawdiff_{x}\}_{x \in \oneptcpt{\ambient} \setminus \mathcal{N}}\bigr),
	\end{equation}
	which is a Hunt process on $\ambient \setminus \mathcal{N}$ by \cite[Lemma A.1.27]{CF}.
	We sometimes assume the following \emph{absolute continuity condition}, abbreviated as \ref{eq:AC}:
	\begin{equation} \tag*{\textup{AC}} \label{eq:AC}
		\textrm{$P_{t}(x, \cdot) \ll m$ (as Borel measures on $\ambient$)}
		\qquad \textrm{for any $(t,x) \in (0,\infty) \times \ambient$.}
	\end{equation}
	If $\diff$ is $\refmeas$-symmetric and satisfies \ref{eq:AC},
	then by \cite[Proof of Theorem 3.8]{BCM} there exists a unique
	Borel measurable function $p=p_{t}(x,y)\colon(0,\infty)\times\ambient\times\ambient\to[0,\infty]$
	such that for any $t,s\in(0,\infty)$ and any $x,y\in\ambient$,
	\begin{equation}\label{eq:AC-transition-density}
		P_{t}(x,dy)=p_{t}(x,z)\,\refmeas(dz),\quad
		p_{t}(x,y)=p_{t}(y,x),\quad
		p_{t+s}(x,y)=\int_{\ambient}p_{t}(x,z)p_{s}(z,y)\,\refmeas(dz).
	\end{equation}
	
	If $\diff$ is $\refmeas$-symmetric, then by \cite[(1.4.13) and Lemma 1.4.3]{FOT}
	the Markovian transition function $P_{t}(x,dy)$ of $\diff$ induces
	a strongly continuous contraction semigroup $(T^{\diff}_{t})_{t>0}$ on 
	$L^{2}(\ambient,\refmeas)$ such that $T^{\diff}_{t}f=P_{t}f$ $\refmeas$-a.e.\ for any
	$\Borel^{*}(\ambient)$-measurable $\refmeas$-version of any $f \in L^{2}(\ambient,\refmeas)$ for each $t \in (0,\infty)$,
	so that a symmetric Dirichlet form $(\form^{(\diff)},\domain^{(\diff)})$ on $L^{2}(\ambient,\refmeas)$,
	called the \emph{Dirichlet form of $\diff$}, is defined by \eqref{e:semigroup} with $(T^{\diff}_{t})_{t>0}$
	in place of $(T_{t})_{t>0}$. Fukushima's theorem \cite[Theorem 7.2.1]{FOT} states that
	\emph{any} regular symmetric Dirichlet form on $L^{2}(\ambient,\refmeas)$ is realized in this manner,
	namely that \emph{any regular symmetric Dirichlet form on $L^{2}(\ambient,\refmeas)$ is the
		Dirichlet form $(\form^{(\diff)},\domain^{(\diff)})$ of some $\refmeas$-symmetric Hunt process
		$\diff$ on $\ambient$}. Moreover, by \cite[Theorem 4.2.8]{FOT}, such a Hunt process
	on $\ambient$ is essentially unique for each given regular symmetric Dirichlet form
	on $L^{2}(\ambient,\refmeas)$ in the following sense:
	if $X$ and $X'$ are $\refmeas$-symmetric Hunt processes on $\ambient$ whose Dirichlet
	forms coincide and are regular, then there exists a common properly exceptional set
	for $X$ and $X'$ outside which the Markovian transition functions of $X$ and $X'$ coincide.
	
	In the rest of this subsection, we assume that
	$(\ambient,\refmeas,\form,\domain)$ is a regular Dirichlet space and that
	$\diff = (\Omega, \events, \{\diff_{t}\}_{t\in[0,\infty]},\{\lawdiff_{x}\}_{x \in \oneptcpt{\ambient}})$
	is a Hunt process on $\ambient$ whose Dirichlet form is $(\form,\domain)$.
	Then by \cite[Theorems 4.2.1-(ii) and 4.1.1]{FOT},
	\begin{equation}\label{eq:cap-zero-properly-exceptional}
		\begin{minipage}{375pt}
			any properly exceptional set $\mathcal{N} \subset \ambient$ for $\diff$
			is $\form$-polar, and any $\form$-polar subset of $\ambient$ is included in
			some properly exceptional set $\mathcal{N} \in \Borel(\ambient)$ for $\diff$.
		\end{minipage}
	\end{equation}
	Furthermore by \cite[Theorem 4.2.3-(i)]{FOT}, for any $[0,\infty]$-valued Borel
	measurable $f\in L^{2}(\ambient,\refmeas)$ and any $t \in (0,\infty)$,
	the Borel measurable function $P_{t}f \colon \ambient \to [0,\infty]$ given by
	\begin{equation} \label{eq:transition-func}
		(P_{t}f)(x) = \int_{\ambient}f(y)\,P_{t}(x,dy) = \expdiff_{x}[ f(\diff_{t}) ]
	\end{equation}
	is an $\form$-quasi-continuous $\refmeas$-version of $T_{t}f$.
	Note also that by \cite[Theorem 4.5.3]{FOT}, $(\form,\domain)$ is strongly local
	if and only if $\diff$ is a diffusion with no killing inside for
	$\form$-q.e.\ starting point, i.e.,
	\begin{equation} \label{eq:sample-path-cont}
		\lawdiff_{x}\bigl(\textrm{$[0,\infty)\ni t \mapsto \diff_{t} \in \oneptcpt{\ambient}$ is continuous}\bigr)=1
	\end{equation}
	for $\form$-a.e.\ $x \in \ambient$, and that by \cite[Theorem 4.5.4-(iii)]{FOT},
	if $\diff$ satisfies \ref{eq:AC}, then $(\form,\domain)$ is strongly local
	if and only if \eqref{eq:sample-path-cont} holds for \emph{any} $x \in \ambient$.
	
	The rest of this subsection is devoted to discussions of the Dirichlet forms on open
	subsets of $\ambient$ induced from $(\form,\domain)$ by assigning boundary conditions.
	We first consider those resulting from Dirichlet boundary condition and their associated
	Hunt processes given as follows.
	
	\begin{definition}[Part Dirichlet form and part process] \label{d:part-form-process}
		Let $D$ be an open subset of $\ambient$.
		\begin{enumerate}[\rm(a)]\setlength{\itemsep}{0pt}\vspace{-5pt}
			\item The \emph{part Dirichlet form} $(\form^{D},\domain^{0}(D))$ of $(\form,\domain)$
			on $D$ is defined by
			\begin{equation}\label{e:FD}
				\domain^{0}(D) \coloneqq \{f \in \domain \mid \textrm{$\wt{f}=0$ $\form$-q.e.\ on $\ambient \setminus D$} \}
				\quad \textrm{and} \quad
				\form^{D} \coloneqq \form|_{\domain^{0}(D) \times \domain^{0}(D)}.
			\end{equation}
			\item The \emph{part process}
			$\diff^{D} = (\Omega, \minaugfilt_{\infty}, \{\diff^{D}_{t}\}_{t\in[0,\infty]},\{\lawdiff_{x}\}_{x \in \oneptcpt{D}})$
			of $\diff$ on $D$ (killed upon exiting $D$) is defined by
			\begin{equation}\label{eq:part-process}
				\diff^{D}_{t} \coloneqq 
				\begin{cases}
					\diff_{t} & \textrm{if $t < \tau_{D}$,} \\
					\cemetery_{D} & \textrm{if $t \geq \tau_{D}$,}
				\end{cases}
				\qquad t \in [0,\infty]
			\end{equation}
			and $\lawdiff_{\cemetery_{D}} \coloneqq \lawdiff_{\cemetery}$,
			where $\oneptcpt{D} = D \cup \{ \cemetery_{D} \}$ denotes the one-point compactification
			of $D$ and $\tau_{D} \coloneqq \dot{\sigma}_{\oneptcpt{\ambient} \setminus D} = \inf \{t \in [0,\infty) \mid \diff_{t} \not\in D\}$.
		\end{enumerate}
	\end{definition}
	
	Let $D$ be a non-empty open subset of $\ambient$. By \cite[Theorem 4.4.3]{FOT},
	$(\form^{D},\domain^{0}(D))$ is a regular symmetric Dirichlet form on $L^{2}(D,\refmeas|_{D})$,
	a subset $\mathcal{N}$ of $D$ is $\form^{D}$-polar if and only if $\mathcal{N}$ is $\form$-polar,
	and a $[-\infty,\infty]$-valued function $f$ defined $\form$-q.e.\ on $D$
	is $\form^{D}$-quasi-continuous on $D$ if and only if $f$ is $\form$-quasi-continuous on $D$.
	By \cite[Theorem 3.4.9]{CF}, the extended Dirichlet space $\domain^{0}(D)_{e}$ of
	$(D,\refmeas|_{D},\form^{D},\domain^{0}(D))$ is identified as
	\begin{equation} \label{e:extdirpart}
		\domain^{0}(D)_{e} = \{ f \in \domain_{e} \mid \textrm{$\wt{f}=0$ $\form$-q.e.\ on $\ambient \setminus D$} \}.
	\end{equation}
	Also, $\diff^{D}$ is an $\refmeas|_{D}$-symmetric Hunt process on $D$ by
	\cite[Theorem A.2.10 and Lemma 4.1.3]{FOT} (see also \cite[Exercise 3.3.7-(ii) and (3.3.4)]{CF}),
	and the Dirichlet form of $\diff^{D}$ is $(\form^{D},\domain^{0}(D))$ by \cite[Theorem 4.4.2]{FOT}.
	
	Assume that the part Dirichlet form $(\form^{D},\domain^{0}(D))$ on $D$ is transient, and
	let $A \subset D$. We define the \emph{($0$-order) capacity} $\Capa_{D}(A)$ of $A$ in $D$ by
	\begin{equation} \label{e:defCapD}
		\Capa_{D}(A) \coloneqq \inf \bigl\{ \form(f,f) \bigm\vert \textrm{$f \in \domain^{0}(D)_{e}$, $f \geq 1$ $\refmeas$-a.e.\ on a neighborhood of $A$} \bigr\},
	\end{equation}
	so that $\Capa_{D}$ is countably subadditive by \cite[the $0$-order version of Lemma 2.1.2 and Theorem A.1.2]{FOT}.
	Then $\Capa_{D}(A)=0$ if and only if $A$ is $\form$-polar (i.e., $\Capa_{1}(A)=0$)
	by \cite[Theorems 2.1.6-(i) and 4.4.3-(ii)]{FOT}.
	By \cite[the $0$-order version of Theorem 2.1.5-(i),(ii)]{FOT}, we have
	\begin{equation}\label{e:defCapD-alt}
		\Capa_{D}(A) = \inf \bigl\{ \form(f,f) \bigm\vert \textrm{$f \in \domain^{0}(D)_{e}$, $\wt{f} \ge 1$ $\form$-q.e.\ on $A$} \bigr\}
	\end{equation}
	and if $\Capa_{D}(A)<\infty$ then there exists a unique function $e_{A,D} \in \domain^{0}(D)_{e}$,
	called the \emph{equilibrium potential} of $A$ in $D$, that attains the infimum in \eqref{e:defCapD-alt}.
	We describe the corresponding equilibrium measures in the following lemma, assuming
	the strong locality of $(\form,\domain)$. The equality \eqref{e:eqmeas} below was
	claimed in \cite[(2.7)]{Fit} without a proof, and here we provide a detailed proof
	of it since it plays an important role in this paper.
	
	\begin{lem} \label{l:eqmeas}
		Assume that $(\ambient,\refmeas,\form,\domain)$ is strongly local, and
		let $D$ be a non-empty open subset of $\ambient$ such that the part Dirichlet form $(\form^{D},\domain^{0}(D))$
		on $D$ is transient. Let $A \Subset D$ be a relatively compact open subset of $D$.
		\begin{enumerate}[\rm(a)]\setlength{\itemsep}{0pt}\vspace{-5pt}
			\item\label{it:eqmeas-inner} There exists a unique $e_{A,D} \in \domain^{0}(D)_{e}$ and a unique Radon measure
			$\lambda^1_{A,D}$ on $\ambient$ charging no $\form$-polar set \textup{(see Definition \ref{d:smooth} below)} such that 
			\begin{equation} \label{e:eqpot1}
				\Capa_{D}(A)=\form( e_{A,D}, e_{A,D}), \quad \wt{e}_{A,D} = 1 \ \textrm{$\form$-q.e.\ on $A$,} \quad \form(u,e_{A,D})= \int_{\ambient} \wt{u}  \,d\lambda^1_{A,D}
			\end{equation}
			for all $u \in \domain^{0}(D)_{e}$.
			Furthermore $\supp_{\ambient}[\lambda_{A,D}^1] \subset \partial A$ and
			\begin{equation} \label{e:eqmass}
				\lambda_{A,D}^1(\ambient)=\lambda_{A,D}^1(\partial A)= \Capa_D(A).
			\end{equation} 
			\item\label{it:eqmeas-outer} Assume in addition that $D$ is relatively compact in $\ambient$. Then there exists a unique
			Radon measure $\lambda^0_{A,D}$ on $\ambient$ charging no $\form$-polar set such that
			\begin{equation} \label{e:eqmeas}
				\form(e_{A,D}, u)= \int_{\partial A}  \wt{u} \, d\lambda^1_{A,D} - \int_{\ambient} \wt{u}\,d\lambda^0_{A,D}
			\end{equation}
			for any $u \in \domain_e \cap L^{\infty}(\ambient,\refmeas)$, where $\lambda^1_{A,D}$ is the measure in part (a).
			Furthermore $\supp_{\ambient}[\lambda_{A,D}^0] \subset \partial D$ and
			\begin{equation} \label{e:eqmass1}
				\lambda_{A,D}^0(\ambient) = \lambda_{A,D}^0(\partial D)=\Capa_D(A).
			\end{equation}
		\end{enumerate}
	\end{lem}
	
	\begin{proof}
		\begin{enumerate}[\rm(a)]\setlength{\itemsep}{0pt}
			\item 
			Note that $(\domain^{0}(D)_{e},\form)$ is a Hilbert space by \cite[Theorem 1.5.3]{FOT}.
			Since $A \Subset D$, the regularity of $(\form,\domain)$ along with \cite[Theorem 2.3.4]{CF} implies that the set
			\[
			\sL_{A,D} \coloneqq \bigl\{f \in \domain^{0}(D)_{e} \bigm| \textrm{$\wt{f} \ge 1$ $\form$-q.e.\ on $A$}\bigr\}
			\]
			is non-empty, closed, convex subset of the Hilbert space  $(\domain^{0}(D)_{e},\form)$. Hence there exists a unique element $	\wt{e}_{A,D} \in \sL_{A,D}$ such that 	$\Capa_D(A)=\form(	{e}_{A,D},	{e}_{A,D})$.
			Since $1 \wedge e_{A,D} \in \sL_{A,D}$ and $\form(1 \wedge e_{A,D},1 \wedge e_{A,D}) \le \form(e_{A,D},e_{A,D})$, we conclude $\wt{e}_{A,D}=1 \wedge \wt{e}_{A,D}$ $\form$-q.e.\ and hence $\wt{e}_{A,D} = 1$ $\form$-q.e.\ on $A$.
			
			Let $v \in \domain^{0}(D)_{e}$ such that $v \ge 0$ $\refmeas$-a.e. Then for any $t>0$,  $	{e}_{A,D} + t v \in \sL_{A,D}$ and hence $\form({e}_{A,D} + t v , {e}_{A,D} + t v ) \ge \form({e}_{A,D},{e}_{A,D})$ or equivalently $\form({e}_{A,D} , v ) + (t/2)\form(v,v) \ge 0$. Letting $t \downarrow 0$, we conclude 
			\[
			\form({e}_{A,D} , v ) \ge 0, \quad \mbox{for all $v \in \domain^{0}(D)_{e}$ such that $v \ge 0$ $\refmeas$-a.e. }
			\]
			The existence of a Radon measure $\lambda^1_{A,D}$ on $D$ satisfying the last equality in \eqref{e:eqpot1} now follows from by applying \cite[Theorem 2.2.5 and Lemma 2.2.10]{FOT} to the Dirichlet form $(\form^D,\domain^0(D))$. We also consider it as a Radon measure on $\ambient$ by setting $\lambda^1_{A,D}(\cdot) \coloneqq \lambda^1_{A,D}(\cdot \cap D)$. This concludes the proof of all claims in \eqref{e:eqpot1}. 
			
			By the strong locality and $\wt{e}_{A,D} = 1$ $\form$-q.e.~on $A$, we conclude that ${e}_{A,D}$ is $\form$-harmonic on $A$.
			By the energy minimizing property of $e_{A,D}$, we have that $e_{A,D}$ is $\form$-harmonic on $D \setminus \ol{A}$.
			Therefore any $u \in \domain \cap (\contfunc_{\mathrm{c}}(A) \cup \contfunc_{\mathrm{c}}(D \setminus \ol{A}))$ we have $\form(u,e_{A,D})=0$, which implies that 
			$\lambda^1_{A,D}(A \cup (D \setminus \ol{A}))=0$, namely $\supp_{\ambient}[\lambda^1_{A,D}] \subset \partial A$. The proof of \eqref{e:eqmass} is contained in \cite[Proof of Proposition 5.21]{BCM}.
			
			\item Let $\phi \in \domain \cap \contfunc_{\mathrm{c}}(\ambient)$ satisfy $\supp_{\ambient}[\phi] \cap A = \emptyset$ and $\phi \le 0$.
			Choose $\psi \in \domain \cap \contfunc_{\mathrm{c}}(\ambient)$ so that $0 \le \psi \le 1$ and $\restr{\psi}{V}=1$, where $V$ is a neighborhood of $\supp_{\ambient}[\phi]$.
			Since $e_{A,D}\psi$ is $\form$-harmonic on $(D^c \setminus A) \cap V$ and $\wt{e}_{A,D}\psi - (\wt{e}_{A,D} \psi + \phi)^{+}=0$ $\form$-q.e.\ on $((D^c \setminus A) \cap V)^c$,
			we have $\form(e_{A,D}\psi, e_{A,D}\psi) = \form(e_{A,D} \psi , (\wt{e}_{A,D} \psi + \phi)^{+} )$ and therefore
			\begin{align*}
				0 & \le \form(({e}_{A,D} \psi + \phi)^{+}-e_{A,D}\psi, ({e}_{A,D} \psi + \phi)^{+}-e_{A,D}\psi) \\
				&= \form(({e}_{A,D} \psi + \phi)^{+}, ({e}_{A,D} \psi + \phi)^{+}) -2 \form(({e}_{A,D} \psi + \phi)^{+}, e_{A,D}\psi) + \form( e_{A,D}\psi, e_{A,D}\psi) \\
				&= \form(({e}_{A,D} \psi + \phi)^{+}, ({e}_{A,D} \psi + \phi)^{+})-\form( e_{A,D}\psi, e_{A,D}\psi) \\
				&\le \form({e}_{A,D} \psi + \phi, {e}_{A,D} \psi + \phi)-\form( e_{A,D}\psi, e_{A,D}\psi) \quad \textrm{(by the Markov property)}\\
				&= \form(\phi,\phi)+2 \form(e_{A,D}\psi,\phi)= \form(\phi,\phi)+2 \form(e_{A,D},\phi) \quad \textrm{(by the strong locality).}
			\end{align*}
			By replacing $\phi$ with $t\phi$ and letting $t \downarrow 0$, we obtain 
			\begin{equation} \label{e:eqm0}
				\form(e_{A,D},\phi) \ge 0 \quad \textrm{for all $\phi \in \domain \cap \contfunc_{\mathrm{c}}(\ambient)$ such that $\phi \le 0$ and $\supp_{\ambient}[\phi] \subset A^c$.}
			\end{equation}
			It follows that there exists a Radon measure $\lambda_{A,D}^0$ on $A^c$ such that for all $\phi \in \domain\cap \contfunc_{\mathrm{c}}(\ambient)$ with $\supp_{\ambient}[\phi] \subset A^c$, we have 
			\begin{equation} \label{e:eqm1}
				\form(\phi,e_{A,D}) = -\int_{A^c} \phi \, d\lambda_{A,D}^0.
			\end{equation}
			Furthermore by the strong locality of $(\form,\domain)$, the $\form$-harmonicity of $e_{A,D}$ on $D^c \setminus A$ and the compactness of $\partial D$, we have 
			\begin{equation} \label{e:eqm2}
				\lambda_{A,D}^0(A^c)= \lambda_{A,D}^0(\partial D) < \infty.
			\end{equation}
			We consider $\lambda_{A,D}^0$ as a finite Borel measure on $\ambient$ by setting $\lambda_{A,D}^0(\cdot) \coloneqq \lambda_{A,D}^0(\cdot \cap A^c)$,
			and then the equality in \eqref{e:eqm2} means that $\supp_{\ambient}[\lambda_{A,D}^0] \subset \partial D$.
			
			%   Arguing similarly as the proof of \eqref{e:eqm0}, for all $\phi \in \domain \cap \contfunc_{\mathrm{c}}(\ambient)$ with $\supp_{\ambient}[\phi] \subset B$ and $\phi \ge 0$, we have 
			%   \[
			%   \form(e_{A,D},\phi) \ge 0.
			%   \]
			%   Thus there exists a Radon measure $\lambda^1_{A,D}$ on $B$ such that for all 
			%   $\phi \in \domain \cap \contfunc_{\mathrm{c}}(\ambient)$ with $\supp_{\ambient}[\phi] \subset B$ and $\phi \ge 0$, we have 
			%   \begin{equation} \label{e:eqm3}
				%   	\form(e_{A,D},\phi) = \int_{B} \phi \, d\lambda_{A,D}^1.
				%   \end{equation}
			As before, we can consider $\lambda_{A,D}^1$ as a Borel measure on $\ambient$ such that 
			\begin{equation} \label{e:eqm4}
				\lambda_{A,D}^1(\ambient)= \lambda_{A,D}^1(\partial A)<\infty.
			\end{equation}
			
			Now let $\phi \in \domain \cap \contfunc_{\mathrm{c}}(\ambient)$ and let $\psi \in \domain \cap \contfunc_{\mathrm{c}}(\ambient)$ satisfy 
			$\restr{\psi}{\unifdom} =1$ for some neighborhood $\unifdom$ of $A$, $0\le \psi \le 1$ on $\ambient$ and $\supp_{\ambient}[\psi] \subset D$. Then 
			\begin{align}
				\form(\phi, e_{A,D}) &= \form(\phi-  \phi\psi, e_{A,D}) + \form(\phi\psi, e_{A,D}) \nonumber \\
				&= -\int_{\partial D} (\phi - \phi \psi) \,d\lambda^0_{A,D} + \int_{\partial A} \phi\psi \, d\lambda^1_{A,D} \nonumber \\
				&=-\int_{\partial D} \phi \,d\lambda^0_{A,D} +\int_{\partial A} \phi \, d\lambda^1_{A,D} \label{e:eqm5}  \quad \mbox{(by \eqref{e:eqm1},\eqref{e:eqm2}, \eqref{e:eqmass},\eqref{e:eqm4}).}
			\end{align}
			Also by \cite[Theorem 4.4.3-(i),(ii) and Lemma 2.2.3]{FOT}, $\lambda_{A,D}^0,\lambda_{A,D}^1$ charge no $\form$-polar set.
			Finally, for any $u \in \domain_e \cap L^\infty(\ambient,\refmeas)$, by \cite[Theorem 2.1.7 and Corollary 1.6.3]{FOT},
			there exists $\{u_n\}_{n \in \bN} \subset \domain \cap \contfunc_{\mathrm{c}}(\ambient)$ with $\norm{u_n}_{\sup} \le \norm{u}_{L^\infty(\ambient,\refmeas)}$,
			$u_n \to \wt{u}$ $\form$-q.e.\ on $\ambient$ and $\lim_{n\to \infty}\form(u-u_n,u-u_n)=0$.
			This along with \eqref{e:eqm5} applied to the sequence $\{u_n\}_{n \in \mathbb{N}}$,
			$\lambda_{A,D}^0(\partial D)<\infty$, $\lambda_{A,D}^1(\partial A)<\infty$ and
			the dominated convergence theorem implies the desired equality \eqref{e:eqmeas}.
			\qedhere\end{enumerate}
	\end{proof}
	
	We close this subsection by introducing the Dirichlet form induced by assigning
	reflected (Neumann) boundary condition, which requires the notion of $\form$-energy measure
	and the space of functions locally in $\domain$ defined as follows. Note that $fg \in \domain$ for any $f,g \in \domain \cap L^{\infty}(\ambient,\refmeas)$
	by \cite[Theorem 1.4.2-(ii)]{FOT}, that $\{(-n)\vee(f\wedge n)\}_{n=1}^{\infty} \subset \domain$
	and $\lim_{n\to\infty}(-n)\vee(f\wedge n)=f$ in norm in $(\domain,\form_{1})$
	for any $f \in \domain$ by \cite[Theorem 1.4.2-(iii)]{FOT}, and that these
	two claims with $(\domain,\form_{1})$ replaced by $(\domain_e,\form)$ hold by
	\cite[Corollary 1.6.3]{FOT}.
	
	\begin{definition}[Energy measure; {\cite[(3.2.13), (3.2.14) and (3.2.15)]{FOT}}]\label{d:EnergyMeas}
		The \emph{$\form$-energy measure} $\Gamma(f,f)$ of $f \in \domain_{e}$
		%associated with $(\ambient,\refmeas,\form,\domain)$
		is defined, first for $f \in \domain \cap L^{\infty}(\ambient,\refmeas)$
		as the unique ($[0,\infty]$-valued) Radon measure on $\ambient$ such that
		\begin{equation}\label{e:EnergyMeas}
			\int_{\ambient} g \, d\Gamma(f,f)= \form(f,fg)-\frac{1}{2}\form(f^{2},g) \qquad \textrm{for all $g \in \domain \cap \contfunc_{\mathrm{c}}(\ambient)$,}
		\end{equation}
		next by
		$\Gamma(f,f)(A) \coloneqq \lim_{n\to\infty}\Gamma\bigl((-n)\vee(f\wedge n),(-n)\vee(f\wedge n)\bigr)(A)$
		for each $A \in \Borel(\ambient)$ for $f \in \domain$, and then by
		$\Gamma(f,f)(A) \coloneqq \lim_{n\to\infty}\Gamma(f_{n},f_{n})(A)$
		for each $A \in \Borel(\ambient)$ for $f \in \domain_{e}$;
		here $\{f_{n}\}_{n \in \mathbb{N}} \subset \domain$ is any sequence such that
		$\lim_{k\wedge l\to\infty}\form(f_{k}-f_{l},f_{k}-f_{l})=0$ and
		$\lim_{n\to \infty}f_{n}= f$ $\refmeas$-a.e.\ on $\ambient$
		(recall Definition \ref{d:ExtDiriSp}). We remark that, if $(\form,\domain)$ is strongly local,
		then $\Gamma(f,f)(\ambient)=\form(f,f)$ for any $f \in \domain_{e}$ by \cite[Lemma 3.2.3]{FOT}.
	\end{definition}
	
	\begin{definition}[Function locally in the form domain and its energy measure] \label{d:domain-local}
		Assume that $(\ambient,\refmeas,\form,\domain)$ is strongly local,
		and let $D$ be an open subset of $\ambient$. We define the
		\emph{space $\domain_{\loc}(D)$ of functions on $D$ locally in $\domain$} as
		\begin{equation}\label{e:Floc}
			\domain_{\loc}(D) \coloneqq \Biggl\{ f \Biggm|
			\begin{minipage}{285pt}
				$f$ is an $m$-equivalence class of $\mathbb{R}$-valued Borel measurable functions
				on $D$ such that $f = f^{\#}$ $m$-a.e.\ on $V$ for some $f^{\#}\in\mathcal{F}$
				for each relatively compact open subset $V$ of $D$
			\end{minipage}
			\Biggr\},
		\end{equation}
		and define the \emph{$\form$-energy measure} $\Gamma_{D}(f,f)$ of $f \in \domain_{\loc}(D)$
		%		associated with $(\ambient,\refmeas,\form,\domain)$
		as the unique Radon measure on $D$ such that
		$\Gamma_{D}(f,f)(A)=\Gamma(f^{\#},f^{\#})(A)$ for any relatively compact
		Borel subset $A$ of $D$ and any $V,f^{\#}$ as in \eqref{e:Floc} with $A\subset V$;
		note that $\Gamma(f^{\#},f^{\#})(A)$ is independent of a particular choice of such $V,f^{\#}$
		by \cite[Corollary 3.2.1]{FOT}.
	\end{definition}
	
	Now we can define the reflected Dirichlet form on an open set as follows.
	
	\begin{definition}[Reflected Dirichlet form] \label{dfn:reflected-form}
		Assume that $(\ambient,\refmeas,\form,\domain)$ is strongly local, and let $D$ be an open
		subset of $\ambient$. We define a linear subspace $\domain(D)$ of $L^{2}(D,\refmeas|_{D})$ by
		\begin{equation} \label{e:defFD}
			\domain(D) \coloneqq \biggl\{f \in \domain_{\loc}(D) \biggm| \int_{D} f^{2}\,dm + \int_{D} \one_{D} \, d\Gamma_{D}(f,f) < \infty \biggr\},
		\end{equation}
		and a non-negative definite symmetric bilinear form $\formrefgen{D} \colon \domain(D) \times \domain(D) \to \mathbb{R}$ by
		\begin{equation} \label{e:neumann}
			\formrefgen{D}(f,g) \coloneqq \frac{1}{4}\biggl(\int_{D} \one_{D} \, d\Gamma_{D}(f+g,f+g) - \int_{D} \one_{D} \, d\Gamma_{D}(f-g,f-g)\biggr).
		\end{equation}
		We call $(\formrefgen{D},\domain(D))$ the \textbf{reflected Dirichlet form} of $(\form,\domain)$ on $D$.
	\end{definition}
	
	The form $(\formrefgen{D},\domain(D))$ need not be a regular symmetric Dirichlet form
	on $L^{2}(\overline{D},\refmeas|_{\overline{D}})$ in general.
	This fails to hold, e.g., when $(\ambient,\refmeas,\form,\domain)$ is the MMD space of the Brownian motion on $\mathbb{R}^{2}$
	and $D = D_{0} \setminus I$ for an open subset $D_{0}$ of $\mathbb{R}^{2}$ and a closed interval $I \subset \mathbb{R}^{2}$ with $D_{0}\cap I\not=\emptyset$;
	see Example \ref{x:slit} below, and also \cite[p.~2793]{BuCh}, which points to another example due to Kolsrud.
	A sufficient condition for this to hold is provided in Theorem \ref{thm:hkeunif}-\eqref{it:hkeunif} below.
	
	\subsection{Sub-Gaussian heat kernel estimates} \label{ss:HKE}
	Let $\scdiff \colon [0,\infty) \to [0,\infty)$ be a homeomorphism such that
	\begin{equation} \label{e:reg}
		C^{-1} \Bigl( \frac{R}{r} \Bigr)^{\beta_1} \le \frac{\scdiff(R)}{\scdiff(r)} \le C \Bigl( \frac{R}{r} \Bigr)^{\beta_2}
	\end{equation}
	for all $0 < r \le R$ for some $C,\beta_{1},\beta_{2}\in(1,\infty)$ with $\beta_{1} \leq \beta_{2}$.
	If necessary, we extend $\scdiff$ by setting $\scdiff(\infty) \coloneqq \infty$.
	Such a function $\scdiff$ is termed a \emph{scale function}.
	For such $\scdiff$, we define $\wt{\scdiff} \colon [0,\infty) \to [0,\infty]$ by
	\begin{equation} \label{e:defPhi}
		\wt{\scdiff}(s)= \sup_{r \in (0,\infty)} \biggl( \frac{s}{r}-\frac{1}{\scdiff(r)} \biggr),
	\end{equation}
	so that $\wt{\scdiff}(0)=0$ and $\wt{\scdiff}(s) \in (0,\infty)$ for any $s\in(0,\infty)$ by \cite[Remark 3.16]{GT12}.
	For example, if $\beta\in(1,\infty)$ and $\scdiff(r)=r^{\beta}$ for any $r\in[0,\infty)$,
	then $\widetilde{\scdiff}(s)=\beta^{-\frac{\beta}{\beta-1}}(\beta-1)s^{\frac{\beta}{\beta-1}}$
	for any $s\in[0,\infty)$.
	
	\begin{definition}[\hypertarget{hke}{$\on{HKE(\scdiff)}$}]\label{d:HKE}
		Let $(\ambient,\refmeas,\form,\domain)$ be a regular Dirichlet space, and
		let $(T_{t})_{t>0}$ denote its associated strongly continuous contraction semigroup.
		A family $\set{p_t}_{t>0}$ of $[0,\infty]$-valued Borel measurable
		functions on $\ambient \times \ambient$ is called the
		\emph{heat kernel} of $(\ambient,\refmeas,\form,\domain)$, if $p_t$ is an integral kernel
		of the operator $T_{t}$ for any $t \in (0,\infty)$, that is, for any $t \in (0,\infty)$ and any $f \in L^{2}(\ambient,\refmeas)$,
		\begin{equation*}
			T_{t} f(x) = \int_X p_t (x, y) f (y)\, dm (y) \qquad \mbox{for $m$-a.e.\ $x \in \ambient$.}
		\end{equation*}
		
		Assuming further that $(\ambient,d,\refmeas,\form,\domain)$ is an MMD space,
		we say that $(\ambient,d,\refmeas,\form,\domain)$ satisfies the \textbf{heat kernel estimates}
		\hyperlink{hke}{$\on{HKE(\scdiff)}$}, if there exist $C_{1},c_{1},c_{2},c_{3},\delta\in(0,\infty)$
		and a heat kernel $\set{p_t}_{t>0}$ of $(\ambient,\refmeas,\form,\domain)$ such that for each $t \in (0,\infty)$,
		\begin{align}\label{e:uhke}
			p_t(x,y) &\le \frac{C_{1}}{m\bigl(B(x,\scdiff^{-1}(t))\bigr)} \exp \biggl( -c_{1} t \wt{\scdiff}\biggl( c_{2}\frac{d(x,y)} {t} \biggr) \biggr)
			\qquad \mbox{for $m$-a.e.\ $x,y \in \ambient$,}\\
			p_t(x,y) &\ge \frac{c_{3}}{m\bigl(B(x,\scdiff^{-1}(t))\bigr)}
			\qquad \mbox{for $m$-a.e.\ $x,y\in \ambient$ with $d(x,y) \le \delta\scdiff^{-1}(t)$,}
			\label{e:nlhke}
		\end{align}
		where $\wt{\scdiff}$ is as defined in \eqref{e:defPhi}.
		%	\begin{equation} \label{e:defPhiRt}
			%		\Phi(R,t) \coloneqq \Phi_{\scdiff}(R,t) \coloneqq \sup_{r>0} \biggl(\frac{R}{r}-\frac{t}{\scdiff(r)}\biggr),
			%		\qquad (R,t)\in[0,\infty)\times(0,\infty).
			%	\end{equation}
	\end{definition}
	
	We recall the following results obtained by the second-named author in \cite{Mur24}
	on the regularity, heat kernel estimates, $1$-capacity and descriptions of the domain
	and the extended Dirichlet space for reflected Dirichlet forms on uniform domains.
	
	\begin{theorem} \label{thm:hkeunif}
		Let $\scdiff$ be a scale function, let $(\ambient,d,\refmeas,\form,\domain)$
		be an MMD space satisfying \hyperlink{VD}{\textup{VD}} and \hyperlink{hke}{$\on{HKE(\scdiff)}$},
		and let $\unifdom$ be a uniform domain in $(\ambient,d)$. Then the following hold:
		\begin{enumerate}[\rm(a)]\setlength{\itemsep}{0pt}\vspace{-5pt}
			\item\label{it:hkeunif} \textup{(\cite[Theorem 2.8]{Mur24})}
			$(\overline{\unifdom},d,\refmeas|_{\overline{\unifdom}},\formrefgen{\unifdom},\domain(\unifdom))$
			is an MMD space satisfying \hyperlink{VD}{\textup{VD}} and \hyperlink{hke}{$\on{HKE(\scdiff)}$},
			where $\domain(\unifdom)$ is considered as a linear subspace of
			$L^{2}(\overline{\unifdom},\refmeas|_{\overline{\unifdom}})$
			via \eqref{eq:volume-doubling-unif-bdry-zero}.
			\item\label{it:formref-cap-equiv} \textup{(\cite[Proposition 5.11-(i)]{Mur24})}
			There exists $C \in (1,\infty)$ such that for any $A \subset \overline{\unifdom}$,
			\begin{equation} \label{eq:formref-cap-equiv}
				\Capa^{\on{ref},\unifdom}_{1}(A) \leq \Capa_{1}(A) \leq C \Capa^{\on{ref},\unifdom}_{1}(A),
			\end{equation}
			where $\Capa^{\on{ref},\unifdom}_{1}(A)$ denotes the $1$-capacity of $A$
			with respect to $(\overline{\unifdom},\refmeas|_{\overline{\unifdom}},\formrefgen{\unifdom},\domain(\unifdom))$.
			\item\label{it:formref-quasi-cont-equiv} \textup{(Cf.\ \cite[Proposition 5.11-(iii)]{Mur24})}
			For each $u \in \domain(\unifdom)_{e}$, let $\widetilde{u}^{\on{ref},\unifdom}$ denote an
			$\formrefgen{\unifdom}$-quasi-continuous $\refmeas|_{\overline{\unifdom}}$-version of $u$. Then
			\begin{align}\label{eq:formref-quasi-cont-equiv}
				\bigl\{ \widetilde{u}^{\on{ref},\unifdom} \bigm| u \in \domain(\unifdom) \bigr\}
				&= \bigl\{ \widetilde{u}|_{\overline{\unifdom}} \bigm| u \in \domain \bigr\}, \\
				\bigl\{ \widetilde{u}^{\on{ref},\unifdom} \bigm| u \in \domain(\unifdom)_{e} \bigr\}
				&= \bigl\{ \widetilde{u}|_{\overline{\unifdom}} \bigm| u \in \domain_{e} \bigr\}
				\label{eq:formref-quasi-cont-equiv-ext}
			\end{align}
			with any two functions defined $\form$-q.e.\ on $\overline{\unifdom}$
			and equal $\form$-q.e.\ on $\overline{\unifdom}$ identified.
			\item\label{it:formref-energy-meas-zero} \textup{(Cf.\ \cite[Theorem 2.9]{Mur24})}
			$\Gamma(u,u)(\partial \unifdom) = 0$ for any $u \in \domain_{e}$.
			In particular, if $\overline{\unifdom} = \ambient$, then
			$(\formrefgen{\unifdom},\domain(\unifdom))=(\form,\domain)$.
		\end{enumerate}
	\end{theorem}
	
	\begin{proof}
		\eqref{it:hkeunif}, \eqref{it:formref-cap-equiv} and \eqref{eq:formref-quasi-cont-equiv}
		are proved in \cite[Theorem 2.8, Proposition 5.11-(i) and Proof of Proposition 5.11-(iii)]{Mur24},
		respectively, and we also have $\Gamma(u,u)(\partial \unifdom) = 0$ for any $u \in \domain$
		by \cite[Theorem 2.9]{Mur24} and then for any $u \in \domain_{e}$ by the definition
		of $\Gamma(u,u)$ presented in Definition \ref{d:EnergyMeas}.
		In particular, if $\overline{\unifdom} = \ambient$, then
		$\domain(\unifdom)=\domain$ by (\eqref{it:formref-cap-equiv} and)
		\eqref{eq:formref-quasi-cont-equiv}, and
		$\formrefgen{\unifdom}(u,u)=\Gamma_{\unifdom}(u,u)(\unifdom)
			=\Gamma(u,u)(\unifdom)+\Gamma(u,u)(\partial \unifdom)=\form(u,u)$
		for any $u \in \domain$.
		
		It thus remains to prove \eqref{eq:formref-quasi-cont-equiv-ext}.
		If $u \in \domain_{e}$, then $\widetilde{u}|_{\overline{\unifdom}}$ is $\formrefgen{\unifdom}$-quasi-continuous
		since $\{ F_{k} \cap \overline{\unifdom} \}_{k \geq 1}$ is an $\formrefgen{\unifdom}$-nest
		for any $\form$-nest $\{ F_{k} \}_{k \geq 1}$ by the lower inequality in \eqref{eq:formref-cap-equiv},
		we see from Definitions \ref{d:ExtDiriSp}, \ref{d:domain-local} and \ref{dfn:reflected-form}
		that $u|_{\unifdom} \in \domain(\unifdom)_{e}$, and therefore $\widetilde{u}|_{\overline{\unifdom}}$ 
		is an $\formrefgen{\unifdom}$-quasi-continuous $\refmeas|_{\overline{\unifdom}}$-version of
		$u|_{\unifdom} \in \domain(\unifdom)_{e}$ by \eqref{eq:volume-doubling-unif-bdry-zero}.
		If $\diam(\unifdom)<\infty$, then the converse inclusion claimed in \eqref{eq:formref-quasi-cont-equiv-ext}
		follows from \eqref{eq:formref-quasi-cont-equiv} and the fact that $\domain(\unifdom)_{e}=\domain(\unifdom)$
		by \eqref{it:hkeunif}, \cite[Proof of Lemma 6.49]{KM23} and \cite[Proof of Proposition 2.9]{HiKu}.
		
		Assume $\diam(\unifdom)=\infty$, let $u \in \domain(\unifdom)_{e}$ and, recalling
		Definition \ref{d:ExtDiriSp}, choose $\{ u_{n} \}_{n\in\mathbb{N}} \subset \domain(\unifdom)$
		so that $\lim_{k\wedge l\to\infty}\formrefgen{\unifdom}(u_{k}-u_{l},u_{k}-u_{l})=0$
		and $\lim_{n\to\infty}u_{n}=u$ $\refmeas$-a.e.\ on $\overline{\unifdom}$.
		Let $E_{Q} \colon L^{2}(\unifdom,\refmeas|_{\unifdom}) \to L^{2}(\ambient,\refmeas)$
		be the linear map defined by \cite[(5.4)]{Mur24} (see also \cite[Lemma 5.6]{Mur24}),
		so that $E_{Q}(f)|_{\unifdom}=f$ for any $f\in L^{2}(\unifdom,\refmeas|_{\unifdom})$
		and by \cite[Proposition 5.8-(c)]{Mur24} and $\diam(\unifdom)=\infty$
		there exists $C_{1} \in (1,\infty)$ such that
		\begin{equation}\label{eq:unifdom-extension-op}
			E_{Q}(f) \in \domain \quad \textrm{and} \quad
			\form(E_{Q}(f),E_{Q}(f)) \leq C_{1} \formrefgen{\unifdom}(f,f)
			\quad\textrm{for any $f \in \domain(\unifdom)$.}
		\end{equation}
		Moreover, since $(\ambient,d,\refmeas,\form,\domain)$ satisfies the
		\emph{Poincar\'{e} inequality} $\on{PI}(\scdiff)$ by \cite[Proof of Theorem 1.2]{GHL15}
		or \cite[Proof of Theorem 3.2]{Lie} (see also \cite[Remark 2.9-(b)]{KM20}),
		it follows from $\lim_{k\wedge l\to\infty}\formrefgen{\unifdom}(u_{k}-u_{l},u_{k}-u_{l})=0$
		and \cite[Proof of Lemma 4.4, the first paragraph]{KM23} that for any
		$(x,r) \in \unifdom \times (0,\infty)$ with $\overline{B(x,r)} \subset \unifdom$,
		\begin{equation}\label{eq:Fe-approx-L2-conv}
			u|_{B(x,r)} \in L^{2}(B(x,r),\refmeas|_{B(x,r)})
			\quad \textrm{and} \quad
			\lim_{n\to\infty}\int_{B(x,r)}(u-u_{n})^{2}\,d\refmeas=0.
		\end{equation}
		By \eqref{eq:Fe-approx-L2-conv}, the definition \cite[(5.4)]{Mur24} of $E_{Q}$
		and \cite[Proposition 3.2-(d)]{Mur24} we can define an extension
		$E_{Q}(u)$ of $u$ to $\ambient$ by \cite[(5.4)]{Mur24} and obtain
		$\lim_{n\to\infty}E_{Q}(u_{n})(x)=E_{Q}(u)(x)\in\mathbb{R}$ for any
		$x \in \ambient \setminus \overline{\unifdom}$ and hence for $\refmeas$-a.e.\ $x\in\ambient$,
		and $\{ E_{Q}(u_{n}) \}_{n\in\mathbb{N}} \subset \domain$ and
		$\lim_{k\wedge l\to\infty}\form(E_{Q}(u_{k})-E_{Q}(u_{l}),E_{Q}(u_{k})-E_{Q}(u_{l}))=0$
		by $\lim_{k\wedge l\to\infty}\formrefgen{\unifdom}(u_{k}-u_{l},u_{k}-u_{l})=0$
		and \eqref{eq:unifdom-extension-op}. Thus $E_{Q}(u) \in \domain_{e}$ and
		$\lim_{n\to\infty}\form(E_{Q}(u_{n}),E_{Q}(u_{n}))=\form(E_{Q}(u),E_{Q}(u))$
		by Definition \ref{d:ExtDiriSp}, hence letting $n\to\infty$ in the inequality
		\eqref{eq:unifdom-extension-op} for $f=u_{n}$ yields the same inequality with
		$u$ in place of $f$, and $\widetilde{E_{Q}(u)}\big|_{\overline{\unifdom}}$ is an
		$\formrefgen{\unifdom}$-quasi-continuous $\refmeas|_{\overline{\unifdom}}$-version of
		$E_{Q}(u)|_{\unifdom} = u \in \domain(\unifdom)_{e}$ by the first paragraph of this proof,
		whence $\widetilde{u}^{\on{ref},\unifdom}=\widetilde{E_{Q}(u)}\big|_{\overline{\unifdom}}$
		$\form$-q.e.\ on $\overline{\unifdom}$ by the $\formrefgen{\unifdom}$-q.e.\ uniqueness of
		$\widetilde{u}^{\on{ref},\unifdom}$ from \cite[Lemma 2.1.4]{FOT} and \eqref{eq:formref-cap-equiv}.
	\end{proof}
	
	\begin{remark} \label{rmk:unifdom-extension-op-Fe}
		The above proof of \eqref{eq:formref-quasi-cont-equiv-ext} in
		Theorem \ref{thm:hkeunif}-\eqref{it:formref-quasi-cont-equiv}
		has shown also the following improvement on \cite[Proposition 5.8-(c)]{Mur24}: 
		\begin{equation}\label{eq:unifdom-extension-op-Fe}
			\mspace{-5mu}\begin{minipage}{416pt}
				\emph{If $\diam(\unifdom)=\infty$, then we can define an extension $E_{Q}(f)$
					of any $f \in \domain(\unifdom)_{e}$ to $\ambient$ by \textup{\cite[(5.4)]{Mur24}}
					and obtain a linear map $E_{Q} \colon \domain(\unifdom)_{e} \to \domain_{e}$
					such that $E_{Q}(\domain(\unifdom)) \subset \domain$ and the inequality in
					\eqref{eq:unifdom-extension-op} holds for any $f \in \domain(\unifdom)_{e}$
					for some $C_{1} \in (1,\infty)$.}
			\end{minipage}
		\end{equation}
		Note that the analogous statement is trivial when $\diam(\unifdom)<\infty$
		since $\domain(\unifdom)_{e}=\domain(\unifdom)$ in this case as we have seen
		in the second paragraph of the above proof of Theorem \ref{thm:hkeunif}.
	\end{remark}
	
	As recalled in Subsection \ref{ss:DF}, the general results \cite[Theorems 7.2.1 and 4.5.3]{FOT}
	from the theory of regular symmetric Dirichlet forms guarantee the existence
	of an associated diffusion with no killing inside which is unique only up to
	a properly exceptional set of starting points. On the other hand, under the
	assumption of \hyperlink{VD}{\textup{VD}} and \hyperlink{hke}{$\on{HKE(\scdiff)}$},
	a continuous heat kernel $p=p_{t}(x,y)$ exists and gives a Markovian transition function
	with the Feller and strong Feller properties, which allow us to define canonically
	an associated diffusion starting from \emph{every} $x \in \ambient$ as we recall below.
	Recall from Notation \ref{ntn:intro}-\eqref{it:topology} that $\contfunc_{0}(\ambient)$
	denotes the space of $\mathbb{R}$-valued continuous functions on $\ambient$ vanishing at infinity.
	
	\begin{prop}\label{p:feller}
		Let $(\ambient,d,\refmeas,\form,\domain)$ be an MMD space satisfying \hyperlink{VD}{\textup{VD}}
		and \hyperlink{hke}{$\on{HKE(\scdiff)}$} for some scale function $\scdiff$. Then the following hold.
		\begin{enumerate}[\rm(a)]\setlength{\itemsep}{0pt}\vspace{-5pt}
			\item\label{it:HKE-conn-irr-cons} $\ambient$ is connected and locally pathwise connected and
			$(\ambient,\refmeas,\form,\domain)$ is irreducible and conservative.
			\item\label{it:HKE-CHK}\textup{(\cite[Theorem 3.1]{BGK})} A (unique) continuous heat kernel
			$p = p_{t}(x,y) \colon (0,\infty) \times \ambient \times \ambient \to [0,\infty)$
			of $(\ambient,\refmeas,\form,\domain)$ exists.
			\item\label{it:HKE-Feller}\textup{(\cite[Proposition 3.2]{Lie})} The Markovian transition function
			$(P_{t})_{t>0}$ on $\ambient$ defined by $P_{t} (x, dy) \coloneqq p_{t}(x,y) \, m(dy)$,
			$(t,x) \in (0,\infty) \times \ambient$, has the \emph{Feller property}:
			$P_{t}(\contfunc_{0}(\ambient)) \subset \contfunc_{0}(\ambient)$ for any $t \in (0,\infty)$
			and $\lim_{t \downarrow 0} \norm{P_{t} f - f}_{\sup} = 0$ for any $f \in \contfunc_{0}(\ambient)$,
			and the \emph{strong Feller property}: $P_{t}f \in \contfunc(\ambient)$
			for any bounded Borel measurable function $f \colon \ambient \to \mathbb{R}$.
			In particular, there exists a diffusion
			$\diff = (\Omega, \events, \{\diff_{t}\}_{t\in[0,\infty]},\{\lawdiff_{x}\}_{x \in \oneptcpt{\ambient}})$
			on $\ambient$ such that $\lawdiff_{x}(\diff_{t} \in dy)= p_{t}(x,y) \, m(dy)$
			for any $(t,x) \in (0,\infty) \times \ambient$, and $\diff$ is \emph{conservative},
			i.e., $\lawdiff_{x}(\diff_{t} \in \ambient)=1$ for any $(t,x)\in(0,\infty)\times\ambient$.
			\item\label{it:HKE-part-CHK-sFeller} Let $\diff$ be a diffusion on $\ambient$ as in \eqref{it:HKE-Feller},
			and let $D$ be a non-empty open subset of $\ambient$. Then a (unique) continuous heat kernel
			$p^{D} = p^{D}_{t}(x,y) \colon (0,\infty) \times D \times D \to [0,\infty)$
			of $(D,\refmeas|_{D},\form^{D},\domain^{0}(D))$ exists, and the part process
			$\diff^{D}$ of $\diff$ on $D$ satisfies the strong Feller property on $D$ and
			$\lawdiff_{x}(\diff^{D}_{t} \in dy)= p^{D}_{t}(x,y) \, m|_{D}(dy)$
			for any $(t,x) \in (0,\infty) \times D$. Moreover, if $D$ is connected, then
			$p^{D}_{t}(x,y) \in (0,\infty)$ for any $(t,x,y) \in (0,\infty) \times D \times D$.
		\end{enumerate}
	\end{prop}
	
	\begin{proof}
		\begin{itemize}\setlength{\itemsep}{0pt}
			\item[\eqref{it:HKE-conn-irr-cons}]$(\ambient,\refmeas,\form,\domain)$
			is irreducible by \eqref{e:nlhke} from \hyperlink{hke}{$\on{HKE(\scdiff)}$}, and
			$\ambient$ is connected and $(\ambient,\refmeas,\form,\domain)$ is conservative
			by \cite[Theorem 7.4 and Lemma 7.3-(a),(b)]{GT12} and \cite[Theorem 3.2]{Lie}.
			Since $(\ambient,d,\refmeas,\form,\domain)$ satisfies \hyperlink{ehi}{$\on{EHI}$}
			by its \hyperlink{VD}{\textup{VD}} and \hyperlink{hke}{$\on{HKE(\scdiff)}$}
			as noted in Remark \ref{r:ehi-hke} below, $\ambient$ is locally pathwise connected
			by \cite[Proposition 5.6]{GH14} and \cite[Remark 5.3 and Lemma 5.2-(a)]{BCM}
			(see also Lemma \ref{l:chain}-\eqref{it:EHI-MD-RBC} below).
			\item[\eqref{it:HKE-CHK}]This is proved in \cite[Theorem 3.1]{BGK}.
			\item[\eqref{it:HKE-Feller}]This is proved in \cite[Proposition 3.2]{Lie}.
			\item[\eqref{it:HKE-part-CHK-sFeller}]The first claim is proved in \cite[Theorem 3.1]{BGK}.
			To show the stated properties of $\diff^{D}$,
			%	For the process $\diff$ whose transition function is both Feller
			%	and strong Feller, \cite[p.~69, Section 1, Proof of Theorem]{Chu} shows
			%	that the part process $\diff^D$ has the semigroup strong Feller property
			%	(as a process on $D$).
			let $(P^{D}_{t})_{t>0}$ denote the Markovian transition function of $\diff^{D}$,
			which satisfies \ref{eq:AC} since $\diff$ satisfies \ref{eq:AC},
			and define a Markovian transition function $(Q^{D}_{t})_{t>0}$ on $D$ by
			$Q^{D}_{t}(x,dy) \coloneqq p^{D}_{t}(x,y)\,\refmeas|_{D}(dy)$, $(t,x)\times(0,\infty)\times D$.
			Then since the Dirichlet form of $\diff^{D}$ is $(\form^{D},\domain^{0}(D))$
			as mentioned after \eqref{e:extdirpart}, we have
			$Q^{D}_{t}(f|_{D}) = P^{D}_{t}(f|_{D}) \leq P_{t}f$ $\refmeas$-a.e.\ on $D$
			for any $f \in L^{2}(\ambient,\refmeas)$ and any $t \in (0,\infty)$, and hence
			$p^{D}_{t}(x,y)\leq p_{t}(x,y)$ for any $(t,x,y)\in(0,\infty)\times D\times D$,
			which together with \hyperlink{VD}{\textup{VD}} and \hyperlink{hke}{$\on{HKE(\scdiff)}$}
			easily implies that $(Q^{D}_{t})_{t>0}$ has the strong Feller property on $D$.
			Now let $f \in \contfunc_{\mathrm{c}}(D)$. Then for any $s,t \in (0,\infty)$ and any $x \in D$,
			by the Markov property of $\diff^{D}$, $P^{D}_{t} f = Q^{D}_{t} f$ $\refmeas$-a.e.\ on $D$
			and $P^{D}_{s}(x,\cdot) \ll \refmeas|_{D}$ we obtain
			\begin{equation*}
				P^{D}_{t}( P^{D}_{s} f )(x) = (P^{D}_{t+s} f)(x) = P^{D}_{s}( P^{D}_{t} f )(x) = P^{D}_{s}( Q^{D}_{t} f )(x),
			\end{equation*}
			and letting $s \downarrow 0$ yields
			\begin{equation} \label{eq:diffD-pD-equal}
				( P^{D}_{t} f )(x) = ( Q^{D}_{t} f )(x)
			\end{equation}
			by the dominated convergence theorem since
			$\lim_{s \downarrow 0}(P^{D}_{s} f)(y) = f(y)$ for any $y \in D$ and
			$\lim_{s \downarrow 0}P^{D}_{s}( Q^{D}_{t} f )(x) = (Q^{D}_{t} f)(x)$
			by the sample-path right-continuity of $\diff^{D}$, $f \in \contfunc_{\mathrm{c}}(D)$,
			and $Q^{D}_{t}f \in \contfunc(D)$ implied by the strong Feller property of
			$Q^{D}_{t}$. We thus conclude from the validity of \eqref{eq:diffD-pD-equal}
			for any $f \in \contfunc_{\mathrm{c}}(D)$ that $P^{D}_{t}( x, \cdot ) = Q^{D}_{t}( x, \cdot )$
			for any $(t,x) \in (0,\infty) \times D$, which together with the strong Feller
			property of $Q^{D}_{t}$ proves the stated properties of $\diff^{D}$.
			%(Note that the continuity of $P_u f$ for $u > 0$ is NOT required in this argument.)
			%Therefore the law of the diffusion of $\diff^D$ is uniquely determined and the transition function defined by the continuous heat kernel on $D$ coincides with the transition function of $\diff^D$. 
			
			Lastly, assume that $D$ is connected, so that $D$ is pathwise connected
			since $\ambient$ is locally pathwise connected by \eqref{it:HKE-conn-irr-cons}.
			If $\diam(D)<\infty$, then $p^{D}_{t}(x,y)>0$ for any $(t,x,y)\in(0,\infty)\times D\times D$
			by \hyperlink{VD}{\textup{VD}}, \eqref{e:uhke} from \hyperlink{hke}{$\on{HKE(\scdiff)}$},
			the properties of $\diff^{D}$ just shown above and \cite[Proposition A.3-(2)]{Kaj10}.
			If $\diam(D)=\infty$, then since $D$ is connected and locally pathwise connected,
			for any $x,y\in D$ we can choose a pathwise connected open subset $D_{0}$ of $D$
			with $\diam(D_{0})<\infty$ so that $x,y\in D_{0}$ and thus
			$p^{D}_{t}(x,y)\geq p^{D_{0}}_{t}(x,y)>0$ for any $t \in (0,\infty)$, completing the proof.
			\qedhere\end{itemize}
	\end{proof}
	
	In view of Proposition \ref{p:feller}, we often impose the following assumption. 
	
	\begin{assumption}\label{a:feller}
		Let $\scdiff$ be a scale function, and let $(\ambient,d,\refmeas,\form,\domain)$
		be an MMD space satisfying \hyperlink{VD}{\textup{VD}} and \hyperlink{hke}{$\on{HKE(\scdiff)}$}.
		We assume that $p = p_{t}(x,y) \colon (0,\infty) \times \ambient \times \ambient \to [0,\infty)$
		is the continuous heat kernel of $(\ambient,\refmeas,\form,\domain)$
		as given in Proposition \ref{p:feller}-\eqref{it:HKE-CHK}, and that
		$\diff = (\Omega, \events, \{\diff_{t}\}_{t\in[0,\infty]},\{\lawdiff_{x}\}_{x \in \oneptcpt{\ambient}})$
		is a diffusion on $\ambient$ with minimum augmented admissible filtration
		$\minaugfilt_{*}=\{\minaugfilt_{t}\}_{t \in [0,\infty]}$,
		life time $\zeta$ and shift operators $\{ \shiftdiff_{t} \}_{t \in [0,\infty]}$
		such that $\lawdiff_{x}(\diff_{t} \in dy)= p_{t}(x,y) \, \refmeas(dy)$
		for any $(t,x)\in(0,\infty)\times\ambient$ as given in Proposition \ref{p:feller}-\eqref{it:HKE-Feller}.
		%	Furthermore, for any uniform domain $\unifdom$ in $\ambient$, we assume that the
		%	reflected diffusion $\diffref$ corresponding to the MMD space
		%	$(\ol{\unifdom},d,\restr{\refmeas}{\ol{\unifdom}},\formref,\domainref)$
		%	is also defined from every starting point in $\ol{\unifdom}$ with a continuous heat kernel, so that the transition
		%	function of $\diffref$ satisfies both the Feller and strong Feller properties as given in Proposition \ref{p:feller}.
	\end{assumption}
	
	\subsection{Harmonic functions and the elliptic Harnack inequality} \label{ssec:harm-func-EHI}
	
	We recall the definition of harmonic functions and the elliptic Harnack inequality.
	
	\begin{definition}[Harmonic function] \label{dfn:harmonic}
		Let $(\ambient,\refmeas,\form,\domain)$ be a strongly local regular Dirichlet space,
		and $D$ an open subset of $\ambient$. We say that a function
		$h \in \domain_{\on{loc}}(D)$ is \emph{$\form$-harmonic} on $D$  if  
		\begin{equation}\label{eq:harmonic} 
			\form(h, v)=0  \quad \textrm{for every $v \in \domain \cap \contfunc_{\mathrm{c}}(D)$.}
		\end{equation}
		Here by the strong locality of $(\form,\domain)$, we can unambiguously define $\form(h,v) \coloneqq \form(h^\#,v)$ where $h^\# \in \domain$ and $h = h^\#$ $\refmeas$-a.e.\ on a neighborhood of  $\supp_{\refmeas}[v]$.
	\end{definition}
	
	\begin{definition}[Elliptic Harnack inequality (EHI)] \label{d:ehi}
		We say that an MMD space $(\ambient,d,\refmeas,\form,\domain)$ satisfies the \textbf{(scale-invariant) elliptic Harnack inequality},
		abbreviated as \hypertarget{ehi}{$\on{EHI}$}, if there exist $C_{H} \in (1,\infty)$ and $\delta \in (0,1)$
		such that for any $(x,r) \in \ambient \times (0,\infty)$ and any $h \in \domain_{\loc}(B(x,r))$
		that is non-negative $\refmeas$-a.e.\ on $B(x,r)$ and $\form$-harmonic on $B(x,r)$,
		\begin{equation} \label{EHI} \tag*{$\on{EHI}$}
			\esssup_{B(x,\delta r)} h \le C_{H} \essinf_{B(x,\delta r)} h.
		\end{equation}
	\end{definition}
	
	There is a close relationship between the heat kernel estimates \hyperlink{hke}{$\on{HKE(\scdiff)}$} and the elliptic Harnack inequality \hyperlink{ehi}{$\on{EHI}$} as we recall below.
	
	\begin{remark} \label{r:ehi-hke}
		If $(\ambient,d,\refmeas,\form,\domain)$ is an MMD space satisfying the volume doubling property \hyperlink{VD}{\textup{VD}} and \hyperlink{hke}{$\on{HKE(\scdiff)}$},
		then it satisfies (the metric doubling property \hyperlink{MD}{\textup{MD}} and) \hyperlink{ehi}{$\on{EHI}$} by \cite[Theorem 1.2]{GHL15} (see also \cite[Theorem 4.5]{KM23}).
		Conversely, if $(\ambient,d,\refmeas,\form,\domain)$ is an MMD space satisfying \hyperlink{MD}{\textup{MD}} and \hyperlink{ehi}{$\on{EHI}$},
		then by \cite[Theorem 7.9]{BCM} (see also \cite{BM18}) there exist a metric $\theta$ on $\ambient$ quasisymmetric to $d$
		and an $\form$-smooth Radon measure $\nu$ on $\ambient$ with full $\form$-quasi-support
		(see Definitions \ref{d:smooth} and \ref{d:quasisupport} below)
		such that the time-changed MMD space $(\ambient,\theta,\nu,\form^{\nu},\domain^{\nu})$,
		where $(\form^{\nu},\domain^{\nu}) \coloneqq (\trform,\trdomain)$ is defined by \eqref{e:def-trace-domain}
		and \eqref{e:def-trace-form} below, satisfies \hyperlink{VD}{\textup{VD}} and \hyperlink{hke}{$\on{HKE(\scdiff)}$}
		for some scale function $\scdiff$.
	\end{remark}
	
	We are often interested in harmonic functions on an open set $V$ with zero (or Dirichlet)
	boundary condition ``along the boundary of a larger open set $U$'' as defined below.
	
	\begin{definition}[Function with Dirichlet boundary condition] \label{d:dbdry}
		Let $(\ambient,\refmeas,\form,\domain)$ be a strongly local regular Dirichlet space,
		and let $V \subset U$ be open subsets of $\ambient$. We define
		\begin{equation} \label{eq:dbdry}
			\domain^0_{\loc}(U,V) \coloneqq \Biggl\{ f \Biggm|
			\begin{minipage}{312.5pt}
				$f$ is an $\refmeas$-equivalence class of $\mathbb{R}$-valued Borel measurable functions
				on $V$ such that $f = f^{\#}$ $\refmeas$-a.e.\ on $A$ for some $f^{\#}\in\domain^0(U)$
				for each open subset $A$ of $V$ with $\overline{A}$ compact and $\overline{A} \cap \overline{U\setminus V} =\emptyset$
			\end{minipage}\Biggr\},
		\end{equation}
		so that $\domain^0_{\loc}(U,V)$ is a linear subspace of $\domain_{\loc}(V)$,
		and call each $u \in \domain^0_{\loc}(U,V)$ a \emph{function on $V$ with Dirichlet boundary condition relative to $U$}.
		Each $u \in \domain^0_{\loc}(U,V)$ that is $\form$-harmonic on $V$ (recall Definition \ref{dfn:harmonic})
		is called an \emph{$\form$-harmonic function on $V$ with Dirichlet boundary condition relative to $U$}.
	\end{definition}
	
	The following lemma shows that harmonicity and Dirichlet boundary condition are preserved under local uniform convergence.
	
	\begin{lem} \label{l:harm-conv}
		Let $(\ambient,\refmeas,\form,\domain)$ be a strongly local regular Dirichlet space.
		\begin{enumerate}[\rm(a)]\setlength{\itemsep}{0pt}\vspace{-5pt}
			\item\label{it:harm-conv} Let $U \subset \ambient$ be open and let $h_n \in \domain_{\on{loc}}(U)$, $n \ge 1$
			be a sequence of locally bounded harmonic functions such that $h_n$ converges to $h$ uniformly on any compact subset of $U$.
			Then $h \in \domain_{\on{loc}}(U)$ and $h$ is $\form$-harmonic on $U$.
			\item\label{it:harm-conv-dbdry} Let $U,V$ be open subsets of $\ambient$ with $V \subset U$ and let $h_n \in \domain_{\on{loc}}^0(U,V), n \ge 1$ be a sequence of bounded harmonic functions on $V$ such that
			$h_n$ converges to $h$ uniformly on $A$ for any $A \subset V$ relatively compact in $\overline{U}$ with $\overline{A} \cap \overline{U \setminus V} = \emptyset$.
			Then $h \in \domain_{\on{loc}}^0(U,V)$ and $h$ is $\form$-harmonic on $V$.
		\end{enumerate}
	\end{lem}
	
	\begin{proof}
		\begin{enumerate}[\rm(a)]\setlength{\itemsep}{0pt}
			\item Let $V$ be relatively compact open subset of $U$.
			Since $\ambient$ is locally compact there is a compact neighborhood $W$ of $\ol{V}$ such that $\ol{V} \subset W \subset U$.
			Since $(\form,\domain)$ is a regular Dirichlet form, there exists $\phi \in \domain \cap \contfunc_{\mathrm{c}}(U)$ such that $0 \le \phi \le 1$, $\restr{\phi}{\ol{V}} \equiv 1$ and $\restr{\phi}{W^c} \equiv 0$.
			Since $h_i$ is locally bounded and $\supp_{\ambient}[\phi]$ is compact, by \cite[Theorem 1.4.2-(ii)]{FOT} we obtain $h_i \phi \in \domain$.
			Since $h_n \to h$ uniformly on compact subsets of $U$, we have that $\phi h_n$ converges to $\phi h$ in $L^2(\ambient,\refmeas)$.
			We claim that $\phi h_n, n \in \bN$ is an $\form_1$-Cauchy sequence that converges to $\phi h \in \domain$.
			To see this, note that by the Leibniz rule \cite[Lemma 3.2.5]{FOT} for $\Gamma$
			and the $\form$-harmonicity of $h_i-h_j$ on $U$,
			\begin{align} \label{e:ha1}
				\form(\phi (h_i-h_j),\phi (h_i-h_j)) &= \int_W (h_i-h_j)^2 \,d\Gamma(\phi,\phi)+ \form(h_i-h_j, \phi^2(h_i-h_j)) \nonumber\\
				&= \int_W (h_i-h_j)^{2} \,d\Gamma(\phi,\phi).
			\end{align}
			Since $h_i$ converges uniformly on $W$, we obtain that $\phi h_i$ is a $\form_1$-Cauchy sequence whose limit is $\phi h_i$.
			By \eqref{e:ha1} and $\lim_{i \to \infty} \phi h_i = h$ $\refmeas$-a.e.\ on $V$, we conclude that $h \in \domain_{\on{loc}}(U)$. 
			
			Let $\psi \in \domain \cap \contfunc_{\mathrm{c}}(U)$.
			Let $V$ be a relatively compact open subset of $U$ with $\supp_{\ambient}[\psi] \subset V$.
			Then choosing $\phi$ as above, by strong locality and harmonicity of $h_i$ we obtain
			\[
			\form(h,\psi)= \form(\phi h,\psi)= \lim_{i \to \infty} \form(\phi h_i,\psi)=\lim_{i \to \infty} \form( h_i,\psi)=0.
			\]
			Therefore $h$ is $\form$-harmonic on $U$.
			\item Let $A \subset V$ be open such that $A$ is relatively compact in $\overline{\unifdom}$ with
			$\overline{A} \cap \overline{U \setminus V} = \emptyset$. Since $\ambient$ is locally compact,
			there exists a neighborhood $W$ of $\ol{A}$ such that $\ol{W}$ is compact and satisfies $\overline{W} \cap \overline{U \setminus V} = \emptyset$.
			Therefore, there exists $\phi \in \domain \cap \contfunc_{\mathrm{c}}(\ambient)$ such that $\phi$ is $[0,1]$-valued, $\restr{\phi}{W} \equiv 1$ and $\supp_{\ambient}[\phi] \cap \overline{U \setminus V} = \emptyset$.
			Let $\wh{h_i} \in \domain^0(U)$ be such that $h_i = \wh{h_i}$ $\refmeas$-a.e.\ on $A$ for all $i \in \bN$.
			By replacing $\wh{h_i}$ with $(-M_i \vee \wh{h_i}) \wedge M_i$, where $M_i= \sup_{A} \abs{h_i}$, we may assume that $\wh{h_i} \in \domain^0(U) \cap L^{\infty}(\ambient,\refmeas)$.
			Therefore $\phi \wh{h_i} \in \domain^0(U)$ has an $\form$-quasi-continuous $\refmeas$-version which vanishes $\form$-a.e.\ on $V^c$ for all $i \in \bN$.
			Therefore $\phi \wh{h_i} \in \domain^0(V)$ for all $i \in \bN$.
			Using the harmonicity of $h_i$ in $V$ and the same argument as used in \eqref{e:ha1}, we conclude that the sequence $\phi \wh{h_i} \in \domain^0(V)$ is $\form_1$-Cauchy and converges to $\phi h \in \domain^0(V)$.
			Since $\phi h = h$ $\refmeas$-a.e.\ on $A$, we conclude that $h \in \domain^0(U,V)$.
			The assertion that $h$ is $\form$-harmonic on $V$ follows from (a).
			\qedhere\end{enumerate}
	\end{proof}
	
	\begin{remark} \label{rmk:harm-conv}
		Let $(\ambient,\refmeas,\form,\domain)$ be a strongly local regular Dirichlet space.
		The argument used in the above proof of Lemma \ref{l:harm-conv} implies also the following facts.
		\begin{enumerate}[\rm(a)]\setlength{\itemsep}{0pt}\vspace{-5pt}
			\item\label{it:rmk:harm-conv} If $U,h_n,h$ are as in Lemma \ref{l:harm-conv}-\eqref{it:harm-conv},
			then for any $\phi \in \domain \cap \contfunc_{\mathrm{c}}(U)$,
			the sequence $\phi h_n \in \domain, n \in \bN$ is $\form_1$-Cauchy and converges to $\phi h\in \domain$.
			\item\label{it:rmk:harm-conv-dbdry} Let $U,V,h_n,h$ be as in Lemma \ref{l:harm-conv}-\eqref{it:harm-conv-dbdry}, and
			extend $h_n,h$ to $V \cup U^{c}$ by setting $\restr{h_n}{U^c} \equiv 0$ for all $n \in \bN$ and $\restr{h}{U^c} \equiv 0$.
			Then for any $\phi \in \domain \cap \contfunc_{\mathrm{c}}(\ambient)$ such that
			$\supp_{\ambient}[\phi] \cap \overline{U \setminus V} = \emptyset$,
			we have $h_n \phi \in \domain$ for all $n \in \bN$ and $h_n \phi$ converges in $\form_1$-norm to $h \phi \in \domain$.
		\end{enumerate}
	\end{remark}
	
	Harnack inequalities are often used along a chain of balls.
	We recall the definition of Harnack chain -- see \cite[Section 3]{JK}.
	For a ball $B=B(x,r)$ in a metric space $(\ambient,d)$ and $\varepsilon \in (0,\infty)$,
	we let $\varepsilon B$ denote the ball $B(x,\varepsilon r)$.
	
	\begin{definition}[Harnack chain; relatively ball connected] \label{d:chain-rbc}
		Let $(\ambient,d)$ be a metric space.
		\begin{enumerate}[\rm(a)]\setlength{\itemsep}{0pt}\vspace{-5pt}
			\item\label{it:chain} Let $D$ be an open subset of $\ambient$ and $M \in (1,\infty)$.
			%For $x,y \in U$, a {\em $M$-Harnack chain from $x_1$ to $x_2$}
			For $x,y \in D$, an \emph{$M$-Harnack chain from $x$ to $y$}
			in $D$ is a sequence of balls $B_1,B_2,\ldots,B_n$ each contained in $D$ such that
			$x \in M^{-1}B_1$, $y \in M^{-1}B_n$ and $M^{-1}B_{i} \cap M^{-1}B_{i+1} \neq \emptyset$ for $i=1,2,\ldots,n-1$. %Note  that consecutive balls must have comparable radius.
			The number $n$ of balls in a Harnack chain is called the {\em length} of the Harnack chain.
			The infimum of the lengths of all $M$-Harnack chains from $x$ to $y$ in $D$ is denoted by $N_D(x,y;M)$.
			\item\label{it:rbc} (\cite[Definition 5.1-(i)]{BCM}) Let $K \in (1,\infty)$. We say that $(\ambient, d)$ is \textbf{$K$-relatively ball connected}
			if for each $\varepsilon \in (0,1)$ there exists $N = N(\varepsilon) \in \mathbb{N}$
			such that for any $(x_{0},R) \in \ambient \times (0,\infty)$ and any
			$x,y \in \overline{B}(x_{0},R) \coloneqq \{ z \in \ambient \mid d(x_{0},z) \leq R\}$
			there exist $\{ z_{i} \}_{i=0}^{N} \subset \ambient$ such that $z_{0}=x$, $z_{N}=y$,
			$B(z_{i}, \varepsilon R) \subset B(x_{0}, KR)$ for any $i \in \{0,\ldots,N\}$
			and $d(z_{i-1}, z_{i}) <  \varepsilon R$ for any $i \in \{1,\ldots,N\}$.
		\end{enumerate}
	\end{definition}
	
	If $K \in (1,\infty)$ and a metric space $(\ambient,d)$ is $K$-relatively ball connected,
	then for any $\varepsilon \in (0,1)$, any $(x_{0},r) \in \ambient \times (0,\infty)$
	and any $x,y \in B(x_{0},r)$, by the triangle inequality we have
	\begin{equation} \label{e:rbc}
		N_{B(x_0,2Kr)}(x,y;\varepsilon^{-1}) \le N(\varepsilon),
	\end{equation}
	where $N(\varepsilon)$ is as given in Definition \ref{d:chain-rbc}-\eqref{it:rbc}.
	
	\begin{remark} \label{r:chain}
		Let $(\ambient,d,\refmeas,\form,\domain)$ be an MMD space satisfying
		\hyperlink{ehi}{$\on{EHI}$} with constants $C_H$ and $\delta$.
		If $u$ is a $[0,\infty)$-valued continuous $\form$-harmonic function
		on an open subset $D$ of $\ambient$, then for any $x_1$, $x_2 \in D$,
		\begin{equation} \label{e:hchain}
			C_H^{-N_D(x_1, x_2; \delta^{-1})} u(x_1) \le u(x_2) \le C_H^{N_D(x_1, x_2; \delta^{-1})} u(x_1).
		\end{equation}
	\end{remark}
	
	The following lemma lists some useful estimates on the lengths of Harnack chains.
	
	\begin{lem} \label{l:chain}
		\begin{enumerate}[\rm(a)]\setlength{\itemsep}{0pt}
			\item\label{it:EHI-MD-RBC} \textup{(\cite[Theorem 5.4]{BCM})} Let $(\ambient,d,\refmeas,\form,\domain)$
			be an MMD space satisfying \hyperlink{MD}{\textup{MD}} and \hyperlink{ehi}{$\on{EHI}$}.
			Then $(\ambient,d)$ is $K$-relatively ball connected for some $K \in (1,\infty)$.
			\item\label{it:chain-unifdom} Let $(\ambient,d)$ be a metric space satisfying
			the metric doubling property \hyperlink{MD}{\textup{MD}}, and let $\unifdom$
			be a $(c_{\unifdom},C_{\unifdom})$-uniform domain in $(\ambient,d)$.
			Then for each $M \in (1,\infty)$ there exists $C \in (0,\infty)$,
			depending only on $c_{\unifdom}$, $C_{\unifdom}$ and $M$, such that for any $x,y \in \unifdom$,
			\begin{equation} \label{e:hc}
				N_\unifdom (x,y;M) \le C \log \biggl( \frac{d(x,y)}{ \min\{\delta_\unifdom(x), \delta_\unifdom (y)\} } + 1 \biggr)  + C.
			\end{equation}
		\end{enumerate}
	\end{lem}
	
	\begin{proof}
		The conclusion in \eqref{it:EHI-MD-RBC} is contained in \cite[Theorem 5.4]{BCM}.
		
		To see \eqref{it:chain-unifdom}, let $\gamma$ be a $(c_U,C_U)$-uniform curve between $x,y \in U$.
		Without loss of generality, we may assume $\delta_U(x) \le \delta_U(y)$. Since
		\[
		\delta_U(z) \ge \max \left( c_U\min\{d(x,z),d(y,z)\}, \delta_{\unifdom}(x)-d(x,z), \delta_{\unifdom}(y)-d(y,z) \right) \mbox{ for any $z\in \gamma$,}
		\]
		we have 
		\begin{equation} \label{e:hc1}
			\delta_{\unifdom}(z) \ge c_U \delta_{\unifdom}(x)/2.
		\end{equation}
		If $d(x,y) \le 4 \delta_{\unifdom}(x)$, we  choose a maximal $M^{-1} c_U \delta_{\unifdom}(x)/2$ subset of $\gamma$. Observing that $\gamma \subset B(x,2C_Ud(x,y)) \subset B(x,8 C_U \delta_{\unifdom}(x))$ and using the metric doubling property we obtain the desired upper bound.
		
		For $i \in \bN$,   choose $z_i \in \gamma$ such that $d(x,z_i)=2^{-i} d(x,y)$ and such that $z_{i+1}$ lies on the subcurve from $x$ to $z_i$. Note that
		\[
		d(z_i,z_{i+1}) \le 2^{-i+1} d(x,y), \quad \delta_{\unifdom}(z_i) \ge  c_U 2^{-i} d(x,y) \quad  \mbox{for all $i \ge 1$.}
		\]
		First we show that 
		\[
		N_\unifdom(z_i,z_{i+1};M) \lesssim 1 \quad \mbox{for all $i \ge 1$.}
		\]
		To see this, we choose a maximal $M^{-1}c_U^2 2^{-i-2}d(x,y)$ subset $N_i$ of a $(c_U,C_U)$-uniform curve $\gamma_i$ from $z_i$ to $z_{i+1}$. 
		Since the balls $\{B(n,M^{-1}c_U^2 2^{-i-2}d(x,y)): n \in N_i\}$ cover $\gamma_i$ and $\diam(\gamma_i) \le C_U 2^{-i+1}d(x,y)$, and are contained in $U$ by \eqref{e:hc1},  the
		metric doubling property \cite[Exercise 10.17]{Hei} implies that 
		\begin{equation} \label{e:hc0}
			N_\unifdom(z_i,z_{i+1};M) \lesssim \#N_i \lesssim 1 \quad \mbox{for all $i \ge 1$.}
		\end{equation}
		
		Let $k\in \bN$ be the smallest number such that $z_{k+1} \in B(x, M^{-1} \delta_{\unifdom}(x))$, so that $k \asymp 1+ \log \left( \frac{d(x,y)}{\delta_{\unifdom}(x)}+1\right)$. By joining $M$-Harnack chains of length $N_\unifdom(z_i,z_{i+1};M)$ from $z_i$ to $z_{i+1}$ successively and using the ball $B(x, M^{-1} \delta_{\unifdom}(x))$, we obtain a $M$-Harnack chain from $x$ to $z_1$ they yields the estimate 
		\begin{equation} \label{e:hc2}
			N_U(x,z_1;M) \le 1 + \sum_{i=1}^k N_\unifdom(z_i,z_{i+1};M) \lesssim \log \left( \frac{d(x,y)}{\delta_\unifdom(x)} + 1 \right)  +1.
		\end{equation}
		
		Similarly for $i \in \bN$, choose    $w_i \in \gamma$ such that $d(y,w_i)=2^{-i} d(x,y)$ and such that $w_{i+1}$ lies on the subcurve from $w_i$ to $y$.  Similar to \eqref{e:hc2}, we obtain
		\begin{equation} \label{e:hc3}
			N_U(y,w_1;M) \le 1 + \sum_{i=1}^k N_\unifdom(w_i,w_{i+1};M) \lesssim \log \left( \frac{d(x,y)}{\delta_\unifdom(y)} + 1 \right)  +1.
		\end{equation}
		Since $\delta_{\unifdom}(z_1) \wedge \delta_{\unifdom}(w_1) \ge c_U d(x,y)/2$ and $d(z_1,w_1) \le 2 d(x,y)$, by the same argument as \eqref{e:hc0},
		we have 
		\begin{equation} \label{e:hc4}
			N_\unifdom(z_1,w_1;M) \lesssim 1.
		\end{equation}
		By \eqref{e:hc2}, \eqref{e:hc3} and \eqref{e:hc4}, we conclude \eqref{e:hc}.
	\end{proof}
	
	We record a few more consequences of Harnack chaining.
	\begin{lem} \label{l:hchain}
		Let $(\ambient,d,\refmeas,\form,\domain)$ be an MMD space satisfying
		\hyperlink{MD}{\textup{MD}} and \hyperlink{ehi}{$\on{EHI}$},
		and let $\unifdom$ be a uniform domain in $(\ambient,d)$.
		Then there exist $A_0,A_1,C_1 \in (1,\infty)$ and $\gamma \in (0,\infty)$
		such that for any $\xi \in \partial \unifdom$, any $0<r< R<\diam(\unifdom)/A_1$
		and any continuous function $h\colon\unifdom \cap B(\xi,A_{0} R) \to (0,\infty)$
		that is $\form$-harmonic on $\unifdom \cap B(\xi,A_{0} R)$, with $\xi_{R},\xi_{r}$ as in Lemma \ref{l:xir},
		\begin{equation} \label{e:hchain1}
			C_1^{-1} \Bigl(\frac{r}{R}\Bigr)^\gamma h(\xi_r) \le h(\xi_R) \le C_1 \Bigl(\frac{R}{r}\Bigr)^\gamma h(\xi_r).
		\end{equation}
		Furthermore if $\xi_R, \xi_R' \in \unifdom$ are two points that satisfy the conclusion of Lemma \ref{l:xir}, that is 
		\begin{equation*}
			d(\xi,\xi_R)=d(\xi,\xi_R')=R \quad \textrm{and} \quad \delta_{\unifdom}(\xi_R) \wedge  \delta_{\unifdom}(\xi_R')> \frac{c_{U}R}{2},
		\end{equation*}
		then 
		\begin{equation} \label{e:hchain2}
			C_1^{-1} h(\xi_R') \le h(\xi_R) \le C_1 h(\xi_R').
		\end{equation}
	\end{lem}
	
	\begin{proof} 
		Let $\delta \in (0,1)$ denote the constant in \hyperlink{EHI}{EHI}.
		By Lemma \ref{l:chain}-\eqref{it:chain-unifdom}, for any $\xi \in \partial \unifdom$
		and any $0<r<R$ we have $N_{\unifdom} (\xi_r,\xi_R; \delta^{-1}) \le C_1$, where $C_1$
		depends only on $\delta$ and the constants associated to the uniformity of $\unifdom$.
		By Lemma \ref{l:xir} and the proof of Lemma \ref{l:chain}-\eqref{it:chain-unifdom}, there exist
		$A_0,A_1 \in (1,\infty)$ depending only on $\delta$ and the constants associated to the uniformity
		of $\unifdom$ such that for all $\xi \in \partial \unifdom$ and all $0<r<R<\diam(\unifdom)$,
		\begin{equation} \label{e:hch1}
			N_{\unifdom} (\xi_r,\xi_R; \delta^{-1}) \le N_{\unifdom \cap B(\xi,A_0 R)} (\xi_r,\xi_R; \delta^{-1}) \le C_1 (1+ \log(R/r)).
		\end{equation}
		The estimate \eqref{e:hchain1} now follows from \eqref{e:hch1} and Remark \ref{r:chain}.
		The estimate \eqref{e:hchain2} also follows from the same argument.
	\end{proof}
	
	\subsection{Trace Dirichlet form} \label{ss:trace}
	
	Throughout this subsection, we assume that $(\ambient,\refmeas,\form,\domain)$
	is a regular Dirichlet space. Recall that the $1$-capacity $\Capa_1(A)$ of $A \subset \ambient$
	with respect to $(\ambient,\refmeas,\form,\domain)$ is defined by \eqref{e:defCap1}.
	
	\begin{definition}[Smooth measure] \label{d:smooth}
		A Radon measure $\nu$ on $\ambient$, i.e., a Borel measure $\nu$ on $\ambient$ which is finite on
		any compact subset of $\ambient$, is said to be \emph{$\form$-smooth} if $\nu$ charges no $\form$-polar set
		(that is, $\nu(A)=0$ for any $A \in \Borel(\ambient)$ with $\Capa_{1}(A)=0$).
		%	We recall that every positive Radon measure charging no set of capacity zero is smooth. %This will be satisfied by our construction.
	\end{definition}
	
	For example, the $\form$-energy measure $\Gamma(f,f)$ of any $f \in \domain_{e}$ is $\form$-smooth by
	\cite[Lemma 3.2.4]{FOT}. An essential feature of an $\form$-smooth Radon measure $\nu$ on $\ambient$
	is that the $\nu$-equivalence class of each $f\in\domain_{e}$ is canonically determined
	by considering an $\form$-quasi-continuous $\refmeas$-version $\widetilde{f}$ of $f$,
	which is $\form$-q.e.\ unique by \cite[Lemma 2.1.4]{FOT} and thus indeed $\nu$-a.e.\ unique.
	
	We say that a subset $D$ of $\ambient$ is \emph{$\form$-quasi-open}
	if there exists an $\form$-nest $\{F_{k}\}_{k\in\mathbb{N}}$ such that $D \cap F_k$
	is an open subset of $F_{k}$ in the relative topology of $F_{k}$
	inherited from $\ambient$ for each $k \in \mathbb{N}$. The complement in
	$\ambient$ of an $\form$-quasi-open set is said to be \emph{$\form$-quasi-closed}.
	Now we recall the definition of an $\form$-quasi-support of an $\form$-smooth Radon measure.
	
	\begin{definition}[Quasi-support; {\cite[(4.6.3) and (4.6.4)]{FOT}}, {\cite[Definition 3.3.4]{CF}}] \label{d:quasisupport}
		Let $\nu$ be an $\form$-smooth Radon measure on $\ambient$.
		A subset $F$ of $\ambient$ is said to be an \emph{$\form$-quasi-support} of $\nu$
		if the following two conditions hold:
		\begin{enumerate}[\rm(a)]\setlength{\itemsep}{0pt}\vspace{-5pt}
			\item\label{it:quasisupport-1} $F$ is $\form$-quasi-closed and $\nu(\ambient \setminus F)=0$.
			\item\label{it:quasisupport-2} If $\widetilde{F} \subset \ambient$ is $\form$-quasi-closed and
			$\nu(\ambient \setminus \widetilde{F})=0$, then $\Capa_{1}(F \setminus \widetilde{F})=0$.
		\end{enumerate}
		By definition, an $\form$-quasi-support of $\nu$ is unique up to $\form$-q.e.\ equivalence;
		that is, if $F_{1}$ and $F_{2}$ are $\form$-quasi-supports of $\nu$, then
		$\Capa_{1}((F_{1} \setminus F_2) \cup (F_{2} \setminus F_{1}))=0$. Furthermore
		by \cite[Theorem 4.6.3]{FOT}, an $\form$-quasi-support of $\nu$ indeed exists.
	\end{definition}
	
	The $\form$-quasi-support of an $\form$-smooth Radon measure can be described more
	explicitly in terms of the corresponding positive continuous additive functional (PCAF)
	of a Hunt process $\diff$ associated with $(\form,\domain)$, as we recall below
	from \cite[Sections A.3 and 4.1]{CF} and \cite[Section 5.1]{FOT}.
	In the rest of this section, we fix an $\refmeas$-symmetric Hunt process 
	$\diff = (\Omega, \events, \{\diff_{t}\}_{t\in[0,\infty]},\{\lawdiff_{x}\}_{x \in \oneptcpt{\ambient}})$
	on $\ambient$ whose Dirichlet form is $(\form,\domain)$, with minimum augmented
	admissible filtration $\minaugfilt_{*}=\{\minaugfilt_{t}\}_{t \in [0,\infty]}$,
	life time $\zeta$ and shift operators $\{ \shiftdiff_{t} \}_{t \in [0,\infty]}$.
	
	A collection $A=\{A_{t}\}_{t\in[0,\infty)}$ of $[0,\infty]$-valued random variables
	on $\Omega$ is called a \emph{positive continuous additive functional} (\emph{PCAF} for short)
	of $\diff$, if the following three conditions hold:
	\begin{enumerate}[\rm(i)]\setlength{\itemsep}{0pt}\vspace{-5pt}
		\item\label{it:PCAF-adapted} $A_{t}$ is $\minaugfilt_{t}$-measurable for any $t \in [0,\infty)$.
		\item\label{it:PCAF-qe-as} There exist $\Lambda \in \minaugfilt_{\infty}$ and a properly exceptional set $\mathcal{N} \subset \ambient$ for $\diff$
		such that $\lawdiff_{x}(\Lambda)=1$ for any $x \in \ambient \setminus \mathcal{N}$ and $\shiftdiff_{t}(\Lambda) \subset \Lambda$ for any $t \in [0,\infty)$.
		\item\label{it:PCAF-sample-path-properties} For any $\omega\in\Lambda$, $[0,\infty) \ni t \mapsto A_{t}(\omega)$ is a $[0,\infty]$-valued continuous function
		with $A_{0}(\omega)=0$ such that for any $s,t \in [0,\infty)$,
		$A_{t}(\omega)<\infty$ if $t < \zeta(\omega)$,
		$A_{t}(\omega)=A_{\zeta(\omega)}(\omega)$ if $t \geq \zeta(\omega)$,
		and $A_{t+s}(\omega)=A_{t}(\omega)+A_{s}(\shiftdiff_{t}(\omega))$.
	\end{enumerate}
	The sets $\Lambda$ and $\mathcal{N}$ are referred to as a \emph{defining set} and
	an \emph{exceptional set}, respectively, of the PCAF $A$. Note that then
	$\Lambda \cap \{ \dot{\sigma}_{\mathcal{N}} \wedge \hat{\sigma}_{\mathcal{N}} = \infty \}$
	is easily seen to be a defining set of $A$ and belongs to $\minaugfilt_{0}$, and
	recall that $\mathcal{N} \subset \mathcal{N}_{1}$ for some properly exceptional set
	$\mathcal{N}_{1}\in\Borel(\ambient)$ for $\diff$ by \eqref{eq:cap-zero-properly-exceptional}.
	Thus by replacing $\mathcal{N}$ with such $\mathcal{N}_{1}$ and then $\Lambda$
	with $\Lambda \cap \{ \dot{\sigma}_{\mathcal{N}_{1}} \wedge \hat{\sigma}_{\mathcal{N}_{1}} = \infty \}$,
	we can always choose a defining set $\Lambda$ and an exceptional set $\mathcal{N}$
	of a given PCAF of $\diff$ so that $\Lambda \in \minaugfilt_{0}$, $\mathcal{N} \in \Borel(\ambient)$
	and $\Lambda \subset \{ \dot{\sigma}_{\mathcal{N}} \wedge \hat{\sigma}_{\mathcal{N}} = \infty \}$.
	If $\mathcal{N}$ can be taken to be the empty set $\emptyset$, then
	we say that $A$ is a \emph{PCAF in the strict sense} of $\diff$.
	
	% Given a PCAF $A$ and a Borel function $f\ge 0$, we denote $(fA)_t \coloneqq \int_0^tf(X_s)\,dA_s$, which is still a random variable.
	
	By \cite[Theorem A.3.5-(i) and Theorem 4.1.1-(i)]{CF} (see also \cite[Theorem 4.1.1-(iii)]{CF} and \cite[Theorem 5.1.3]{FOT}),
	for each PCAF $A$ of $\diff$ there exists a unique Borel measure $\nu$ on $\ambient$,
	called the \emph{Revuz measure} of $A$, such that
	\begin{equation}\label{Revuz}
		\int_{\ambient}f\,d\nu = \lim_{t \downarrow 0}\frac{1}{t}\expdiff_{\refmeas}\biggl[\int_{0}^{t}f(\diff_{s})\,dA_{s}\biggr]
		%		\expdiff_{h\cdot\mu} \biggl[ \int_0^t f(X_s(\omega))\,dA_s(\omega) \biggr] = \int_0^t\langle f\cdot\nu,P_s h\rangle \, ds,
	\end{equation}
	for any Borel measurable function $f \colon \ambient \to [0,\infty]$,
	%	where $P_s$ denotes the Markov semigroup corresponding to the Hunt process.
	and then this measure $\nu$ charges no $\form$-polar set and satisfies $\nu(F_{k})<\infty$
	for any $k\in\mathbb{N}$ for some $\form$-nest $\{F_{k}\}_{k\in\mathbb{N}}$.
	Conversely, by \cite[Lemma 5.1.8 and Theorem 5.1.3]{FOT} (see also \cite[Theorem 4.1.1-(ii)]{CF}),
	given an $\form$-smooth Radon measure $\nu$ on $\ambient$, there exists a PCAF $A$ of $\diff$
	whose Revuz measure is $\nu$, and any two such PCAFs
	$A=\{A_{t}\}_{t\in[0,\infty)},A'=\{A'_{t}\}_{t\in[0,\infty)}$ of $\diff$
	are \emph{equivalent}, i.e., have a common defining set $\Lambda$
	and a common exceptional set $\mathcal{N}$ such that
	$A_{t}(\omega)=A'_{t}(\omega)$ for any $(t,\omega)\in[0,\infty)\times\Lambda$.
	Moreover, if $\diff$ satisfies \ref{eq:AC}, then PCAFs in the strict sense of
	$\diff$ satisfy a pointwise analogue of \eqref{Revuz} as in the following proposition.
	
	\begin{prop}\label{prop:Revuz-correspondence-pointwise}
		Assume that $\diff$ satisfies \ref{eq:AC}, let $A=\{A_{t}\}_{t\in[0,\infty)}$ be
		a PCAF in the strict sense of $\diff$, let $\nu$ be the Revuz measure of $A$, and
		let $D$ be an open subset of $\ambient$. Then for any $(t,x)\in(0,\infty]\times D$
		and any Borel measurable function $f\colon D\to[0,\infty]$,
		\begin{equation}\label{eq:Revuz-correspondence-pointwise}
			\expdiff_{x}\biggl[\int_{0}^{t\wedge\tau_{D}}f(\diff_{s})\,dA_{s}\biggr]
			=\int_{0}^{t}\int_{D}p^{D}_{s}(x,y)f(y)\,\nu(dy)\,ds
		\end{equation}
		\textup{(note that $\nu$ is $\sigma$-finite)}, where $p^{D}$ denotes the unique Borel measurable function
		$p^{D}=p^{D}_{t}(x,y)\colon(0,\infty)\times D \times D\to[0,\infty]$
		satisfying \eqref{eq:AC-transition-density} for the part process $\diff^{D}$ of $\diff$ on $D$
		\textup{(recall that $\diff^{D}$ is an $\refmeas|_{D}$-symmetric Hunt process on $D$ and satisfies \ref{eq:AC})}.
	\end{prop}
	
	\begin{proof}
		Let $( P^{D}_{s} )_{s>0}$ denote the Markovian transition function of $\diff^{D}$. Then we obtain
		\begin{equation}\label{eq:Revuz-correspondence-pointwise-proof}
			\begin{split}
				&\expdiff_{x}\biggl[\int_{0}^{t\wedge\tau_{D}} f(\diff_{s})\,dA_{s}\biggr]
				= \lim_{\delta \downarrow 0} \expdiff_{x}\biggl[ \int_{\delta\wedge\tau_{D}}^{t\wedge\tau_{D}} f(\diff_{s}) \,dA_{s} \biggr] \\
				&= \lim_{\delta \downarrow 0} \expdiff_{p^{D}_{\delta}(x,\cdot)\cdot \refmeas|_{D}}\biggl[ \int_{0}^{(t-\delta)\wedge\tau_{D}} f(\diff_{s}) \,dA_{s} \biggr] \quad \textrm{(by the Markov property of $\diff$ and \ref{eq:AC})} \\
				&= \lim_{\delta \downarrow 0}\int_{0}^{t-\delta} \int_{D} \bigl( P^{D}_{s} p^{D}_{\delta}(x,\cdot) \bigr)(y)f(y) \,\nu(dy)\,ds \quad \textrm{(by \cite[(4.1.25)]{CF})} \\
				&= \lim_{\delta \downarrow 0}\int_{\delta}^{t} \int_{D} p^{D}_{s}(x,y) f(y) \,\nu(dy)\,ds
				= \int_{0}^{t} \int_{D} p^{D}_{s}(x,y)f(y) \,\nu(dy)\,ds.
				\qedhere\end{split}
		\end{equation}
	\end{proof}
	
	Now let $\nu$ be an $\form$-smooth Radon measure on $\ambient$ and
	let $A=\{A_{t}\}_{t\in[0,\infty)}$ be a PCAF of $\diff$ whose Revuz measure is $\nu$
	with a defining set $\Lambda \in \minaugfilt_{0}$ and an exceptional set $\mathcal{N} \in \Borel(\ambient)$
	such that $\Lambda \subset \{ \dot{\sigma}_{\mathcal{N}} \wedge \hat{\sigma}_{\mathcal{N}} = \infty \}$.
	Since $\{ A_{t}\one_{\Lambda} \}_{t\in[0,\infty)}$ is easily seen to be a PCAF
	equivalent to $A$ with defining set $\Lambda \cup \{\zeta = 0\} \in \minaugfilt_{0}$
	and exceptional set $\mathcal{N}$, we may and do assume without loss of generality that
	$A_{t}(\omega)=0$ for any $(t,\omega)\in[0,\infty)\times(\Omega\setminus\Lambda)$
	and that $\{ \zeta=0 \} \subset \Lambda$. Then the \emph{support} $F$ of $A$ defined by
	\begin{equation} \label{e:supportA}
		F \coloneqq \{x \in \ambient \setminus \mathcal{N} \mid \lawdiff_{x}(R=0)=1 \},
		\quad \textrm{where $R \coloneqq \inf \{t \in (0,\infty) \mid A_{t} > 0\}$,}
	\end{equation}
	is $\diff$-nearly Borel measurable and $\form$-quasi-closed as shown in \cite[the paragraph of (5.2.1)]{CF}
	and in fact an $\form$-quasi-support of $\nu$ by \cite[Theorem 5.1.5]{FOT} or
	\cite[Theorem 5.2.1-(i)]{CF}. Moreover, the \textbf{time-changed process}
	$\difftr=\bigl(\widecheck{\Omega},\widecheck{\events},\{\difftr_{t}\}_{t\in[0,\infty]},\{\lawdiff_{x}\}_{x \in \oneptcpt{F}}\bigr)$
	of $\diff$ by the PCAF $A$, defined for $(t,\omega)\in[0,\infty]\times\Omega$ by
	\begin{align}
		\tau_{t}(\omega) & \coloneqq \inf\{ s \in (0,\infty) \mid A_{s}(\omega)>t\}, \mspace{12mu}
		\difftr_{t}(\omega) \coloneqq \diff_{\tau_{t}(\omega)}(\omega), \mspace{12mu}
		\widecheck{\zeta}(\omega) \coloneqq A_{\infty}(\omega) \coloneqq \lim_{s\to\infty}A_{s}(\omega), \nonumber \\
		\widecheck{\Omega}& \coloneqq \Lambda \cap \bigl\{ \textrm{$\difftr_{s} \in \oneptcpt{F}$ for any $s\in[0,\infty)$} \bigr\}, \mspace{26mu}
		\widecheck{\events} \coloneqq \minaugfilt_{\infty}|_{\widecheck{\Omega}}, \mspace{26mu}
		\widecheck{\shiftdiff}_{t}(\omega) \coloneqq \shiftdiff_{\tau_{t}(\omega)}(\omega),
		\label{eq:time-changed-process}
	\end{align}
	is a $\nu$-symmetric right-continuous strong Markov process on $(\oneptcpt{F},\Borel^{*}(\oneptcpt{F}))$
	with life time $\widecheck{\zeta}$ and shift operators $\bigl\{\widecheck{\shiftdiff}_{t}\bigr\}_{t\in[0,\infty]}$
	by \cite[Theorems A.3.9 and 5.2.1-(ii)]{CF}, where $\oneptcpt{F} \coloneqq F\cup\{\cemetery\}$.
	More precisely, from \cite[Proposition A.3.8-(iv),(vi)]{CF} we easily obtain
	\begin{equation}\label{eq:time-changed-process-in-support}
		\bigl\{ \textrm{$\difftr_{s} \in \oneptcpt{F}$ for any $s\in[0,\infty)$} \bigr\} \in \minaugfilt_{0}, \quad
		\lawdiff_{x}\bigl( \textrm{$\difftr_{s} \in \oneptcpt{F}$ for any $s\in[0,\infty)$} \bigr) = 1
		\mspace{10mu}\textrm{for any $x \in \oneptcpt{\ambient}$,}
	\end{equation}
	$\tau_{t}$ is an $\minaugfilt_{*}$-stopping time and $\difftr_{t}$ is
	$\minaugfilt_{\tau_{t}}/\Borel^{*}(\oneptcpt{\ambient})$-measurable
	for any $t\in[0,\infty]$ by \cite[Proposition A.3.8-(i) and Exercise A.1.20-(ii)]{CF}, the family
	$\widecheck{\minaugfilt}_{*} \coloneqq \bigl\{\widecheck{\minaugfilt}_{t}\bigr\}_{t\in[0,\infty]}$ defined by
	\begin{equation}\label{eq:time-changed-filtration}
		\widecheck{\minaugfilt}_{t} \coloneqq \minaugfilt_{\tau_{t}},\quad t\in[0,\infty],
	\end{equation}
	is a right-continuous filtration in $\Omega$ by \cite[Proposition A.3.8-(iii)]{CF},
	and $\difftr$ is strong Markov with respect to $\widecheck{\minaugfilt}_{*}$
	by \cite[Theorem A.3.9]{CF}.
	
	In this situation, it turns out that the Dirichlet form of the time-changed process
	$\difftr$ is identified as the trace Dirichlet form of $(\form,\domain)$ on
	$L^{2}(F,\nu)$, whose definition given in Definition \ref{d:traceDF} below involves
	the hitting distribution of $\diff$ to $F$ defined as follows.
	
	\begin{definition}[Hitting distribution; harmonic measure] \label{d:hmeas}
		Let $F$ be an $\diff$-nearly Borel measurable $\form$-quasi-closed subset of $\ambient$.
		Recalling the $\minaugfilt_{*}$-stopping time $\sigma_{F}$ from \eqref{eq:hitting-times},
		we define the \emph{($0$-order) hitting distribution} $H_{F}$ of $\diff$ to $F$ by
		\begin{equation} \label{e:defhitdist-meas}
			H_{F}(x,A) \coloneqq \lawdiff_{x}( \textrm{$\diff_{\sigma_{F}} \in A$, $\sigma_{F}<\infty$} ), \quad \textrm{$x\in \ambient$, $A \in \Borel(\ambient)$.}
		\end{equation}
		Then by \cite[Theorem 3.4.8]{CF}, letting $\widetilde{u}$ denote any $\form$-quasi-continuous
		$\refmeas$-version of $u \in \domain_{e}$, we can define an $\form$-q.e.\ defined,
		$\form$-quasi-continuous function $H_{F} \wt{u} \in \domain_{e}$ by
		\begin{equation} \label{e:defhitdist}
			H_{F} \wt{u}(x) \coloneqq \expdiff_{x}[\wt{u}(\diff_{\sigma_{F}}) \one_{ \{ \sigma_{F}<\infty \} }],
		\end{equation}
		which is independent of $\widetilde{u}|_{\ambient \setminus F}$ for $\form$-q.e.\ $x \in \ambient$
		since $\lawdiff_{x}( \textrm{$\diff_{\sigma_{F}} \in \oneptcpt{\ambient} \setminus F$, $\sigma_{F}<\infty$} ) = 0$
		for $\form$-q.e.\ $x \in \ambient$ by \cite[Theorem 3.3.3-(i)]{CF} and \cite[Lemma A.2.7]{FOT}, and
		$H_{F} \widetilde{u}$ is \emph{$\form$-harmonic} on $\ambient \setminus F$, i.e., satisfies
		\begin{equation} \label{eq:hit-dist-harm-test}
			\form(H_{F} \widetilde{u},f)=0 \qquad \textrm{for any $f\in\domain_e$ with $\widetilde{f}=0$ $\form$-q.e.\ on $F$.}
		\end{equation}
		Moreover, since $\lawdiff_{x}(\sigma_{F}=0)=1$ for $\form$-q.e.\ $x\in F$ by
		\cite[Theorems A.2.6-(i), 4.1.3 and 4.2.1-(ii)]{FOT}, it follows from
		\eqref{e:defhitdist} and \eqref{eq:hit-dist-harm-test} that
		\begin{equation} \label{eq:hit-dist-harm-var}
			H_{F}\widetilde{u}=\widetilde{u} \quad \textrm{$\form$-q.e.\ on $F$}
			\qquad \textrm{and} \qquad
			\form(H_{F} \widetilde{u},H_{F} \widetilde{u}) \leq \form(u,u).
		\end{equation}
		Lastly, for an open subset $D$ of $\ambient$, we define the
		\textbf{$\form$-harmonic measure} $\hmeas{D}{x}$ of $D$ with base point $x\in D$ by
		\begin{equation} \label{e:defharm}
			\hmeas{D}{x}(A) \coloneqq H_{\ambient \setminus D}(x,A) \quad \textrm{for each $A \in \Borel(\ambient)$.}
		\end{equation}
	\end{definition}
	
	Before starting our discussion of trace Dirichlet forms, we recall
	some basic properties of the harmonic measure in the following lemma.
	
	\begin{lem} \label{l:harmonicm}
		Let an MMD space $(\ambient,d,\refmeas,\form,\domain)$ and a diffusion $\diff$ on $\ambient$
		satisfy Assumption \ref{a:feller}. Let $D$ be a non-empty open subset of $\ambient$.
		\begin{enumerate}[\rm(a)]\setlength{\itemsep}{0pt}\vspace{-5pt}
			\item\label{it:hmeas-smooth-support} \textup{(\cite[Lemma 3.2]{Lie})} For any $x \in D$,
			the measure $\hmeas{D}{x}$ charges no $\form$-polar set and $\supp_{\ambient}[\hmeas{D}{x}] \subset \partial D$.
			\item\label{it:hmeas-continuous} \textup{(\cite[Lemma 3.2]{Lie})}
			For any bounded Borel measurable function $f\colon \partial D \to \mathbb{R}$,
			the function $h \colon D \to \mathbb{R}$ defined by
			\begin{equation*}
				h(x) \coloneqq \int_{\partial D} f(y)\,\hmeas{D}{x}(dy)
			\end{equation*}
			belongs to $\domain_{\on{loc}}(D)$ and is continuous on $D$ and $\form$-harmonic on $D$.
			\item\label{it:hmeas-AC} If $D$ is connected, then $\hmeas{D}{x} \ll \hmeas{D}{y}$ for any $x,y \in D$.
			\item\label{it:hmeas-prob1} If $D$ is relatively compact in $\ambient$ and $\ambient\setminus D$
			is not $\form$-polar, then $\hmeas{D}{x}(\partial D)=1$ for any $x\in D$.
			\item\label{it:hmeas-quasi-support} Let $\unifdom$ be a uniform domain in $(\ambient,d)$ and
			$\diffref=\bigl(\Omega^{\on{ref}},\events^{\on{ref}},\{\diffref_{t}\}_{t\in[0,\infty]},\{\lawref_{x}\}_{x\in\overline{\unifdom}\cup\{\cemetery\}}\bigr)$
			be a diffusion on $\overline{\unifdom}$ as in Assumption \ref{a:feller}
			for the MMD space $(\overline{\unifdom},d,\refmeas|_{\overline{\unifdom}},\formrefgen{\unifdom},\domain(\unifdom))$
			\textup{(recall Theorem \ref{thm:hkeunif}-\eqref{it:hkeunif})}. Then for any $x \in \unifdom$, the
			$\formrefgen{\unifdom}$-harmonic measure of $\unifdom$ with base point $x$
			coincides with $\hmeas{\unifdom}{x}\big|_{\overline{\unifdom}}$, and $\partial \unifdom$
			is an $\formrefgen{\unifdom}$-quasi-support of $\hmeas{\unifdom}{x}\big|_{\overline{\unifdom}}$.
		\end{enumerate}
	\end{lem}
	
	\begin{proof}
		\begin{enumerate}\setlength{\itemsep}{0pt}
			\item[(\ref{it:hmeas-smooth-support},\ref{it:hmeas-continuous})] We have
			$\supp_{\ambient}[\hmeas{D}{x}] \subset \partial D$ by the sample-path continuity
			\eqref{eq:sample-path-cont} of $\diff$, which holds for any $x \in \ambient$
			by \ref{eq:AC} of $\diff$. The remaining properties are proved in \cite[Lemma 3.2]{Lie}.
			Although \cite[Lemma 3.2]{Lie} assumes that $D$ is a relatively compact open subset of $\ambient$,
			the proofs presented there work for an arbitrary open subset.
			\item[\eqref{it:hmeas-AC}] Let $A \in \Borel( \partial D )$ satisfy $\hmeas{D}{y}(A)=0$.
			By \eqref{it:hmeas-continuous}, the function $h_A(z) \coloneqq \int_{\partial D} \one_{A}(\xi) \hmeas{D}{z}(d\xi)= \hmeas{D}{z}(A)$ on $D$
			is continuous, non-negative, $\form$-harmonic on $D$ and belongs to $\domain_{\on{loc}}(D)$.
			Since $h_A(y)= \hmeas{D}{y}(A)=0$, we conclude from \hyperlink{ehi}{$\on{EHI}$}
			(recall Remark \ref{r:ehi-hke}) and the connectedness of $D$ that
			$h_A^{-1}(0)$ is non-empty, both closed and open in $D$ and thus coincides with $D$.
			In particular, $\hmeas{D}{x}(A)=h_A(x)=0$ and hence $\hmeas{D}{x} \ll \hmeas{D}{y}$.
			\item[\eqref{it:hmeas-prob1}] We have $\lawdiff_{x}(\tau_{D}<\infty)=1$ for $\form$-q.e.\ $x\in D$
			by \cite[Proposition 3.2]{BCM} and the irreducibility of $(\ambient,\refmeas,\form,\domain)$
			from Proposition \ref{p:feller}-\eqref{it:HKE-conn-irr-cons}, hence $\lawdiff_{x}(\tau_{D}<\infty)=1$
			for any $x\in D$ by \ref{eq:AC} of $\diff^{D}$ and the Markov property of $\diff$, and therefore
			$\hmeas{D}{x}(\partial D)=\lawdiff_{x}(\tau_{D}<\infty)=1$ for any $x\in D$
			by the sample-path continuity \eqref{eq:sample-path-cont} and \ref{eq:AC} of $\diff$.
			\item[\eqref{it:hmeas-quasi-support}] We see from \cite[Exercise 1.4.1 and Theorem 1.4.2-(ii)]{FOT}
			that $\domain(\unifdom) \cap \contfunc_{\mathrm{c}}(\unifdom) = \domain^{0}(\unifdom) \cap \contfunc_{\mathrm{c}}(\unifdom)$,
			from \cite[Corollary 3.2.1]{FOT} that $\formrefgen{\unifdom}(f,g) = \form^{\unifdom}(f,g)$
			for any $f,g \in \domain^{0}(\unifdom) \cap \contfunc_{\mathrm{c}}(\unifdom)$, and
			thus from the denseness of $\domain(\unifdom) \cap \contfunc_{\mathrm{c}}(\unifdom) = \domain^{0}(\unifdom) \cap \contfunc_{\mathrm{c}}(\unifdom)$
			in $\bigl(\domain(\unifdom),\formrefgen{\unifdom}+\langle\cdot,\cdot\rangle_{L^{2}(\overline{\unifdom},\refmeas|_{\overline{\unifdom}})}\bigr)$
			and in $\bigl(\domain^{0}(\unifdom),\form^{\unifdom}+\langle\cdot,\cdot\rangle_{L^{2}(\unifdom,\refmeas|_{\unifdom})}\bigr)$
			that the part Dirichlet form of $(\formrefgen{\unifdom},\domain(\unifdom))$
			on $\unifdom$ coincides with $(\form^{\unifdom},\domain^{0}(\unifdom))$.
			Namely, the Dirichlet form of the part process $(\diffref)^{\unifdom}$
			of $\diffref$ on $\unifdom$ coincides with that of $\diff^{\unifdom}$,
			which together with Proposition \ref{p:feller}-\eqref{it:HKE-part-CHK-sFeller}
			implies that $(\diffref)^{\unifdom}$ and $\diff^{\unifdom}$
			have the same Markovian transition function on $\unifdom$. It then follows
			by the Markov property of $(\diffref)^{\unifdom}$ and $\diff^{\unifdom}$ that
			for any $x \in \unifdom$, the law of $\{ (\diffref)^{\unifdom}_{t} \}_{t\in[0,\infty)}$
			under $\lawref_{x}$ and that of $\{ \diff^{\unifdom}_{t} \}_{t\in[0,\infty)}$
			under $\lawdiff_{x}$ as $\contfunc_{\cemetery}([0,\infty),\oneptcpt{\unifdom})$-valued
			random variables coincide, where
			\begin{equation*}
				\contfunc_{\cemetery}([0,\infty),\oneptcpt{\unifdom})
				 \coloneqq \biggl\{ \gamma \colon [0,\infty) \to \oneptcpt{\unifdom} \biggm|
				\begin{minipage}{170pt}
					$\gamma$ is continuous, $\gamma(t) = \cemetery_{\unifdom}$
					for any $t \in [0,\infty)$ with $t \geq \inf\gamma^{-1}(\cemetery_{\unifdom})$
				\end{minipage}
				\biggr\},
			\end{equation*}
			equipped with the $\sigma$-algebra generated by its subsets of the form
			$\{ \gamma \in \contfunc_{\cemetery}([0,\infty),\oneptcpt{\unifdom}) \mid \gamma(t) \in A \}$
			for some $t \in [0,\infty)$ and $A \in \Borel(\oneptcpt{\unifdom})$.
			In particular, for any $x \in \unifdom$, any $f \in \contfunc(\oneptcpt{\ambient})$
			with $\norm{f}_{\sup}<\infty$ and any $\varepsilon \in (0,\infty)$, we have
			\begin{equation*}
				\expref_{x}\bigl[ f(\diffref_{(\tau_{\unifdom}-\varepsilon)^{+}})\one_{\{\tau_{\unifdom}<\infty\}} \bigr]
				=\expdiff_{x}\bigl[ f(\diff_{(\tau_{\unifdom}-\varepsilon)^{+}})\one_{\{\tau_{\unifdom}<\infty\}} \bigr],
			\end{equation*}
			and letting $\varepsilon \downarrow 0$ yields
			$\expref_{x}\bigl[ f(\diffref_{\tau_{\unifdom}})\one_{\{\tau_{\unifdom}<\infty\}} \bigr]
			=\expdiff_{x}\bigl[ f(\diff_{\tau_{\unifdom}})\one_{\{\tau_{\unifdom}<\infty\}} \bigr]$
			by \eqref{eq:sample-path-cont} and the dominated convergence theorem, whence
			$\lawref_{x}( \diffref_{\tau_{\unifdom}} \in dy ) = \lawdiff_{x}( \diff_{\tau_{\unifdom}} \in dy ) = \hmeas{\unifdom}{x}(dy)$,
			i.e., the $\formrefgen{\unifdom}$-harmonic measure of $\unifdom$ with base point $x$
			coincides with $\hmeas{\unifdom}{x}\big|_{\overline{\unifdom}}$.
			
			Next, let $x \in \unifdom$ and, to see that $\partial \unifdom$ is an $\formrefgen{\unifdom}$-quasi-support of $\hmeas{\unifdom}{x}\big|_{\overline{\unifdom}}$,
			define the $1$-order hitting distribution $H^{1}_{\partial \unifdom}$ of $\diffref$ to $\partial \unifdom$ by
			\[
			H^{1}_{\partial \unifdom}(y,B) \coloneqq \mathbb{E}^{\on{ref}}_y \left[ e^{- \sigma_{\partial \unifdom}} \one_{B}(\diffref_{\sigma_{\partial \unifdom}}) \one_{\set{\sigma_{\partial \unifdom} < \infty}} \right],
			\quad \sigma_{\partial \unifdom}= \inf\{ t \in (0,\infty) \mid \diffref_t \in \partial \unifdom \}
			\]
			for $y \in \overline{\unifdom}$ and $B \in \Borel( \partial \unifdom)$.
			Then by the result of the previous paragraph and \eqref{it:hmeas-AC} we have
			$H^1_{\partial \unifdom}(y,\cdot) \ll \hmeas{\unifdom}{x}\big|_{\overline{\unifdom}}$ for any $y \in \unifdom$,
			which implies by \cite[Exercise 4.6.1]{FOT} that $\partial \unifdom$
			is an $\formrefgen{\unifdom}$-quasi-support of $\hmeas{\unifdom}{x}\big|_{\overline{\unifdom}}$
			since $\refmeas(\partial \unifdom)=0$ by Lemma \ref{l:doubling-unif}.
			\qedhere\end{enumerate}
	\end{proof}
	
	The rest of this subsection is devoted to a discussion of trace Dirichlet forms,
	which are the Dirichlet forms of the time-changed processes given by \eqref{eq:time-changed-process}
	and defined as follows.
	
	\begin{definition} [Trace Dirichlet form] \label{d:traceDF}
		Let $\nu$ be an $\form$-smooth Radon measure on $\ambient$, set $F^{*} \coloneqq \supp_{\ambient}[\nu]$,
		and let $F$ be an $\diff$-nearly Borel measurable $\form$-quasi-support of $\nu$.
		Since $\Capa_{1}(F \setminus F^{*})=0$ by Definition \ref{d:quasisupport}-\eqref{it:quasisupport-1},\eqref{it:quasisupport-2},
		replacing $F$ with $F \setminus \mathcal{N}$ for an arbitrary $\diff$-nearly Borel measurable
		$\form$-polar set $\mathcal{N}\subset\ambient$ including $F \setminus F^{*}$,
		we may and do assume that $F \subset F^{*}$. We define
		\begin{equation} \label{e:def-trace-domain}
			\trdomain \coloneqq \biggl\{ \widetilde{u}|_{F} \biggm| \textrm{$u \in \domain_e$, $\int_{F} \widetilde{u}^{2} \, d\nu < \infty$} \biggr\},
		\end{equation}
		where we identify functions that coincide $\form$-q.e.\ on $F$; since,
		for each $u,v \in \domain_e$, $\widetilde{u}=\widetilde{v}$ $\form$-q.e.\ on $F$ if and only if
		$\widetilde{u}=\widetilde{v}$ $\nu$-a.e.\ on $\ambient$ by \cite[Theorem 3.3.5]{CF},
		and since $\nu(\ambient \setminus F)=0$ and $F \subset F^{*}$,
		we can canonically consider $\trdomain$ as a linear subspace of $L^{2}(F^{*},\nu)$.
		Then we further define a non-negative definite symmetric bilinear form
		$\trform \colon \trdomain \times \trdomain \to \mathbb{R}$ by
		\begin{equation} \label{e:def-trace-form}
			\trform(\widetilde{u}|_{F},\widetilde{v}|_{F}) \coloneqq \form(H_{F} \widetilde{u}, H_{F} \widetilde{v})
			\qquad \textrm{for $u,v \in \domain_{e}$ with $\widetilde{u}|_{F},\widetilde{v}|_{F} \in \trdomain$,}
		\end{equation}
		and call $(\trform,\trdomain)$ the \textbf{trace Dirichlet form} of $(\form,\domain)$ on $L^{2}(F^{*},\nu)$.
	\end{definition}
	
	%We remark that the value of $\form^\mu(\phi,\phi)$ in \eqref{e:def-trace-form} does not choice of the quasi-support.
	
	Let $\nu,F^{*},F,\trdomain,\trform$ be as in Definition \ref{d:traceDF}, and assume that
	$\nu(\ambient)>0$, or equivalently, $\Capa_{1}(F)>0$. Then $(\trform,\trdomain)$ is indeed
	a regular symmetric Dirichlet form on $L^{2}(F^{*},\nu)$ and $F^{*}\setminus F$ is $\trform$-polar
	by \cite[Theorem 5.2.13-(i)]{CF}, a subset $\mathcal{N}_{1}$ of $F$ is $\trform$-polar
	if and only if $\mathcal{N}_{1}$ is $\form$-polar by \cite[Theorem 5.2.8 and Proof of Theorem 5.2.13-(ii)]{CF},
	and $f|_{F\setminus\mathcal{N}_{1}}$ is $\trform$-quasi-continuous on $F^{*}$
	for any $\form$-quasi-continuous function $f \colon \ambient \setminus \mathcal{N}_{1} \to [-\infty,\infty]$
	defined $\form$-q.e.\ for some $\form$-polar $\mathcal{N}_{1} \subset \ambient$
	by \cite[Theorem 5.2.6 and Proof of Theorem 5.2.13-(ii)]{CF}.
	Furthermore by \cite[Theorem 5.2.15]{CF}, the extended Dirichlet space
	$\trdomain_{e}$ of $(F^{*},\nu,\trform,\trdomain)$ and the values of
	$\trform$ on $\trdomain_{e}\times\trdomain_{e}$ are identified as
	\begin{equation} \label{e:trace-ExtDiriSp}
		\trdomain_{e} = \{ \widetilde{u}|_{F} \mid u \in \domain_{e} \}
		\mspace{30mu} \textrm{and} \mspace{30mu}
		\trform(\widetilde{u}|_{F},\widetilde{v}|_{F}) = \form(H_{F} \widetilde{u}, H_{F} \widetilde{v})
		\quad \textrm{for any $u,v \in \domain_{e}$.}
	\end{equation}
	In probabilistic terms, for the time-changed process $\difftr$ of $\diff$
	by a PCAF $A$ of $\diff$ with Revuz measure $\nu$, which is a $\nu$-symmetric right-continuous
	strong Markov process on $(\oneptcpt{F},\Borel^{*}(\oneptcpt{F}))$ defined by \eqref{eq:time-changed-process},
	its Dirichlet form is $(\trform,\trdomain)$ by \cite[Theorem 5.2.2]{CF}; here, since
	the support $F_{A}$ of $A$ defined by \eqref{e:supportA} is an $\diff$-nearly Borel measurable
	$\form$-quasi-support of $\nu$, as the sets $F$ and $\mathcal{N}$ in Definition \ref{d:traceDF}
	we can choose $F_{A}$ and an exceptional set $\mathcal{N}_{A}$ of $A$ including $F \setminus F^{*}$,
	respectively, and therefore we may and do assume that $F$ in Definition \ref{d:traceDF}
	is the support of $A$, on which we can define the time-changed process
	$\difftr$ by \eqref{eq:time-changed-process}.
	
	In order to analyze the trace Dirichlet form $(\trform,\trdomain)$, it is desirable to compute
	its \emph{Beurling--Deny decomposition} \cite[Theorems 3.2.1 and 4.5.2]{FOT} (see also \cite[Theorem 4.3.3]{CF}).
	For the regular Dirichlet space $(\ambient,\refmeas,\form,\domain)$, this decomposition
	can be stated as follows: there exists a unique triple $(\form^{(c)},\jumpmeas,\kappa)$
	of a strongly local non-negative definite symmetric bilinear form
	$\form^{(c)} \colon \domain_{e} \times \domain_{e} \to \mathbb{R}$,
	a symmetric Radon measure $\jumpmeas$ on $\offdiag{\ambient}$ and
	a Radon measure $\kappa$ on $\ambient$, such that
	$\jumpmeas((\ambient \times \mathcal{N}_{1}) \cap \offdiag{\ambient}) = 0 = \kappa(\mathcal{N}_{1})$
	for any $\form$-polar $\mathcal{N}_{1} \in \Borel(\ambient)$ and
	\begin{equation} \label{e:Beurling-Deny}
		\form(u,v) = \form^{(c)}(u,v) + \frac{1}{2} \int_{\offdiag{\ambient}} (\wt{u}(x)-\wt{u}(y))(\wt{v}(x)-\wt{v}(y))\, \jumpmeas(dx\,dy)
		+ \int_{\ambient} \wt{u}(x)\wt{v}(x)\,\kappa(dx)
	\end{equation}
	for any $u,v \in \domain_{e}$, where $\wt{u},\wt{v}$ denote $\form$-quasi-continuous $\refmeas$-versions of $u,v$ respectively.
	We call $\form^{(c)},\jumpmeas,\kappa$ the \emph{strongly local part}, the \emph{jumping measure}
	and the \emph{killing measure}, respectively, of $(\ambient,\refmeas,\form,\domain)$.
	Moreover, we can define the \emph{strongly local part $\Gamma_{c}(u,u)$ of the
	$\form$-energy measure $\Gamma(u,u)$ of $u\in\domain_{e}$} by replacing $\form$
	with $\form^{(c)}$ in the argument in Definition \ref{d:EnergyMeas} on the basis of
	\cite[(3.2.19), (3.2.20) and (3.2.21)]{FOT}, and $\Gamma_{c}(u,u)(\ambient)=\form^{(c)}(u,u)$
	for any $u\in\domain_{e}$ by \cite[Lemma 3.2.3]{FOT}.
	
	An identification of the Beurling--Deny decomposition of the trace Dirichlet form
	$(\trform,\trdomain)$ is given in \cite[Theorems 5.6.2 and 5.6.3]{CF}.
	In the following proposition, we provide a new simple proof of the result
	for the strongly local part of $(\trform,\trdomain)$ in \cite[Theorem 5.6.2]{CF}.
	
	\begin{prop}[{\cite[Theorem 5.6.2]{CF}}] \label{prop:trace-slocal}
		Let $(\ambient,\refmeas,\form,\domain)$ be a regular Dirichlet space,
		$\nu$ an $\form$-smooth Radon measure on $\ambient$ with $\nu(\ambient)>0$,
		set $F^{*} \coloneqq \supp_{\ambient}[\nu]$,
		let $F$ be an $\form$-quasi-support of $\nu$ satisfying $F \subset F^{*}$, and
		let $(\trform,\trdomain)$ be the trace Dirichlet form of $(\form,\domain)$
		on $L^{2}(F^{*},\nu)$ defined by \eqref{e:def-trace-domain} and \eqref{e:def-trace-form}.
		Let $\Gamma_{c}$ denote the strongly local part of the $\form$-energy measures, and
		$\trformc$ the strongly local part of the $\trform$-energy measures. Then
		\begin{equation} \label{eq:trace-slocal-energymeas}
			\trformc( \widetilde{u}|_{F}, \widetilde{u}|_{F} )(B) = \Gamma_{c}( u, u )(B \cap F)
			\qquad \textrm{for any $u \in \domain_{e}$ and any $B \in \Borel(F^{*})$.}
		\end{equation}
		In particular, the strongly local part $\trform^{(c)}$ of $(F^{*},\nu,\trform,\trdomain)$ is given by
		\begin{equation}\label{eq:trace-slocal}
			\trform^{(c)}( \widetilde{u}|_{F}, \widetilde{u}|_{F} ) = \Gamma_{c}( u, u )(F)
			\qquad \textrm{for any $u\in\domain_{e}$.}
		\end{equation}
	\end{prop}
	
	\begin{proof}
		Since $F$ is $\form$-quasi-closed, we can choose an $\form$-nest $\{F_{k}\}_{k\in\mathbb{N}}$
		so that $F\cap\bigcup_{k\in\mathbb{N}}F_{k}\in\Borel(\ambient)$, and therefore by replacing
		$F$ with $F\cap\bigcup_{k\in\mathbb{N}}F_{k}$ we may and do assume that $F \in \Borel(\ambient)$.
		For any $u \in \domain_{e}$, \eqref{eq:trace-slocal} follows from \eqref{eq:trace-slocal-energymeas} with $B=F^{*}$ and
		$\trform^{(c)}( \widetilde{u}|_{F}, \widetilde{u}|_{F} )=\trformc( \widetilde{u}|_{F}, \widetilde{u}|_{F} )(F^{*})$,
		and $\trformc( \widetilde{u}|_{F}, \widetilde{u}|_{F} )(F^{*} \setminus F) = 0$
		since $\trformc( \widetilde{u}|_{F}, \widetilde{u}|_{F} )$ charges no $\trform$-polar set
		by \cite[Lemma 3.2.4]{FOT} and $F^{*} \setminus F$ is $\trform$-polar.
		It thus suffices to prove \eqref{eq:trace-slocal-energymeas} for $B \in \Borel(F)$.
		Our simple proof of \eqref{eq:trace-slocal-energymeas} is based on Mosco's
		proof of the domination principle in \cite[p.~389, Proof of Proposition]{Mosco}
		and goes as follows. Let $u \in \domain \cap \contfunc_{\mathrm{c}}(\ambient)$.
		
		We write $\form(v) \coloneqq \form(v,v)$, $\Gamma_{c}(v) \coloneqq \Gamma_{c}(v,v)$,
		$\trform(\widetilde{v}|_{F}) \coloneqq \trform(\widetilde{v}|_{F},\widetilde{v}|_{F})$
		and $\trformc(\widetilde{v}|_{F}) \coloneqq \trform(\widetilde{v}|_{F},\widetilde{v}|_{F})$
		for $v\in\domain_{e}$ in this proof. Let $f \in \domain \cap \contfunc_{\mathrm{c}}(\ambient)$ and $\lambda \in (0,\infty)$.
		Computing both sides of the inequality (recall the second half of \eqref{eq:hit-dist-harm-var} and \eqref{e:def-trace-form})
		\begin{equation*}
			\trform( f \cos( \lambda u )|_F ) + \trform( f \sin( \lambda u )|_F )
			\leq \form( f \cos( \lambda u ) ) + \form( f \sin( \lambda u ) )
		\end{equation*}
		on the basis of the chain rule \cite[Theorem 3.2.2]{FOT} as in Mosco's argument
		in \cite[p.~389]{Mosco}, dividing the resulting inequality by $\lambda^2$ and
		letting $\lambda \to \infty$ via the dominated convergence theorem, we obtain
		\begin{equation} \label{eq:Mosco-argument-upper}
			\int_{F^{*}} f^2 \, d\trformc( u|_F, u|_F ) \leq \int_{\ambient} f^2 \, d\Gamma_c( u, u ).
		\end{equation}
		Then for any compact subset $K$ of $F$ and any open subset $G$ of $\ambient$
		with $K \subset G$, by \cite[Exercise 1.4.1]{FOT} we can choose
		$f \in \domain \cap \contfunc_{\mathrm{c}}(\ambient)$ so that
		$\one_{K} \leq f \leq \one_{G}$, thus from \eqref{eq:Mosco-argument-upper} we obtain
		\begin{equation*}
			\trformc( u|_{F}, u|_{F} )( K ) \leq \Gamma_{c}( u, u )( G ).
		\end{equation*}
		Since $\Gamma_{c}(u,u)$ and $\trformc( u|_F, u|_F )$ are outer and inner regular
		by \cite[Theorem 2.18]{Rud}, taking the infimum over the open subsets $G$ of
		$\ambient$ with $K \subset G$ yields
		\begin{equation*}
			\trformc( u|_F, u|_F )( K ) \leq \Gamma_c( u, u )( K ),
		\end{equation*}
		and now for any $B \in \Borel(F)$, taking the supremum over the compact subsets
		$K$ of $B$ shows
		\begin{equation} \label{e:slb1}
			\trformc( u|_{F}, u|_{F} )(B) \leq \Gamma_c( u, u )(B).
		\end{equation}
		
		Next, we show the lower bound matching the upper bound \eqref{e:slb1}.
		Let $K$ be any compact subset of $F$, $G$ any open subset of $\ambient$ with $K\subset G$,
		and choose $f \in \domain \cap \contfunc_{\mathrm{c}}(\ambient)$ so that
		$\one_K \leq f \leq \one_G$. For any $\lambda \in (0,\infty)$, since
		$\abs{\Gamma_{c}(v)(B)^{1/2}-\Gamma_{c}(g)(B)^{1/2}}\leq\Gamma_{c}(v-g)(B)^{1/2}=0$
		for any $v,g \in \domain_{e}$ and any $B \in \Borel(\ambient)$ with $B \subset (\widetilde{v}-\widetilde{g})^{-1}(0)$
		by \cite[Theorem 4.3.8]{CF}, we see from the first half of \eqref{eq:hit-dist-harm-var},
		$K \subset F$, \eqref{e:Beurling-Deny} and \eqref{e:def-trace-form} that
		\begin{align}
			&\Gamma_{c}( f \cos( \lambda u ) )(K) + \Gamma_{c}( f \sin( \lambda u ) )(K) \nonumber \\
			&= \Gamma_{c}\bigl( H_{F}( f \cos( \lambda u ) ) \bigr)(K) + \Gamma_{c}\bigl( H_{F}( f \sin( \lambda u ) ) \bigr)(K) \nonumber \\
			&\leq \form\bigl( H_{F}( f \cos( \lambda u ) ) \bigr) + \form\bigl( H_{F}( f \sin( \lambda u ) ) \bigr) \nonumber \\
			&= \trform( f \cos( \lambda u )|_{F} ) + \trform( f \sin( \lambda u )|_{F} ),
			\label{eq:Mosco-lower-slocal}
		\end{align}%
		and applying Mosco's argument in \cite[p.~389]{Mosco} to \eqref{eq:Mosco-lower-slocal}
		in the same way as the above proof of \eqref{eq:Mosco-argument-upper}, we obtain
		\begin{equation*}
			\Gamma_{c}( u, u )(K) = \int_{K} f^{2}\,d\Gamma_{c}(u,u)
			\leq \int_{F^{*}} f^{2} \, d\trformc( u|_{F}, u|_{F} )
			\leq \trformc( u|_{F}, u|_{F} )( G \cap F^{*} ).
		\end{equation*}
		Now since $\Gamma_{c}(u,u)$ and $\trformc( u|_{F}, u|_{F} )$
		are outer and inner regular by \cite[Theorem 2.18]{Rud}, by taking the infimum over
		the open subsets $G$ of $\ambient$ with $K \subset G$ and then the supremum
		over the compact subsets $K$ of any given $B \in \Borel(F)$, we obtain
		\begin{equation*}
			\Gamma_{c}( u, u )( B ) \leq \trformc( u|_{F}, u|_{F} )( B ),
		\end{equation*}
		which together with \eqref{e:slb1} proves \eqref{eq:trace-slocal-energymeas}
		for $u \in \domain \cap \contfunc_{\mathrm{c}}( \ambient )$.
		
		Lastly, for any $u \in \domain_{e}$, by \cite[Theorem 2.1.7]{FOT} we can choose
		$\{ u_{n} \}_{n \in \mathbb{N}} \subset \domain \cap \contfunc_{\mathrm{c}}( \ambient )$
		so that $\lim_{n\to\infty}\form(u-u_{n})=0$, hence
		$\lim_{n\to\infty}\trform(\widetilde{u}|_{F}-u_{n}|_{F})=0$
		by the second half of \eqref{eq:hit-dist-harm-var} and \eqref{e:trace-ExtDiriSp},
		and then \eqref{eq:trace-slocal-energymeas} for $u$ follows by
		letting $n\to\infty$ in \eqref{eq:trace-slocal-energymeas} for $u_{n}$
		on the basis of the triangle inequalities for $\Gamma_{c}(\cdot)(B\cap F)^{1/2}$ and
		$\trformc(\cdot)(B)^{1/2}$, $\Gamma_{c}(u-u_{n})(\ambient) \leq \form(u-u_{n})$ and
		$\trformc(\widetilde{u}|_{F}-u_{n}|_{F})(F^{*}) \leq \trform(\widetilde{u}|_{F}-u_{n}|_{F})$.
	\end{proof}
	
	We will use the following proposition to show that the killing measure of the boundary trace form is zero.
	This could alternatively be deduced from \cite[Theorem 5.6.3]{CF}, but we give a new
	self-contained proof that does not rely on the notion of supplementary Feller measure.
	
	\begin{prop} \label{p:killing}
		Let $(\ambient,\refmeas,\form,\domain)$ be a regular Dirichlet space,
		$\nu$ an $\form$-smooth Radon measure on $\ambient$ with $\nu(\ambient)>0$,
		set $F^{*} \coloneqq \supp_{\ambient}[\nu]$,
		let $F$ be an $\form$-quasi-support of $\nu$ satisfying $F \subset F^{*}$, and
		let $(\trform,\trdomain)$ be the trace Dirichlet form of $(\form,\domain)$ on
		$L^{2}(F^{*},\nu)$ defined by \eqref{e:def-trace-domain} and \eqref{e:def-trace-form}.
		Let $\kappa$ denote the killing measure of $(\ambient,\refmeas,\form,\domain)$,
		and $\widecheck{\kappa}$ the killing measure of $(F^{*},\nu,\trform,\trdomain)$.
		If $\kappa(\ambient)=0$ and
		\begin{equation} \label{eq:quasi-support-hit-as}
			\lawdiff_{x}(\sigma_{F} < \infty) = 1 \quad \textrm{for $\form$-q.e.\ $x \in \ambient$,}
		\end{equation}
		then $\widecheck{\kappa}(F^{*})=0$.
	\end{prop}
	
	\begin{proof}
		By \cite[Exercise 1.4.1]{FOT}, we can choose
		$\{ f_{n} \}_{n\in\mathbb{N}} \subset \domain \cap \contfunc_{\mathrm{c}}(\ambient )$
		so that for any $x \in \ambient$ we have $0 \leq f_{n}(x) \leq f_{n+1}(x) \leq 1$
		for any $n \in \mathbb{N}$ and $\lim_{n\to\infty}f_{n}(x)=1$.
		Let $u \in \domain \cap \contfunc_{\mathrm{c}}(\ambient )$.
		Then by \eqref{e:EnergyMeas}, \cite[(3.2.23), Lemma 3.2.3]{FOT} and
		\eqref{e:Beurling-Deny} applied to $(F^{*},\nu,\trform,\trdomain)$
		and the monotone convergence theorem, we have
		\begin{equation} \label{e:killm1}
			\trform( u|_{F}, f_{n}u|_{F} ) - \frac{1}{2} \trform( u^{2}|_{F}, f_{n}|_{F} )
			\xrightarrow{n\to\infty} \trform( u|_{F}, u|_{F} ) - \frac{1}{2} \int_{ F^{*} } u^{2} \, d\widecheck{\kappa}.
		\end{equation}
		
		On the other hand, for any $n \in \mathbb{N}$, by \eqref{e:def-trace-form},
		the first half of \eqref{eq:hit-dist-harm-var} and \eqref{eq:hit-dist-harm-test}
		applied to $H_{F}u,H_{F}f_{n}$, and the extension of \eqref{e:EnergyMeas}
		to $\form$-quasi-continuous $\refmeas$-versions of functions in
		$\domain_{e} \cap L^{\infty}(\ambient,\refmeas)$
		proved in \cite[Proof of Theorem 4.3.11]{CF}, we obtain
		\begin{align} \label{e:killm2}
			\trform( u|_{F}, f_{n}u|_{F} ) - \frac{1}{2} \trform( u^{2}|_{F}, f_{n}|_{F} ) 
			&= \form( H_{F}u, H_{F}(f_{n}u) ) - \frac{1}{2} \form( H_{F}(u^{2}),  H_{F}f_{n} ) \nonumber \\
			&= \form( H_{F}u, (H_{F}f_{n})(H_{F}u) ) - \frac{1}{2} \form( (H_{F}u)^{2}, H_{F}f_{n} ) \nonumber \\
			&= \int_{ \ambient } H_{F}f_{n} \,d\Gamma( H_{F}u, H_{F}u ).
		\end{align}
		Since $\lawdiff_{x}( \textrm{$\diff_{\sigma_{F}} =\cemetery$, $\sigma_{F}<\infty$} ) = 0$
		for any $x \in \ambient$, $\{ H_{F}f_{n}(x) \}_{n \in \mathbb{N}} \subset [0,1]$
		is non-decreasing and converges to $\lawdiff_{x}( \sigma_{F}<\infty ) = 1$
		for $\form$-q.e.\ $x \in \ambient$ by \eqref{e:defhitdist}
		and \eqref{eq:quasi-support-hit-as}, and hence
		\begin{align}
			\int_{ \ambient } H_{F}f_{n} \,d\Gamma( H_{F}u, H_{F}u )
			&\xrightarrow{n \to \infty} \Gamma ( H_{F}u, H_{F}u )( \ambient ) \nonumber \\
			&\mspace{28mu}= \form( H_{F}u,  H_{F}u ) - \frac{1}{2} \int_{\ambient} (H_{F}u)^{2} \, d\kappa
			= \trform( u|_{F}, u|_{F} )
			\label{e:killm3}
		\end{align}
		by the monotone convergence theorem and the fact that $\Gamma( H_{F}u, H_{F}u )$
		charges no $\form$-polar set by \cite[Lemma 3.2.4]{FOT}. Here the second equality in
		\eqref{e:killm3} follows from \eqref{e:def-trace-form} and $\kappa(\ambient)=0$,
		and the first one in \eqref{e:killm3} is a special case of the following general equality
		\begin{equation} \label{eq:energy-meas-total}
			\Gamma(v,v)(\ambient)=\form(v,v)-\frac{1}{2}\int_{\ambient}\widetilde{v}^{2}\,d\kappa
			\qquad\textrm{for any $v \in \domain_{e}$}
		\end{equation}
		(which holds for \emph{any} regular Dirichlet space $(\ambient,\refmeas,\form,\domain)$);
		we can verify \eqref{eq:energy-meas-total} first for $v \in \domain \cap \contfunc_{\mathrm{c}}(\ambient )$
		in the same way as \eqref{e:killm1} above, and for general $v \in \domain_{e}$
		by using \cite[Theorem 2.1.7]{FOT} to choose
		$\{ v_{n} \}_{n \in \mathbb{N}} \subset \domain \cap \contfunc_{\mathrm{c}}( \ambient )$
		so that $\lim_{n\to\infty}\form(v-v_{n},v-v_{n})=0$ and then by
		letting $n\to\infty$ in \eqref{eq:energy-meas-total} for $v_{n}$
		on the basis of the triangle inequalities for
		$\Gamma(\cdot,\cdot)(\ambient)^{1/2},\form(\cdot,\cdot)^{1/2},\norm{\cdot}_{L^{2}(\ambient,\kappa)}$
		and $\int_{\ambient}(\widetilde{v}-v_{n})^{2}\,d\kappa\leq\form(v-v_{n},v-v_{n})$
		implied by \eqref{e:Beurling-Deny}.
		
		It thus follows from \eqref{e:killm1}, \eqref{e:killm2} and \eqref{e:killm3} that
		\[
		\int_{ F^{*} } u^{2} d\widecheck{\kappa} = 0
		\qquad \textrm{for any $u \in \domain \cap \contfunc_{\mathrm{c}}( \ambient )$,}
		\]
		and hence $\widecheck{\kappa}( F^{*} ) = \lim_{n\to\infty}\int_{F^{*}}f_{n}^{2}\,d\kappa = 0$.
	\end{proof}
	\subsection{Stable-like heat kernel estimates} \label{s:stablehke}
	
	We recall a generalization of scale function considered in Subsection \ref{ss:HKE}
	from \cite[Defintion 7.2]{BCM} (see also \cite[Definition 5.4]{BM18}).
	
	\begin{definition} \label{d:regscale}
		Let $(\ambient,d)$ be a metric space. We say that a function $\scjump \colon \ambient \times [0,\infty) \to [0,\infty)$
		is a \emph{regular scale function on $(\ambient,d)$ with threshold $\scjthres{\scjump}\in(0,\infty]$}
		if $\scjump(x,\cdot) \colon [0,\infty)\to [0,\infty)$
		is a homeomorphism for all $x \in \ambient$, $\diam(\ambient)\leq\scjthres{\scjump}$
		and there exist $C_1,\beta_1,\beta_2 \in (0,\infty)$ such that for all $x,y \in \ambient$
		and all $s,r\in(0,\infty)$ with $s \leq r \leq \scjthres{\scjump}$,
		\begin{equation} \label{e:Psireg1}
			C_1^{-1} \biggl( \frac{r}{d(x,y) \vee r} \biggr)^{\beta_2} \biggl( \frac{d(x,y) \vee r}{s} \biggr)^{\beta_1}
			\leq \frac{ \scjump(x,r)}{\scjump(y,s)} 
			\leq C_1 \biggl( \frac{r}{d(x,y) \vee r} \biggr)^{\beta_1} \biggl( \frac{d(x,y) \vee r}{s} \biggr)^{\beta_2}.
		\end{equation}
	\end{definition}
	
	The definition in \cite[Defintion 7.2]{BCM} does not state that $\scjump(x,\cdot)\colon [0,\infty)\to [0,\infty)$ is a homeomorphism but this 
	condition can be achieved by replacing $\scjump$ with a comparable function if necessary as we will see in the proof of Lemma \ref{l:scale}.
	
	\begin{definition}
		Let $(\ambient,d,\refmeas,\form,\domain)$ be an NLMMD space, and let $\scjump \colon \ambient \times [0,\infty) \to [0,\infty)$
		be a regular scale function on $(\ambient,d)$ with threshold $\scjthres{\scjump}$.
		\begin{enumerate}[\rm(a)]\setlength{\itemsep}{0pt}\vspace{-5pt}
			\item (Jump kernel estimate) We say that $(\ambient,d,\refmeas,\form,\domain)$ satisfies
			the \emph{jump kernel estimate} \hypertarget{jphi}{$\on{J}(\scjump)$} if there exist
			a symmetric Borel measurable function $\jumpker \colon \offdiag{\ambient} \to (0,\infty)$ and $C \in (1,\infty)$ such that 
			\begin{equation} \label{eq:jump-kernel-estimate}
				\frac{C^{-1}}{\refmeas\bigl(B(x,d(x,y))\bigr) \scjump(x,d(x,y))} \le \jumpker(x,y)
				\le \frac{C}{\refmeas\bigl(B(x,d(x,y))\bigr) \scjump(x,d(x,y))}
			\end{equation}
			for all $(x,y) \in \offdiag{\ambient}$ and
			\begin{equation} \label{eq:DF-pure-jump-ker}
				\form(u,u)=\frac{1}{2}\int_{\ambient}\int_{\ambient} (u(x)-u(y))^{2} \jumpker(x,y)\, \refmeas(dx)\,\refmeas(dy)
			\end{equation}
			for all $u \in \domain$.
			\item (Exit time estimate) We say that $(\ambient,d,\refmeas,\form,\domain)$
			satisfies the \emph{exit time lower estimate} \hypertarget{exit}{$\on{E}(\scjump)_{\geq}$},
			if there exist $C,A \in (1,\infty)$ such that an $\refmeas$-symmetric Hunt process
			$\diff = (\Omega, \events, \{\diff_{t}\}_{t\in[0,\infty]},\{\lawdiff_{x}\}_{x \in \oneptcpt{\ambient}})$
			on $\ambient$ whose Dirichlet form is $(\form,\domain)$ satisfies
			\begin{equation} \label{eq:exit-times-lower-estimate}
				\expdiff_{x}[\tau_{B(x,r)}] \geq C^{-1} \scjump(x,r)
			\end{equation}
			for all $x \in \ambient \setminus \mathcal{N}$ and all $r \in (0,\diam(\ambient)/A)$
			for some properly exceptional set $\mathcal{N} \subset \ambient$ for $\diff$.
			We denote the corresponding upper estimate and the two-sided estimate by
			$\on{E}(\scjump)_{\le}$ and $\on{E}(\scjump)$, respectively.
			\item (Stable-like heat kernel estimates)  We say that $(\ambient,d,\refmeas,\form,\domain)$
			satisfies the \textbf{stable-like heat kernel estimates} \hypertarget{shk}{$\on{SHK}(\scjump)$}
			if there exist $C_1  \in (1,\infty)$ and a heat kernel $\{ p_{t} \}_{t>0}$ of
			$(\ambient,\refmeas,\form,\domain)$ such that for each $t \in (0,\infty)$,
			\begin{equation} \label{eq:shk-upper}
				p_{t}(x,y) \le C_1 \biggl( \frac{1}{\refmeas\bigl(B(x,\scjump^{-1}(x,t))\bigr)} \wedge \frac{t}{\refmeas\bigl(B(x,d(x,y))\bigr) \scjump(x,d(x,y))} \biggr)
			\end{equation}
			and
			\begin{equation} \label{eq:shk-lower}
				p_{t}(x,y) \ge C_1^{-1} \biggl( \frac{1}{\refmeas\bigl(B(x,\scjump^{-1}(x,t))\bigr)} \wedge \frac{t}{\refmeas\bigl(B(x,d(x,y))\bigr) \scjump(x,d(x,y))} \biggr)
			\end{equation}
			for $\refmeas$-a.e.\ $x,y \in \ambient$, where $\scjump^{-1}(x,\cdot)$ denotes the
			inverse of the homeomorphism $\scjump(x,\cdot) \colon [0,\infty)\to [0,\infty)$
			and $B(x,0) \coloneqq \emptyset$.
		\end{enumerate}
	\end{definition}
	
	The following result plays a key role in our proof of heat kernel estimates for the boundary trace process.
	It characterizes stable-like heat kernel estimates \hyperlink{shk}{$\on{SHK}(\scjump)$}
	by the conjunction of the jump kernel estimate \hyperlink{jphi}{$\on{J}(\scjump)$}
	and exit time lower estimate \hyperlink{exit}{$\on{E}(\scjump)_{\geq}$} stated above.
	If $\ambient$ is unbounded then this characterization is essentially contained in \cite{CKW}. It is a slight modification of the equivalence between (1) and (2) in \cite[Theorem 1.15]{CKW}. If $\ambient$ is bounded, we argue using results in \cite{GHH-mv}.
	In Theorem \ref{t:shkchar}, we assume that $(\ambient,\refmeas,\form,\domain)$ is a
	regular Dirichlet space of \emph{pure jump type}, i.e., the strongly local part $\form^{(c)}$ and the
	killing measure $\kappa$ of $(\ambient,\refmeas,\form,\domain)$ in its
	Beurling--Deny decomposition \eqref{e:Beurling-Deny} are identically zero, or in other
	words, there exists a symmetric Radon measure $\jumpmeas$ on $\offdiag{\ambient}$ such that 
	\begin{equation} \label{eq:pure-jump-DF}
		\form(f,g) = \frac{1}{2} \int_{\offdiag{\ambient}} (\wt{f}(x)-\wt{f}(y))(\wt{g}(x)-\wt{g}(y))\, \jumpmeas(dx \, dy)
	\end{equation}
	for all $f,g \in \domain_{e}$, where $\wt{f},\wt{g}$ denote $\form$-quasi-continuous $\refmeas$-versions of $f,g$ respectively.
	
	\begin{theorem} \label{t:shkchar}
		Let $(\ambient,d,\refmeas,\form,\domain)$ be an NLMMD space of pure jump type
		satisfying \hyperlink{VD}{\textup{VD}}, and assume that $(\ambient,d)$ is uniformly perfect.
		Let $\scjump \colon \ambient \times [0,\infty) \to [0,\infty)$ be a regular scale function
		on $(\ambient,d)$ with threshold $\scjthres{\scjump}$. Then the following are equivalent:
		\begin{enumerate}[\rm(1)]\setlength{\itemsep}{0pt}
			\item\label{it:shkchar-shk} $(\ambient,d,\refmeas,\form,\domain)$ satisfies \hyperlink{shk}{$\on{SHK}(\scjump)$}.
			\item\label{it:shkchar-jphiE} $(\ambient,d,\refmeas,\form,\domain)$ satisfies \hyperlink{jphi}{$\on{J}(\scjump)$} and \hyperlink{exit}{$\on{E}(\scjump)_{\geq}$}.
		\end{enumerate}
		Furthermore, either of the above conditions implies that the following hold:
		\begin{enumerate}[\rm(a)]\setlength{\itemsep}{0pt}\vspace{-5pt}
			\item\label{it:shkchar-irr-conserv} $(\ambient,\refmeas,\form,\domain)$ is irreducible and conservative.
			\item\label{it:shkchar-CHK} A (unique) continuous heat kernel
			$p = p_{t}(x,y) \colon (0,\infty) \times \ambient \times \ambient \to [0,\infty)$
			of $(\ambient,\refmeas,\form,\domain)$ exists and satisfies \eqref{eq:shk-upper}
			and \eqref{eq:shk-lower} for any $(t,x,y) \in (0,\infty) \times \ambient \times \ambient$
			for some $C_{1} \in (1,\infty)$.
			\item\label{it:shkchar-Feller} Proposition \ref{p:feller}-\eqref{it:HKE-Feller}
			with ``Hunt process'' in place of ``diffusion'' holds.
			\item\label{it:shkchar-domain} Let $\jumpker \colon \offdiag{\ambient} \to (0,\infty)$
			be as given in \hyperlink{jphi}{$\on{J}(\scjump)$}. Then
			\begin{equation}\label{eq:shkchar-domain}
				\domain=\biggl\{u\in L^{2}(\ambient,\refmeas) \biggm| \int_{\ambient}\int_{\ambient} (u(x)-u(y))^{2} \jumpker(x,y)\, \refmeas(dx)\,\refmeas(dy)<\infty \biggr\}.
			\end{equation}
		\end{enumerate}
	\end{theorem}
	
	\begin{proof} 
		We note that uniform perfectness implies the reverse volume doubling property by Lemma \ref{l:rvd}. By a quasisymmetric change of metric as given in \cite[Proposition 5.2]{BM18} and \cite[(5.7), Proof of Lemma 5.7]{BM18}, it suffices to consider the case $\scjump(x,r)=r^\beta$ for all $x \in \ambient, r>0$, where $\beta>0$ (see also \cite{Kig-qs} where this kind of metric change first appeared). Therefore we will assume without loss of generality that  $\scjump(x,r)=r^\beta$ for all $x \in \ambient, r>0$, for some $\beta>0$.
		
		The implication from \eqref{it:shkchar-shk} to \eqref{it:shkchar-jphiE} follows
		from the same argument as \cite[Proof of $\textup{(1)}\Rightarrow\textup{(2)}$ of Theorem 1.15]{CKW}
		regardless of whether or not $\ambient$ is bounded. 
		
		For the converse implication from \eqref{it:shkchar-jphiE} to \eqref{it:shkchar-shk},
		the proof splits into two cases depending on whether or not $\ambient$ is bounded. \\
		\emph{Case 1: $\ambient$ is unbounded.} By \cite[Theorem 1.15]{CKW}, 
		it suffices to show the exit time upper bound \hyperlink{exit}{$\on{E}(\scjump)_\le$}.
		The exit time upper estimate \hyperlink{exit}{$\on{E}(\scjump)_\le$} follows from
		the Faber--Krahn inequality shown in \cite[Section 4.1]{CKW} along with \cite[Lemma 4.14]{CKW}.
		
		\noindent \emph{Case 2: $\ambient$ is bounded.} 
		The exit time upper estimate \hyperlink{exit}{$\on{E}(\scjump)_\le$} stated
		in the unbounded case also holds in the bounded case with almost the same proof.
		Since the proof of the Faber--Krahn inequality relies on the reverse volume doubling property,
		the statement of the Faber--Krahn inequality has to be modified so that
		for some $\sigma \in (0,1]$ it holds for all balls of radii
		$r \in (0,\sigma\diam(\ambient))$ as formulated in \cite[Definition 2.4]{GHH-mv}.
		
		Once the on-diagonal upper bound in the conclusion of \cite[Theorem 4.25]{CKW} is obtained,
		then the two-sided estimates on the jump kernel \hyperlink{jphi}{$\on{J}(\scjump)$} and
		exit time \hyperlink{exit}{$\on{E}(\scjump)$} imply the stable-like heat kernel estimates
		\hyperlink{shk}{$\on{SHK}(\scjump)$} by the arguments in \cite[Chapter 5]{CKW}
		with minor modifications to take into account that $\ambient$ is bounded.
		Therefore it is enough to prove the on-diagonal upper bound.
		
		In order to show the on-diagonal bound, by \cite[Theorems 2.10 and 2.12]{GHH-mv},
		it suffices to show the condition (Gcap) in \cite[Definition 2.3]{GHH-mv}, which
		in turn follows from \cite[Proposition 13.4 and Lemma 13.5]{GHH-lb} or \cite[Theorem 14.1]{GHH-lb}
		along with the two-sided exit time estimate \hyperlink{exit}{$\on{E}(\scjump)$},
		completing the proof that \eqref{it:shkchar-jphiE} implies \eqref{it:shkchar-shk}.
		
		We next assume \eqref{it:shkchar-shk} (and \eqref{it:shkchar-jphiE}) and prove \eqref{it:shkchar-irr-conserv},
		\eqref{it:shkchar-CHK}, \eqref{it:shkchar-Feller} and \eqref{it:shkchar-domain}.
		\begin{enumerate}[\rm(a)]\setlength{\itemsep}{0pt}\vspace{-5pt}
			\item[\eqref{it:shkchar-irr-conserv}]The irreducibility of $(\ambient,\refmeas,\form,\domain)$
			is immediate by \eqref{eq:shk-lower} from \hyperlink{shk}{$\on{SHK}(\scjump)$}
			or by \hyperlink{jphi}{$\on{J}(\scjump)$} and Lemma \ref{l:irreducible} below.
			For the conservativeness of $(\form,\domain)$, we consider two cases depending on whether or not $(\ambient,d)$ is bounded.
			If $(\ambient,d)$ is bounded, then $\one_{\ambient} \in \domain$ by the compactness of
			$\ambient$ and \cite[Exercise 1.4.1]{FOT}, $\form(\one_{\ambient},\one_{\ambient})=0$
			by \eqref{eq:pure-jump-DF} and thus $(\ambient,\refmeas,\form,\domain)$ is conservative.
			If $(\ambient,d)$ is unbounded, then \eqref{eq:shk-lower} from \hyperlink{shk}{$\on{SHK}(\scjump)$}
			implies that there exists $c_0 \in (0,\infty)$ such that $T_t \one_{\ambient}(x) \ge c_0$ $\refmeas$-a.e.\ on $\ambient$
			for each $t \in (0,\infty)$. This along with \cite[Proposition 3.1-(1)]{CKW}
			implies that $(\ambient,\refmeas,\form,\domain)$ is conservative.
			\item[\eqref{it:shkchar-CHK}]The existence of a continuous heat kernel $p = p_{t}(x,y)$
			of $(\ambient,\refmeas,\form,\domain)$ follows from \cite[Lemma 5.6]{CKW},
			and the validity of \eqref{eq:shk-upper} and \eqref{eq:shk-lower} for any $(t,x,y) \in (0,\infty) \times \ambient \times \ambient$
			is immediate from \hyperlink{shk}{$\on{SHK}(\scjump)$}, \hyperlink{VD}{\textup{VD}} and \eqref{e:Psireg1}.
			\item[\eqref{it:shkchar-Feller}] This is proved in the same way as \cite[Proposition 3.2]{Lie}
			on the basis of \hyperlink{VD}{\textup{VD}}, \eqref{it:shkchar-CHK} and the
			conservativeness of $(\ambient,\refmeas,\form,\domain)$ from \eqref{it:shkchar-irr-conserv}.
			\item[\eqref{it:shkchar-domain}]Using \eqref{e:semigroup} and
			the conservativeness of $(\ambient,\refmeas,\form,\domain)$, we obtain
			\begin{equation}\label{e:dmn1}
				\domain= \biggl\{ u \in L^2(\ambient,\refmeas) \biggm| \lim_{t \downarrow 0}\frac{1}{2t} \int_{\ambient}\int_{\ambient}(u(x)-u(y))^2 p_t(x,y) \,m(dx)\,m(dy) < \infty \biggr\},
			\end{equation}
			where $\{ p_{t} \}_{t>0}$ denotes a heat kernel of $(\ambient,\refmeas,\form,\domain)$.
			We then see from \hyperlink{shk}{$\on{SHK}(\scjump)$} and \hyperlink{jphi}{$\on{J}(\scjump)$}
			that there exists $C_1 \in (0,\infty)$ such that for each $t\in(0,\infty)$ we have
			\begin{equation}\label{e:dmn2}
				\jumpker(x,y) \le C_1 \frac{p_{t}(x,y)}{t} \quad \textrm{for $\refmeas$-a.e.\ $x,y \in \ambient$ with $d(x,y) \ge t^{1/\beta}$}
			\end{equation}
			and
			\begin{equation}\label{e:dmn3}
				\frac{p_{t}(x,y)}{t}\le C_1 \jumpker(x,y) \quad \textrm{for $\refmeas$-a.e.\ $x,y \in \ambient$.}
			\end{equation}
			The conclusion \eqref{eq:shkchar-domain} now follows from \eqref{e:dmn1}, \eqref{e:dmn2}, \eqref{e:dmn3} and the monotone convergence theorem.
			\qedhere\end{enumerate}
	\end{proof}
	
	\begin{remark} \label{r:bddhke}
		If $(\ambient,d)$ is unbounded, the on-diagonal upper bound in the proof of the implication from \eqref{it:shkchar-jphiE} to \eqref{it:shkchar-shk} above follows from \cite[Theorem 4.25]{CKW}.
		However, the proof there does not directly generalize to the case when $\ambient$ is bounded.
		This is because \cite[Proof of Theorem 4.25]{CKW} relies on \cite[Proposition 4.23]{CKW} which in turn uses \cite[Lemma 4.18]{CKW} on a sequence of radii going to infinity.
		However, the generalization of \cite[Lemma 4.18]{CKW}, which relies on the Faber--Krahn inequality, requires the radii to satisfy $r < c \diam(\ambient)$ for some $c \in (0,\infty)$,
		which seems insufficient for the argument in \cite[Proof of Proposition 4.23]{CKW}.
		See \cite[Remark 8.3]{CC26} for a more direct argument to extend the main results
		of \cite{CKW} to the case where the state space is bounded.
	\end{remark}
	
	We also give a simple sufficient condition for the irreducibility of a pure-jump Dirichlet form,
	which in particular applies to any NLMMD space $(\ambient,d,\refmeas,\form,\domain)$ satisfying
	\hyperlink{jphi}{$\on{J}(\scjump)$} for some regular scale function $\scjump$ on $(\ambient,d)$.
	
	\begin{lem} \label{l:irreducible}
		Let $(\ambient,\refmeas,\form,\domain)$ be a regular Dirichlet space satisfying \eqref{eq:DF-pure-jump-ker}
		for any $u \in \domain$ for some symmetric Borel measurable function $\jumpker \colon \offdiag{\ambient} \to (0,\infty)$.
		Then $(\ambient,\refmeas,\form,\domain)$ is irreducible.
	\end{lem}
	
	\begin{proof}
		Let $A \in \Borel(\ambient)$ be $\form$-invariant.
		Then for any $u,v \in \domain$, by \cite[Theorem 1.6.1]{FOT}, we have $\one_{A}u, \one_{A}v, \one_{\ambient \setminus A}u, \one_{\ambient \setminus A}v \in \domain$ and 
		\begin{equation} \label{e:irr1}
			0 = \form(\one_{A}u, \one_{\ambient \setminus A}v )+ \form(\one_{\ambient \setminus A} u, \one_{A} v ).
		\end{equation}
		Let $K_1 \subset A$ and $K_2 \subset \ambient \setminus A$ be arbitrary compact subsets.
		By the regularity of $(\form,\domain)$ there exist $u,v \in \domain \cap \contfunc_{\mathrm{c}}(\ambient)$
		such that $u,v$ are $[0,1]$-valued, $u|_{K_1} \equiv 1$ and $v|_{K_2} \equiv 1$.
		By using \eqref{e:irr1}, we have
		\begin{equation}\label{e:irr2}
			0 = \int_{A} \int_{B} u(x) v(y) \jumpker(x,y)\,\refmeas(dy) \,\refmeas(dx)
			\ge \int_{K_1} \int_{K_2} \jumpker(x,y)\,\refmeas(dy) \,\refmeas(dx),
		\end{equation}
		which together with the strict positivity of $\jumpker$ shows that $\refmeas(K_1) \refmeas(K_2) = 0$
		for any compact sets $K_1,K_2$ with $K_1 \subset A$ and $K_2 \subset \ambient \setminus A$.
		By the inner regularity of $\refmeas$ (see, e.g., \cite[Theorem 2.18]{Rud}), we conclude
		$\refmeas(A) \refmeas(\ambient \setminus A) = 0$, which means the irreducibility of $(\form,\domain)$.  
	\end{proof}

	\subsection{Capacity good measures and their corresponding PCAFs} \label{ss:capgood}
	
	To define a trace process, we need an $\form$-smooth measure and need to identify the support of
	the corresponding positive continuous additive functional (PCAF). To this end, in this subsection,
	we provide a general sufficient condition for a measure to be $\form$-smooth in the strict sense
	and for its support to coincide with the support of the corresponding PCAF in the strict sense of $\diff$.
	We remark that the former notion of support can be larger by a non-$\form$-polar set than the latter
	for a general $\form$-smooth Radon measure in the strict sense; see \cite[Example 5.1.2]{FOT}
	for such an example, which is originally due to Sturm \cite[Section 9]{Stu92}.
	
	The class of measures we consider in this subsection are capacity good measures.
	The following definition is a slight variant of \cite[Definition 4.1]{BM18} and \cite[Definition 6.2]{BCM}.
	
	\begin{definition}[Capacity good measure]\label{d:goodmeas}
		Let $(\ambient,d,\refmeas,\form,\domain)$ be an MMD space that satisfies Assumption \ref{a:feller}.
		Let $\goodmeas$ be a Borel measure on $\ambient$ and let $F \coloneqq \supp_{\ambient}[\goodmeas]$ denote its support.
		We say that $\goodmeas$ is \textbf{$\form$-capacity good} if $\goodmeas(\ambient)>0$ and
		there exist $C_0,A_0,A_1 \in (1,\infty)$ and a regular scale function
		$\scjump \colon F \times [0,\infty) \to [0,\infty)$ on $(F,d)$
		with threshold $\scjthres{\scjump}\in(0,\diam(\ambient)]$
		(in the sense of Definition \ref{d:regscale}) such that
		\begin{equation} \label{e:cgood}
			C_0^{-1} \scjump(x,r) \le
			\frac{\goodmeas (B(x,r))}{\Capa_{B(x,A_0r)} (B(x,r))} 
			\le C_0  \scjump(x,r) \quad \textrm{for all $(x,r) \in F \times (0,\scjthres{\scjump}/A_1)$.}
		\end{equation}
		By \cite[Lemmas 5.22 and 5.23]{BCM}, by changing $C_0,A_1 \in (1,\infty)$ if necessary, we may assume that $A_0=2$ in \eqref{e:cgood}. 
		%	Since $\goodmeas$ is locally finite, any capacity good measure $\goodmeas$ is a Radon measure on $D$, if $D$ is open in $\ambient$.
	\end{definition}
	
	We make the following assumption for the remainder of this subsection.
	
	\begin{assumption}\label{a:capgood}
		Let a scale function $\scdiff$, an MMD space $(\ambient,d,\refmeas,\form,\domain)$
		and a diffusion $\diff = (\Omega, \events, \{\diff_{t}\}_{t\in[0,\infty]},\{\lawdiff_{x}\}_{x \in \oneptcpt{\ambient}})$
		on $\ambient$ satisfy Assumption \ref{a:feller}.
		Let $\goodmeas$ be an $\form$-capacity good Borel measure on $\ambient$
		with support $F \coloneqq \supp_{\ambient}[\nu]$ and with a regular scale function
		$\scjump \colon F \times [0,\infty) \to [0,\infty)$ on $(F,d)$
		with threshold $\scjthres{\scjump}\in(0,\diam(\ambient)]$
		as given in Definition \ref{d:goodmeas}.
	\end{assumption}
	
	If $\goodmeas$ is as given in Assumption \ref{a:capgood}, then
	$(F,d,\goodmeas)$ satisfies \hyperlink{VD}{\textup{VD}} by \cite[Lemma 5.23]{BCM}.
	In particular by \eqref{e:vd}, there exist $C \in (1,\infty)$ and $\beta\in(0,\infty)$ such that
	\begin{equation} \label{e:cgvd}
		\frac{\goodmeas(B(\xi,R))}{\goodmeas(B(\xi,r))} \le C \biggl( \frac{R}{r} \biggr)^\beta \quad \textrm{for all $\xi \in F$ and all $0<r\leq R$.}
	\end{equation}
	
	The following lemma is an upper bound on the integral of the heat kernel with respect to
	an $\form$-capacity good measure $\goodmeas$. This upper bound is later used to show that
	any $\form$-capacity good measure is $\form$-smooth in the strict sense (Lemma \ref{l:strict})
	and to identify the support of the corresponding PCAF in the strict sense as
	the topological support of $\goodmeas$ (Proposition \ref{prop:goodpcafsupp}).
	
	\begin{lem} \label{l:hkbdy}
		Let $\scdiff,(\ambient,d,\refmeas,\form,\domain),\goodmeas,F$ be as in Assumption \ref{a:capgood}.
		Then there exists $C \in (1,\infty)$ such that for any $(t,x) \in (0,\infty) \times \ambient$,
		\begin{equation} \label{e:hkb}
			\int_{F}	p_t(x,y)\,\goodmeas(dy) \le C \frac{\goodmeas\left( B(\xi_x,\scdiff^{-1}(t))\right)}{\refmeas(B(\xi_x,\scdiff^{-1}(t))},
		\end{equation}
		where $\xi_x \in F$ is any point such that $\dist(x,F)=d(x,\xi_x)$.
	\end{lem}
	
	\begin{proof}
		By \hyperlink{hke}{$\on{HKE}(\scdiff)$}, \cite[Lemma 3.19]{GT12} and \eqref{e:reg},
		there exist $C_1 \in(1,\infty)$, $c_2 \in (0,1)$ and $0<\alpha_1 < \alpha_2<\infty$
		such that for all $x,y \in \ambient$, we have 
		\begin{equation} \label{e:hb1}
			p_t(x,y)=p_t(y,x) \le \frac{C_1}{\refmeas(B(x,\scdiff^{-1}(t)))} \exp \biggl( -c_2 \min\biggl\{ \biggl(\frac{d(x,y)}{\scdiff^{-1}(t)}\biggr)^{\alpha_1},\biggl(\frac{d(x,y)}{\scdiff^{-1}(t)}\biggr)^{\alpha_2}\biggr\}\biggr).
		\end{equation}
		If $\xi_x \in F$ satisfies $\dist(x,F)=d(x,\xi_x)$, then 
		\begin{equation} \label{e:hb2}
			d(\xi_x,y)\le d(x,y)+d(x,\xi_x) \le 2d(x,y) \quad 
			\mbox{for all $y \in F$.}
		\end{equation}
		By \eqref{e:hb1}, \eqref{e:hb2} and \eqref{e:vd}, there exist $C_2 \in(1,\infty)$ and $c_3 \in (0,1)$
		such that for all $x \in \ambient$ and all $y,\xi_{x} \in F$ with $\dist(x,F)=d(x,\xi_x)$, we have
		\begin{align} \label{e:hb3}
			p_t(x,y) &\le \frac{C_1}{\refmeas(B(y,\scdiff^{-1}(t)))} \exp \biggl( -c_3 \min\biggl\{ \biggl(\frac{d(\xi_x,y)}{\scdiff^{-1}(t)}\biggr)^{\alpha_1},\biggl(\frac{d(\xi_x,y)}{\scdiff^{-1}(t)}\biggr)^{\alpha_2}\biggr\}\biggr) \nonumber \\
			&\le \frac{C_2}{\refmeas(B(\xi_x,\scdiff^{-1}(t)))} \exp \biggl( -\frac{c_3}{2} \min\biggl\{ \biggl(\frac{d(\xi_x,y)}{\scdiff^{-1}(t)}\biggr)^{\alpha_1},\biggl(\frac{d(\xi_x,y)}{\scdiff^{-1}(t)}\biggr)^{\alpha_2}\biggr\}\biggr).
		\end{align}
		
		Now for all $(t,x) \in (0,\infty) \times \ambient$ and all $\xi_{x} \in F$ with
		$\dist(x,F)=d(x,\xi_x)$, using \eqref{e:hb3} and \eqref{e:vd}, we obtain
		\begin{align}
			\lefteqn{\int_{F} p_t(x,y)\,  \goodmeas(dy) }\nonumber \\
			&= \int_{B(\xi_x,\scdiff^{-1}(t))} p_t(x,\cdot)\,d\goodmeas   +  \sum_{k=1}^\infty \int_{B(\xi_x,2^k\scdiff^{-1}(t))\setminus B(\xi_x,2^{k-1}\scdiff^{-1}(t))}p_t(x,\cdot)\, d\goodmeas \nonumber \\
			&\overset{\eqref{e:hb3}}{\lesssim} \frac{\goodmeas(  B(x,\scdiff^{-1}(t)))}{\refmeas(B(\xi_x,\scdiff^{-1}(t))} + \sum_{k =1}^\infty \frac{\goodmeas(  B(\xi_x,2^k\scdiff^{-1}(t)))}{\refmeas(B(\xi_x,\scdiff^{-1}(t))} \exp(-c2^{\alpha_1 k})\nonumber  \\
			&\lesssim \frac{\goodmeas(  B(\xi_x,\scdiff^{-1}(t)))}{\refmeas(B(\xi_x,\scdiff^{-1}(t))} + \sum_{k =1}^\infty \frac{\goodmeas(  B(\xi_x,2^k\scdiff^{-1}(t)))}{\refmeas(B(\xi_x,\scdiff^{-1}(t))} \exp(-c2^{\alpha_1 k})\nonumber  \\
			&\lesssim \frac{\goodmeas( B(\xi_x,\scdiff^{-1}(t)))}{\refmeas(B(\xi_x,\scdiff^{-1}(t))} \left[1+ \sum_{k=1}^\infty 2^{k \beta} \exp(-c2^{\alpha_3 k}) \right] \quad \mbox{(by \eqref{e:cgvd})}
			\label{e:hb4} \\
			&\lesssim \frac{\goodmeas( B(\xi_x,\scdiff^{-1}(t)))}{\refmeas(B(\xi_x,\scdiff^{-1}(t))}. \nonumber
			\qedhere\end{align}
	\end{proof}
	
	Next, we show that $\goodmeas$ is an $\form$-smooth measure in the strict sense as defined in \cite[p. 238]{FOT}.
	
	\begin{lem} \label{l:strict}
		Let $\scdiff,(\ambient,d,\refmeas,\form,\domain),\goodmeas,F$ be as in Assumption \ref{a:capgood}.
		Then $\goodmeas$ is  $\form$-smooth in the strict sense.
	\end{lem}
	
	\begin{proof}
		%	We only consider the case when $\ambient$ is unbounded (the bounded case is similar and easier).
		Let $\scjump,\scjthres{\scjump},A_1$ be as in Assumption \ref{a:capgood} and Definition \ref{d:goodmeas}. By \cite[Theorem 1.2]{GHL15}
		and \cite[Lemmas 5.22 and 5.23]{BCM}, there exists $C_0 \in (1,\infty)$ such that 
		\begin{equation} \label{e:capest}
			C^{-1} \frac{m(B(x,r))}{\scdiff(r)} \le \Capa_{B(x,2r)}(B(x,r)) \le C \frac{m(B(x,r))}{\scdiff(r)}
			\quad \textrm{for all $x \in \ambient$ and all $r \in (0,\infty)$.}
		\end{equation}
		
		For $\xi \in F$ and $r \in (0,\scjthres{\scjump}/A_{1})$, we consider the measure
		$\goodmeas_{\xi,r}(\cdot) \coloneqq \goodmeas(\cdot \cap B(\xi,r))$. By the same argument as for \eqref{e:hb1}
		there exist $C_1 \in(1,\infty)$, $c_2 \in (0,1)$ and $0<\alpha_1 < \alpha_2<\infty$
		such that for all $x,y,z \in \ambient$ with $d(x,y) \le d(x,z)$, we have
		\begin{align} \label{e:str1}
			p_t (x,z) &\le \frac{C_1}{m(B(y,\scdiff^{-1}(t)))} \exp \biggl( - c_1\min\biggl\{ \biggl(\frac{d(x,y)}{\scdiff^{-1}(t)}\biggr)^{\alpha_1}, \biggl(\frac{d(x,y)}{\scdiff^{-1}(t)}\biggr)^{\alpha_2}\biggr\}\biggr).
		\end{align}
		Note that for any $x \in \ambient \setminus B(\xi,2r)$ and $z \in B(\xi,r)$,
		we have $d(x,z) \ge d(\xi,z)$. Hence by \eqref{e:str1} and the same argument as \eqref{e:hb4}, we obtain
		\begin{equation} \label{e:str2}
			\int_{F} p_s(x,\cdot)\, d\goodmeas_{\xi,r} \lesssim \frac{\goodmeas( B(\xi,\scdiff^{-1}(s)))}{m(B(\xi,\scdiff^{-1}(s))}
			\qquad \textrm{for all $x \in \ambient \setminus B(\xi,2r)$}.
		\end{equation}
		
		If $x \in B(\xi,2r)$, then by Lemma \ref{l:hkbdy},
		\begin{align} \label{e:str3}
			\int_{F} p_s(x,y) \goodmeas_{\xi,r}(dy) \le \int_{F} p_s(x,y) \goodmeas(dy) &\lesssim \frac{\goodmeas(  B(\xi_x,\scdiff^{-1}(s)))}{m(B(\xi_x,\scdiff^{-1}(s))},
		\end{align}
		where $\xi_x \in B(\xi, 3r) \cap F$ satisfies $\dist(x,F)=d(x,\xi_x)$.
		% 	where $\xi_x \in B(\xi, 3n) \cap F$ is chosen so that $d(x,\xi_x)= \delta_\unifdom(\xi_x)$. To obtain \eqref{e:str3} for $x, \xi_x$ as above, we use the sub-Gaussian heat kernel bounds for $p_s^{\on{ref}}(x,\cdot)$ and the volume doubling property for $\goodmeas$ and $\refmeas$ to estimate 
		% 	\begin{align*}
			% 		\lefteqn{\int_{F} p_s^{\on{ref}}(x,y)\, d\goodmeas_{\xi,r}(y)}\nonumber \\
			% 		&=	\int_{B(x,\scdiff^{-1}(s))} p_s^{\on{ref}}(\xi,\cdot)\,d\goodmeas(y)  +  \sum_{k=1}^\infty \int_{B(\xi,M^k\scdiff^{-1}(s))\setminus B(\xi,M^{k-1}\scdiff^{-1}(s))}p_s^{\on{ref}}(\xi,\cdot)\, d\goodmeas \\
			% 		& \lesssim \frac{\goodmeas(F \cap B(x,\scdiff^{-1}(s)))}{m(B(x,\scdiff^{-1}(s))} + \sum_{k =1}^\infty \frac{\goodmeas(F \cap B(x,M^k\scdiff^{-1}(s)))}{m(B(x,\scdiff^{-1}(s))} \exp(-cM^{\alpha k}) \\
			% 		& \lesssim \frac{\goodmeas(F \cap B(\xi_x,\scdiff^{-1}(s)))}{m(B(\xi_x,\scdiff^{-1}(s))} + \sum_{k =1}^\infty \frac{\goodmeas(F \cap B(\xi_x,M^k\scdiff^{-1}(s)))}{m(B(\xi_x,\scdiff^{-1}(s))} \exp(-cM^{\alpha k}) \\
			% 		&\lesssim  \frac{\goodmeas(F \cap B(\xi_x,\scdiff^{-1}(s)))}{m(B(\xi_x,\scdiff^{-1}(s))} \left[1+ \sum_{k=1}^\infty M^{k \beta} \exp(-cM^{\alpha k}) \right] \\
			% 		&\lesssim  \frac{\goodmeas(F \cap B(\xi_x,\scdiff^{-1}(s)))}{m(B(\xi_x,\scdiff^{-1}(s))}.
			% 	\end{align*}
		For all $\eta \in F$  using the doubling property of $\refmeas$ and $\goodmeas$, we have
		\begin{align} \label{e:str4}
			\int_0^1 e^{-s} \frac{\goodmeas(  B(\eta,\scdiff^{-1}(s)))}{m(B(\eta,\scdiff^{-1}(s))} \, ds &= \sum_{k=0}^\infty \int_{2^{-k}}^{2^{-k+1}} e^{-s} \frac{\goodmeas(  B(\eta,\scdiff^{-1}(s)))}{m(B(\eta,\scdiff^{-1}(s))} \, ds \nonumber \\
			&\asymp \sum_{k=0}^\infty  \frac{\goodmeas(  B(\eta,\scdiff^{-1}(2^{-k} )))}{m(B(\eta,\scdiff^{-1}(2^{-k} ))} 2^{-k}  \nonumber \\
			& {\asymp} \sum_{k=0}^\infty \scjump(\eta, \scdiff^{-1}(2^{-k} )) \quad \textrm{(by \eqref{e:cgood} and \eqref{e:capest})} \nonumber\\
			&\asymp \scjump(\eta, \scdiff^{-1}(1)) \quad \mbox{(by \eqref{e:Psireg1} and \cite[Lemma 3.19]{GT12})},
		\end{align}
		and
		\begin{align} \label{e:str5}
			\int_1^\infty e^{-s} \frac{\goodmeas( B(\eta,\scdiff^{-1}(s)))}{\refmeas(B(\eta,\scdiff^{-1}(s))} \, ds
			&= \sum_{k=1}^\infty \int_{2^{k-1}}^{2^k}  e^{-s} \frac{\goodmeas(B(\eta,\scdiff^{-1}(s)))}{\refmeas(B(\eta,\scdiff^{-1}(s))} \, ds \nonumber\\
			&\lesssim \sum_{k=1}^\infty  \frac{\goodmeas(  B(\eta,\scdiff^{-1}(2^{k} )))}{m(B(\eta,\scdiff^{-1}(2^{k} ))} 2^{ k}  e^{-2^{k-1}}\nonumber \\
			&\asymp \sum_{k=0}^\infty \scjump(\eta, \scdiff^{-1}(2^{k} ))e^{-2^{k-1}}  \quad \mbox{(by \eqref{e:capest} and \eqref{e:cgood})}\nonumber \\
			&\overset{\eqref{e:Psireg1}}{\lesssim}   \sum_{k=0}^\infty \scjump(\eta, \scdiff^{-1}(1 )) 2^{k \beta_2}e^{-2^{k-1}}   \lesssim \scjump(\eta, \scdiff^{-1}(1)).
		\end{align}
		Combining \eqref{e:str2}, \eqref{e:str3}, \eqref{e:str4},  \eqref{e:str5} and using \eqref{e:Psireg1}, we obtain
		\begin{equation} \label{e:str6}
			\int_{F} \int_{0}^\infty e^{-t} p_t(x,y) \,dt\, \goodmeas_{\xi,r}(dy) \lesssim \sup_{\eta \in F \cap B(\xi,3r)}  \scjump(\eta, \scdiff^{-1}(1)) \lesssim \scjump(\xi,r) 
		\end{equation}
		for all $x \in \ambient$. Since $\goodmeas_{\xi,r}$ is a finite measure such that the corresponding $1$-potential
		$x \mapsto \int_{F} \int_{0}^\infty e^{-t} p_t(x,y) \,dt\, \goodmeas_{\xi,r}(dy)$ is bounded,
		we conclude from \cite[Exercise 4.2.2]{FOT} that $\goodmeas_{\xi,r}$ is of finite energy integral
		for all $\xi\in F$ and all $r \in (0,\scjthres{\scjump}/A_{1})$.
		By covering the set $B(\xi,R) \cap F$ with finitely many balls centered at $F$
		and of radii less than $\scjthres{\scjump}/A_{1}$, we conclude that
		$\goodmeas_{\xi,R}(\cdot)=\goodmeas(\cdot \cap B(\xi,R))$ is a finite measure
		of finite energy integral and that the corresponding $1$-potential
		$x \mapsto \int_{F} \int_{0}^\infty e^{-t} p_t(x,y) \,dt\, \goodmeas_{\xi,R}(dy)$
		is bounded for all $\xi \in F$ and all $R \in (0,\infty)$.
		Therefore $\goodmeas$ is $\form$-smooth in the strict sense.
	\end{proof}
	
	We record another upper bound on an integral of heat kernel with respect to $\goodmeas$ similar to Lemma \ref{l:hkbdy}.
	
	\begin{lem} \label{l:ballbdy}
		Let $\scdiff,(\ambient,d,\refmeas,\form,\domain),\goodmeas,F,\scjump,\scjthres{\scjump}$
		be as in Assumption \ref{a:capgood}. Then there exist $C \in (1,\infty)$ and $A \in (4,\infty)$
		such that for all $(\xi,r) \in F \times (0,\scjthres{\scjump}/A)$ and all $x \in B(\xi,r)$,
		\begin{equation} \label{e:ubb1}
			\int_{{F} \cap B(\xi,r)} \int_0^\infty p^{B(\xi,r)}_{t}(x,y)\,dt\, \goodmeas(dy) \le   C \scjump(\xi,r),
		\end{equation}
		where $p^{B(\xi,r)} = p^{B(\xi,r)}_{t}(x,y) \colon (0,\infty) \times B(\xi,r) \times B(\xi,r) \to [0,\infty)$
		denotes the continuous heat kernel of $\bigl(B(\xi,r),\refmeas|_{B(\xi,r)},\form^{B(\xi,r)},\domain^{0}(B(\xi,r))\bigr)$
		as given in Proposition \ref{p:feller}-\eqref{it:HKE-part-CHK-sFeller}.
	\end{lem}
	\begin{proof}
		%Let $K_0 \in (1,\infty)$ be such that $\scdiff(K_0^{-1}r) \le \scdiff(r)/2$ for all $r >0$.
		By Fubini's theorem and Lemma \ref{l:hkbdy}, there exists $A_1 \in (1,\infty)$ such that
		for all $(\xi,r) \in F \times (0,\scjthres{\scjump}/A_1)$ and all $x \in B(\xi,r)$ we have 
		\begin{align} \label{e:bb1}
			\MoveEqLeft{\int_{{F} \cap B(\xi,r)} \int_0^{\scdiff(r)} p^{B(\xi,r)}_{t}(x,y)\,dt\, \goodmeas(dy)} \nonumber \\
			&\le \int_{F} \int_0^{\scdiff(r)} p_{t}(x,y)\,dt\, \goodmeas(dy) \quad\mbox{(since $ p^{B(\xi,r)}(\cdot,\cdot) \le  p(\cdot,\cdot)$)} \nonumber\\
			&\lesssim \int_0^{\scdiff(r)} \frac{\goodmeas\bigl( B(\xi_x,\scdiff^{-1}(t))\bigr)}{\refmeas\bigl(B(\xi_x,\scdiff^{-1}(t)\bigr)} \,dt=\sum_{k=0}^\infty \int_{\scdiff(2^{-k}r)}^{\scdiff(2^{-(k-1)}r)} \frac{\goodmeas\bigl( B(\xi_x,\scdiff^{-1}(t)) \bigr)}{\refmeas\bigl(B(\xi_x,\scdiff^{-1}(t) \bigr)} \,dt \quad\textrm{(by \eqref{e:hkb})} \nonumber\\
			&\lesssim \sum_{k=0}^\infty \frac{\goodmeas( B(\xi_x,2^{-k}r))}{\refmeas(B(\xi_x,2^{-k}r))} \scdiff(2^{-k}r) \overset{\eqref{e:cgood}, \eqref{e:capest}}{\asymp} \sum_{k=0}^\infty \scjump(\xi_x,2^{-k}r) \overset{\eqref{e:Psireg1}}{\asymp} \scjump(\xi_x, r)\overset{\eqref{e:Psireg1}}{\asymp} \scjump(\xi, r),
		\end{align}
		where $\xi_x \in F$ is chosen as in Lemma \ref{l:hkbdy}.
		
		By \cite[Proof of Theorem 2.5]{HS}, there exist $C_1,A_1 \in (1,\infty)$ such that
		for all $(x,r) \in \ambient \times (0,\diam(\ambient)/A_1)$, the first Dirichlet eigenvalue 
		\begin{equation*}
			\lambda_0(B(x,r)) \coloneqq \inf \biggl\{\frac{\form(f,f)}{\int_{B(x,r)}f^2\,d\refmeas} \biggm| f \in \domain^{0}(B(x,r)) \biggr\}
		\end{equation*}
		satisfies
		\begin{equation} \label{e:bb2}
			\frac{C_1^{-1}}{\scdiff(r)} \le	\lambda_0(B(x,r)) \le \frac{C_1}{\scdiff(r)}.
		\end{equation} 
		Hence by \cite[Proof of Lemma 3.9-(3)]{HS} and   \eqref{e:bb2}, there exist
		$C_2,A_1 \in (1,\infty)$ and $c_1 \in (0,\infty)$ such that for all
		$(x,r) \in \ambient \times (0,\diam(\ambient)/A_1)$, all $y,z \in  B(x,r)$ and all $t \in [\scdiff(r),\infty)$, we have 
		\begin{equation}\label{e:bb3}
			p^{B(x,r)}_{t}(y,z) \le \frac{ C_2}{\refmeas(B(x,r))} \exp \biggl(-\frac{c_1t}{\scdiff(r)}\biggr).
		\end{equation}
		Therefore for all $(\xi,r) \in F \times (0,\scjthres{\scjump}/A_1)$ and all $x \in  B(\xi,r)$ we have 
		\begin{align} \label{e:bb4}
			\MoveEqLeft{\int_{{F} \cap B(\xi,r)} \int_{\scdiff(r)}^\infty p^{B(\xi,r)}_{t}(x,y)\,dt\, \goodmeas(dy)} \nonumber \\
			&\le \int_{{F} \cap B(\xi,r)}\int_{\scdiff(r)}^\infty \frac{C_2}{\refmeas(B(\xi,r)} \exp \biggl(-\frac{c_1t}{\scdiff(r)}\biggr) \,dt\, \goodmeas(dy) \quad\mbox{(by \eqref{e:bb3})} \nonumber\\
			&\lesssim \int_{{F} \cap B(\xi,r)} \frac{ \scdiff(r)}{\refmeas(B(\xi,r)} \, \goodmeas(dy) \asymp \frac{\goodmeas(B(\xi,r))\scdiff(r)}{\refmeas(B(\xi,r)} \overset{\eqref{eq:bdry-scale-estimate}}{\asymp} \scjump(\xi,r).
		\end{align}
		By \eqref{e:bb1} and \eqref{e:bb4}, we obtain the desired upper bound \eqref{e:ubb1}.
	\end{proof}
	
	Since $\goodmeas$ is an $\form$-smooth measure in the strict sense as proved in Lemma \ref{l:strict},
	it defines a PCAF in the strict sense due to the Revuz correspondence by \cite[Theorem 4.1.11]{CF} or \cite[Theorem 5.1.7]{FOT}.
	
	\begin{definition} \label{d:goodpcaf}
		Let $(\ambient,d,\refmeas,\form,\domain),\diff,\goodmeas$ be as in Assumption \ref{a:capgood}.
		We let $\goodpcaf = \{ \goodpcaf_{t} \}_{ t \in [0,\infty) }$ denote
		a positive continuous additive functional (PCAF) in the strict sense of $\diff$
		whose Revuz measure is $\goodmeas$, with a defining set $\Lambda \in \minaugfilt_{0}$ such that
		$\goodpcaf_{t}(\omega)=0$ for any $(t,\omega)\in[0,\infty)\times(\Omega\setminus\Lambda)$
		and $\{\zeta=0\} \subset \Lambda$; the existence of such $\goodpcaf$
		follows from Lemma \ref{l:strict} and \cite[Theorem 5.1.7]{FOT}.
	\end{definition} 
	
	The state space of the time-changed process $\difftr$ of $\diff$
	by the PCAF $\goodpcaf$ is the support of $\goodpcaf$ (recall \eqref{e:supportA},
	\eqref{eq:time-changed-process} and \eqref{eq:time-changed-process-in-support}).
	We now identify the support of $\goodpcaf$ as $F$ in the following proposition.
	
	\begin{prop} \label{prop:goodpcafsupp}
		Let $(\ambient,d,\refmeas,\form,\domain),\diff,\goodmeas,F$ be as in Assumption \ref{a:capgood}.
		Then the support of the PCAF $\goodpcaf$ is $F$, i.e.,
		\begin{equation} \label{eq:goodpcafsupp}
			F = \bigl\{x \in \ambient \bigm| \lawdiff_{x}\bigl( \textrm{$\goodpcaf_{t} > 0$ for any $t \in (0,\infty)$} \bigr) = 1 \bigr\}.
		\end{equation}
		In particular, the topological support $F$ of $\goodmeas$ is an $\form$-quasi-support of $\goodmeas$.
	\end{prop}
	
	\begin{remark}\label{rmk:goodmeas-use-Green-function}
		In the proofs of Propositions \ref{prop:goodpcafsupp} and \ref{prop:greeninv} given below
		we will use some basic properties of the Green functions of $(\form,\domain)$ from
		Proposition \ref{p:goodgreen} and Lemma \ref{l:green}; we are indeed allowed to do so,
		because the proofs of the latter results are independent of the former ones and their proofs below.
	\end{remark}
	
	\begin{proof}[Proof of Proposition \ref{prop:goodpcafsupp}]
		Set	 
		\[R \coloneqq \inf\{ t \in (0,\infty) \mid \goodpcaf_t > 0 \}, \quad
		S(\goodmeas) \coloneqq \{ x \in \ambient \mid \lawdiff_{x}( R = 0 ) = 1 \}
		.\]
		First we show that
		\begin{equation} \label{e:pc1}
			\lawdiff_{x}(R \ge \sigma_{F})=1 \quad \textrm{for all $x \in \ambient \setminus F$}.
		\end{equation}
		Indeed, for all $x \in \ambient \setminus F$, we see from \eqref{eq:Revuz-correspondence-pointwise}
		with $D=\ambient \setminus F$ and $f=\one_{\ambient \setminus F}$ and $\goodmeas(\ambient \setminus F)=0$
		that $\expdiff_{x}\bigl[\goodpcaf_{\tau_{\ambient \setminus F}}\bigr]=0$,
		%	\begin{align*}
			%		\expdiff_{x}\bigl[\goodpcaf_{\tau_{\ambient \setminus F}}\bigr]
			%		&= \lim_{t \downarrow 0} \expdiff_{x}\biggl[  \int_t^{\tau_{\ambient \setminus F}} d\goodpcaf_s \biggr] \quad \textrm{(by the monotone convergence theorem)} \\
			%		&= \lim_{t \downarrow 0} \expdiff_{p^{F^c}_t(x,\cdot) \cdot \refmeas}\biggl[ \int_0^{\tau_{\ambient \setminus F}} d\goodpcaf_s \biggr] =0 \quad \textrm{(by \cite[(4.1.25)]{CF})} \\
			%		&=\int_{0}^{\infty}\int_{\ambient \setminus F}p^{\ambient \setminus F}_{s}(x,y)\,\goodmeas(dy)\,ds=0.
			%	\end{align*}
		therefore $\lawdiff_{x}\bigl(\goodpcaf_{\tau_{\ambient \setminus F}}=0\bigr)=1$ and we thus obtain \eqref{e:pc1}.
		By the sample-path right-continuity of $\diff$, $\lawdiff_{x}(\tau_{\ambient \setminus F}>0)=1$ for all $x \in \ambient \setminus F$,
		and hence by \eqref{e:pc1} we conclude 
		\begin{equation} \label{e:pc2}
			F \supset \bigl\{x \in \ambient \bigm| \lawdiff_{x}\bigl( \textrm{$\goodpcaf_{t} > 0$ for any $t \in (0,\infty)$} \bigr) = 1 \bigr\}.
		\end{equation}
		Note that by \cite[(A.3.12) in Proposition A.3.6]{CF}, we have
		\begin{equation} \label{e:pc2a}
			\lawdiff_{x}( R = \sigma_{ S(\goodmeas) } ) = 1 \quad \textrm{for any $x \in \ambient$}.
		\end{equation}
		Therefore in order to obtain \eqref{eq:goodpcafsupp}, by \eqref{e:pc2} and \eqref{e:pc2a}
		it suffices to prove that
		\begin{equation} \label{e:hit1}
			\lawdiff_{x}( \sigma_{ S(\goodmeas) } = 0 ) = 1 \quad \textrm{for any $x \in F$.}
		\end{equation}

		We adapt \cite[Proof of Proposition 6.16]{BCM} to obtain \eqref{e:hit1}.
		Let $\scjump,\scjthres{\scjump}$ be as in Assumption \ref{a:capgood}.
		We collect a few preliminary estimates on the Green functions.
		By Lemma \ref{l:ballbdy}, there exist $C_1,A_1 \in (1,\infty)$ such that
		for all $(\xi,r) \in F \times (0,\scjthres{\scjump}/A_1)$,
		\begin{equation} \label{e:sm4}
			\int_{B(\xi,r)} \gren{B(\xi,r)}(y,z)\,\goodmeas(dz) \le C_1  \scjump(\xi,r) \quad\mbox{for all $y \in B(\xi,r)$.}
		\end{equation}
		By increasing $A_1$ if necessary and by \cite[Lemmas 5.10, 5.22 and 5.23]{BCM},
		there exist $C_2,A_0 \in (1,\infty)$ such that 
		for all $x \in \ambient$ and all $r \in (0,\diam(\ambient)/A_1)$, we have 
		\begin{equation}\label{e:sm1}
			C_2^{-1} \Capa_{B(x,2r)}(B(x,r))^{-1} \le	\gren{B(x,r)}(x,A_0^{-1}r) \le C_2 \Capa_{B(x,2r)}(B(x,r))^{-1}.
		\end{equation}
		We also recall the following inequality for capacity (\cite[p.~441, Solution to Exercise 2.2.2]{FOT};
		see also \cite[the $0$-order version of Exercise 4.2.2]{FOT} and \cite[Proof of Proposition 5.21]{BCM}):
		for any $(\xi,r) \in F \times (0,\scjthres{\scjump}/A_1)$, any $K \in \Borel(B(\xi,r))$
		and any Borel measure $\mu$ on $B(\xi,r)$ with $\mu(B(\xi,r))<\infty$ and
		$\int_{B(\xi,r)} \gren{B(\xi,r)}(\cdot,z)\,\mu(dz) \leq 1$ $\form$-q.e.\ on $B(\xi,r)$,
		\begin{equation}\label{eq:capacity-variational-formula}
			\Capa_{B(\xi,r)}(K) \geq \mu(K),
		\end{equation}
		which is applicable to the measure $\mu \coloneqq (C_1 \scjump(\xi,r))^{-1}\goodmeas|_{B(\xi,r)}$ by \eqref{e:sm4} and yields
		\begin{equation} \label{e:sm5}
			\Capa_{B(\xi,r)}(K) \geq C_{1}^{-1}\frac{\goodmeas(K)}{\scjump(\xi,r)}.
		\end{equation} 
		
		Next we show that
		\begin{equation} \label{e:hit2}
			\lawdiff_{x}( \sigma_{\oneptcpt{\ambient} \setminus \{x\}} = 0 ) = 1
			\quad \textrm{for all $x \in \ambient$.}
		\end{equation}
		Indeed, for any $x \in \ambient$ and any $t \in (0,\infty)$, since
		$(\ambient,d,\refmeas)$ satisfies \hyperlink{RVD}{\textup{RVD}} by
		\hyperlink{VD}{\textup{VD}}, Proposition \ref{p:feller}-\eqref{it:HKE-conn-irr-cons}
		and Lemma \ref{l:rvd}, we have $\refmeas( \{ x \} ) = 0$, hence
		$\lawdiff_{x}( \diff_{t} = x ) = 0$ by \ref{eq:AC} of $\diff$ from Assumption \ref{a:feller},
		thus $\lawdiff_{x}( \sigma_{\oneptcpt{\ambient} \setminus \{x\}} \leq t ) = 1$,
		and letting $t \downarrow 0$ yields \eqref{e:hit2}.
		
		Now fix any $\xi \in {F}$, and let $t,\varepsilon \in (0,\infty)$ be arbitrary. By \eqref{e:hit2}, we have
		\begin{equation} \label{e:fqs1}
			\lawdiff_{\xi}( T < t ) > 1-\varepsilon, \quad \textrm{where $T = \tau_{B(\xi,r)}$,}
		\end{equation}
		for some $r=r(\xi,t,\varepsilon) \in (0,\infty)$.
		By decreasing $r=r(\xi,t,\varepsilon)$ if necessary, we may assume that
		$r \in (0,\scjthres{\scjump}/A_1)$, where $A_1 \in (1,\infty)$ is as above.
		Fixing $r=r(\xi,t,\varepsilon)$ as above, we define
		\begin{equation*}
			K \coloneqq B(\xi, A_0^{-1}r) \cap S(\goodmeas).
		\end{equation*}
		We show that there exists a constant $c_{0} \in (0,1)$ that depends only on the constants involved in the assumption such that
		\begin{equation} \label{e:sm2}
			\lawdiff_{\xi}( \sigma_{K} < T ) \geq c_{0}.
		\end{equation}
		Let $e$ denote the equilibrium measure for $K$ such that $e(\overline{K})=\Capa_{B}(K)$,
		where $B \coloneqq B(\xi,r)$. To prove \eqref{e:sm2}, we observe that
		\begin{equation} \label{e:sm3a}
			\lawdiff_{z}(\sigma_{K} < \tau_B)= \int_{\overline{K}} \gren{B}(z,y)\, e(dy)
			\quad \textrm{for all $z \in B$.}
		\end{equation}
		To see \eqref{e:sm3a}, we use \cite[Theorem 4.3.3 and the $0$-order version of Exercise 4.2.2]{FOT}
		to conclude that both sides of \eqref{e:sm3a} are $\form$-quasi-continuous versions of the $0$-order
		equilibrium potential for $K$ with respect to the part Dirichlet form $(\form^{B},\domain^{0}(B))$ on $B$.
		Since both sides of \eqref{e:sm3a} are $\diff^{B}$-excessive by \cite[Lemma A.2.4-(ii)]{CF} and
		Lemma \ref{l:green}, respectively, we obtain \eqref{e:sm3a} by \ref{eq:AC} of $\diff^{B}$ from
		Proposition \ref{p:feller}-\eqref{it:HKE-part-CHK-sFeller} and \cite[Theorem A.2.17-(iii)]{CF}.
		Then by \eqref{e:sm3a} and the maximum principle \eqref{e:max},
		\begin{equation} \label{e:sm3}
			\lawdiff_{\xi}(\sigma_{K} < T)= \int_{\overline{K}} \gren{B}(\xi,y)\,e(dy) \ge \gren{B}(\xi,A_0^{-1}r) \Capa_{B}(K).
		\end{equation}
		
		Recalling that $S(\goodmeas)$ is an $\form$-quasi-support of the Revuz measure
		$\goodmeas$ of $\goodpcaf$ by \cite[Theorem 5.1.5]{FOT} or \cite[Theorem 5.2.1-(i)]{CF},
		we have $\goodmeas(\ambient \setminus S(\goodmeas))=0$ and hence
		\begin{equation} \label{e:fqs2}
			\goodmeas(K) = \goodmeas\bigl(B(\xi,A_0^{-1}r)\bigr).
		\end{equation}
		Now \eqref{e:sm2} follows by estimating $\lawdiff_{\xi}(\sigma_{K} < T)$ as
		\begin{align*}
			\lawdiff_{\xi}(\sigma_{K} < T)
			&\overset{\eqref{e:sm3}}{\ge} \gren{B}(\xi,A_0^{-1}r) \Capa_{B}(K)
			\overset{\eqref{e:sm5},\eqref{e:fqs2}}{\ge} C_1^{-1} \frac{\gren{B}(\xi,A_0^{-1}r)\goodmeas\bigl(B(\xi,A_0^{-1}r)\bigr)}{\scjump(\xi,r)} \nonumber \\
			&\overset{\eqref{e:sm1}}{\ge} (C_1 C_2)^{-1} \frac{ \goodmeas\bigl(B(\xi,A_0^{-1}r)\bigr)}{\Capa_{B(\xi,2r)}(B(\xi,r))\scjump(\xi,r)}
			\overset{\eqref{e:cgood}}{\gtrsim} \frac{\goodmeas\bigl(B(\xi,A_0^{-1}r)\bigr)}{\goodmeas(B(\xi, r))}
			\overset{\eqref{e:cgvd}}{\gtrsim} 1. \nonumber
		\end{align*}
		Then by choosing $\varepsilon=\frac{1}{2}c_{0}$ and using
		$\{ \sigma_{K} < T \} \subset \{ \sigma_{S(\goodmeas)} \leq t \} \cup \{ T \geq t \}$, we obtain
		\begin{equation*}
			\lawdiff_{\xi}( \sigma_{S(\goodmeas)} \leq t )
			\geq \lawdiff_{\xi}(\sigma_{K}<T) - \lawdiff_{\xi}(T \geq t)
			\overset{\eqref{e:fqs1},\eqref{e:sm2}}{>} c_{0} -\varepsilon
			= {\textstyle\frac{1}{2}} c_{0},
		\end{equation*}
		hence $\lawdiff_{\xi}( \sigma_{S(\goodmeas)} = 0 ) > \frac{1}{2}c_{0} > 0$ since $t \in (0,\infty)$ is arbitrary,
		and thus $\lawdiff_{\xi}( \sigma_{S(\goodmeas)} = 0 ) = 1$ by the Blumenthal 0-1 law \cite[Lemma A.2.5]{CF},
		proving \eqref{e:hit1} and thereby \eqref{eq:goodpcafsupp}.
		
		Lastly by \eqref{eq:goodpcafsupp} and \cite[Theorem 5.1.5]{FOT}, $F$ is an $\form$-quasi-support of $\goodmeas$.
	\end{proof}
	
	It turns out that the time-changed process $\difftr$ of $\diff$ by the PCAF $\goodpcaf$
	is a Hunt process on $F$ and satisfies \ref{eq:AC} with respect to $\goodmeas$
	and the occupation density formula in \eqref{e:odtr} below.
	The latter formula means that the Green function of $\difftr$ is the same as that of the diffusion $\diff$,
	and we will use it in the proof of the exit time lower estimate \hyperlink{exit}{$\on{E}(\scjump)_{\geq}$}
	for the boundary trace process (Proposition \ref{p:bdry-trace-exit}). Note that the relative topology
	of $\oneptcpt{F}=F\cup\{\cemetery\}$ inherited from $\oneptcpt{\ambient}$ coincides with
	its topology as the one-point compactification of $F$.
	
	\begin{prop} \label{prop:greeninv}
		Let $(\ambient,d,\refmeas,\form,\domain),\diff,\goodmeas,F$ be as in Assumption \ref{a:capgood}, and let
		$\difftr=\bigl(\widecheck{\Omega},\widecheck{\events},\{\difftr_{t}\}_{t\in[0,\infty]},\{\lawdiff_{x}\}_{x \in \oneptcpt{F}}\bigr)$
		be the time-changed process of $\diff$ by the PCAF $\goodpcaf$ defined by \eqref{eq:time-changed-process}
		with $\goodpcaf=\{\goodpcaf_{t}\}_{t\in[0,\infty)}$ in place of $A=\{A_{t}\}_{t\in[0,\infty)}$.
		Then the following hold:
		\begin{enumerate}[\rm(a)]\setlength{\itemsep}{0pt}\vspace{-5pt}
			\item\label{it:time-change-Hunt-AC} The subset $\widecheck{\Omega}_{0}$ of $\widecheck{\Omega}$ defined by
			\begin{equation} \label{eq:time-changed-process-Hunt-sample-space}
				\widecheck{\Omega}_{0} \coloneqq \widecheck{\Omega} \cap
				\Bigl( \bigl\{ \widecheck{\zeta} \in \{0,\infty\} \bigr\} \cup \Bigl\{ \lim_{s\to\infty}\diff_{s}=\cemetery \Bigr\} \Bigr)
			\end{equation}
			satisfies $\widecheck{\Omega}_{0} \in \minaugfilt_{0}$,
			$\lawdiff_{x}(\widecheck{\Omega}_{0})=1$ for any $x\in\oneptcpt{\ambient}$ and
			$\widecheck{\shiftdiff}_{t}(\widecheck{\Omega}_{0})\subset \widecheck{\Omega}_{0}$
			for any $t\in[0,\infty]$, and the time-changed process $\difftr$ with
			$\widecheck{\Omega}$ in \eqref{eq:time-changed-process} replaced by
			$\widecheck{\Omega}_{0}$ is a $\goodmeas$-symmetric Hunt process on $F$
			with life time $\widecheck{\zeta}$
			and shift operators $\bigl\{\widecheck{\shiftdiff}_{t}\bigr\}_{t\in[0,\infty]}$
			whose Dirichlet form is the regular symmetric Dirichlet form $(\trform,\trdomain)$
			on $L^{2}(F,\goodmeas)$ defined by \eqref{e:def-trace-domain} and \eqref{e:def-trace-form}.
			Moreover, $\difftr$ satisfies \ref{eq:AC}, i.e.,
			$\lawdiff_{x}(\difftr_{t} \in dy) \ll \goodmeas(dy)$ for any $(t,x)\in(0,\infty)\times F$.
			\item\label{it:greeninv}Let $B$ be a closed subset of $F$ such that the part Dirichlet form
			$(\form^{\ambient \setminus B},\domain^{0}(\ambient \setminus B))$ of $(\form,\domain)$
			on $\ambient \setminus B$ is transient. Then for any $x \in \ambient \setminus B$
			and any $\Borel^{*}(F \setminus B)$-measurable function $f \colon F \setminus B \to [0,\infty]$,
			$\int_{0}^{\widecheck{\tau}_{F \setminus B}} f(\difftr_{s})\, ds$
			is $\minaugfilt_{\infty}$-measurable and
			\begin{equation} \label{e:odtr}
				\expdiff_{x}\biggl[ \int_{0}^{\widecheck{\tau}_{F \setminus B}} f(\difftr_{s})\, ds \biggr]
				= \int_{F \setminus B} \gren{\ambient \setminus B}(x,y) f(y)\, \goodmeas(dy),
			\end{equation}
			%	\textup{(recall \eqref{eq:time-changed-process-in-support}),}
			where $\widecheck{\tau}_{F \setminus B} \coloneqq \inf\bigl\{ t \in [0,\infty) \bigm| \difftr_{t} \not \in F \setminus B \bigr\}$
			and $\gren{\ambient \setminus B}(x,y) \coloneqq \int_{0}^{\infty}p^{\ambient \setminus B}_{t}(x,y)\,dt$
			for the continuous heat kernel
			$p^{\ambient \setminus B} = p^{\ambient \setminus B}_{t}(x,y) \colon (0,\infty) \times (\ambient \setminus B) \times (\ambient \setminus B) \to [0,\infty)$
			of $\bigl(\ambient \setminus B,\refmeas|_{\ambient \setminus B},\form^{\ambient \setminus B},\domain^{0}(\ambient \setminus B)\bigr)$
			as given in Proposition \ref{p:feller}-\eqref{it:HKE-part-CHK-sFeller}.
		\end{enumerate}
	\end{prop}
	
	\begin{remark} \label{r:greeninv}
		A weaker version of \eqref{e:odtr} with \emph{every} $x$ replaced with
		\emph{$\form$-quasi-every} $x$ can be obtained by following
		\cite[Proof of Lemma 6.2.2]{FOT} (see in particular \cite[(6.2.10) and (6.2.11)]{FOT}).
	\end{remark}
	
	\begin{proof}[Proof of Proposition \ref{prop:greeninv}]
		\begin{enumerate}[\rm(a)]\setlength{\itemsep}{0pt}
			\item[\eqref{it:greeninv}]Note that $\widecheck{\tau}_{F \setminus B}$
			is an $\widecheck{\minaugfilt}_{*}$-stopping time by \cite[Proof of Proposition A.3.8-(vi)]{CF}
			(recall \eqref{eq:time-changed-filtration}). Set $D \coloneqq \ambient \setminus B$ and
			let $( P^{D}_{t} )_{t>0}$ denote the Markovian transition function of $\diff^{D}$.
			We easily see from \eqref{eq:time-changed-process-in-support},
			the sample path properties \eqref{it:PCAF-sample-path-properties} of $\goodpcaf$,
			the strong Markov property of $\diff$ (see, e.g., \cite[Theorem A.1.21]{CF}),
			$B \subset F$ and \eqref{eq:goodpcafsupp} that
			\begin{equation}\label{eq:exit-time-tcp}
				\lawdiff_{x}\bigl(\widecheck{\tau}_{F \setminus B}=\goodpcaf_{\tau_{D}}\bigr)=1 \quad \textrm{for any $x \in \ambient$.}
			\end{equation}
			
			Let $x \in D$, let $u \colon F \setminus B \to [0,\infty]$ be Borel measurable,
			and extend $u$ to $\ambient$ by setting $u|_{\ambient \setminus (F \setminus B)} \coloneqq 0$.
			Then since $\bigl[0,\goodpcaf_{\tau_{D}}\bigr)=\{s\in[0,\infty)\mid\tau_{s}<\tau_{D}\}$
			on $\Omega$, we obtain
			\begin{equation}\label{eq:greeninv-proof}
				\begin{split}
					&\expdiff_{x}\biggl[ \int_{0}^{\widecheck{\tau}_{F \setminus B}} u(\difftr_{s}) \,ds \biggr]
					= \expdiff_{x}\biggl[ \int_{0}^{\goodpcaf_{\tau_{D}}} u(\diff_{\tau_{s}}) \,ds \biggr] \quad \textrm{(by \eqref{eq:exit-time-tcp})} \\
					&= \expdiff_{x}\biggl[ \int_{0}^{\infty} \one_{[0,\tau_{D})}(\tau_{s}) u(\diff_{\tau_{s}}) \,ds \biggr] \\
					&= \expdiff_{x}\biggl[ \int_{0}^{\tau_{D}} u(\diff_{s}) \,d\goodpcaf_{s} \biggr] \quad \textrm{(by \cite[Lemma A.3.7-(i)]{CF})} \\
					%		&= \lim_{\delta \downarrow 0} \expdiff_{x}\biggl[ \int_{\delta}^{\infty} u(\diff_{s\wedge \tau_{D}}) \,d\goodpcaf_{s} \biggr] \\
					%		&= \lim_{\delta \downarrow 0} \expdiff_{p^{D}_{\delta}(x,\cdot)\cdot \refmeas|_{D}}\biggl[ \int_{0}^{\infty} u(\diff_{s \wedge \tau_{D}}) \,d\goodpcaf_{s} \biggr] \\
					%		&= \lim_{\delta \downarrow 0}\int_{0}^{\infty} \int_{F \setminus B} \bigl( P^{D}_{s} p^{D}_{\delta}(x,\cdot) \bigr)(y)u(y) \,\goodmeas(dy)\,ds \quad \textrm{(by \cite[(4.1.25)]{CF})} \\
					%		&= \lim_{\delta \downarrow 0}\int_{\delta}^{\infty} \int_{F \setminus B} p^{D}_{s}(x,y) u(y) \,\goodmeas(dy)\,ds
					&= \int_{0}^{\infty} \int_{F \setminus B} p^{D}_{s}(x,y)u(y) \,\goodmeas(dy)\,ds \quad \textrm{(by \eqref{eq:Revuz-correspondence-pointwise} and $\goodmeas(\ambient \setminus F)=0$)} \\
					&= \int_{F \setminus B} \biggl( \int_{0}^{\infty} p^{D}_{s}(x,y) \,ds \biggr) u(y) \,\goodmeas(dy)
					= \int_{F \setminus B} \gren{D}(x,y) u(y)\, \goodmeas(dy).
				\end{split}
			\end{equation}
			
			Now let $f \colon F \setminus B \to [0,\infty]$ be $\Borel^{*}(F \setminus B)$-measurable.
			Then $\int_{0}^{\widecheck{\tau}_{F \setminus B}} f(\difftr_{s})\, ds$ is
			$\minaugfilt_{\infty}$-measurable by \cite[Chapter 0, Exercise 3.3]{BlGe}
			and Fubini's theorem (see also \cite[Proof of Theorem A.1.22]{CF}), and there exist
			Borel measurable functions $f_{1},f_{2}\colon F \setminus B \to [0,\infty]$
			such that $f_{1}\leq f\leq f_{2}$ on $F \setminus B$ and $f_{1}=f_{2}$ $\goodmeas$-a.e.\ on $F \setminus B$.
			It follows from \eqref{eq:greeninv-proof} for $u=\one_{\{y\in F\setminus B\mid f_{1}(y)<f_{2}(y)\}}$
			and Fubini's theorem that $f_{1}(\difftr_{s})=f_{2}(\difftr_{s})=f(\difftr_{s})$
			for a.e.\ $s\in[0,\widecheck{\tau}_{F \setminus B})$ $\lawdiff_{x}$-a.s.\ and hence that
			$\int_{0}^{\widecheck{\tau}_{F \setminus B}} f(\difftr_{s})\, ds=\int_{0}^{\widecheck{\tau}_{F \setminus B}} f_{1}(\difftr_{s})\, ds$
			$\lawdiff_{x}$-a.s., which, together with \eqref{eq:greeninv-proof} for $u=f_{1}$ and
			$f_{1}=f$ $\goodmeas$-a.e.\ on $F \setminus B$, shows \eqref{e:odtr}.
			\item[\eqref{it:time-change-Hunt-AC}]We first prove that for any $t\in[0,\infty)$ and any $A \in \Borel(F)$,
			\begin{equation}\label{eq:time-changed-process-Borel}
				\textrm{the function}\quad
				F \ni x\mapsto\lawdiff_{x}(\difftr_{t} \in A)
				\quad\textrm{is Borel measurable,}
			\end{equation}
			which for $t=0$ is immediate from \eqref{eq:goodpcafsupp}.
			If $\{x\}$ is $\form$-polar for any $x\in F$, then $\{x\}$ is also $\trform$-polar
			for any $x\in F$ by \cite[Theorem 5.2.6]{CF} and \eqref{eq:goodpcafsupp}, hence
			any function defined $\trform$-q.e.\ on $F$ and $\trform$-quasi-continuous on $F$
			is an $\mathbb{R}$-valued Borel measurable function on $F$, and
			\eqref{eq:time-changed-process-Borel} holds since the function in
			\eqref{eq:time-changed-process-Borel} is $\trform$-quasi-continuous on $F$
			for any $t\in(0,\infty)$ and any $A \in \Borel(F)$ with $\goodmeas(A)<\infty$
			by \cite[Theorem 6.2.1-(iv)]{FOT} or the conjunction of \cite[Theorem 5.2.7]{CF} and \eqref{eq:goodpcafsupp}.
			
			Thus we may and do assume that $\{x_{0}\}$ is $\form$-polar for some $x_{0}\in F$.
			Note that then $\ambient\setminus\{x_{0}\}$ is connected. Indeed, if
			$\ambient\setminus\{x_{0}\}$ were not connected, then it would easily follow from
			\cite[Exercise 1.4.1, Theorems 1.4.2-(ii) and 1.6.1]{FOT} and the regularity and strong locality of
			$(\ambient\setminus\{x_{0}\},\refmeas|_{\ambient\setminus\{x_{0}\}},\form^{\ambient\setminus\{x_{0}\}},\domain^{0}(\ambient\setminus\{x_{0}\}))$
			that this Dirichlet space would not be irreducible. This would contradict the fact that
			$(\form^{\ambient\setminus\{x_{0}\}},\domain^{0}(\ambient\setminus\{x_{0}\}))$
			coincides with $(\form,\domain)$ as symmetric Dirichlet forms on $L^{2}(\ambient,\refmeas)$
			by \eqref{e:FD} and $\Capa_{1}(\{x_{0}\})=0$ and is hence irreducible by
			the irreducibility of $(\ambient,\refmeas,\form,\domain)$ from
			Proposition \ref{p:feller}-\eqref{it:HKE-conn-irr-cons}.
			
			To show \eqref{eq:time-changed-process-Borel}, let $x_{1}\in F\setminus\{x_{0}\}$
			and, recalling that $\ambient$ is locally pathwise connected by Proposition \ref{p:feller}-\eqref{it:HKE-conn-irr-cons},
			let $D_{1}$ be a pathwise connected open neighborhood of $x_{1}$ in $\ambient$ with $x_{0}\not\in\overline{D_{1}}$.
			Then noting that $\goodmeas(F\cap B(x_{0},r))=\goodmeas(B(x_{0},r))\in(0,\infty)$ for any $r\in(0,\infty)$
			by \eqref{e:cgvd} and $x_{0}\in F=\supp_{\ambient}[\goodmeas]$ and that
			$\lim_{r\downarrow 0}\goodmeas(B(x_{0},r))=\goodmeas(\{x_{0}\})=0$
			by $\Capa_{1}(\{x_{0}\})=0$ and Lemma \ref{l:strict}, we can choose
			$x_{2}\in F\setminus(\overline{D_{1}}\cup\{x_{0}\})$, and see from
			$\ambient\setminus\{x_{0}\}$ being connected and locally pathwise connected that
			$D_{1}\cup\{x_{2}\}\subset D$ for some pathwise connected open subset $D$ of
			$\ambient$ with $x_{0}\not\in\overline{D}$.
			
			Let $r\in(0,\infty)$ satisfy $\overline{B(x_{0},r)}\cap D=\emptyset$,
			set $B \coloneqq B_{r} \coloneqq F\cap\overline{B(x_{0},r)}$, and let $\widecheck{\tau}_{F \setminus B}$
			be as in \eqref{it:greeninv}. We claim that for any $t\in[0,\infty)$ and any $A \in \Borel(D)$,
			\begin{equation}\label{eq:time-changed-process-killed-Borel}
				\textrm{the function}\quad
				\ambient\setminus B\ni x\mapsto\lawdiff_{x}(\textrm{$\difftr_{t} \in A$, $t<\widecheck{\tau}_{F \setminus B}$})
				\quad\textrm{is Borel measurable.}
			\end{equation}
			To see \eqref{eq:time-changed-process-killed-Borel}, for each $\alpha\in[0,\infty)$ and
			each $\Borel^{*}(\ambient)$-measurable function $f\colon \ambient \to [0,\infty]$, noting that
			$\int_{0}^{\widecheck{\tau}_{F \setminus B}} e^{-\alpha s} f(\difftr_{s})\, ds$ is
			$\minaugfilt_{\infty}$-measurable by \cite[Chapter 0, Exercise 3.3]{BlGe} and Fubini's theorem,
			define $\widecheck{R}^{F\setminus B}_{\alpha}f\colon \ambient \to [0,\infty]$ by
			\begin{equation}\label{eq:time-changed-process-killed-resolvent-dfn}
				\widecheck{R}^{F\setminus B}_{\alpha}f(x)
				 \coloneqq \expdiff_{x}\biggl[ \int_{0}^{\widecheck{\tau}_{F \setminus B}} e^{-\alpha s} f(\difftr_{s})\, ds \biggr],
			\end{equation}
			so that $\widecheck{R}^{F\setminus B}_{\alpha}f$ is $\Borel^{*}(\ambient)$-measurable
			by \cite[Exercise A.1.20-(i)]{CF}. Then for any $\Borel^{*}(\ambient)$-measurable function
			$f\colon \ambient \to [0,\infty]$, any $\alpha,\beta\in[0,\infty)$ with $\alpha<\beta$ and any $x\in\ambient$,
			we easily see from \eqref{eq:time-changed-process-killed-resolvent-dfn} that
			\begin{equation}\label{eq:time-changed-process-killed-resolvent-zero-equiv}
				\widecheck{R}^{F\setminus B}_{\alpha}f(x)=0
				\quad \textrm{if and only if} \quad
				\widecheck{R}^{F\setminus B}_{\beta}f(x)=0,
			\end{equation}
			and from the strong Markov property of $\diff$ (see, e.g., \cite[Theorem A.1.21]{CF}) and Fubini's theorem that 
			\begin{equation}\label{eq:time-changed-process-killed-resolvent-eq}
				\widecheck{R}^{F\setminus B}_{\alpha}f(x)
				=\widecheck{R}^{F\setminus B}_{\beta}f(x)
				+(\beta-\alpha)\widecheck{R}^{F\setminus B}_{\alpha}(\widecheck{R}^{F\setminus B}_{\beta}f)(x).
			\end{equation}
			Choose $\delta\in(0,d(x_{1},x_{2})/2)$ so that $B(x_{1},\delta)\cup B(x_{2},\delta)\subset D$,
			and define $f_{B}\colon\ambient\to[0,1]$ by
			$f_{B} \coloneqq \min_{j\in\{1,2\}}\widecheck{R}^{F\setminus B}_{1}(\one_{B(x_{j},\delta)})$.
			Note that $\Capa_{1}(B)>0$ by $\goodmeas(B)>0$ and Lemma \ref{l:strict} and thus that
			$(\ambient \setminus B,\refmeas|_{\ambient \setminus B},\form^{\ambient \setminus B},\domain^{0}(\ambient \setminus B))$
			is transient by the irreducibility of $(\ambient,\refmeas,\form,\domain)$ from
			Proposition \ref{p:feller}-\eqref{it:HKE-conn-irr-cons} and \cite[Proposition 2.1]{BCM}.
			Therefore for any $x \in \ambient \setminus B$, by \eqref{eq:time-changed-process-killed-resolvent-eq},
			\eqref{e:odtr} from \eqref{it:greeninv}, $B(x_{1},\delta)\cup B(x_{2},\delta)\subset D\subset\ambient\setminus B$,
			\eqref{e:cgvd} and the finiteness and continuity of $\gren{\ambient \setminus B}|_{\offdiagp{\ambient \setminus B}}$ from Lemma \ref{l:green} we have
			\begin{equation}\label{eq:time-changed-process-killed-0-resolvent-finite}
				\widecheck{R}^{F\setminus B}_{0}f_{B}(x)
				\leq\min_{j\in\{1,2\}}\widecheck{R}^{F\setminus B}_{0}(\one_{B(x_{j},\delta)})(x)
				=\min_{j\in\{1,2\}}\int_{F\cap B(x_{j},\delta)} \gren{\ambient \setminus B}(x,y) \, \goodmeas(dy)
				<\infty,
			\end{equation}
			and hence by \eqref{eq:time-changed-process-killed-resolvent-eq} and \eqref{e:odtr} from \eqref{it:greeninv},
			for any $\alpha\in(0,\infty)$, any $\Borel^{*}(\ambient)$-measurable function
			$f\colon\ambient\to[0,\infty]$ and any $n\in\mathbb{N}$,
			\begin{align*}
				&\widecheck{R}^{F\setminus B}_{\alpha}(f\wedge(nf_{B}))(x)
				=\widecheck{R}^{F\setminus B}_{0}(f\wedge(nf_{B}))(x) - \alpha\widecheck{R}^{F\setminus B}_{0}\bigl(\widecheck{R}^{F\setminus B}_{\alpha}(f\wedge(nf_{B}))\bigr)(x) \\
				&=\int_{F\setminus B}\gren{\ambient \setminus B}(x,y)(f\wedge(nf_{B}))(y)\,\goodmeas(dy) - \alpha\int_{F\setminus B}\gren{\ambient \setminus B}(x,y)\widecheck{R}^{F\setminus B}_{\alpha}(f\wedge(nf_{B}))(y)\,\goodmeas(dy),
			\end{align*}
			which, as well as its limit $\widecheck{R}^{F\setminus B}_{\alpha}(f\one_{\{y\in\ambient\mid f_{B}(y)>0\}})(x)$ as $n\to\infty$,
			is Borel measurable in $x \in \ambient \setminus B$ by the Borel measurability of $\gren{\ambient \setminus B}$
			and Fubini's theorem. Moreover, since $D$ is a connected open subset of $\ambient\setminus B$,
			$\gren{\ambient \setminus B}|_{D\times D}$ is $(0,\infty]$-valued by Proposition \ref{p:feller}-\eqref{it:HKE-part-CHK-sFeller},
			and we obtain $f_{B}(x)>0$ for any $x\in D$ by combining the strict positivity of $\gren{\ambient \setminus B}|_{D\times D}$
			with the equality in \eqref{eq:time-changed-process-killed-0-resolvent-finite},
			$B(x_{1},\delta)\cup B(x_{2},\delta)\subset D$, $x_{1},x_{2}\in F=\supp_{\ambient}[\goodmeas]$
			and \eqref{eq:time-changed-process-killed-resolvent-zero-equiv}. Thus for each
			$[0,\infty)$-valued $f\in\contfunc_{\mathrm{c}}(\ambient)$ with $\supp_{\ambient}[f]\subset D$,
			$(\widecheck{R}^{F\setminus B}_{\alpha}f)|_{\ambient\setminus B}$ is Borel measurable
			for any $\alpha\in(0,\infty)$, and the right-hand side of \eqref{eq:time-changed-process-killed-resolvent-dfn}
			with the function $e^{-\alpha(\cdot)}$ replaced by any $[0,\infty)$-valued $\varphi\in\contfunc_{0}([0,\infty))$
			is also Borel measurable in $x\in\ambient\setminus B$ by the Stone--Weierstrass theorem
			(see, e.g., \cite[Corollary V.8.3]{Con}) applied to the subalgebra of $\contfunc_{0}([0,\infty))$
			generated by $\{e^{-\alpha(\cdot)}\mid\alpha\in(0,\infty)\}$. The same holds
			also when $e^{-\alpha(\cdot)}$ in \eqref{eq:time-changed-process-killed-resolvent-dfn}
			is replaced by $\varepsilon^{-1}\one_{(t,t+\varepsilon)}$ for any $t,\varepsilon\in[0,\infty)$ with $\varepsilon>0$,
			and letting $\varepsilon\downarrow 0$ shows the Borel measurability of
			$\ambient\setminus B\ni x\mapsto\expdiff_{x}\bigl[f(\difftr_{t})\one_{\{t<\widecheck{\tau}_{F \setminus B}\}}\bigr]$
			by the sample-path right-continuity of $\{f(\difftr_{s})\one_{\{s<\widecheck{\tau}_{F \setminus B}\}}\}_{s\in[0,\infty)}$
			and dominated convergence, whence \eqref{eq:time-changed-process-killed-Borel} follows
			since $f\in\contfunc_{\mathrm{c}}(\ambient)$ with $\supp_{\ambient}[f]\subset D$ is arbitrary.
			
			Now \eqref{eq:time-changed-process-Borel} follows from \eqref{eq:time-changed-process-killed-Borel}.
			Indeed, with $r$ as above, set $\tau_{n} \coloneqq \tau_{\ambient\setminus B_{r/n}}$ and
			$\widecheck{\tau}_{n} \coloneqq \widecheck{\tau}_{F \setminus B_{r/n}}$
			for each $n\in\mathbb{N}$, $\tau \coloneqq \sup_{n\in\mathbb{N}}\tau_{n}$ and
			$\widecheck{\tau} \coloneqq \sup_{n\in\mathbb{N}}\widecheck{\tau}_{n}$, so that
			$\{\tau_{n}\}_{n\in\mathbb{N}}$ is a non-decreasing sequence of $\minaugfilt_{*}$-stopping times and
			$\{\widecheck{\tau}_{n}\}_{n\in\mathbb{N}}$ is a non-decreasing sequence of $\widecheck{\minaugfilt}_{*}$-stopping times.
			Then by the sample-path right-continuity of $\diff,\difftr$, for any $n\in\mathbb{N}$ we have
			$\diff_{\tau_{n}}\in B_{r/n}\cup\{\cemetery\}$ on $\Omega$,
			$\diff_{\tau_{\widecheck{\tau}_{n}}}=\difftr_{\widecheck{\tau}_{n}}\in B_{r/n}\cup\{\cemetery\}$
			on $\widecheck{\Omega}$, hence $\tau_{n}\leq\tau_{\widecheck{\tau}_{n}}$ on $\widecheck{\Omega}$
			and $\tau\leq\tau_{\widecheck{\tau}}$ on $\widecheck{\Omega}$.
			Now let $x\in\ambient\setminus\{x_{0}\}$. Since $\lawdiff_{x}(\dot{\sigma}_{\{x_{0}\}}=\infty)=1$
			by \cite[Theorems 4.2.4 and 4.1.2]{FOT} (or \cite[Theorem A.2.17-(i),(ii)]{CF}),
			\ref{eq:AC} of $\diff$, $\Capa_{1}(\{x_{0}\})=0$ and \eqref{eq:cap-zero-properly-exceptional},
			the quasi-left-continuity on $(0,\zeta)$ of $\diff$ as in \cite[Definition A.1.23 and Theorem A.1.24]{CF}
			(or the sample-path continuity \eqref{eq:sample-path-cont} of $\diff$ along with \ref{eq:AC} of $\diff$)
			yields $\lawdiff_{x}(\tau\geq\zeta)=1$. Thus, recalling \eqref{eq:time-changed-process-in-support}, we have
			$1=\lawdiff_{x}(\tau\geq\zeta)\leq\lawdiff_{x}(\tau_{\widecheck{\tau}}\geq\zeta)\leq\lawdiff_{x}(\widecheck{\tau}\geq\goodpcaf_{\infty}=\widecheck{\zeta})\leq 1$,
			and therefore for any $t\in[0,\infty)$ and any $A \in \Borel(D)$,
			\begin{equation}\label{eq:time-changed-process-Borel-proof}
				\lim_{n\to\infty}\lawdiff_{x}(\textrm{$\difftr_{t} \in A$, $t<\widecheck{\tau}_{F \setminus B_{r/n}}$})
				=\lawdiff_{x}(\textrm{$\difftr_{t} \in A$, $t<\widecheck{\zeta}$})
				=\lawdiff_{x}(\difftr_{t} \in A),
			\end{equation}
			so that $\ambient\ni x\mapsto\lawdiff_{x}(\difftr_{t} \in A)$ is Borel measurable by
			\eqref{eq:time-changed-process-killed-Borel}, which proves \eqref{eq:time-changed-process-Borel}
			since $\lawdiff_{x}(\difftr_{t}=x_{0})=0$ for any $x\in\ambient\setminus\{x_{0}\}$
			by $\lawdiff_{x}(\dot{\sigma}_{\{x_{0}\}}=\infty)=1$ and
			$F\setminus\{x_{0}\}\subset\bigcup_{k\in\mathbb{N}}D_{k}$ for some sequence
			$\{D_{k}\}_{k\in\mathbb{N}}$ of pathwise connected open subsets of $\ambient$
			with $x_{0}\not\in\bigcup_{k\in\mathbb{N}}\overline{D_{k}}$.
			
			We next prove the stated properties of $\widecheck{\Omega}_{0}$. It is clear that
			$\widecheck{\shiftdiff}_{t}(\widecheck{\Omega}_{0})\subset \widecheck{\Omega}_{0}$
			for any $t\in[0,\infty]$, and $\widecheck{\Omega}_{0}\in\minaugfilt_{\infty}$
			by $\Lambda\in\minaugfilt_{0}$, \eqref{eq:time-changed-process},
			\eqref{eq:time-changed-process-in-support}, \eqref{eq:time-changed-process-Hunt-sample-space}
			and the sample-path right-continuity of $\diff$. Since $(\ambient,\refmeas,\form,\domain)$
			is irreducible by Proposition \ref{p:feller}-\eqref{it:HKE-conn-irr-cons},
			$(\ambient,\refmeas,\form,\domain)$ is either transient or recurrent by
			\cite[Proposition 2.1.3-(iii)]{CF} or \cite[Lemma 1.6.4-(iii)]{FOT}.
			If $(\ambient,\refmeas,\form,\domain)$ is transient, then
			$\lawdiff_{x}(\lim_{s\to\infty}\diff_{s}=\cemetery)=1$ for any $x\in\oneptcpt{\ambient}$
			by \cite[Theorem 3.5.2]{CF} combined with \ref{eq:AC} and the conservativeness of $\diff$
			from Proposition \ref{p:feller}-\eqref{it:HKE-Feller}. Otherwise
			$(\ambient,\refmeas,\form,\domain)$ is irreducible and recurrent, so the function
			$\ambient\ni x\mapsto 1-\expdiff_{x}\bigl[e^{-\goodpcaf_{\infty}}\bigr]$,
			which is $\diff$-excessive as noted in \cite[Proof of Proposition 3.5]{Kaj12},
			is constant on $\ambient$ by \cite[Lemma 3.5.5-(ii) and Theorem A.2.17-(i),(iii)]{CF}
			and \ref{eq:AC} of $\diff$. Its constant value is actually $1$ and thus
			$\lawdiff_{x}(\widecheck{\zeta}=\goodpcaf_{\infty}=\infty)=1$ for any $x\in\ambient$;
			indeed, since $\one_{\ambient}\in\domain_{e}$ and $\form(\one_{\ambient},\one_{\ambient})=0$
			by the recurrence of $(\ambient,\refmeas,\form,\domain)$, we have
			$\one_{F}\in\trdomain_{e}$ and $\trform(\one_{F},\one_{F})=0$
			by \eqref{e:trace-ExtDiriSp} and \eqref{eq:hit-dist-harm-var}, namely the Dirichlet form
			$(\trform,\trdomain)$ of $\difftr$ on $L^{2}(F,\goodmeas)$ is recurrent, and hence
			conservative by \cite[Proposition 2.1.10]{CF} or \cite[Lemma 1.6.5]{FOT}, so that
			$\lawdiff_{x}(\widecheck{\zeta}=\goodpcaf_{\infty}=\infty)=\lim_{t\to\infty}\lawdiff_{x}(t<\widecheck{\zeta})=1$
			and $\expdiff_{x}\bigl[e^{-\goodpcaf_{\infty}}\bigr]=0$ for $\goodmeas$-a.e.\ $x\in F$
			and in particular for some $x\in F$ by $\goodmeas(F)>0$. By these observations,
			$\lawdiff_{\cemetery}(\widecheck{\zeta}=0)=1$ and \eqref{eq:time-changed-process-in-support}
			we obtain $\lawdiff_{x}(\widecheck{\Omega}_{0})=1$ for any $x\in\oneptcpt{\ambient}$
			and thus $\widecheck{\Omega}_{0}\in\minaugfilt_{0}$.
			
			For any $(t,\omega) \in (0,\infty) \times \widecheck{\Omega}_{0}$, the left limit
			$\difftr_{t-}(\omega) \coloneqq \lim_{s\uparrow t}\difftr_{s}(\omega)$ in $\oneptcpt{F}$ exists;
			indeed, setting $\tau_{t-}(\omega) \coloneqq \lim_{s\uparrow t}\tau_{s}(\omega)$ and recalling
			\eqref{eq:time-changed-process} and \eqref{eq:time-changed-process-Hunt-sample-space}, we have
			$\lim_{s\uparrow t}\difftr_{s}(\omega)=\diff_{\tau_{t-}(\omega)-}(\omega) \in \oneptcpt{F}$
			if either $t=\widecheck{\zeta}(\omega)$ and $\tau_{t-}(\omega)<\infty$ or $t<\widecheck{\zeta}(\omega)$,
			$\lim_{s\uparrow t}\difftr_{s}(\omega)=\lim_{s\to\infty}\diff_{s}(\omega)=\cemetery$
			if $t=\widecheck{\zeta}(\omega)$ and $\tau_{t-}(\omega)=\infty$, and
			$\lim_{s\uparrow t}\difftr_{s}(\omega)=\cemetery$ if $t>\widecheck{\zeta}(\omega)$.
			
			To see the quasi-left-continuity on $(0,\infty)$ of $\difftr$,
			recalling \eqref{eq:time-changed-filtration}, let $\{\sigma_{n}\}_{n\in\mathbb{N}}$
			be a non-decreasing sequence of $\widecheck{\minaugfilt}_{*}$-stopping times,
			set $\sigma \coloneqq \lim_{n\to\infty}\sigma_{n}$, and let $\mu$ be a finite Borel measure on $\oneptcpt{\ambient}$.
			Then $\{\tau_{\sigma_{n}}\}_{n\in\mathbb{N}}$ is a non-decreasing sequence of
			$\minaugfilt_{*}$-stopping times by \cite[Proposition A.3.8-(v)]{CF} and, setting
			$\tau \coloneqq \lim_{n\to\infty}\tau_{\sigma_{n}}$, we see from the quasi-left-continuity on $(0,\infty)$
			(or the sample-path continuity \eqref{eq:sample-path-cont} and \ref{eq:AC}) of $\diff$ that
			\begin{equation}\label{eq:diff-quasi-left-cont-apply}
				\lawdiff_{\mu}\Bigl(\textrm{$\lim_{n\to\infty}\difftr_{\sigma_{n}}=\diff_{\tau}\in \oneptcpt{F}$, $\tau<\infty$}\Bigr)
				=\lawdiff_{\mu}(\tau<\infty).
			\end{equation}
			On the other hand, by \cite[Lemma A.3.7-(ii)]{CF} and
			$\goodpcaf_{\zeta}=\goodpcaf_{\infty}=\widecheck{\zeta}$ we have
			\begin{gather}\label{eq:difftr-limit-stopping-time-goodpcaf}
				\goodpcaf_{\tau}=\lim_{n\to\infty}\goodpcaf_{\tau_{\sigma_{n}}}=\lim_{n\to\infty}\sigma_{n}=\sigma
				\qquad\textrm{on $\{\sigma\leq\widecheck{\zeta}\}$,} \\
				\{\sigma<\widecheck{\zeta}\}\subset\{\tau<\zeta\}=\{\diff_{\tau}\in\ambient\}
				\subset\bigcap_{n\in\mathbb{N}}\{\sigma_{n}<\widecheck{\zeta}\}
				\subset\{\sigma\leq\widecheck{\zeta}\},
				\label{eq:difftr-limit-stopping-time-before-life-time}
			\end{gather}
			and it further follows from \eqref{eq:diff-quasi-left-cont-apply}, \eqref{eq:goodpcafsupp},
			the strong Markov property of $\diff$ at time $\tau$ (see, e.g., \cite[Theorem A.1.21]{CF})
			and the sample path properties \eqref{it:PCAF-sample-path-properties} of $\goodpcaf$ that
			$\lawdiff_{\mu}\bigl(\tau_{\goodpcaf_{\tau}}=\tau<\zeta\bigr)=\lawdiff_{\mu}(\tau<\zeta)$,
			which together with \eqref{eq:difftr-limit-stopping-time-before-life-time} and
			\eqref{eq:difftr-limit-stopping-time-goodpcaf} yields
			\begin{equation}\label{eq:difftr-limit-stopping-time-identify}
				\lawdiff_{\mu}(\tau_{\sigma}=\tau<\zeta)=\lawdiff_{\mu}(\tau<\zeta).
			\end{equation}
			By the first inclusion in \eqref{eq:difftr-limit-stopping-time-before-life-time}, we also obtain
			\begin{equation}\label{eq:difftr-limit-stopping-time-after-life-time}
				\{\tau\geq\zeta\}\subset\{\sigma\geq\widecheck{\zeta}\}
				\qquad\textrm{and hence}\qquad
				\diff_{\tau}=\cemetery=\difftr_{\sigma}\quad\textrm{on $\{\tau\geq\zeta\}$.}
			\end{equation}
			Combining \eqref{eq:diff-quasi-left-cont-apply}, \eqref{eq:difftr-limit-stopping-time-identify}
			and \eqref{eq:difftr-limit-stopping-time-after-life-time}, we conclude that
			\begin{equation}\label{eq:difftr-quasi-left-cont-non-explosion}
				\lawdiff_{\mu}\Bigl(\textrm{$\lim_{n\to\infty}\difftr_{\sigma_{n}}=\difftr_{\sigma}$, $\tau<\infty$}\Bigr)
				=\lawdiff_{\mu}(\tau<\infty).
			\end{equation}
			Moreover, on $\{\sigma<\infty=\tau\}$, which is equal to
			$\{\widecheck{\zeta}\leq\sigma<\infty=\tau\}$ by \eqref{eq:difftr-limit-stopping-time-after-life-time}, we have
			\begin{equation}\label{eq:difftr-quasi-left-cont-explosion}
				\lawdiff_{\mu}\Bigl(\textrm{$\lim_{n\to\infty}\difftr_{\sigma_{n}}=\difftr_{\sigma}$, $\sigma<\infty=\tau$}\Bigr)
				=\lawdiff_{\mu}(\sigma<\infty=\tau);
			\end{equation}
			indeed, clearly $\lim_{n\to\infty}\difftr_{\sigma_{n}}=\cemetery=\difftr_{\sigma}$
			on $\{\sigma>\widecheck{\zeta}\}\cup\{\sigma=\widecheck{\zeta}=0\}$,
			$\lawdiff_{\mu}(\Omega\setminus\widecheck{\Omega}_{0})=0$, and on
			$\widecheck{\Omega}_{0}\cap\{0<\sigma=\widecheck{\zeta}<\infty=\tau\}$
			we have $\lim_{s\to\infty}\diff_{s}=\cemetery$ by
			\eqref{eq:time-changed-process-Hunt-sample-space} and $\widecheck{\zeta}\in(0,\infty)$ and therefore
			$\difftr_{\sigma_{n}}=\diff_{\tau_{\sigma_{n}}}\xrightarrow{n\to\infty}\cemetery=\difftr_{\sigma}$
			by $\lim_{n\to\infty}\tau_{\sigma_{n}}=\tau=\infty$ and $\sigma=\widecheck{\zeta}$.
			Now \eqref{eq:difftr-quasi-left-cont-non-explosion} and \eqref{eq:difftr-quasi-left-cont-explosion}
			together imply \eqref{eq:difftr-quasi-left-cont-non-explosion} with $\sigma$ in place of $\tau$,
			i.e., that $\difftr$ is quasi-left-continuous on $(0,\infty)$ with respect to
			$\widecheck{\minaugfilt}_{*}$. Thus $\difftr$ with $\widecheck{\Omega}$ replaced by
			$\widecheck{\Omega}_{0}$ is a Hunt process on $F$, and the other stated properties
			of $\difftr$ except \ref{eq:AC} have been already noted in the paragraphs of
			\eqref{eq:time-changed-process} and \eqref{e:trace-ExtDiriSp}.
			
			Lastly, to see \ref{eq:AC} of $\difftr$, we first apply the same argument as \eqref{eq:greeninv-proof} above
			to show the absolute continuity of the Markovian resolvent kernel of $\difftr$.
			%	Recall the continuous heat kernel $p=p_{t}(x,y)$ of $(\ambient,\refmeas,\form,\domain)$
			%	from Assumption \ref{a:feller}, and let $( P_{t} )_{t>0}$ denote the Markovian transition function of $\diff$.
			Let $x \in F$, $\alpha \in (0,\infty)$, and let $B \in \Borel(F)$ satisfy $\goodmeas(B)=0$. Then
			\begin{equation}\label{eq:time-change-AC-resolvent}
				\begin{split}
					&\expdiff_{x}\biggl[ \int_{0}^{\infty} e^{-\alpha s} \one_{B}(\difftr_{s}) \,ds \biggr]
					= \expdiff_{x}\biggl[ \int_{0}^{\infty} e^{-\alpha s} \one_{B}(\diff_{\tau_{s}}) \,ds \biggr] \\
					&= \expdiff_{x}\biggl[ \int_{0}^{\infty} e^{-\alpha \goodpcaf_{s}} \one_{B}(\diff_{s}) \,d\goodpcaf_{s} \biggr] \quad \textrm{(by \cite[Lemma A.3.7-(i)]{CF})} \\
					&\leq \expdiff_{x}\biggl[ \int_{0}^{\infty} \one_{B}(\diff_{s}) \,d\goodpcaf_{s} \biggr]
					= 0 \quad \textrm{(by \eqref{eq:Revuz-correspondence-pointwise} with $D=\ambient$ and $\goodmeas(B)=0$).}
					%		&= \lim_{\delta \downarrow 0} \lim_{t \to \infty} \expdiff_{x}\biggl[ \int_{\delta}^{t+\delta} e^{-\alpha \goodpcaf_{s}} \one_{B}(\diff_{s}) \,d\goodpcaf_{s} \biggr] \\
					%		&\leq \lim_{\delta \downarrow 0} \lim_{t \to \infty} \expdiff_{x}\biggl[ \int_{\delta}^{t+\delta} \one_{B}(\diff_{s}) \,d\goodpcaf_{s} \biggr] \\
					%		&= \lim_{\delta \downarrow 0} \lim_{t \to \infty} \expdiff_{p_{\delta}(x,\cdot)\cdot \refmeas}\biggl[ \int_{0}^{t} \one_{B}(\diff_{s}) \,d\goodpcaf_{s} \biggr] \\
					%		&= \lim_{\delta \downarrow 0} \lim_{t \to \infty} \int_{0}^{t} \int_{F} \bigl( P_{s} p_{\delta}(x,\cdot) \bigr)(y) \one_{B}(y) \,\goodmeas(dy)\,ds \quad \textrm{(by \cite[(4.1.2)]{CF})} \\
					%		&= 0 \quad \textrm{(by $\goodmeas(B)=0$).}
				\end{split}
			\end{equation}
			Given \eqref{eq:time-change-AC-resolvent} and the fact that $\difftr$ is a $\goodmeas$-symmetric
			Hunt process on $F$ whose Dirichlet form $(\trform,\trdomain)$ on $L^{2}(F,\goodmeas)$ is regular,
			we obtain \ref{eq:AC} of $\difftr$ from \cite[Theorem 4.2.4]{FOT} or \cite[Proposition 3.1.11]{CF}.
			\qedhere\end{enumerate}
	\end{proof}
	
	\section{Green function,  Martin kernel, and Na\"im kernel} \label{sec:GreenMartinNaim}
	
	\subsection{Properties of Green function} \label{ssec:Green}
	
	The elliptic Harnack inequality implies the existence of Green functions as shown in \cite[Theorem 4.4]{BCM}, which we recall below.
	
	\begin{prop} \label{p:goodgreen}
		Let $(\ambient,d,\refmeas,\form,\domain)$ be an MMD space satisfying \hyperlink{ehi}{$\on{EHI}$},
		and let $\diff$ be an $\refmeas$-symmetric diffusion on $\ambient$ whose Dirichlet form is $(\form,\domain)$.
		Let $D$ be a non-empty open subset of $\ambient$ such that the part Dirichlet form $(\form^D, \domain^0(D))$ on $D$ is transient.
		Then there exist a Borel measurable function $\gren{D}\colon D \times D \to [0,\infty]$ and
		a Borel properly exceptional set $\mathcal{N}$ for $\diff$ such that the following hold:
		\begin{enumerate}[\rm(i)]\setlength{\itemsep}{0pt}\vspace{-5pt}
			\item\label{it:goodgreen-sym} (Symmetry) $\gren{D}(x,y)= \gren{D}(y,x)$ for all $(x,y) \in D \times D$.
			\item\label{it:goodgreen-cont} (Continuity) $\gren{D}|_{\offdiag{D}}$ is $[0,\infty)$-valued and continuous.
			\item\label{it:goodgreen-odf} (Occupation density formula) 
			For any Borel measurable function $f\colon D\to[0,\infty]$,
			\begin{equation}\label{e:odf}
				\expdiff_{x}\biggl[ \int_0^{\tau_{D}} f(\diff_s) ds \biggr] = \int_D \gren{D}(x,y) f(y)\, \refmeas(dy)
				\quad \textrm{for every $x\in D\setminus \mathcal{N}$.}
			\end{equation}
			\item\label{it:goodgreen-excessive} (Excessiveness)
			For each $y\in D$, $x\mapsto \gren{D}(x,y)$ is $\diff^D|_{D\setminus \mathcal{N}}$-excessive.
			\item\label{it:goodgreen-harm} (Harmonicity) For any fixed $y \in D$,
			the function $D \setminus \{y\} \ni x \mapsto \gren{D}(x,y)$
			belongs to $\domain_{\on{loc}}(D \setminus \{y\})$ and is $\form$-harmonic on $D \setminus \{y\}$, and 
			$\gren{D}(x,y)=\lawdiff_{x}[ \gren{D}(\diff^D_{\tau_{V}},y) ]$ for any open subset $V$ of $D$ with $y \not\in \overline{V}$
			and any $x\in D\setminus \mathcal{N}$, where we adopt the convention that
			$\gren{D}(x,\cemetery_{D})=\gren{D}(\cemetery_{D},x)=0$ for all $x \in D$.
			\item\label{it:goodgreen-max} (Maximum principles) If $V$ is a relatively compact open subset of $D$ and $x_0 \in V$, then
			\begin{equation}\label{e:max}
				\inf_{\overline{V} \setminus \{x_0\}} \gren{D}(x_0,\cdot) = \inf_{\partial V} \gren{D}(x_0,\cdot),
				\qquad \sup_{D \setminus V} \gren{D}(x_0, \cdot) = \sup_{\partial V}\gren{D}(x_0,\cdot).
			\end{equation}
		\end{enumerate}
		We call $\gren{D}$ the \emph{\textbf{Green function}} of $(\form, \domain)$ on $D$.
	\end{prop}
	
	\begin{proof}
		All parts except \eqref{it:goodgreen-harm} follows from \cite[Theorem 4.4]{BCM}. 
		
		The claims that $x \mapsto \gren{D}(x,y)$ belongs to $\domain_{\on{loc}}(D\setminus \{y\})$
		and is harmonic in $D \setminus \{y\}$ follow from \cite[Remark 2.7-(ii), Proposition 2.9-(iii) and Theorem 4.4]{BCM}.
		The remaining claims in \eqref{it:goodgreen-harm} are proved in \cite[Proof of Theorem 4.4]{BCM}.
	\end{proof}
	
	\begin{definition} \label{d:greenop}
		Let $(\ambient,d,\refmeas,\form,\domain)$ be an MMD space satisfying \hyperlink{ehi}{$\on{EHI}$},
		and $D$ a non-empty open subset of $\ambient$ such that the part Dirichlet form $(\form^D, \domain^0(D))$ on $D$ is transient.
		For a Borel measurable function $f \colon D \to [0,\infty]$, we define
		\[
		G_D f(x) \coloneqq \begin{cases}
			\int_D \gren{D}(x,y)f(y)\,\refmeas(dy) & \textrm{if $x \in D$,} \\
			0 & \textrm{if $x \notin D$.}
		\end{cases}
		\] 
		By \cite[Theorem 4.2.6]{FOT}, if $f \colon D \to [0,\infty]$ is Borel measurable
		and $\int_D f G_D f\, d\refmeas<\infty$, then $G_D f$ is an
		$\form$-quasi-continuous $\refmeas$-version of the Green operator
		defined in \eqref{e:transient} for the part Dirichlet form $(\form^D,\domain^0(D))$
		on $D$ and $G_D f \in \domain^{0}(D)_{e}$.
	\end{definition}

	We see that the exceptional set $\mathcal{N}$ in Proposition \ref{p:goodgreen} can be taken to be the empty set
	if the diffusion process is defined from every starting point as given in Proposition \ref{p:feller}.
	
	\begin{lem} \label{l:green}
		Let an MMD space $(\ambient,d,\refmeas,\form,\domain)$ and a diffusion $\diff$ on $\ambient$ satisfy Assumption \ref{a:feller},
		and let $D$ be a non-empty open subset of $\ambient$ such that the part Dirichlet form $(\form^D,\domain^0(D))$ on $D$ is transient.
		Define $g_D^p \colon D \times D \to [0,\infty]$ by
		\begin{equation} \label{e:fep1}
			g^p_D(x,y) \coloneqq \int_0^\infty p_t^D(x,y)\,dt, \quad x,y \in D,
		\end{equation}
		where $p^{D}_{t}(\cdot,\cdot)$ is the continuous heat kernel of $(D,\refmeas|_{D},\form^{D},\domain^{0}(D))$
		as given in Proposition \ref{p:feller}-\eqref{it:HKE-part-CHK-sFeller},
		and let $g_D(\cdot,\cdot)$ denote the Green function on $D$ from Proposition \ref{p:goodgreen},
		which is applicable by Remark \ref{r:ehi-hke}. Then, with $\mathcal{N}$ as in Proposition \ref{p:goodgreen},
		\begin{equation} \label{e:fep2}
			g^p_D(x,y) = \gren{D}(x,y) \qquad \textrm{for all $(x,y) \in (D \times D) \setminus \mathcal{N}_{\diag}$,}
		\end{equation}
		and Proposition \ref{p:goodgreen}-\eqref{it:goodgreen-sym},\eqref{it:goodgreen-cont},\eqref{it:goodgreen-odf},\eqref{it:goodgreen-excessive},\eqref{it:goodgreen-harm},\eqref{it:goodgreen-max}
		with $g^{p}_{D},\emptyset$ in place of $\gren{D},\mathcal{N}$ hold.
		Moreover, if $D$ is connected, then $g^{p}_{D}(x,y)\in(0,\infty]$ for any $x,y\in D$.
	\end{lem}
	
	\begin{proof}
		The occupation density formula \eqref{e:odf} for $g^p_D$ follows from Fubini's theorem as
		\begin{equation*}
			\expdiff_{x}\biggl[ \int_{0}^{\tau_{D}} f(X_{s}) ds \biggr]
			= \int_0^\infty \int_D f(y)p_t^D(x,y)f(y) \,\refmeas(dy) \,dt
			= \int_D f(y)g^p_D(x,y)f(y) \,\refmeas(dy).
		\end{equation*}
		By the transience of $\diff^{D}$, we have 
		\begin{equation} \label{e:gfin}
			g^p_D(x,y) <\infty \quad \textrm{for $\refmeas$-a.e.\ $x,y \in D$.}
		\end{equation}
		
		By the heat kernel estimate \hyperlink{hke}{$\on{HKE(\scdiff)}$}, the function 
		\begin{equation*}
			(x,y) \mapsto \int_{\delta}^\infty p_t^D(x,y) \,dt
		\end{equation*}
		converges uniformly on compact subsets of $\offdiag{D}$ as $\delta \downarrow 0$. Therefore it suffices to show that for each $\delta>0$ and $(x_0,y_0) \in \offdiag{D}$, the function $(x,y) \mapsto \int_{\delta}^\infty p_t^D(x,y) \,dt$ is continuous at $(x_0,y_0)$.
		Indeed, by the parabolic Harnack inequality \cite[Theorem 3.1]{BGK}, we can choose disjoint open neighborhoods $B_1$ and $B_2$ of $x_0,y_0$ and constants $C_1,C_2>0$ such that 
		\begin{equation*}
			\sup_{(x,y) \in B_1 \times B_2} p^D_t(x,y) \le C_1 \inf_{(x,y) \in B_1 \times B_2}p^D_{C_2^{-1}t}(x,y) \le  C_1 p^D_{C_2^{-1}t}(x',y')  \quad \mbox{for all  $t \ge \delta$,} 
		\end{equation*}
		where $(x',y') \in B_1 \times B_2$ is chosen using \eqref{e:gfin} such that $g^p_D(x',y')<\infty$.
		Combining the above estimate with the transience of $\diff^D$, and the dominated convergence theorem, we conclude that
		$(x,y) \mapsto \int_{\delta}^\infty p_t^D(x,y) \,dt$ is continuous at $(x_0,y_0)$.
		
		The equality \eqref{e:fep2} for $(x,y) \in \offdiag{D}$ follows from
		the continuity of $g^p_D,\gren{D}$ along with \eqref{e:odf} for $g^p_D,\gren{D}$.
		The equality $\gren{D}(x,y)=\expdiff_{x}[ \gren{D}(\diff^D_{\tau_{V}},y) ]$
		for any $x,y\in D$ and any open subset $V$ of $D$ with $y \not\in \overline{V}$
		follows from Proposition \ref{p:goodgreen}-\eqref{it:goodgreen-harm},
		the continuity of $g^p_D, g_D$ and the continuity of
		$V\ni z\mapsto\expdiff_{z}[ \gren{D}(\diff^D_{\tau_{V}},y) ]$
		from Lemma \ref{l:harmonicm}-\eqref{it:hmeas-continuous}.
		The $\diff^{D}$-excessiveness of $g_D^p(\cdot,y)$ for $y \in D$
		follows easily from \eqref{e:fep1} and \eqref{eq:AC-transition-density} for $p^{D}$,
		and then for each $y \in D\setminus\mathcal{N}$, since $\refmeas(\{y\})=0$ as observed in the paragraph of \eqref{e:hit2}
		and $g_D^p(\cdot,y)|_{D\setminus\mathcal{N}},\gren{D}(\cdot,y)|_{D\setminus\mathcal{N}}$ are
		$\diff^{D}|_{D\setminus\mathcal{N}}$-excessive by Proposition \ref{p:goodgreen}-\eqref{it:goodgreen-excessive}
		and equal on $(D\setminus\mathcal{N})\setminus\{y\}$, we have $g_D^p(y,y)=\gren{D}(y,y)$
		by \ref{eq:AC} of $\diff^{D}$ and \cite[Theorem A.2.17-(i),(iii)]{CF}.
		Lastly, if $D$ is connected, then $g^{p}_{D}$ is $(0,\infty]$-valued by \eqref{e:fep1}
		and the last claim in Proposition \ref{p:feller}-\eqref{it:HKE-part-CHK-sFeller}.
	\end{proof}
	
	Due to Lemma \ref{l:green}, if an MMD space $(\ambient,d,\refmeas,\form,\domain)$
	and a diffusion $\diff$ on $\ambient$ satisfy Assumption \ref{a:feller} and
	$D$ is a non-empty open subset of $\ambient$ such that the part Dirichlet form $(\form^D, \domain^0(D))$ on $D$
	is transient, we adopt the convention to redefine the $\gren{D}(\cdot,\cdot)$
	from Proposition \ref{p:goodgreen} to be equal to $g^p_D(\cdot,\cdot)$ from Lemma \ref{l:green}.
	In particular, $\gren{D}(x,\cdot)$ is $\diff^D$-excessive for all $x \in D$.
	
	In the next lemma, we show that the Green function has Dirichlet boundary condition in the sense of Definition \ref{d:dbdry}.
	
	\begin{lem}[Dirichlet boundary condition of Green function] \label{l:dbdy}
		Let $(\ambient,d,\refmeas,\form,\domain)$ be an MMD space satisfying \hyperlink{ehi}{$\on{EHI}$},
		and $D$ a non-empty open subset of $\ambient$ such that the part Dirichlet form
		$(\form^D, \domain^0(D))$ on $D$ is transient. Then for any $y_0 \in D$,
		the function $D \setminus \{y_0\} \ni x \mapsto \gren{D}(x,y_{0})$ belongs to
		$\domain^0_{\loc}(D, D \setminus \set{y_0})$ and is $\form$-harmonic on $D \setminus \set{y_0}$.
	\end{lem}
	
	\begin{proof} %***May be using \cite[Lemma 3.4]{BCM} along with \cite[Corollary 1.5.1.]{FOT} is better.
		The following argument is a variant of \cite[Proof of Lemma 4.10]{BM19}.
		
		By \cite[Theorems 1.5.4-(i) and 4.2.6]{FOT}, there exists a $(0,\infty)$-valued function
		$f_0 = f_{D,0} \in L^{1}(D,\refmeas|_{D})$ such that $\int_D f_0 G_D f_0 \, d\refmeas < \infty$
		and $G_D f_0 \in \domain^{0}(D)_{e}$. Let us adopt the convention that $f_0$ is extended to
		$\ambient$ by setting $f_0 \coloneqq 0$ on $\ambient \setminus D$, and similarly for $G_D f$
		for any Borel measurable function $f\colon D\to[0,\infty]$.
		
		Let $y_0 \in D$, and let $K$ be any compact subset of $\ambient$
		such that $y_0 \not\in K$. Choose $\phi$ so that
		$\phi \in \mathcal{F} \cap \contfunc_{\mathrm{c}}(\ambient)$, $\phi$ is $[0,1]$-valued, $\phi = 1$ on $K$,
		and $y_0 \not\in \supp_{\ambient}[\phi]$. For each $r > 0$ with $B( y_0, 2r ) \subset D$
		and $r < \dist( y_0, \supp_{\ambient}[\phi] )$, consider the function
		\begin{equation} \label{e:deffr}
			g_r \coloneqq \phi \min\{ 1/r, G_D( f_r ) \}, \quad \textrm{where $f_r \coloneqq \biggl( \int_{B( y_0, r )} f_0 \, d\refmeas \biggr)^{-1} \one_{ B( y_0, r ) } f_0$.}
		\end{equation}
		Then $G_D( f_r)$ is an element of $\domain^{0}(D)_{e}$ $\form$-quasi-continuous
		on $D$ by \cite[Corollary 1.5.1 and Theorem 4.2.6]{FOT}, and hence is $\form$-quasi-continuous on $\ambient$ by
		\cite[Theorem 3.4.9]{CF}, \cite[Theorem 4.4.3]{FOT} and our convention that
		$g_r = 0$ on $\ambient \setminus D$. Since
		$\domain^{0}(D)_{e} \cap L^2( \ambient, \refmeas ) = \domain^0(D)$,
		it follows that $g_r \in \domain^{0}(D)$. Also, $G_D( f_r )$ and $g_r$ are
		continuous on $D \setminus \overline{ B( y_0, r ) }$ by the continuity
		of Green's function $\gren{D}$ on $D$ and dominated convergence.
		Note that for any $r_0>0$ such that  $B( y_0, 2r_0 ) \subset D$
		and $r_0 < \dist( y_0, \supp_{\ambient}[\phi] )$, the function $(x,y)\mapsto \gren{D}( x, y )$ stays bounded for $x \in D \setminus B( y_0, 2r_0 )$
		and $y \in B( y_0, r_0 )$ by the latter of the maximum principles \eqref{e:max}
		and the joint continuity of $\gren{D}$. Therefore, there exists $\delta \in (0,\infty)$
		such that $B( y_0, 2\delta ) \subset D$, $\delta < \dist( y_0, \supp_{\ambient}[\phi] )$
		and for any $r \in (0,\delta)$ we have
		\begin{equation*} 
			g_r = \phi \min\{ 1/r, G_D( f_r ) \} = \phi G_D (f_r) \in \domain^{0}(D) \cap L^{\infty}(\ambient,\refmeas).
		\end{equation*}
		Thus for all $r,s \in (0,\delta)$, by \cite[Theorem 1.4.2-(ii)]{FOT} and \eqref{e:FD} we have
		$\phi^{2}(G_{D}(f_{r})-G_{D}(f_{s})) \in \domain^{0}(D)$, and hence by \cite[(1.5.9)]{FOT} 
		\begin{equation} \label{e:db1}
			\form\bigl(G_{D}(f_{r})-G_{D}(f_{s}),\phi^{2}(G_{D}(f_{r})-G_{D}(f_{s}))\bigr)
			= \int_{\ambient} (f_{r}-f_{s}) \phi^{2}(G_{D}(f_{r})-G_{D}(f_{s})) \,d\refmeas =0.
		\end{equation}
		
		Now, as $r \downarrow 0$, $g_{r} = \phi G_{D}(f_{r})$ converges pointwise on $\ambient$
		to $\phi \gren{D}( \cdot, y_0 )$ (and uniformly on any compact subset of $D \setminus \{ y_0 \}$)
		by the joint continuity of $\gren{D}$, and it thus remains to prove that
		this convergence takes place also in $( \domain, \form_{1} )$.
		The convergence in $L^2( \ambient, \refmeas )$ is clear by dominated convergence
		because these functions are uniformly bounded and supported on
		$\supp_{\ambient}[\phi]$. These functions form an $\form$-Cauchy family
		as $r \downarrow 0$ since we can apply dominated convergence to the
		right-hand side of the equality
		\begin{equation*}
			\form( g_r - g_s, g_r - g_s )
			= \int_{ \supp_{\ambient}[\phi] } ( G_D( f_r ) - G_D( f_s ) )^{2} d\Gamma(\phi,\phi),
		\end{equation*}
		which is implied by the Leibniz rule \cite[Lemma 3.2.5]{FOT} for $\Gamma$,
		\eqref{e:db1} and the same calculation as in \eqref{e:ha1}.
		(\cite[Lemma 3.2.5]{FOT} is stated only for functions in $\domain \cap L^{\infty}(\ambient,\refmeas)$,
		but can be easily verified also for ones in $\domain_{e} \cap L^{\infty}(\ambient,\refmeas)$
		by extending \cite[Corollary 3.2.1]{FOT} from $u\in\domain$ to $u\in\domain_{e}$
		on the basis of \cite[Exercise 1.4.1, Lemma 2.1.4 and Theorem 2.3.3-(i)]{FOT} and applying it
		together with \cite[Exercise 1.4.1]{FOT} and $\domain_{e}\cap L^{2}(\ambient,\refmeas)=\domain$.)
	\end{proof}
	
	The following Dynkin--Hunt type formula is a basic ingredient in comparing the Green function on two domains.
	
	\begin{lem}[Dynkin--Hunt formula]\label{l:dhformula}
		Let $(\ambient,d,\refmeas,\form,\domain)$ be an MMD space satisfying \hyperlink{ehi}{$\on{EHI}$},
		and let $\diff$ be an $\refmeas$-symmetric diffusion on $\ambient$ whose Dirichlet form is $(\form,\domain)$.
		Let $D_1 \subset D_2$ be open subsets of $\ambient$ such that the part Dirichlet form
		$(\form^{D_2}, \domain^0(D_2))$ on $D_2$ is transient. Then there exists
		a properly exceptional set $\mathcal{N}_{D_2}$ for $\diff^{D_2}$ such that
		for all $(x,y) \in \offdiagp{D_1}$ with $x \not\in \mathcal{N}_{D_2}$,
		\begin{equation} \label{e:dh}
			\gren{D_2}(x,y) = \gren{D_1}(x,y)+ \expdiff_{x}\bigl[ \one_{\{\diff_{\tau_{D_1}} \in D_2 \}} \gren{D_2}(\diff_{\tau_{D_1}},y) \bigr].
		\end{equation}
		In addition, if the MMD space $(\ambient,d,\refmeas,\form,\domain)$ and the diffusion $\diff$ on $\ambient$ satisfy Assumption \ref{a:feller},
		then \eqref{e:dh} holds for all $(x,y) \in \offdiagp{D_1}$.
	\end{lem}
	
	\begin{proof}
		By the occupation density formula (Proposition \ref{p:goodgreen}-\eqref{it:goodgreen-odf}) and \cite[Lemma 4.5]{BCM},
		there exists a Borel properly exceptional set $\mathcal{N}_{D_2}$ for $\diff^{D_2}$ such that for all
		Borel measurable function $f \colon D_2 \to [0,\infty]$ and all $x \in D_1 \setminus \mathcal{N}_{D_2}$ we have
		\begin{equation}\label{e:odf-D1D2}
			\expdiff_{x}\biggl[ \int_0^{\tau_{D_i}} f(\diff_s) ds \biggr] = \int_{D_i} \gren{D_i}(x, y) f(y)\, \refmeas(dy),
			\quad \textrm{for $i=1,2$.}
		\end{equation}  
		Therefore for any such $f$ and $x$, we have 
		\begin{align} \label{e:dh1}
			&\int_{D_2} \gren{D_2}(x, z) f(z)\, m(dz) \nonumber \\
			&\overset{\eqref{e:odf-D1D2}}{=} \expdiff_{x}\biggl[ \int_0^{\tau_{D_2}} f(\diff_s)\, ds \biggr] = \expdiff_{x}\biggl[ \int_0^{\tau_{D_1}} f(\diff_s)\, ds \biggr] + \expdiff_{x}\biggl[ \int_{\tau_{D_1}}^{\tau_{D_2}} f(\diff_s)\, ds \biggr] \nonumber\\
			&\overset{\eqref{e:odf-D1D2}}{=} \int_{D_1} \gren{D_1}(x, z) f(z)\, \refmeas(dz) + \expdiff_{x}\biggl[ \one_{\{\diff_{\tau_{D_1}} \in D_2\}} \expdiff_{\diff_{\tau_{D_1}}}\biggl[ \int_0^{\tau_{D_2}} f(\diff_s)\, ds \biggr] \biggr]\nonumber  \\
			&\overset{\eqref{e:odf-D1D2}}{=} \int_{D_1} \gren{D_1}(x, z) f(z)\, \refmeas(dz) + \int_{D_2} \expdiff_{x} \bigl[ \one_{\{\diff_{\tau_{D_1}} \in D_2\}} \gren{D_2}(\diff_{\tau_{D_1}},z) \bigr] f(z) \, \refmeas(dz),
		\end{align}
		where we used the strong Markov property \cite[Theorem A.1.21]{CF} of $\diff$
		and Fubini's theorem in the third and fourth lines, respectively. Now for any 
		$y \in D_1 \setminus \{x\}$, setting $f \coloneqq (\refmeas(B(y,r)))^{-1} \one_{B(y,r)}$
		and letting $r \downarrow 0$ in \eqref{e:dh1}, we obtain \eqref{e:dh} by the
		continuity of $\gren{D_1},\gren{D_2}$, the maximum principle for $\gren{D_2}$
		(Proposition \ref{p:goodgreen}-\eqref{it:goodgreen-cont},\eqref{it:goodgreen-max})
		and the dominated convergence theorem.
		
		If $(\ambient,d,\refmeas,\form,\domain)$ and $\diff$ satisfy Assumption \ref{a:feller},
		then we have \eqref{e:odf-D1D2} for any $x \in D_1$ by Lemma \ref{l:green},
		so that the above argument shows \eqref{e:dh} for any $(x,y) \in \offdiagp{D_1}$.
	\end{proof}
	
	For an MMD space $(\ambient,d,\refmeas, \form, \domain)$ satisfying \hyperlink{ehi}{$\on{EHI}$},
	and for a non-empty open subset $D \subset \ambient$ such that the part Dirichlet form
	$(\form^{D}, \domain^{0}(D))$ on $D$ is transient, we define (by a slight abuse of notation)
	\begin{equation} \label{e:defgdr}
		\gren{D}(x,r) \coloneqq \inf_{y \in S(x,r)} \gren{D}(x,y)
		\qquad \textrm{for $x \in D$ and $r \in (0,\delta_D(x))$,}
	\end{equation}
	where $S(x,r) \coloneqq \partial B(x,r)$ as defined in Notation \ref{ntn:intro}-\eqref{it:ball-diam-dist}.
	
	We collect various useful estimates on the Green function from \cite{BCM}.
	
	\begin{lem} \label{l:hit}
		Let $(\ambient,d,\refmeas,\form,\domain)$ be an MMD space satisfying \hyperlink{MD}{\textup{MD}} and \hyperlink{ehi}{$\on{EHI}$},
		and let $\diff$ be an $\refmeas$-symmetric diffusion on $\ambient$ whose Dirichlet form is $(\form,\domain)$.
		Let $D$ be a non-empty open subset of $\ambient$ such that the part Dirichlet form $(\form^{D}, \domain^0(D))$ on $D$ is transient.
		Then there exist $C_0, C_1, C_2, A_0, \theta \in (1,\infty)$ depending only on the constants
		associated with the assumptions \hyperlink{MD}{\textup{MD}} and \hyperlink{ehi}{$\on{EHI}$}
		such that the following hold:
		\begin{enumerate}[\rm(a)]\setlength{\itemsep}{0pt}\vspace{-5pt}
			\item\label{it:Green-func-bounds}  For all $x \in D$ and all $r \in (0,\delta_D(x)/A_0)$,
			\begin{equation} \label{e:gradial}
				\sup_{y \in S(x,r)} g_D(x,y) \le C_1 \inf_{y \in S(x,r)} g_D(x,y), \quad g_D(x,r) \le \Capa_D(B(x,r))^{-1} \le C_1 g_D(x,r).
			\end{equation}
			Furthermore,
			\begin{equation} \label{e:grcomp}
				g_D(x,R) \le g_D(x,r) \le C_2 \biggl(\frac{R}{r}\biggr)^\theta g_D(x,R) \quad \mbox{for all $x \in D$ and $0<r< R \le \delta_D(x)/A_1$.}
			\end{equation}
			\item\label{it:hitting-prob-bounds} For each $y \in D$ and each $R \in (0,\delta_D(y)/A_{0})$,
			\begin{equation} \label{e:hitp}
				C_0^{-1} \frac{\gren{D}(x,y)}{\gren{D}(y,R)} \le \lawdiff_{x}\bigl(\sigma_{\ol{B(y,R)}} < \sigma_{D^c}\bigr) \le C_0 \frac{\gren{D}(x,y)}{\gren{D}(y,R)} \quad \mbox{for $\form$-q.e.~$x \in D \setminus \ol{B(y,R)}$.}
			\end{equation}
			If $(\ambient,d,\refmeas,\form,\domain)$ and $\diff$ satisfy Assumption \ref{a:feller}, then \eqref{e:hitp} holds for all $x \in D \setminus \ol{B(y,R)}$. 
		\end{enumerate}
	\end{lem}
	
	\begin{proof}
		\begin{enumerate}[\rm(a)]\setlength{\itemsep}{0pt}
			\item The estimate \eqref{e:gradial} follows from \cite[Lemma 5.10 and Proposition 5.7]{BCM}
			and \eqref{e:grcomp} follows from \cite[Corollary 5.15]{BCM} and the maximum principle (Proposition \ref{p:goodgreen}-\eqref{it:goodgreen-max}).
			\item By Lemma \ref{l:chain}-\eqref{it:EHI-MD-RBC}, we can choose $K \in (1,\infty)$ so that $(\ambient,d)$ is $K$-relatively ball connected.
			Let $A_1 \in (1,\infty)$ be as given in \eqref{it:Green-func-bounds}. By \cite[Lemma 5.10]{BCM}
			and \eqref{it:Green-func-bounds}, there exist $A_1 \in (K,\infty)$ and $C_1 \in (1,\infty)$ such that 
			\begin{equation} \label{e:hitp1}
				\gren{D}(y,R) \le \Capa_D(B(y,R))^{-1} \le C_1	\gren{D}(y,R), \quad \gren{D}(y,R) \le \gren{D}(y,z) \le C_1	\gren{D}(y,R) 
			\end{equation}
			for all $y \in D$, all $R \in (0,A_1^{-1} \delta_D(y))$ and all $z \in S(y,R)$.
			%By   \cite[Corollary 5.15]{BCM} and increasing $A_1$ if necessary, there exist $\theta, c_1 \in (0,1)$ such that 
			% \begin{equation} \label{e:hitp2}
				% 	\frac{\gren{D}(x,r)}{\gren{D}(y,s)} \ge c_1 \left( \frac{s}{r}\right)^\theta \quad \mbox{for all $y \in D, 0<s < r \le  A_1^{-1} K \delta_D(y)$}.
				% \end{equation}
			Let  $y \in D$, $R \in (0,A_1^{-1} \delta_D(y))$, and let $\nu$ denote the equilibrium measure on $S(y,R)$ corresponding to $\Capa_{D}(B(y,R))$. \\
			\emph{Case 1}, $d(x,y)\ge 2KR$: In this case, $\gren{D}(x,\cdot)$ is $\form$-harmonic on $B(y,2KR)$ and hence by \eqref{e:hchain} and \eqref{e:rbc}, there exists $C_2 \in (1,\infty)$ such that 
			\begin{equation} \label{e:hitp3}
				C_2^{-1} \gren{D}(x,y) \le	\gren{D}(x,z) \le C_2 \gren{D}(x,y) \quad \mbox{for all $z \in S(y,R)$.}
			\end{equation}
			Therefore by \cite[Theorem 4.3.3]{FOT}, for $\form$-q.e.\ $x \in D \setminus B(y,KR)$,
			\begin{align}
				\lawdiff_{x}\bigl( \sigma_{\ol{B(y,R)}} < \sigma_{D^c} \bigr) &= \int_{S(y,R)} \gren{D}(x,z)\,\nu(dy) \overset{\eqref{e:hitp3}}{\le} C_2 \gren{D}(x,y) \Capa_D(B(y,R))\nonumber \\
				&\overset{\eqref{e:hitp1}}{\le} C_2 \frac{\gren{D}(x,y)}{\gren{D}(y,R)}, \label{e:hitp4} \\
				\lawdiff_{x}\bigl( \sigma_{\ol{B(y,R)}} < \sigma_{D^c} \bigr) &= \int_{S(y,R)} \gren{D}(x,z)\,\nu(dy) \overset{\eqref{e:hitp3}}{\ge} C_2^{-1} \gren{D}(x,y) \Capa_D(B(y,R))\nonumber \\
				&\overset{\eqref{e:hitp1}}{\ge} C_2^{-1} C_1^{-1} \frac{\gren{D}(x,y)}{\gren{D}(y,R)}. \label{e:hitp5}
			\end{align}
			
			\noindent \emph{Case 2}, $R \le d(x,y) < 2KR$: By \cite[Theorem 4.3.3]{FOT},
			for $\form$-q.e.\ $x \in D$ with $R \le d(x,y)\le KR$,
			\begin{align}
				\lawdiff_{x}\bigl( \sigma_{\ol{B(y,R)}} < \sigma_{D^c} \bigr) &\ge \lawdiff_{x}( \sigma_{B(y,R/(2K))} < \sigma_{D^c} )
				\overset{\eqref{e:hitp5}}{\ge}   C_2^{-1} C_1^{-1} \frac{\gren{D}(x,y)}{\gren{D}(y,R/(2K))} \nonumber \\ & \overset{\eqref{e:grcomp}}{\ge}  C_2^{-1} C_1^{-1} c_1 (2K)^{-\theta} \frac{\gren{D}(x,y)}{\gren{D}(y,R )}, \label{e:hitp6} \\
				\lawdiff_{x}\bigl( \sigma_{\ol{B(y,R)}} < \sigma_{D^c} \bigr) & \le 1 \overset{\eqref{e:grcomp}}{\le} c_1^{-1}K^{\theta} \frac{\gren{D}(x,y)}{\gren{D}(y,R )}. \label{e:hitp7}
			\end{align}
			By \eqref{e:hitp4}, \eqref{e:hitp5}, \eqref{e:hitp6}, and \eqref{e:hitp7}, we obtain \eqref{e:hitp}.
			
			If the MMD space $(\ambient,d,\refmeas,\form,\domain)$ and the associated diffusion $\diff$ satisfies Assumption \ref{a:feller}, then by
			Lemma \ref{l:harmonicm}-\eqref{it:hmeas-continuous} we obtain \eqref{e:hitp} for all $x \in D \setminus \ol{B(y,R)}$.
			\qedhere\end{enumerate}
	\end{proof}
	
	\subsection{Boundary Harnack principle} \label{ssec:BHP}
	
	In this work, we need to understand the behavior of Green function near the boundary of a uniform domain.
	The following scale-invariant boundary Harnack principle is useful to describe the behavior of Green function near the boundary of a uniform domain.
	Boundary Harnack principle has been obtained in increasing generality over a long period of time \cite{Kem,Anc78,Dah,Wu,JK,Aik,GS,Lie,BM19}.
	\begin{definition}[Boundary Harnack principle (BHP)] \label{d:bhp}
		Let $(\ambient,d,\refmeas,\form,\domain)$ be an MMD space and let $\unifdom$ be an open subset of $\ambient$.
		We say that $\unifdom$ satisfies the \textbf{(scale-invariant) boundary Harnack principle}, abbreviated as \textup{BHP},
		if there exist $A_0, A_1, C_1 \in (1,\infty)$ such that for all $\xi \in \partial \unifdom$,
		all $r \in (0,\diam(\unifdom)/A_1)$ and for any two non-negative $\form$-harmonic functions 
		$u,v$ on $\unifdom \cap B(\xi,A_{0} r)$ with Dirichlet boundary condition relative to $\unifdom$
		such that $v>0$ $\refmeas$-a.e.\ on $\unifdom \cap B(\xi,r)$, we have
		\begin{equation} \tag*{\textup{BHP}} \label{eq:BHP}
			\esssup_{x \in \unifdom \cap B(\xi,r)}	\frac{u(x)}{v(x)} 
			\le C_1  \essinf_{x \in \unifdom \cap B(\xi,r)} \frac{u(x)}{v(x)}.
			% \frac{v(x)}{v(x')} \qq \forall x,x' \in B_{\unifdom}(\xi,r).
		\end{equation}
	\end{definition}
	
	The elliptic Harnack inequality implies the boundary Harnack principle for uniform domains on any doubling metric space as shown in a recent work  \cite{Che}.
	This recent  work  \cite{Che} along with earlier works in more restrictive settings in \cite{GS, Lie, BM19} use an approach due to Aikawa \cite{Aik}.
	
	\begin{theorem}[Boundary Harnack principle for uniform domains; {\cite[Theorem 1.1]{Che}}] \label{t:bhp}
		Let $(\ambient,d,\refmeas,\form,\domain)$ be an MMD space satisfying \hyperlink{MD}{\textup{MD}}
		and \hyperlink{ehi}{$\on{EHI}$}, and let $\unifdom$ be a uniform domain in $(\ambient,d)$.
		Then $\unifdom$ satisfies \ref{eq:BHP}.
	\end{theorem}
	
	\begin{remark}\label{rmk:unifdom-bdry-non-empty-diam-pos}
		Note that $\partial \unifdom\not=\emptyset$ and $\diam(\unifdom)\in(0,\infty]$
		in the setting of Theorem \ref{t:bhp}; indeed, otherwise $\unifdom$ would be both open and
		closed in $\ambient$ and satisfy $\emptyset\not=\unifdom\not=\ambient$, which is impossible
		since $\ambient$ is connected by Lemma \ref{l:chain}-\eqref{it:EHI-MD-RBC} and \cite[Lemma 5.2-(a)]{BCM}.
	\end{remark}
	
	%(instead of \emph{length} uniform domains considered in Theorem \ref{t:bhp}).
	%In other words, Theorem \ref{t:bhp} can be generalized to uniform domains to metric spaces that need not contain any non-constant rectifiable curves.  
	
	The following oscillation lemma is a standard consequence of the boundary Harnack principle and follows from \cite[Proof of Theorem 2]{Aik}.
	It is an analogue of Moser's oscillation lemma for the elliptic Harnack inequality \cite[\textsection 5]{Mos} and has a similar proof.
	
	\begin{lem} \label{l:bhpholder}
		Let $(\ambient,d,\refmeas,\form,\domain)$ be an MMD space and let $\unifdom$ be an open subset of $\ambient$ satisfying \ref{eq:BHP}.
		Then there exist $A_0, A_1, C_0 \in (1,\infty)$ and $\gamma\in(0,\infty)$ such that for all
		$\xi \in \partial \unifdom$, all $0<r<R<\diam(\unifdom)/A_1$ and for any two non-negative continuous $\form$-harmonic functions
		$u,v$ on $\unifdom \cap B(\xi,A_0 R)$ with Dirichlet boundary condition relative to $\unifdom$
		such that $v(x)>0$ for any $x \in \unifdom \cap B(\xi,R)$, we have 
		\begin{equation} \label{eq:bhpholder}
			\osc_{\unifdom \cap B(\xi,r)} \frac{u}{v}
			\le C_0 \Bigl( \frac{r}{R} \Bigr)^\gamma \osc_{\unifdom\cap B(\xi,R)} \frac{u}{v}.
		\end{equation}
	\end{lem}
	
	Another important consequence of the boundary Harnack principle is the Carleson estimate.
	The proof is a variant of \cite[Proof of Theorem 2]{Aik08} where we use estimates on Green function from \cite{BM18, BCM} instead of known estimates of the Euclidean space. 
	The basic idea is that Carleson estimate for one harmonic function with Dirichlet boundary condition (say, the Green function at a suitably chosen point) along with boundary Harnack principle implies Carleson estimate in general.
	The Carleson estimate for Green function can be obtained by using the maximum principle and comparison estimates for the Green function obtained in \cite{BM18,BCM}.
	This is a modification of the argument in \cite[Proof of (4.28)]{GS}.
	
	\begin{prop}[Carleson estimate] \label{p:carleson}
		Let $(\ambient,d,\refmeas,\form,\domain)$ be an MMD space satisfying \hyperlink{MD}{\textup{MD}}
		and \hyperlink{ehi}{$\on{EHI}$}, and let $\unifdom$ be a uniform domain in $(\ambient,d)$.
		Then there exist $A_0, A_1, C_0 \in (1,\infty)$ such that for all $\xi \in \partial \unifdom$,
		all $R \in (0,\diam(\unifdom)/A_1)$ and any non-negative continuous $\form$-harmonic function $u$
		on $\unifdom \cap B(\xi,A_0R)$ with Dirichlet boundary condition relative to $\unifdom$,
		\begin{equation} \label{eq:carleson}
			\sup_{x \in B(\xi,R)} u(x)\le C u(\xi_{R/2}).
		\end{equation}
	\end{prop}
	
	\begin{proof}
		Let $u$ be an $\form$-harmonic function as in the statement of the proposition.
		Noting that $\unifdom$ satisfies \ref{eq:BHP} by Theorem \ref{t:bhp},
		let us choose $A_0, A_1, C_1$ as the constants in Definition \ref{d:bhp}.
		First, we note that there exist $C_2, A_3 \in (1,\infty)$ and $A_4 \in (A_1,\infty)$ such that 
		\begin{equation} \label{e:carl1}
			\sup_{\unifdom  \cap B(\xi,R)} \gren{\unifdom \cap B(\xi,A_3 R)}(\xi_{2A_0 R}, \cdot)
			\le C_2 \gren{\unifdom \cap B(\xi,A_3 R)}(\xi_{2A_0 R},\xi_{R/2}),  
		\end{equation}
		for all $\xi \in \partial \unifdom$ and all $R \in (0,A_4^{-1}\diam(\unifdom))$.
		This follows from the chaining using \hyperlink{ehi}{$\on{EHI}$}
		by a similar argument as given in the proof of Lemma \ref{l:chain}-\eqref{it:chain-unifdom},
		the maximum principle (Proposition \ref{p:goodgreen}-\eqref{it:goodgreen-max}) and
		the comparison of Green functions in \cite[Corollary 5.8]{BCM}.
		Then by \ref{eq:BHP} (Definition \ref{d:bhp}) from Theorem \ref{t:bhp}, we have
		\begin{equation} \label{e:carl2}
			\sup_{B(\xi,R)} \frac{u(\cdot)}{\gren{\unifdom \cap B(\xi,A_3 R)}(\xi_{2A_0 R},\cdot)}
			\leq C_{1} \frac{u(\xi_{R/2})}{\gren{\unifdom \cap B(\xi,A_3 R)}(\xi_{2A_0 R}, \xi_{R/2})}
		\end{equation}
		for all $\xi \in \partial \unifdom$ and all $R \in (0,A_4^{-1}\diam(\unifdom))$.
		Therefore by \eqref{e:carl1} and \eqref{e:carl2}, we conclude that for all $\xi \in \partial \unifdom$,
		all $R \in (0,A_4^{-1}\diam(\unifdom))$ and any non-negative continuous $\form$-harmonic function $u$
		on $\unifdom \cap B(\xi,A_0R)$ with Dirichlet boundary condition relative to $\unifdom$, we have
		\begin{equation*}
			\sup_{B(\xi,R)}u(\cdot)
			\leq C_{1} \frac{u(\xi_{R/2})}{\gren{\unifdom \cap B(\xi,A_3 R)}(\xi_{2A_0 R}, \xi_{R/2})} \sup_{B(\xi,R)} \gren{\unifdom \cap B(\xi,A_3 R)}(\xi_{2A_0 R},\cdot)
			\leq C_1 C_2 u(\xi_{R/2}).
			\qedhere\end{equation*}
	\end{proof}
	
	\subsection{Na\"im kernel}
	
	We introduce the Na\"im kernel and study some of its properties.
	For the remainder of the section we make the following running assumption.
	
	\begin{assumption} \label{a:ehi-bhp}
		Let $(\ambient,d,\refmeas,\form,\domain)$ be an MMD space satisfying \hyperlink{MD}{\textup{MD}}
		and \hyperlink{ehi}{$\on{EHI}$}, and let $\unifdom$ be a uniform domain in $(\ambient,d)$
		such that the part Dirichlet form $(\form^\unifdom, \domain^0(\unifdom))$ on $\unifdom$ is transient.
		Note that $\unifdom$ satisfies \ref{eq:BHP} by Theorem \ref{t:bhp} and that
		$\partial \unifdom\not=\emptyset$ and $\diam(\unifdom)\in(0,\infty]$ by Remark \ref{rmk:unifdom-bdry-non-empty-diam-pos}.
	\end{assumption}
	%  By the result of A.~Chen mentioned above, the assumption that $\unifdom$ satisfies the boundary Harnack principle is redundant but since \cite{Che} is not yet available, we made this additional assumption throughout this work. In the case of \emph{length} uniform domains in a length space, we can use Theorem \ref{t:bhp} instead of the upcoming work \cite{Che} to remove the assumption concerning the boundary Harnack principle.
	
	Recalling that $\gren{\unifdom}|_{\offdiag{\unifdom}}$ is $(0,\infty)$-valued by Remark \ref{r:ehi-hke} and Lemma \ref{l:green},
	for each $x_0 \in \unifdom$ we define $\naimker^\unifdom_{x_0} \colon \offdiagp{\unifdom \setminus \{x_0\}} \to (0,\infty)$ by
	\begin{equation} \label{e:defnaim}
		\naimker^\unifdom_{x_0}(x,y) \coloneqq \frac{\gren{\unifdom}(x,y)}{\gren{\unifdom}(x_0,x) \gren{\unifdom}(x_0,y)}.
	\end{equation}
	
	The function $\naimker^{\unifdom}_{x_0}$ satisfies the following \emph{local H\"older regularity} and bounds.
	The proofs are variants of Moser's oscillation inequality \cite[\textsection 5]{Mos}.
	
	\begin{lem} \label{l:naimholder}
		Let an MMD space $(\ambient,d,\refmeas,\form,\domain)$ and a uniform domain $\unifdom$ in $(\ambient,d)$ satisfy Assumption \ref{a:ehi-bhp}.
		Then there exist $A, C_{1}, C_{2}, C_{3} \in (1,\infty)$ and $\gamma \in (0,\infty)$ such that the following estimates hold for any $x_0 \in \unifdom$:
		\begin{enumerate}[\rm(a)]\setlength{\itemsep}{0pt}\vspace{-5pt}
			\item For any $\eta \in \partial \unifdom$, $z \in \unifdom \setminus \{x_0\}$ and any $0 < r< R < (2A)^{-1} (d(\eta,x_0) \wedge d(z,x_0) \wedge \delta_{\unifdom}(z) )$,
			\begin{equation*}
				\osc_{ (B(\eta,r) \cap(\unifdom \setminus \{x_0\}) ) \times (B(z,r) \cap(\unifdom \setminus \{x_0\}) )} \naimker^\unifdom_{x_0}
				\le A \Bigl(\frac{r}{R}\Bigr)^\gamma \osc_{ \left(B(\eta,R) \cap(\unifdom \setminus \{x_0\}) \right) \times (B(z,R) \cap(\unifdom \setminus \{x_0\}) )}\naimker^\unifdom_{x_0}.
			\end{equation*}
			\item For any $(\eta,\xi) \in \offdiagp{\partial \unifdom}$ and any $0 < r< R < (2A)^{-1} \left(d(\eta,x_0) \wedge d(\eta,\xi) \wedge d(\xi,x_0) \right)$,
			\begin{equation*}
				\osc_{ (B(\eta,r) \cap(\unifdom \setminus \{x_0\}) ) \times (B(\xi,r) \cap(\unifdom \setminus \{x_0\}) )} \naimker^\unifdom_{x_0}
				\le A \Bigl(\frac{r}{R}\Bigr)^\gamma \osc_{ (B(\eta,R) \cap(\unifdom \setminus \{x_0\}) ) \times (B(\xi,R) \cap(\unifdom \setminus \{x_0\}) )}\naimker^\unifdom_{x_0}.
			\end{equation*}
			\item For any $\eta \in \partial \unifdom$, any $z \in \unifdom \setminus \{x_0\}$ and any $0 < R < (2A)^{-1} (d(\eta,x_0) \wedge d(z,x_0) \wedge \delta_{\unifdom}(z) )$,
			\begin{equation*}
				\sup_{ (B(\eta,R) \cap(\unifdom \setminus \{x_0\}) ) \times (B(z,R) \cap(\unifdom \setminus \{x_0\}) )} \naimker^\unifdom_{x_0}
				\le C_1 \frac{ \gren{\unifdom}(z,\eta_{R/2})}{\gren{\unifdom}(x_0,z) \gren{\unifdom}(x_0,\eta_{R/2})}
			\end{equation*}
			and
			\begin{equation*}
				\inf_{ (B(\eta,R) \cap(\unifdom \setminus \{x_0\}) ) \times (B(z,R) \cap(\unifdom \setminus \{x_0\}) )} \naimker^\unifdom_{x_0}
				\geq C_{1}^{-1} \frac{\gren{\unifdom}(z,\eta_{R/2})}{\gren{\unifdom}(x_0,z) \gren{\unifdom}(x_0,\eta_{R/2})}.
			\end{equation*}
			\item For any $(\eta,\xi) \in \offdiagp{\partial \unifdom}$ and any $0 < R < (2A)^{-1} \bigl(d(\eta,x_0) \wedge d(\eta,\xi) \wedge d(\xi,x_0) \bigr)$,
			\begin{equation*}
				\sup_{ (B(\eta,R) \cap(\unifdom \setminus \{x_0\}) ) \times ( B(\xi,R) \cap(\unifdom \setminus \{x_0\}) )} \naimker^\unifdom_{x_0} \le C_2 \frac{\gren{\unifdom}(\eta_{R/2},\xi_{R/2})}{\gren{\unifdom}(x_0,\eta_{R/2})\gren{\unifdom}(x_0,\xi_{R/2})},
			\end{equation*}
			and
			\begin{equation*}
				\inf_{ (B(\eta,R) \cap(\unifdom \setminus \{x_0\}) ) \times (B(\xi,R) \cap(\unifdom \setminus \{x_0\}) )} \naimker^\unifdom_{x_0} \ge C_2^{-1} \frac{\gren{\unifdom}(\eta_{R/2},\xi_{R/2})}{\gren{\unifdom}(x_0,\eta_{R/2})\gren{\unifdom}(x_0,\xi_{R/2})}.
			\end{equation*}
			\item For any $(x,\xi) \in (\unifdom \setminus \{x_0\}) \times \partial \unifdom$
			with $d(\xi,x_0) \le d(\xi,x)$ and any $0<r<R <A^{-1}d(\xi,x_0)$,
			\begin{equation} \label{e:naimbnd}
				\sup_{y \in \unifdom \cap B(\xi,R)}\naimker^\unifdom_{x_0}(x,y) \le C_{3} \naimker^\unifdom_{x_0}(x,\xi_{R/2}), \quad
				\inf_{y \in \unifdom \cap B(\xi,R)}\naimker^\unifdom_{x_0}(x,y) \ge C_{3}^{-1} \naimker^\unifdom_{x_0}(x,\xi_{R/2}),
			\end{equation}
			and
			\begin{equation} \label{e:naimequi}
				\osc_{y \in \unifdom \cap B(\xi,r)} \naimker^\unifdom_{x_0}(x,y) \le C_{3} \Bigl(\frac{r}{R}\Bigr)^{\gamma} \naimker^\unifdom_{x_0}(x,\xi_{R/2}).
			\end{equation}
		\end{enumerate}
	\end{lem}
	
	\begin{proof}
		Let $A \in (1,\infty)$ be the maximum of the constants $\delta^{-1}$ in \hyperlink{ehi}{EHI},
		$A_0$ and $A_1$ in Definition \ref{d:bhp}. Let $C_{\on{EHI}}$ and $C_{\on{BHP}}$ denote the
		corresponding constants $C_{H}$ and $C_1$, respectively. We will use \hyperlink{ehi}{EHI} and
		\ref{eq:BHP} several times in this proof with these constants $A,C_{\on{EHI}},C_{\on{BHP}}$.
		\begin{enumerate}[\rm(a)]\setlength{\itemsep}{0pt}\vspace{-5pt}
			\item
			For any $0<r< (2A)^{-1} \bigl(d(\eta,x_0) \wedge d(z,x_0) \wedge \delta_{\unifdom}(z)\bigr)$, define
			\begin{align*}
				M(r)& \coloneqq \sup_{ (B(\eta,r) \cap(\unifdom \setminus \{x_0\}) ) \times (B(z,r) \cap(\unifdom \setminus \{x_0\}) )} \naimker^{\unifdom}_{x_0},\\ m(r)& \coloneqq \inf_{ \left(B(\eta,r) \cap(\unifdom \setminus \{x_0\}) \right) \times \left(B(z,r) \cap(\unifdom \setminus \{x_0\}) \right)} \naimker^{\unifdom}_{x_0}.
			\end{align*}
			
			For any $(x_1,y_1), (x_2,y_2) \in (B(\eta,R/A) \cap(\unifdom \setminus \{x_0\}) ) \times (B(z,R/A) \cap(\unifdom \setminus \{x_0\}) )$, we have
			\begin{align} \label{e:naim1}
				\MoveEqLeft{\frac{M(R)\gren{\unifdom}(x_0,x_1) \gren{\unifdom}(x_0,y_1) -\gren{\unifdom}(x_1,y_1)}{\gren{\unifdom}(x_0,x_1) \gren{\unifdom}(x_0,y_1)}} \nonumber\\ 
				&\le  C_{\on{BHP}}\frac{M(R)\gren{\unifdom}(x_0,x_2) \gren{\unifdom}(x_0,y_1) -\gren{\unifdom}(x_2,y_1)}{\gren{\unifdom}(x_0,x_2) \gren{\unifdom}(x_0,y_1)} \nonumber \\
				& \le C_{\on{BHP}} C_{\on{EHI}}^2 \frac{M(R)\gren{\unifdom}(x_0,x_2) \gren{\unifdom}(x_0,y_2) -\gren{\unifdom}(x_1,y_2)}{\gren{\unifdom}(x_0,x_2) \gren{\unifdom}(x_0,y_2)};
			\end{align}
			here, for the first inequality we apply \ref{eq:BHP} to the functions
			$M(R)\gren{\unifdom}(x_0,\cdot) \gren{\unifdom}(x_0,y_1) -\gren{\unifdom}(\cdot,y_1),\gren{\unifdom}(x_0,\cdot) \gren{\unifdom}(x_0,y_1) \in \domain_{\on{loc}}^0(\unifdom,B(\eta,Ar) \cap \unifdom)$,
			which are non-negative and $\form$-harmonic on $B(\xi,Ar) \cap \unifdom$, and
			for the second inequality we apply \hyperlink{ehi}{$\on{EHI}$} to
			$M(R)\gren{\unifdom}(x_0,x_2) \gren{\unifdom}(x_0,\cdot) -\gren{\unifdom}(x_2,\cdot), \gren{\unifdom}(x_0,x_2) \gren{\unifdom}(x_0,\cdot) \in \domain_{\on{loc}}(B(z,R))$,
			which are non-negative and $\form$-harmonic on $B(z,R)$.
			
			Taking supremum over $(x_1,y_1)$ and infimum over $(x_2,y_2)$ in \eqref{e:naim1}, % over $\left(B(\eta,R/A) \cap(\unifdom \setminus \{x_0\}) \right) \times \left(B(z,r) \cap(\unifdom \setminus \{x_0\}) \right)$
			we obtain
			\begin{equation} \label{e:naim2}
				M(R)-m(R/A) \le C_{\on{BHP}} C_{\on{EHI}}^2  ( M(R)- M(R/A)).
			\end{equation}
			By considering $(x,y) \mapsto  \naimker^{\unifdom}_{x_0}(x,y)-m(R)=\frac{\gren{\unifdom}(x,y)-m(R)\gren{\unifdom}(x_0,x) \gren{\unifdom}(x_0,y)}{\gren{\unifdom}(x_0,x) \gren{\unifdom}(x_0,y)}$ and using a similar argument as the proof of \eqref{e:naim2}, we obtain
			\begin{equation} \label{e:naim3}
				M(R/A)-m(R) \le C_{\on{BHP}} C_{\on{EHI}}^2  ( m(R/A)- m(R)).
			\end{equation}
			Combining \eqref{e:naim2} and \eqref{e:naim3}, we obtain
			\begin{equation*}
				M(R/A)-m(R/A) \le \frac{ C_{\on{BHP}} C_{\on{EHI}}^2 -1}{ C_{\on{BHP}} C_{\on{EHI}}^2 +1} \left( M(R)-m(R)\right).
			\end{equation*}
			Iterating the above estimate, we obtain (a) with $\gamma= (\log A)^{-1}\log \frac{ C_{\on{BHP}} C_{\on{EHI}}^2 +1}{ C_{\on{BHP}} C_{\on{EHI}}^2 -1}$.
			\item For any $0<r< (2A)^{-1} \bigl(d(\eta,x_0) \wedge d(\xi,x_0) \wedge d(\eta,\xi)\bigr)$, define
			\begin{align*}
				M(r)& \coloneqq \sup_{ (B(\eta,r) \cap(\unifdom \setminus \{x_0\}) ) \times (B(\xi,r) \cap(\unifdom \setminus \{x_0\}) )} \naimker^{\unifdom}_{x_0},\\ m(r)& \coloneqq \inf_{ (B(\eta,r) \cap(\unifdom \setminus \{x_0\}) ) \times (B(\xi,r) \cap(\unifdom \setminus \{x_0\}) )} \naimker^{\unifdom}_{x_0}.
			\end{align*}
			
			For any $(x_1,y_1), (x_2,y_2) \in (B(\eta,R/A) \cap(\unifdom \setminus \{x_0\}) ) \times (B(\xi,R/A) \cap(\unifdom \setminus \{x_0\}) )$, we have
			\begin{align} \label{e:naim4}
				\MoveEqLeft{\frac{M(R)\gren{\unifdom}(x_0,x_1) \gren{\unifdom}(x_0,y_1) -\gren{\unifdom}(x_1,y_1)}{\gren{\unifdom}(x_0,x_1) \gren{\unifdom}(x_0,y_1)}} \nonumber\\ 
				&\le  C_{\on{BHP}}\frac{M(R)\gren{\unifdom}(x_0,x_2) \gren{\unifdom}(x_0,y_1) -\gren{\unifdom}(x_2,y_1)}{\gren{\unifdom}(x_0,x_2) \gren{\unifdom}(x_0,y_1)} \nonumber \\
				& \le C_{\on{BHP}}^2 \frac{M(R)\gren{\unifdom}(x_0,x_2) \gren{\unifdom}(x_0,y_2) -\gren{\unifdom}(x_1,y_2)}{\gren{\unifdom}(x_0,x_2) \gren{\unifdom}(x_0,y_2)};
			\end{align}
			here, for the first line we apply \ref{eq:BHP} to the functions
			$M(R)\gren{\unifdom}(x_0,\cdot) \gren{\unifdom}(x_0,y_1) -\gren{\unifdom}(\cdot,y_1),\gren{\unifdom}(x_0,\cdot) \gren{\unifdom}(x_0,y_1) \in \domain_{\on{loc}}^0(\unifdom,B(\eta,Ar) \cap \unifdom)$,
			which are non-negative and $\form$-harmonic on $B(\xi,Ar) \cap \unifdom$, and
			for the second inequality we apply \ref{eq:BHP} to
			$M(R)\gren{\unifdom}(x_0,x_2) \gren{\unifdom}(x_0,\cdot) -\gren{\unifdom}(x_2,\cdot), \gren{\unifdom}(x_0,x_2) \gren{\unifdom}(x_0,\cdot) \in \domain_{\on{loc}}^0(\unifdom,\unifdom \cap B(\xi,R))$,
			which are non-negative and $\form$-harmonic on $\unifdom \cap B(\xi,R)$.
			
			Taking supremum over $(x_1,y_1)$ and infimum over $(x_2,y_2)$ in \eqref{e:naim1}, % over $\left(B(\eta,R/A) \cap(\unifdom \setminus \{x_0\}) \right) \times \left(B(z,r) \cap(\unifdom \setminus \{x_0\}) \right)$
			we obtain
			\begin{equation} \label{e:naim5}
				M(R)-m(R/A) \le C_{\on{BHP}}^2  ( M(R)- M(R/A)).
			\end{equation}
			By considering $(x,y) \mapsto \naimker^{\unifdom}_{x_0}(x,y)-m(R)=\frac{\gren{\unifdom}(x,y)-m(R)\gren{\unifdom}(x_0,x) \gren{\unifdom}(x_0,y)}{\gren{\unifdom}(x_0,x) \gren{\unifdom}(x_0,y)}$ and using a similar argument as the proof of \eqref{e:naim2}, we obtain
			\begin{equation} \label{e:naim6}
				M(R/A)-m(R) \le C_{\on{BHP}}^2  ( m(R/A)- m(R)).
			\end{equation}
			Combining \eqref{e:naim2} and \eqref{e:naim3}, we obtain
			\begin{equation*}
				M(R/A)-m(R/A) \le \frac{ C_{\on{BHP}}^2 -1}{ C_{\on{BHP}}^2 +1} ( M(R)-m(R) ).
			\end{equation*}
			Iterating the above estimate, we obtain (a) with $\gamma= (\log A)^{-1}\log \frac{ C_{\on{BHP}}^2 +1}{ C_{\on{BHP}}^2 -1}$.
			\item Let $(x,y) \in (B(\eta,R) \cap(\unifdom \setminus \{x_0\}) ) \times (B(z,R) \cap(\unifdom \setminus \{x_0\}) )$,
			where $\eta,z,R$ are as given in the statement of the lemma. Then by applying \ref{eq:BHP} to the $\form$-harmonic functions
			$\gren{\unifdom}(\cdot,y)$ and $\gren{\unifdom}(x_0,\cdot)$ on $\unifdom \cap B(\eta,AR)$ and \hyperlink{ehi}{$\on{EHI}$}
			to the $\form$-harmonic functions $\gren{\unifdom}(\eta_{R/2},\cdot)$ and $\gren{\unifdom}(x_0,\cdot)$ on $B(z,AR)$, we obtain
			\begin{equation*}
				\naimker^\unifdom_{x_0}(x,y) \le C_{\on{BHP}} \frac{\gren{\unifdom}(\eta_{R/2},y)}{\gren{\unifdom}(x_0,\eta_{R/2}) \gren{\unifdom}(x_0,y)} \le  C_{\on{BHP}}C_{\on{EHI}}^2\frac{\gren{\unifdom}(\eta_{R/2},z)}{\gren{\unifdom}(x_0,\eta_{R/2}) \gren{\unifdom}(x_0,z)}.
			\end{equation*}
			This proves the first estimate, and the second one also follows from a similar argument.
			\item Let $(x,y) \in (B(\eta,R) \cap(\unifdom \setminus \{x_0\}) ) \times (B(\xi,R) \cap(\unifdom \setminus \{x_0\}) )$, where $\eta,\xi,R$ as given.
			Then by using \ref{eq:BHP} for the $\form$-harmonic functions $\gren{\unifdom}(\cdot,y)$ and $\gren{\unifdom}(x_0,\cdot)$ on $\unifdom \cap B(\eta,AR)$
			and for the $\form$-harmonic functions $\gren{\unifdom}(\eta_{R/2},\cdot)$ and $\gren{\unifdom}(x_0,\cdot)$ on $\unifdom \cap B(\xi,AR)$, we deduce
			\begin{equation*}
				\naimker^\unifdom_{x_0}(x,y) \le C_{\on{BHP}} \frac{\gren{\unifdom}(\eta_{R/2},y)}{\gren{\unifdom}(x_0,\eta_{R/2}) \gren{\unifdom}(x_0,y)} \le  C_{\on{BHP}}^2\frac{\gren{\unifdom}(\eta_{R/2},\xi_{R/2})}{\gren{\unifdom}(x_0,\eta_{R/2}) \gren{\unifdom}(x_0,\xi_{R/2})}
			\end{equation*}
			and
			\begin{equation*}
				\naimker^\unifdom_{x_0}(x,y) \ge C_{\on{BHP}}^{-1} \frac{\gren{\unifdom}(\eta_{R/2},y)}{\gren{\unifdom}(x_0,\eta_{R/2}) \gren{\unifdom}(x_0,y)} \ge  C_{\on{BHP}}^{-2}\frac{\gren{\unifdom}(\eta_{R/2},\xi_{R/2})}{\gren{\unifdom}(x_0,\eta_{R/2}) \gren{\unifdom}(x_0,\xi_{R/2})}.
			\end{equation*}
			\item By \ref{eq:BHP} applied to the $\form$-harmonic functions $\gren{\unifdom}(x,\cdot)$
			and $\gren{\unifdom}(x_0,x)\gren{\unifdom}(x_0,\cdot)$ on $\unifdom \cap B(\xi,AR)$ we obtain \eqref{e:naimbnd}.
			By Lemma \ref{l:bhpholder}, we have
			\begin{equation*}
				\osc_{y \in \unifdom \cap B(\xi,r)} \naimker^\unifdom_{x_0}(x,y)
				\le C_0 \Bigl(\frac{r}{R}\Bigr)^\gamma \osc_{y \in \unifdom \cap B(\xi,R)} \naimker^\unifdom_{x_0}(x,y)
				\le C_0 \Bigl(\frac{r}{R}\Bigr)^\gamma \sup_{y \in \unifdom \cap B(\xi,R)} \naimker^\unifdom_{x_0}(x,y),
			\end{equation*}
			which together with \eqref{e:naimbnd} yields \eqref{e:naimequi}.
			\qedhere\end{enumerate}
	\end{proof}
	
	Thanks to the H\"older regularity estimates obtained in Lemma \ref{l:naimholder},
	we can extended $\naimker^{\unifdom}_{x_0}$ to $\offdiagp{\ol{\unifdom} \setminus \{x_0\}}$ as shown below.
	
	\begin{prop} \label{p:naim}
		Let an MMD space $(\ambient,d,\refmeas,\form,\domain)$ and a uniform domain $\unifdom$ in $(\ambient,d)$ satisfy Assumption \ref{a:ehi-bhp}.
		Let $x_0 \in \unifdom$. Then the function $\naimker^\unifdom_{x_0}(\cdot,\cdot)$ defined in \eqref{e:defnaim} has a continuous extension to $\offdiagp{\ol{\unifdom} \setminus \{x_0\}}$, which is again denoted by $\naimker^\unifdom_{x_0} \colon \offdiagp{\ol{\unifdom} \setminus \{x_0\}} \to [0,\infty)$, and there exist $C_1,C_2,A_1 \in (1,\infty)$, $c_0 \in (0,1/4)$ and $\gamma \in (0,\infty)$ depending only on the constants associated with Assumption \ref{a:ehi-bhp} such that the following hold:
		\begin{align} \label{e:naimest}
			C_1^{-1} \frac{\gren{\unifdom}(\xi_r,\eta_r)}{\gren{\unifdom}(x_0,\xi_r) \gren{\unifdom}(x_0,\eta_r)} \le \naimker^\unifdom_{x_0} (\xi,\eta) \le C_1 \frac{\gren{\unifdom}(\xi_r,\eta_r)}{\gren{\unifdom}(x_0,\xi_r) \gren{\unifdom}(x_0,\eta_r)}
		\end{align}
		for all $(\xi,\eta) \in \offdiagp{\partial \unifdom}$ and all $0 < r \leq c_0 ( d(x_0,\xi) \wedge d(x_0,\eta) \wedge d(\xi,\eta) )$, and
		\begin{equation} \label{e:naimholder}
			\abs{\naimker^\unifdom_{x_0}(\xi,\eta) - \naimker^\unifdom_{x_0}(x,y)} \le C_2 \naimker^\unifdom_{x_0}(\xi,\eta) \biggl( \frac{d(\xi,x)^\gamma}{R^\gamma}+ \frac{d(\eta,y)^\gamma}{R^\gamma} \biggr)
		\end{equation}
		for all $(\xi,\eta) \in \offdiagp{\partial \unifdom}$, $0 < R < (2A_{1})^{-1} ( d(x_0,\xi) \wedge d(x_0,\eta) \wedge d(\xi,\eta) )$, $x \in \overline{\unifdom} \cap B(\xi,R)$ and $y \in \overline{\unifdom} \cap B(\eta,R)$.
		Furthermore $\naimker^\unifdom_{x_0}(\xi,\eta) = \naimker^\unifdom_{x_0}(\eta,\xi)$ for all $(\xi,\eta) \in \offdiagp{\overline{\unifdom} \setminus \{x_0\}}$.
	\end{prop}
	
	\begin{proof}
		The existence of a  continuous extension to $\offdiagp{\ol{\unifdom} \setminus \{x_0\}}$ of the function defined in \eqref{e:defnaim} follows from Lemma \ref{l:naimholder}. More precisely, the existence of a continuous extension at all points in $\partial \unifdom \times (\unifdom \setminus \{x_0\}) $ and $(\unifdom \setminus \{x_0\}) \times \partial \unifdom$ follows from Lemma \ref{l:naimholder}(a,c) along with the symmetry of Green function.
		On the other hand, the existence of a continuous extension at all points in $\offdiagp{\partial \unifdom}$ follows from Lemma \ref{l:naimholder}(b,d).
		
		The estimates \eqref{e:naimest} and \eqref{e:naimholder} are direct consequences of Lemma \ref{l:naimholder}(b,d). 
		The symmetry of $\naimker^\unifdom_{x_0}$ follows from the symmetry of $\gren{\unifdom}$ and the continuity of $\naimker^\unifdom_{x_0}$.
	\end{proof}
	
	\begin{definition} \label{dfn:Naim-kernel}
		Let an MMD space $(\ambient,d,\refmeas,\form,\domain)$ and a uniform domain $\unifdom$ in $(\ambient,d)$ satisfy Assumption \ref{a:ehi-bhp}.
		The function $\naimker^\unifdom_{x_0} \colon \offdiagp{\ol{\unifdom} \setminus \{x_0\}} \to [0,\infty)$ defined as the continuous extension of \eqref{e:defnaim} is called the \textbf{Na\"im kernel} of the domain $\unifdom$ with base point $x_0 \in \unifdom$.
	\end{definition}
	
	This function is essentially same as the one introduced by L.~Na\"im in \cite{Nai} where  she extends to function considered in \eqref{e:defnaim} to the Martin boundary instead of the topological boundary as considered above. Another difference from \cite{Nai} is the use of Martin topology and   fine topology of H.~Cartan instead of the topology arising from the metric. 
	
	\subsection{Martin kernel}
	
	We recall the definition of the closely related \emph{Martin kernel} introduced by R.~S.~Martin \cite{Mar}.
	\begin{definition}
		Let an MMD space $(\ambient,d,\refmeas,\form,\domain)$ and a uniform domain $\unifdom$ in $(\ambient,d)$ satisfy Assumption \ref{a:ehi-bhp}.
		Let $x_0 \in \unifdom$. We define $\martinker^\unifdom_{x_0} \colon \unifdom \times (\overline{\unifdom} \setminus \{x_0\}) \setminus \unifdom_{\diag} \to [0,\infty)$ by
		\begin{equation} \label{e:defMartin}
			\martinker^\unifdom_{x_0}(x,\xi) \coloneqq 
			\begin{cases}
				\displaystyle \frac{\gren{\unifdom}(x,\xi)}{\gren{\unifdom}(x_0,\xi)}  & \textrm{if $\xi \in U \setminus \{ x_0, x \}$},\\
				\displaystyle \lim_{\unifdom \ni y \to \xi} \frac{\gren{\unifdom}(x,y)}{\gren{\unifdom}(x_0,y)}  & \textrm{if $\xi \in  \partial \unifdom$,}
			\end{cases}
		\end{equation} 
		where the limit in the second case exists by \ref{eq:BHP} and Lemmas \ref{l:dbdy} and \ref{l:bhpholder}.
		The function $\martinker^\unifdom_{x_0}$ is called the \textbf{Martin kernel} of $\unifdom$ with base point $x_0$.
	\end{definition}
	
	The following oscillation lemma is an analogue of Lemma \ref{l:naimholder}.
	
	\begin{lem} \label{l:martinholder}
		Let an MMD space $(\ambient,d,\refmeas,\form,\domain)$ and a uniform domain $\unifdom$ in $(\ambient,d)$ satisfy Assumption \ref{a:ehi-bhp}.
		Then there exist $C,A \in (1,\infty)$ and $\gamma \in (0,\infty)$ such that the following estimates hold for any $x_0 \in \unifdom$:
		\begin{enumerate}[\rm(a)]\setlength{\itemsep}{0pt}\vspace{-5pt}
			\item For any $z \in \unifdom$, any $\xi \in \partial \unifdom$ and any $0 <r<R < (2A)^{-1}(\delta_\unifdom(z) \wedge d(x_0,\xi))$,
			\begin{equation} \label{e:martinholder}
				\osc_{(\unifdom \cap B(z,r)) \times (\overline{\unifdom} \cap B(\xi,r)) } \martinker^\unifdom_{x_0}(\cdot,\cdot) \le C \Bigl( \frac{r}{R}\Bigr)^{\gamma} \osc_{(\unifdom \cap B(z,R)) \times (\overline{\unifdom} \cap B(\xi,R)) } \martinker^\unifdom_{x_0}(\cdot,\cdot).
			\end{equation}
			\item For any $z \in \unifdom$, any $\xi \in \partial \unifdom$ and any $0 <r<R < (2A)^{-1}(\delta_\unifdom(z) \wedge d(x_0,\xi))$,
			\begin{equation} \label{e:martin-carleson-1}
				\sup_{(\unifdom \cap B(z,R)) \times (\overline{\unifdom} \cap B(\xi,R)) } \martinker^\unifdom_{x_0}(\cdot,\cdot) \le C \martinker^\unifdom_{x_0}(z,\xi_{R/2}).
			\end{equation}
			\item For any $(\eta,\xi) \in \offdiagp{\partial \unifdom}$ and any $0<r<R< (2A)^{-1} (d(\xi,x_0)\wedge d(\eta,x_0)\wedge d(\xi,\eta))$,
			\begin{equation} \label{e:martin-carleson-2}
				\sup_{x \in \unifdom \cap B(\eta,R)} \osc_{y \in \ol{\unifdom} \cap B(\xi,r)} \martinker^\unifdom_{x_0}(x,y) \le C \Bigl( \frac{r}{R}\Bigr)^\gamma \martinker^\unifdom_{x_0}(\eta_{R/2}, \xi_{R/2}).
			\end{equation}
		\end{enumerate}
	\end{lem}
	
	\begin{proof}
		We omit the proofs of (a) and (b) as they are similar to that of Lemma \ref{l:bhpholder}.
		Both estimates follow from applying \hyperlink{ehi}{$\on{EHI}$} and \ref{eq:BHP}
		to the first and second arguments respectively of the Martin kernel.
		
		\noindent(c) By Lemma \ref{l:bhpholder} 
		\begin{equation*}
			\osc_{y \in \ol{\unifdom} \cap B(\xi,r)} \martinker^\unifdom_{x_0}(x,y) \lesssim  \Bigl( \frac{r}{R}\Bigr)^\gamma \osc_{y \in \ol{\unifdom} \cap B(\xi,R)} \martinker^\unifdom_{x_0}(x,y) \lesssim  \Bigl( \frac{r}{R}\Bigr)^\gamma \martinker^\unifdom_{x_0}(x,\xi_{R/2}) 
		\end{equation*}
		for all $x \in  \unifdom \cap B(\eta,R)$. By Carleson's estimate (Proposition \ref{p:carleson}), we have
		\begin{equation*}
			\sup_{x \in \unifdom \cap B(\eta,R)} \martinker^\unifdom_{x_0}(x,\xi_{R/2}) \lesssim  \martinker^\unifdom_{x_0}(\eta_{R/2}, \xi_{R/2}).
		\end{equation*}
		Combining the above two estimates, we obtain the desired result. 
	\end{proof}
	
	We discuss the $\form$-harmonicity and Dirichlet boundary condition of the Martin kernel $\martinker^\unifdom_{x_0}(\cdot,\xi)$, where $\xi \in \partial \unifdom$.
	
	\begin{lem} \label{l:martinharm}
		Let an MMD space $(\ambient,d,\refmeas,\form,\domain)$ and a uniform domain $\unifdom$ in $(\ambient,d)$ satisfy Assumption \ref{a:ehi-bhp}.
		For all $\xi \in \partial \unifdom$, the function $\martinker^\unifdom_{x_0}(\cdot,\xi)\colon \unifdom \to [0,\infty)$ belongs to $\domain_{\loc}(\unifdom)$ and is $\form$-harmonic on $\unifdom$.
		Furthermore $\martinker^\unifdom_{x_0}(\cdot,\xi)$ satisfies Dirichlet boundary condition relative to $\unifdom$ off $\xi$ in the following sense: for any open  subset $V$ of $\unifdom$ such that $\xi \notin \overline{V}$, $\martinker^\unifdom_{x_0}(\cdot,\xi) \in \domain_{\on{loc}}^0(\unifdom,V)$.
	\end{lem}
	
	\begin{proof}
		Let $y_n \in \unifdom$ be a sequence with $\lim_{n \to \infty} y_n = \xi$.
		Define $h_n \colon \unifdom \setminus \{y_n\} \to [0,\infty)$ as $h_n \coloneqq \martinker^\unifdom_{x_0}(\cdot,y_n)$ for all $n \ge 1$. 
		
		If $K \subset \unifdom$ is compact then $K \subset \unifdom \setminus \{y_n\}$ for all but finitely many $n$.
		By Lemma \ref{l:martinholder}-(a),(b), the sequence $h_n$ converges uniformly on compact subsets of $\unifdom$
		and is bounded on compact sets. Therefore by Proposition \ref{p:goodgreen}-\eqref{it:goodgreen-harm} and Lemma \ref{l:harm-conv},
		the function $\martinker^\unifdom_{x_0}(\cdot,\xi)\colon\unifdom \to [0,\infty)$ belongs to $\domain_{\on{loc}}(\unifdom)$ and is $\form$-harmonic in $\unifdom$.
		
		Let $V$ be an open subset of $\unifdom$ such that $\xi \notin \overline{V}$ and let $A \subset V$
		be relatively compact in $\overline{\unifdom}$ with $\overline{A} \cap \overline{\unifdom \setminus V} = \emptyset$.
		Then by Lemma \ref{l:martinholder}-(c), $h_n$ converges uniformly to $\martinker^\unifdom_{x_0}(\cdot,\xi)$ on $A$.
		Therefore by Lemma \ref{l:harm-conv}-\eqref{it:harm-conv-dbdry},
		$\martinker^\unifdom_{x_0}(\cdot,\xi) \in \domain_{\on{loc}}^0(\unifdom,V)$.
	\end{proof}
	
	Next, we relate the Martin and Na\"im kernels. Due to Lemma \ref{l:martinharm} and the continuity of
	$\naimker^\unifdom_{x_0}$, the Na\"im kernel can be expressed in terms of the Martin kernel as
	\begin{equation} \label{e:naimb}
		\naimker^\unifdom_{x_0}(x,y) =
		\begin{cases}
			\displaystyle \frac{\martinker^\unifdom_{x_0}(x,y)}{\gren{\unifdom}(x_0,x)}  & \textrm{if $x \in \unifdom$,} \\[10pt]
			\displaystyle \lim_{\unifdom \ni z \to x} \frac{\martinker^\unifdom_{x_0}(z,y)}{\gren{\unifdom}(x_0,z)}  & \textrm{if $x \in \partial \unifdom$,}
		\end{cases}
	\end{equation}
	where the limit in the second case exists by \ref{eq:BHP} and Lemmas \ref{l:martinharm} and \ref{l:bhpholder}.
	We chose the approach based on Lemma \ref{l:naimholder} because the symmetry of $\naimker^\unifdom_{x_0}$ and the joint continuity
	are immediate through our approach while these properties need to be shown if we use \eqref{e:naimb}.
	The equality \eqref{e:naimb} is closer to the original approach to \emph{define} Na\"im kernel
	as the extension to the boundary is done for one variable at a time in \cite{Nai}.
	
	It is well known that any unbounded domain satisfying the boundary Harnack principle
	has a unique Martin kernel point at infinity. Following \cite[Chapter 4]{GS},
	we call the Martin kernel point at infinity the \textbf{$\form$-harmonic profile} of $\unifdom$.
	We recall the short argument to prove its uniqueness.
	
	\begin{lem}[Uniqueness of harmonic profile] \label{l:uniqueprofile}
		Let $(\ambient,d,\refmeas,\form,\domain)$ be an MMD space and let $\unifdom$ be an unbounded
		open subset of $\ambient$ satisfying \ref{eq:BHP} and $\partial \unifdom \not=\emptyset$.
		Let $h_1\colon \unifdom \to [0,\infty)$ and $h_2\colon \unifdom \to (0,\infty)$
		be two continuous functions such that $h_1,h_2 \in \domain_{\on{loc}}^0(\unifdom,\unifdom)$
		and $h_1,h_2$ are $\form$-harmonic on $\unifdom$. Then there exists $c\in [0,\infty)$
		such that $h_1(x)=ch_2(x)$ for all $x \in \unifdom$.
	\end{lem}
	
	\begin{proof}
		Let $A \in (1,\infty)$ be the largest among the constants $A_0, A_1$ in Definition \ref{d:bhp} and Lemma \ref{l:bhpholder}.
		Let $C$ be the largest among the constants $C_1, C_0$ in Definition \ref{d:bhp} and Lemma \ref{l:bhpholder} respectively.
		Let $\gamma$ be as given in Lemma \ref{l:bhpholder}.
		
		Let $\xi \in \partial \unifdom$ and $x_0 \in \unifdom$.
		For all $R \in (d(\xi,x_0),\infty)$, by Definition \ref{d:bhp} we have
		\begin{equation*}
			\sup_{B(\xi,R) \cap \unifdom} \frac{h_1(\cdot)}{h_2(\cdot)} \le C \frac{h_1(x_0)}{h_2(x_0)}.
		\end{equation*}
		Letting $R \to \infty$, we obtain 
		\begin{equation*}
			\osc_{\unifdom} \frac{h_1(\cdot)}{h_2(\cdot)} \le \sup_{\unifdom} \frac{h_1(\cdot)}{h_2(\cdot)} \le C \frac{h_1(x_0)}{h_2(x_0)}.
		\end{equation*}
		For any $d(\xi,x_0) < r < R <\infty$, by Lemma \ref{l:bhpholder} we have
		\begin{equation*}
			\osc_{B(\xi,r) \cap \unifdom} \frac{h_1(\cdot)}{h_2(\cdot)}
			\le C \Bigl(\frac{r}{R}\Bigr)^\gamma \osc_{B(\xi,R) \cap \unifdom} \frac{h_1(\cdot)}{h_2(\cdot)}
			\le C \Bigl(\frac{r}{R}\Bigr)^\gamma \sup_\unifdom \frac{h_1(\cdot)}{h_2(\cdot)}
			\le C^2 \Bigl(\frac{r}{R}\Bigr)^\gamma \frac{h_1(x_0)}{h_2(x_0)}.
		\end{equation*}
		Letting $R \to \infty$, we obtain $\osc_{B(\xi,r) \cap \unifdom} \frac{h_1(\cdot)}{h_2(\cdot)}=0$ for any
		$r \in (d(\xi,x_0),\infty)$. Letting $r \to \infty$, we obtain $\osc_{\unifdom} \frac{h_1(\cdot)}{h_2(\cdot)}=0$.
	\end{proof}
	
	We recall a standard construction of the harmonic profile \cite[Chapter 4]{GS}.
	
	\begin{prop}[Existence of harmonic profile] \label{p:hprofile}
		Let an MMD space $(\ambient,d,\refmeas,\form,\domain)$ and a uniform domain $\unifdom$ in $(\ambient,d)$ satisfy Assumption \ref{a:ehi-bhp},
		and assume that $\unifdom$ is unbounded. Then for any $x_0 \in \unifdom$ and a sequence $\{y_n\}_{n \in \mathbb{N}}$ in $\unifdom$ such that
		$\lim_{n \to \infty} d(x_0,y_n)=\infty$, the sequence $\martinker^\unifdom_{x_0}(\cdot,y_n) \colon \unifdom \setminus \{y_n\} \to (0,\infty)$
		converges uniformly on any bounded subset of $\unifdom$ to a continuous function $\hprof{x_0} \colon \unifdom \to (0,\infty)$
		such that $\hprof{x_0} \in \domain_{\on{loc}}^0(\unifdom,\unifdom)$, $\hprof{x_0}(x_0)=1$,
		$\hprof{x_0}$ is bounded on any bounded subset of $\unifdom$ and is $\form$-harmonic on $\unifdom$.
		Furthermore, the limit $\hprof{x_0}$ depends only on $\unifdom, x_0$ and not on the sequence $\{y_n\}_{n \in \mathbb{N}}$.
	\end{prop}
	
	\begin{proof}
		Let $A \in (1,\infty)$ be the largest among the constants $A_0, A_1$ in Definition \ref{d:bhp} and Lemma \ref{l:bhpholder}.
		Let $C$ be the largest among the constants $C_1, C_0$ in Definition \ref{d:bhp} and Lemma \ref{l:bhpholder} respectively.
		Let $\gamma$ be as given in Lemma \ref{l:bhpholder}.
		
		Let $\xi \in \partial \unifdom$ and let $A d(x_0,\xi)< r <R$. Then for any $n,k \in \bN$ such that $AR< d(\xi,y_n) \wedge d(\xi,y_k)$, by Lemma \ref{l:bhpholder} and Definition \ref{d:bhp} we estimate 
		\begin{align*}
			\sup_{\unifdom \cap B(\xi,r)} \abs{\frac{\martinker^\unifdom_{x_0}(\cdot,y_n)}{\martinker^\unifdom_{x_0}(\cdot,y_k)}-1} &=    \sup_{\unifdom \cap B(\xi,r)} \abs{\frac{\martinker^\unifdom_{x_0}(\cdot,y_n)}{\martinker^\unifdom_{x_0}(\cdot,y_k)}-\frac{\martinker^\unifdom_{x_0}(x_0,y_n)}{\martinker^\unifdom_{x_0}(x_0,y_k)} }
			\le 
			\osc_{\unifdom \cap B(\xi,r)} \frac{\martinker^\unifdom_{x_0}(\cdot,y_n)}{\martinker^\unifdom_{x_0}(\cdot,y_k)} \nonumber \\   &\le C \Bigl(\frac{r}{R}\Bigr)^\gamma  \osc_{\unifdom \cap B(\xi,r)}  \frac{\martinker^\unifdom_{x_0}(\cdot,y_n)}{\martinker^\unifdom_{x_0}(\cdot,y_k)} \nonumber \\
			&\le  C \Bigl(\frac{r}{R}\Bigr)^\gamma  \sup_{\unifdom \cap B(\xi,r)}  \frac{\martinker^\unifdom_{x_0}(\cdot,y_n)}{\martinker^\unifdom_{x_0}(\cdot,y_k)} \nonumber\\
			&\le C^2 \Bigl(\frac{r}{R}\Bigr)^\gamma \frac{\martinker^\unifdom_{x_0}(x_0,y_n)}{\martinker^\unifdom_{x_0}(x_0,y_k)}= C^2 \Bigl(\frac{r}{R}\Bigr)^\gamma
		\end{align*}
		By letting $R=(2A)^{-1} (d(\xi,y_n) \wedge d(\xi,y_k))$, we obtain that for all $n,k$ such that $d(\xi,y_n)\wedge d(\xi,y_k) > 2A^2 d(\xi,x_0)$, we have 
		\begin{equation} \label{e:hp1}
			\sup_{\unifdom \cap B(\xi,r)}  \abs{\frac{\martinker^\unifdom_{x_0}(\cdot,y_n)}{\martinker^\unifdom_{x_0}(\cdot,y_k)}-1} \le C^2 (2A)^\gamma r^{\gamma} (d(\xi,y_n) \wedge d(\xi,y_k))^{-\gamma}.
		\end{equation}
		By Carelson's estimate (Proposition \ref{p:carleson}) for any $\xi \in \partial \unifdom, r>0$, there exist  $C_1>0, N \in \bN$ such that 
		\begin{equation} \label{e:hp2}
			\sup_{\unifdom \cap B(\xi,r)} \martinker^{\unifdom}_{x_0} (\cdot,y_n) \lesssim  \martinker^{\unifdom}_{x_0} (\xi_{r/2},y_n) \quad \mbox{for all $n \ge N$}.
		\end{equation}
		By Harnack chaining along a uniform curve in $\unifdom$ between $\xi_{r/2}$ and $x_0$ and using \eqref{e:hchain}, there exist $N \in \bN$, $C_2=C_2(x_0,\xi,r)$ such that
		\begin{equation} \label{e:hp3}
			\martinker^\unifdom_{x_0}(\xi_{r/2},y_n) \le C_2 \quad \mbox{for all $n \ge N$.}
		\end{equation}
		Combining \eqref{e:hp1}, \eqref{e:hp2}, and \eqref{e:hp3}, we obtain
		\begin{align*}
			\lim_{n,k \to \infty}\sup_{\unifdom \cap B(\xi,r)}  \abs{\martinker^\unifdom_{x_0}(\cdot,y_n) - \martinker^\unifdom_{x_0}(\cdot,y_k) }  \le   \lim_{n,k \to \infty}    C_1 C_2  C^2 (2A)^\gamma r^{\gamma} (d(\xi,y_n) \wedge d(\xi,y_k))^{-\gamma}=0.
		\end{align*}
		Since $r\in(0,\infty)$ is arbitrary, letting $r \to \infty$, we conclude that the sequence $\{\martinker^{\unifdom}_{x_0}(\cdot,y_n)\}_{n \in \mathbb{N}}$
		converges uniformly on any bounded subset of $\unifdom$ to some $\hprof{x_0}\colon \unifdom \to (0,\infty)$,
		then $\hprof{x_0}$ is continuous by the continuity of $\martinker^{\unifdom}_{x_0}(\cdot,y_n)$
		and bounded on any bounded subset of $\unifdom$ by \eqref{e:hp2}, and Lemma \ref{l:harm-conv}
		implies that $\hprof{x_0}\in \domain_{\on{loc}}^0(\unifdom,\unifdom)$
		and that $\hprof{x_0}$ is $\form$-harmonic on $\unifdom$.
		
		The assertion that the limit $\hprof{x_0}$ depends only on $\unifdom,x_0$
		follows from $\hprof{x_0}(x_0)=1$ and Lemma \ref{l:uniqueprofile}.
	\end{proof}
	
	\section{Estimates for harmonic and elliptic measures} \label{sec:harm-ell-meas}
	
	To goal of this section is to estimate the harmonic measure of balls on the boundary of a uniform domain using ratio of Green functions. We restrict to the class of uniform domains that satisfy the following \textbf{capacity density condition}. 
	
	\subsection{The capacity density condition} \label{ssec:CDC}
	
	This is a slight variant of similar conditions considered in \cite{Anc,AH}.
	
	\begin{definition}[Capacity density condition (CDC)] \label{d:cdc}
		Let $(\ambient,d,\refmeas,\form,\domain)$ be an MMD space satisfying \hyperlink{MD}{\textup{MD}}
		and \hyperlink{ehi}{$\on{EHI}$}. Recalling Lemma \ref{l:chain}-\eqref{it:EHI-MD-RBC}, let
		$K \in (1,\infty)$ be such that $(\ambient,d)$ is $K$-relatively ball connected. 
		We say that a uniform domain $\unifdom$ in $(\ambient,d)$ satisfies the
		\textbf{capacity density condition}, abbreviated as \textup{CDC},
		if there exist $A_0 \in (8K,\infty)$ and $A_1, C \in (1,\infty)$
		such that for all $\xi \in \partial \unifdom$ and all $R \in (0,\diam(\unifdom)/A_1)$,
		\begin{equation} \tag*{\textup{CDC}} \label{eq:CDC}
			\Capa_{B(\xi,A_0 R)}(B(\xi,R)) \le C \Capa_{B(\xi,A_0 R)}(B(\xi,R)\setminus \unifdom).
		\end{equation}
	\end{definition}
	
	We note that the capacity density condition implies the transience of $(\form^{\unifdom},\domain^{0}(\unifdom))$.
	
	\begin{remark} \label{r:transient}
		Let $(\ambient,d,\refmeas,\form,\domain)$ and $K$ be as in Definition \ref{d:cdc},
		and let $\unifdom$ be a uniform domain in $(\ambient,d)$ satisfying \ref{eq:CDC}. Then
		$\ambient \setminus \unifdom$ is not $\form$-polar by Remark \ref{rmk:unifdom-bdry-non-empty-diam-pos}
		and \cite[Theorems 2.1.6 and 4.4.3-(ii)]{FOT}, and hence the part Dirichlet form
		$(\form^{\unifdom},\domain^{0}(\unifdom))$ on $\unifdom$ is transient by \cite[Theorem 4.8 and Proposition 2.1]{BCM}.
	\end{remark}
	
	Due to Remark \ref{r:ehi-hke}, it would be convenient to assume \hyperlink{VD}{\textup{VD}} and \hyperlink{hke}{$\on{HKE}(\scdiff)$} instead,
	under which \hyperlink{MD}{\textup{MD}} and \hyperlink{ehi}{$\on{EHI}$} hold and Definition \ref{d:cdc} thus remains applicable.
	Therefore, we make the following assumption. 
	
	\begin{assumption} \label{a:hkecdcbhp}
		Let a scale function $\scdiff$, an MMD space $(\ambient,d,\refmeas,\form,\domain)$
		and a diffusion $\diff = (\Omega, \events, \{\diff_{t}\}_{t\in[0,\infty]},\{\lawdiff_{x}\}_{x \in \oneptcpt{\ambient}})$
		on $\ambient$ satisfy Assumption \ref{a:feller}.
%		In particular, by Remark \ref{r:ehi-hke} and
%		Lemma \ref{l:chain}-\eqref{it:EHI-MD-RBC}, $(\ambient,d)$ is $K$-relatively ball connected for some $K \in (1,\infty)$.
		Let $\unifdom$ be a uniform domain in $(\ambient,d)$ satisfying \ref{eq:CDC},
		set $\oneptcpt{\overline{\unifdom}} \coloneqq \overline{\unifdom}\cup\{\cemetery\}$,
		$\formref \coloneqq \formrefgen{\unifdom}$ and, recalling Theorem \ref{thm:hkeunif}-\eqref{it:hkeunif}, let
		$\diffref=\bigl(\Omega^{\on{ref}},\events^{\on{ref}},\{\diffref_{t}\}_{t\in[0,\infty]},\{\lawref_{x}\}_{x\in\oneptcpt{\overline{\unifdom}}}\bigr)$
		be a diffusion on $\overline{\unifdom}$ as in Assumption \ref{a:feller}
		for the MMD space $(\overline{\unifdom},d,\refmeas|_{\overline{\unifdom}},\formref,\domain(\unifdom))$.
		
		Then Assumption \ref{a:ehi-bhp} holds by Remarks \ref{r:ehi-hke} and \ref{r:transient}.
		In particular, $\unifdom$ satisfies \ref{eq:BHP} by Theorem \ref{t:bhp}, and
		$\partial \unifdom\not=\emptyset$ and $\diam(\unifdom)\in(0,\infty]$ by Remark \ref{rmk:unifdom-bdry-non-empty-diam-pos}.
	\end{assumption}
	
	Ancona \cite[Definition 2 and Lemma 3]{Anc} showed that the capacity density condition
	\ref{eq:CDC} in a Euclidean domain is equivalent to an estimate on the harmonic measure
	called the \emph{uniform $\Delta$-regularity}. Such a result can be extended to an arbitrary
	open set in any MMD space satisfying \hyperlink{MD}{\textup{MD}} and \hyperlink{ehi}{$\on{EHI}$}
	by using the estimates on hitting probabilities from \cite{BM18,BCM}.
	More precisely, we have the following relationships between hitting probabilities
	and \ref{eq:CDC}. Part \eqref{it:cdc1-anyA0} of the lemma below is meant to 
	justify our requirement $A_0 \in (8K,\infty)$ in Definition \ref{d:cdc}.
	
	\begin{lem} \label{l:cdc1}
		Let $(\ambient,d,\refmeas,\form,\domain)$ be an MMD space, and let $ D $ be an open subset of $\ambient$.
		\begin{enumerate}[\rm(a)]\setlength{\itemsep}{0pt}\vspace{-5pt}
			\item\label{it:hmeas-cdc1}Let $A_0,A_1 \in (1,\infty)$, $\gamma \in (0,1)$ and assume that for each $\xi \in \partial D$ and each $R \in (0,\diam(D)/A_1)$,
			\begin{equation} \label{e:delta1}
				\hmeas{D \cap B(\xi,A_0 R)}{x} (D \cap S(\xi,A_0 R)) \le 1 -\gamma \quad \textrm{for $\form$-q.e.\ $x\in B(\xi, R) \cap D$.}
			\end{equation}
			Then for all $\xi \in \partial D$ and all $R \in (0,\diam(D)/A_1)$,
			\begin{equation} \label{e:cdc1}
				\Capa_{B(\xi,A_0 R)}(B(\xi,R)) \le \gamma^{-2}\Capa_{B(\xi,A_0 R)}(B(\xi,R)\setminus D).
			\end{equation}
			\item\label{it:cdc1-consequences} Assume that $(\ambient,d,\refmeas,\form,\domain)$ satisfies \hyperlink{MD}{\textup{MD}}
			and \hyperlink{ehi}{$\on{EHI}$} and, recalling Lemma \ref{l:chain}-\eqref{it:EHI-MD-RBC},
			let $K \in(1,\infty)$ be such that $(\ambient,d)$ is $K$-relatively ball connected.
			Suppose that there exist $A_0 \in (8K,\infty)$ and $A_1, C \in (1,\infty)$ such that
			for all $\xi \in \partial D$ and all $R \in (0,\diam(D)/A_1)$,
			\begin{equation} \label{e:cdc3}
				\Capa_{B(\xi, A_0 R)}(B(\xi, R)) \le C \Capa_{B(\xi, A_0 R)}(B(\xi, R) \setminus D).
			\end{equation}
			Then the following hold:
			\begin{enumerate}[\rm(1)]\setlength{\itemsep}{0pt}\vspace{-5pt}
				\item\label{it:cdc1-anyA0}For any $\wh{A_0} \in (1,\infty)$, there exist $\wh{A_1},\wh{C} \in (1,\infty)$
				such that for all $\xi \in \partial D$ and all $R \in (0,\diam(D)/\wh{A_1})$,
				\begin{equation} \label{e:cdc4}
					\Capa_{B(\xi,\wh{A_0} R)}(B(\xi, R)) \le \wh{C} \Capa_{B(\xi, \wh{A_0}R)}(B(\xi,R) \setminus D).
				\end{equation}
				\item\label{it:cdc1-hmeas}There exist $\wh{A_0},\wh{A_1} \in (1,\infty)$ and $\gamma \in (0,1)$ such that
				for each $\xi \in \partial D$ and each $R \in (0,\diam(D)/\wh{A_1})$,
				\begin{equation} \label{e:delta2}
					\hmeas{D \cap B(\xi,\wh{A_0} R)}{x} (D \cap S(\xi,\wh{A_0} R)) \le 1 -\gamma
					\quad \textrm{for $\form$-q.e.\ $x\in B(\xi, R) \cap D$.}
				\end{equation}
				If in addition $(\ambient,d,\refmeas,\form,\domain)$ satisfies Assumption \ref{a:feller},
				then \eqref{e:delta2} holds for all $x\in B(\xi, R) \cap D$.
			\end{enumerate}
		\end{enumerate}
	\end{lem} 
	
	\begin{proof}
		\begin{enumerate}[(a1)]\setlength{\itemsep}{0pt}
			\item[\eqref{it:hmeas-cdc1}]Let $e \coloneqq e_{B(\xi,R)\setminus D, B(\xi,A_0R)} \in \domain^{0}(B(\xi,A_0R))_{e}$
			denote the equilibrium potential for $\Capa_{B(\xi,A_0 R)}(B(\xi,R)\setminus D)$.
			Then by \cite[Theorem 4.3.3]{FOT}, for $\form$-q.e.\ $x \in B(\xi,R) \cap D$,
			\begin{align*}
				\wt{e}(x)&= \lawdiff_{x}(\sigma_{B(\xi,R) \setminus D} < \sigma_{B(\xi,A_0R)^c})  \nonumber \\
				& \ge \lawdiff_{x}( \sigma_{D \cap S(\xi,A_0 R)} > \sigma_{D^c} ) = 1 - \lawdiff_{x}( \sigma_{D \cap S(\xi,A_0 R)} < \sigma_{D^c} ) \overset{\eqref{e:delta1}}{\ge} \gamma.
			\end{align*}
			Therefore $\gamma^{-1} \wt{e} \ge 1$ $\form$-q.e.\ on $B(\xi,A_0^{-1}r)$ and
			$\Capa_{B(\xi,A_0 R)}(B(\xi,R)) \le \form(\gamma^{-1}e, \gamma^{-1}e)=\gamma^{-2}\Capa_{B(\xi,A_0 R)}(B(\xi,R)\setminus D)$.
			\item[\eqref{it:cdc1-anyA0}]By \cite[Lemma 5.22]{BCM} and domain monotonicity of capacity, in order to show \eqref{e:cdc4},
			we may and do assume that $\wh{A_0} > A_0$. By \cite[Lemma 5.18]{BCM}
			there exist $C_2 \in (1,\infty)$ and $\wh{A_1} \in [A_1,\infty)$ such that
			for all $\xi \in \partial D $ and all $R \in (0,\diam(D)/\wh{A_1})$,
			\begin{equation} \label{e:cd1}
				\gren{B(\xi,A_0R)}(y,z) \le	\gren{B(\xi,\wh{A_0}R)}(y,z) \le C_1 \gren{B(\xi,A_0R)}(y,z) \quad \mbox{for all $y,z \in B(\xi,R)$.}
			\end{equation}
			Let $\xi \in \partial D $, $R \in (0,\diam( D )/\wh{A_1})$, and let
			$e_1,\nu$ be the equilibrium potential and measure for $\Capa_{B(\xi,\wh{A_1}R)}(B(\xi,R) \setminus  D )$
			such that $\Capa_{B(\xi,\wh{A_1}r)}(B(\xi,R) \setminus  D )=\form(e_1,e_1)$ and
			$e_1 =\int \gren{B(\xi,\wh{A_1}r)}(\cdot,z) \, \nu(dz)$. Define
			\begin{equation*}
				e \coloneqq \int \gren{B(\xi,A_1r)}(\cdot,z) \, \nu(dz).
			\end{equation*}
			By \eqref{e:cd1}, for $\form$-q.e.\ $y \in B(\xi,R)\setminus D$, we have
			\begin{equation}
				e(y)=\int \gren{B(\xi,A_1R)}(y,z) \, \nu(dz) \ge C_1^{-1}\int \gren{B(\xi,\wh{A_1}R)}(y,z) \, \nu(dz) \ge C_1^{-1}. \nonumber
			\end{equation}
			Therefore 
			\begin{align*}
				\Capa_{B(\xi,A_1R)}(B(\xi,R) \setminus D) &\le \form(C_1 e,C_1 e)= C_1^2 \int e(z) \,\nu(dz) \le C_1^2 \int e_1(z) \,\nu(dz) \\
				& = C_1^2 \form(e_1,e_1)= C_1^2 \Capa_{B(\xi,\wh{A_1}R)}(B(\xi,R) \setminus D).
			\end{align*}
			The above estimate along with \eqref{e:cdc3} and \cite[Lemma 5.22]{BCM} implies \eqref{e:cdc4}.
			\item[\eqref{it:cdc1-hmeas}]By \cite[Lemma 5.9]{BCM}, there exist $\wh{A_0}, \wh{A_1}, C_1 \in (1,\infty)$ such that
			for all $\xi \in D$, all $R \in (0,\diam(D)/\wh{A_1})$ and all $x,y \in \ol{B(\xi,R)}$, we have
			\begin{equation} \label{e:cd2}
				\gren{B(\xi,\wh{A_0}r)}(x, y) \ge C_1^{-1}  \gren{B(\xi,\wh{A_0}r)}(\xi,r).
			\end{equation}
			By \eqref{it:cdc1-anyA0} and increasing $\wh{A_0}, \wh{A_1}$ if necessary, we may assume that \eqref{e:cdc4} holds.
			By further increasing $\wh{A_0}, \wh{A_1}$ if necessary and using \cite[Lemma 5.10]{BCM},
			we may assume that there exists $C_2 \in (1,\infty)$ such that for all
			$\xi \in \partial D$ and $R \in (0,\diam(D)/\wh{A_1})$,
			\begin{equation} \label{e:cd3}
				\gren{B(\xi,\wh{A_0}r)}(\xi,r) \le \Capa_{B(\xi,\wh{A_0}r)}\left( B(\xi,r) \right)^{-1} \le C_2	\gren{B(\xi,\wh{A_0}r)}(\xi,r).
			\end{equation}
			Let $\xi \in \partial D$, $R \in (0,\diam(D)/\wh{A_1})$ and let
			$e \coloneqq e_{B(\xi,R)\setminus D, B(\xi,\wh{A_0}R)},\nu$ denote the equilibrium potential
			and measure, respectively, for $\Capa_{B(\xi,\wh{A_0} R)}(B(\xi,R)\setminus D)$.
			By \cite[Theorem 4.3.3]{FOT}, for $\form$-q.e.\ $x \in B(\xi,R) \cap D$, we have
			\begin{align} \label{e:cd4}
				\wt{e}(x)&= \lawdiff_{x}\bigl( \sigma_{B(\xi,R) \setminus D} < \sigma_{B(\xi,\wh{A_0}R)^c} \bigr) = \int_{\ol{B(\xi,R) \setminus D}} \gren{B(\xi,\wh{A_0}R)^c}(x,y) \, \nu(dy) \nonumber \\
				&\overset{\eqref{e:cd2}}{\ge} C^{-1} \gren{B(\xi,\wh{A_0}r)}(\xi,r) \nu\bigl( \ol{B(\xi,R) \setminus D} \bigr) \nonumber  \\
				&= C^{-1} \gren{B(\xi,\wh{A_0}r)}(\xi,r) \Capa_{B(\xi, \wh{A_0}R)}(B(\xi,R) \setminus D) \nonumber \\
				&\overset{\eqref{e:cdc4}}{\ge} C^{-1} \wh{C}^{-1} \gren{B(\xi,\wh{A_0}r)}(\xi,r) \Capa_{B(\xi, \wh{A_0}R)}(B(\xi,R)) \overset{\eqref{e:cd3}}{\ge}  C^{-1}  \wh{C}^{-1} C_2^{-1}.
			\end{align}
			Setting $\gamma \coloneqq C^{-1} \wh{C}^{-1} C_2^{-1} \in (0,1)$, we conclude that
			\[
			\hmeas{D \cap B(\xi,\wh{A_0} R)}{x} (D \cap S(\xi,\wh{A_0} R)) \le \lawdiff_{x}\bigl( \sigma_{B(\xi,R) \setminus D} > \sigma_{B(\xi,\wh{A_0}R)^c} \bigr) \overset{\eqref{e:cd4}}{\le} 1 - \gamma.
			\]
			The final assertion under Assumption \ref{a:feller} follows from the continuity of
			$\form$-harmonic measure from Lemma \ref{l:harmonicm}-\eqref{it:hmeas-continuous}.
			\qedhere\end{enumerate}
	\end{proof}
	
	The estimate \eqref{e:delta2} in Lemma \ref{l:cdc1}-\eqref{it:cdc1-hmeas} above
	can be used repeatedly to obtain certain polynomial type decay rates on the harmonic measure. 
	
	\begin{lem}[Uniform $\Delta$-regularity] \label{l:Delta}
		Let an MMD space $(\ambient,d,\refmeas,\form,\domain)$ and a uniform domain $\unifdom$
		in $(\ambient,d)$ satisfy Assumption \ref{a:hkecdcbhp}. Then the following hold:
		\begin{enumerate}[\rm(a)]\setlength{\itemsep}{0pt}\vspace{-5pt}
			\item\label{it:Delta-hmeas} There exist $C_{1}, A_{1} \in (1,\infty)$ and $\delta \in (0,\infty)$
			such that for all $\xi \in \partial \unifdom$ and all $0<r<R < \diam(\unifdom)/A_1$,
			\begin{equation} \label{e:delta}
				\hmeas{\unifdom \cap B(\xi,R)}{x}(\unifdom \cap  S(\xi,R)) \le C_1 \Bigl( \frac{r}{R}\Bigr)^\delta \quad \textrm{for all $x \in \unifdom \cap B(\xi,r)$.}
			\end{equation}
			\item\label{it:Delta-hfunc} There exist $C_2, A_0, A_1 \in (1,\infty)$ and $\delta \in (0,\infty)$
			such that for all $\xi \in \partial \unifdom$, all $0<r<R<\diam(\unifdom)/A_1$ and
			all $(0,\infty)$-valued continuous $\form$-harmonic function $h$ on $\unifdom \cap B(\xi, A_0 R)$
			with Dirichlet boundary condition relative to $\unifdom$,
			\begin{equation} \label{e:hdelta}
				\frac{h(\xi_r)}{h(\xi_R)} \le C_2  \Bigl( \frac{r}{R}\Bigr)^\delta.
			\end{equation}
		\end{enumerate}
	\end{lem}
	
	\begin{proof}
		\begin{enumerate}[\rm(a)]\setlength{\itemsep}{0pt}
			\item By Lemma \ref{l:cdc1}-\eqref{it:cdc1-hmeas}, there exist $A_0,A_1 \in (1,\infty)$ and $\gamma \in (0,1)$ such that
			\begin{equation} \label{e:cc1}
				\hmeas{\unifdom \cap B(\xi, R)}{x} (\unifdom \cap S(\xi, R)) \le 1 -\gamma
			\end{equation}
			for all $\xi \in \partial \unifdom$, all $R \in (0,\diam(\unifdom)/A_1)$ and all $x\in B(\xi, A_0^{-1} R)$.
			By the strong Markov property, for all $i \in \mathbb{N}$, all $\xi \in \partial \unifdom$,
			all $R \in (0,\diam(\unifdom)/A_1)$ and all $x\in B(\xi, A_0^{-i} R)$,
			\begin{align*}
				\MoveEqLeft{\hmeas{\unifdom \cap B(\xi, R)}{x} (\unifdom \cap S(\xi, R))} \nonumber \\
				&\le \hmeas{\unifdom \cap B(\xi,A_0^{-i} R)}{x} (\unifdom \cap S(\xi,A_0^{-i} R)) \sup_{y \in \unifdom \cap S(\xi,A_0^{-i} R)} \hmeas{\unifdom \cap B(\xi,  R)}{y} (\unifdom \cap S(\xi, R)) \nonumber \\
				&\overset{\eqref{e:cc1}}{\le} (1 -\gamma) \sup_{y \in \unifdom \cap S(\xi,A_0^{-i+1} R)}  \hmeas{U\cap B(\xi, R)}{y} (\unifdom \cap S(\xi, R)).
			\end{align*}
			By repeatedly using the above estimate, we obtain
			\begin{equation*}
				\hmeas{\unifdom\cap B(\xi, R)}{x} (\unifdom \cap S(\xi,R)) \le (1-\gamma)^i
			\end{equation*}
			for all $i \in \mathbb{N}$, all $\xi \in \partial \unifdom$, all $R \in (0,\diam(\unifdom)/A_1)$
			and all $x\in B(\xi, A_0^{-i} R)$. This implies \eqref{e:delta}.
			\item 
			By \ref{eq:BHP} from Theorem \ref{t:bhp}, Proposition \ref{p:goodgreen} and
			Lemma \ref{l:dbdy}, it suffices to consider the case when $h$ is a Green function.
			More precisely, it suffices to show that there exist $C_3, A_0, A_1 \in (1,\infty)$
			and $\delta \in (0,\infty)$ such that for all $\xi \in \partial \unifdom$,
			all $0<r<R < \diam(\unifdom)/A_1$ and all $x_0 \in \unifdom$ such that $d(\xi,x_0)> A_0 R$, we have
			\begin{equation} \label{e:hdeltag}
				\frac{\gren{\unifdom}(\xi_r,x_0)}{\gren{\unifdom}(\xi_R,x_0)} \le C_3 \Bigl( \frac{r}{R}\Bigr)^\delta.
			\end{equation}
			Let us choose $A_0, A_1 \in (1,\infty)$ such that the conclusion of \eqref{it:Delta-hmeas},
			and \ref{eq:BHP} and Carleson's estimate (Proposition \ref{p:carleson}) hold. Then
			for all $\xi \in \partial \unifdom$, all $0<r<R < \diam(\unifdom)/A_1$ and
			all $x_0 \in \unifdom$ such that $d(\xi,x_0)> A_0 R$, we have
			\begin{align*}
				\gren{\unifdom}(\xi_r,x_0)&= \expdiff_{\xi_r}\bigl[ \gren{\unifdom}\bigl( \diff^\unifdom_{\tau_{\unifdom \cap B(\xi,R)}}, x_0 \bigr) \bigr] \quad \textrm{(by Lemma \ref{l:green})}\\
				&\le \Bigl( \sup_{\unifdom \cap S(\xi,R)} \gren{\unifdom}(\cdot,x_0) \Bigr) \hmeas{\unifdom \cap B(\xi,R)}{\xi_r}(\unifdom \cap S(\xi,R)) \\
				&\lesssim \gren{\unifdom}(\xi_R,x_0)\hmeas{\unifdom \cap B(\xi,R)}{\xi_r}(\unifdom \cap  S(\xi,R)) \quad \textrm{(by Carleson's estimate)} \\
				&\lesssim \gren{\unifdom}(\xi_R,x_0) \Bigl(\frac{r}{R}\Bigr)^\delta \quad \textrm{(by \eqref{e:delta}).}
				\qedhere\end{align*}
		\end{enumerate}
	\end{proof}
	
	\subsection{Two-sided bounds on harmonic measure} \label{ssec:hmeas}
	
	The following estimate of harmonic measure is the main result of this subsection.
	It is an extension of \cite[Lemmas 3.5 and 3.6]{AH} obtained for the Brownian motion
	and uniform domains satisfying the capacity density condition in Euclidean space,
	which in turn generalize similar results obtained by Jerison and Kenig for NTA domains
	in \cite[Lemma 4.8]{JK} and by Dahlberg for Lipschitz domains in \cite[Lemma 1]{Dah}.
	
	\begin{theorem} \label{t:hmeas}
		Let an MMD space $(\ambient,d,\refmeas,\form,\domain)$ and a uniform domain $\unifdom$ in $(\ambient,d)$
		satisfy Assumption \ref{a:hkecdcbhp}. Then there exist $C,A \in (1,\infty)$ such that 
		\begin{equation} \label{e:harmonicm}
			C^{-1} \gren{\unifdom}(x_0, \xi_{r}) \Capa_{B(\xi,2r)}(B(\xi,r)) \le \hmeas{\unifdom}{x_0}(\partial \unifdom \cap B(\xi,r))  \le C  \gren{\unifdom}(x_0, \xi_{r})  \Capa_{B(\xi,2r)}(B(\xi,r)) 
		\end{equation}
		for all $\xi \in \partial \unifdom$, all $x_0 \in \unifdom$ and all $r \in (0,d(\xi,x_0)/A)$.
	\end{theorem}
	
	While it is possible to prove Theorem \ref{t:hmeas} by adapting an iteration argument (called the \emph{box argument})
	from \cite[Proof of Lemma 3.6]{AH} based on the notion of \emph{capacitary width}, our proof given below
	after Lemma \ref{l:gbnd} is new and more probabilistic and avoids the use of such a complicated argument.
	Combining Theorem \ref{t:hmeas} with the $\form$-harmonicity of $\gren{U}(x_{0},\cdot)$ on $\unifdom \setminus \{x_{0}\}$, Harnack chaining (Lemma \ref{l:hchain}),
	Remark \ref{r:ehi-hke} and \cite[Lemma 5.23]{BCM}, we obtain the following volume doubling property of the harmonic measure.
	
	\begin{cor}	\label{c:asdouble}
		Let an MMD space $(\ambient,d,\refmeas,\form,\domain)$ and a uniform domain $\unifdom$ in $(\ambient,d)$
		satisfy Assumption \ref{a:hkecdcbhp}. Then there exist $C,A \in (1,\infty)$ such that
		\begin{equation} \label{e:harmonicd}
			\hmeas{\unifdom}{x_0}(\partial \unifdom \cap B(\xi, r)) \le C \hmeas{\unifdom}{x_0}(\partial \unifdom \cap B(\xi,r/2))   
		\end{equation}
		for all $\xi \in \partial \unifdom$, all $x_0 \in \unifdom$ and all $r \in (0,d(\xi,x_0)/A)$.
		In particular, $\supp_{\ambient}[\hmeas{\unifdom}{x_0}] = \partial \unifdom$.
	\end{cor}
	
	Thanks to the capacity density condition \ref{eq:CDC}, we can compare the Green function
	on the domain $\unifdom$ with that on a ball chosen at a suitable scale.
	The following is an analogue of a lemma of Aikawa and Hirata for uniform domains
	in Euclidean space \cite[Lemma 3.2]{AH}. Our proof follows an argument in
	\cite[Proof of Lemma 3.12]{BM18} to compare Green functions on different open sets.
	
	\begin{lem} \label{l:gbnd}
		Let an MMD space $(\ambient,d,\refmeas,\form,\domain)$ and a uniform domain $\unifdom$
		in $(\ambient,d)$ satisfy Assumption \ref{a:hkecdcbhp}. Then there exist $A_1 \in (1,\infty)$
		and $c_0 \in (0,1)$ such that for each $c \in (0,c_{0}]$ the following holds for some
		$C_1 \in (1,\infty)$: for all $\xi \in \partial \unifdom$ and all $r \in (0,\diam(\unifdom)/A_1)$,
		\begin{equation} \label{e:greenlocglob}
			C_1^{-1} \Capa_{B(\xi,2r)}(B(\xi,r))^{-1} \le \gren{\unifdom}(\xi_{r},c r) \le C_1 \Capa_{B(\xi,2r)}(B(\xi,r))^{-1}.
		\end{equation}
	\end{lem}
	
	\begin{proof}
		By Lemma \ref{l:Delta}-\eqref{it:Delta-hmeas}, there exist $A_1, A_0\in (4,\infty)$ such that
		for all $\xi \in \partial \unifdom$ and all $r \in (0,\diam(\unifdom)/A_1)$, we have
		\begin{equation} \label{e:glb1}
			\sup_{z \in \unifdom \cap B(\xi,2r)} \hmeas{\unifdom \cap B(\xi,A_0r)}{z}(\unifdom \cap S(\xi,A_0r))\le \frac{1}{2}.
		\end{equation}
		
		By \eqref{e:gradial} and \cite[Lemmas 5.20-(c), 5.22 and 5.23]{BCM}, there exist $c_0 \in (0,c_{\unifdom}/2)$ and $\wt{A}_1 \in (4,\infty)$
		such that for all $c \in (0,c_0]$ there exists $C_2 \in (1,\infty)$ satisfying the following estimate:
		for all $\xi \in \partial \unifdom$, all $r \in (0,\diam(\unifdom)/\wt{A}_1)$ and all $y \in S(\xi_r,cr)$, we have
		\begin{equation} \label{e:greenb}
			C_2^{-1} \Capa_{B(\xi,2r)}(B(\xi,r))^{-1}  \le   \gren{B(\xi_r,c_{\unifdom} r/2)}(\xi_r, y) \le \gren{B(\xi,A_0r)}(\xi_r, y) \le C_2 \Capa_{B(\xi,2r)}(B(\xi,r))^{-1}.
		\end{equation}
		%The existence of $c_0,C_2,C_3$ as above follows from \cite[Theorem 1.2]{GHL15} (The arguments in \cite{GHL15} work even if $\ambient$ is bounded). 
		Also by \eqref{e:gradial} and by reducing $c_0$ further if necessary, there exists $C_3 \in (1,\infty)$ such that 
		\begin{equation} \label{e:grb}
			\sup_{S(\xi_r,cr)} \gren{\unifdom}(\xi_r,\cdot)
				\le C_3 \inf_{S(\xi_r,cr)} \gren{\unifdom}(\xi_r,\cdot)
				= C_3 \gren{\unifdom}(\xi_r,cr)
		\end{equation}
		for all $c \in (0,c_0]$, all $\xi \in \partial \unifdom$ and all $r \in (0,\diam(\unifdom)/\wt{A}_1)$.
		On the other hand, for all $c \in (0,c_0]$, all $\xi \in \partial \unifdom$ and all $r \in (0,\diam(\unifdom)/A_1)$,
		by choosing $\eta \in  S(\xi_r, cr)$ satisfying
		\begin{equation}\label{e:glb2}
			\gren{\unifdom}(\xi_r, \eta)=\sup_{y \in S(\xi_r, cr)}\gren{\unifdom}(\xi_r,y),
		\end{equation}
		and by the Dynkin--Hunt formula (Lemma \ref{l:dhformula}) and
		the maximum principle (the latter equality in \eqref{e:max}), we obtain
		\begin{align} 
			\gren{\unifdom}(\xi_r, \eta)&= \gren{\unifdom \cap B(\xi,A_0r)}(\xi_r,\eta)+ \expdiff_{\eta}\bigl[\one_{ \{\tau_{\unifdom \cap B(\xi,A_0r)}<\infty, X_{\tau_{\unifdom \cap B(\xi,A_0r)}} \in \unifdom \}} \gren{\unifdom}(X_{\tau_{\unifdom \cap B(\xi,A_0r)}}, \xi_r) \bigr] \nonumber \\
			&\le \gren{\unifdom \cap B(\xi,A_0r)}(\xi_r,\eta) + \gren{\unifdom}(\xi_r,\eta) \lawdiff_{\eta}\bigl( \tau_{\unifdom \cap B(\xi,A_0r)}<\infty, X_{\tau_{\unifdom \cap B(\xi,A_0r)}} \in \unifdom \bigr) \nonumber \\
			&\le \gren{\unifdom \cap B(\xi,A_0r)}(\xi_r,\eta) + \frac{1}{2}\gren{\unifdom}(\xi_r,\eta) \quad \textup{(by \eqref{e:glb1}),} \nonumber
		\end{align}
		and hence
		\begin{equation}\label{e:glb3}
			\gren{B(\xi_{r},c_{\unifdom} r/2)}(\xi_r,\eta)
				%\le \gren{\unifdom \cap B(\xi,A_0r)}(\xi_r,\eta)
				\le \gren{\unifdom}(\xi_r, \eta)
				\le 2 \gren{\unifdom \cap B(\xi,A_0r)}(\xi_r,\eta)
				\le 2 \gren{B(\xi,A_0 r)}(\xi_r,\eta).
		\end{equation}
		Combining \eqref{e:greenb}, \eqref{e:glb3}, \eqref{e:glb2} and \eqref{e:grb},
		we obtain \eqref{e:greenlocglob}.
	\end{proof}
	
	\begin{proof}[Proof of Theorem \ref{t:hmeas}]
		We first show the lower bound on the harmonic measure which is considerably easier than the upper bound. \\
		\noindent\textbf{Lower bound on harmonic measure}: By Lemmas \ref{l:harmonicm}-\eqref{it:hmeas-prob1}
		and \ref{l:Delta}-\eqref{it:Delta-hmeas}, there exists $c_1 \in (0,1/2)$ such that for all
		$\xi \in \partial \unifdom$, all $r \in (0,\diam(\unifdom)/A_1)$ and all $y \in \unifdom \cap B(\xi,2c_1r)$,
		\begin{equation} \label{e:lbh1}
			\hmeas{\unifdom}{y}(B(\xi,r) \cap \partial \unifdom ) \ge 1- \hmeas{\unifdom \cap B(\xi,r)}{y}(\unifdom \cap S(\xi,r)) \ge \frac{1}{2}.
		\end{equation}
		By Lemmas \ref{l:hit}-\eqref{it:hitting-prob-bounds} and \ref{l:gbnd} and increasing $A_1$ if necessary,
		there exist $c_2 \in (0,c_1)$ and $C_1, C_2 \in (1,\infty)$ such that 
		\begin{equation}\label{e:lbh2}
			C_1^{-1} \frac{\gren{\unifdom}(x_0,\xi_{c_1r})}{\gren{\unifdom}(\xi_{c_1r}, c_2r)}
			\le \lawdiff_{x_0}\bigl( \sigma_{\ol{B(\xi_{c_1 r},c_2 r)}} < \sigma_{\unifdom^c} \bigr)
			\le C_1 \frac{\gren{\unifdom}(x_0,\xi_{c_1 r})}{\gren{\unifdom}(\xi_{c_1 r}, c_2 r)}
		\end{equation}
		and
		\begin{equation} \label{e:lbh3}
			C_2^{-1} \Capa_{B(\xi,2r)}(B(\xi,r))^{-1} \le \gren{\unifdom}(\xi_{c_1r}, c_2r)
			\le C_2 \Capa_{B(\xi,2r)}(B(\xi,r))^{-1}
		\end{equation}
		for all $\xi \in \partial \unifdom$, all $r \in (0,\diam(\unifdom)/A_1)$ and all $x_0 \in \unifdom \setminus B(\xi,2r)$.
		
		The lower bound on the harmonic measure is obtained by estimating the probability of the event
		that the diffusion $\diff$ first hits the set $\overline{B(\xi_{c_1 r},c_2 r)}$ before exiting $\unifdom$
		along $\partial \unifdom \cap B(\xi,r)$. Setting  $B_0 \coloneqq \overline{B(\xi_{c_1 r}, c_2 r)}$,
		we estimate the harmonic measure as
		\begin{align} \label{e:hml}
			\hmeas{\unifdom}{x}&(\partial \unifdom \cap B(\xi,r)) \ge \lawdiff_{x}\bigl( \textrm{$\sigma_{B_0}<\sigma_{\unifdom^c}$, $\diff_{\sigma_{\unifdom^c}} \in \partial \unifdom \cap B(\xi,r)$} \bigr) \nonumber \\
			&= \lawdiff_{x}(\sigma_{B_0} < \sigma_{\unifdom^c}) \expdiff_{x}\bigl[ \hmeas{\unifdom}{\diff_{\sigma_{B_0}}}(\partial \unifdom \cap B(\xi,r)) \bigr] \quad \mbox{(by the strong Markov property of $\diff$)}\nonumber \\
			&\ge \lawdiff_{x}(\sigma_{B_0} < \sigma_{\unifdom^c})  \inf_{y \in B_0} \hmeas{\unifdom}{y}( \partial \unifdom \cap B(\xi,r)) \overset{\eqref{e:lbh1}}{\ge}\frac{1}{2} \lawdiff_{x}(\sigma_{B_0} < \sigma_{\unifdom^c}) \nonumber \\
			&\overset{\eqref{e:lbh2}}{\ge} (2C_1)^{-1} \frac{\gren{\unifdom}(x,\xi_{c_1 r})}{\gren{\unifdom}(\xi_{c_1 r}, c_2 r)} \overset{\eqref{e:lbh3}}{\ge} (2C_1C_2)^{-1} \gren{\unifdom}(x,\xi_{c_1r/2}) \Capa_{B(\xi,2r)}(B(\xi,r)) 
		\end{align}
		for all $\xi \in \partial \unifdom$, all $r \in (0,\diam(\unifdom)/A_1)$ and all $x \in \unifdom \setminus B(\xi,2r)$.
		On the other hand, by Lemma \ref{l:hchain} and increasing $A_1$ if necessary,
		there exist $A_0,C_3 \in (1,\infty)$ such that
		\begin{equation} \label{e:hm1a}
			C_3 \gren{\unifdom}(x,\xi_{r}) \ge \gren{\unifdom}(x,\xi_{c_1r/2}) \ge C_3^{-1}\gren{\unifdom}(x,\xi_{r})
		\end{equation}
		for all $\xi \in \partial \unifdom$, all $r \in (0,\diam(\unifdom)/A_{1})$ and all $x \in \unifdom \setminus B(\xi,A_0 r)$.
		Combining \eqref{e:hml} and \eqref{e:hm1a}, we obtain the desired lower bound.
		
		\noindent\textbf{Upper bound on harmonic measure}: Fix $\xi \in \partial \unifdom$ and
		$r \in (0,\diam(\unifdom)/A_{1})$, where $A_1 \in (1,\infty)$ may be increased from its current value
		in the course of this proof. We note that all the constants in the argument below are independent of
		the choice of $\xi,r$ and depend only on the constants involved in Assumption \ref{a:hkecdcbhp}.
		We consider two cases depending on whether or not $(B(\xi,4r) \setminus B(\xi,2r)) \cap \partial \unifdom$ is empty.
		
		\noindent\emph{Case 1: $(B(\xi,4r) \setminus B(\xi,2r)) \cap \partial \unifdom =\emptyset$.}
		In this case, we use the estimate 
		\begin{equation} \label{e:uhm1}
			\hmeas{\unifdom}{x}(B(\xi,r) \cap \partial \unifdom) \le \lawdiff_{x}\bigl( \sigma_{S(\xi, 3r ) \cap \unifdom}<\sigma_{\unifdom^c} \bigr).
		\end{equation}
		
		By Lemma \ref{l:hit}-\eqref{it:hitting-prob-bounds} and the same argument as the proof of Lemma \ref{l:gbnd}
		(by using \cite[Lemmas 5.10, 5.20-(a) and 5.23]{BCM}) and by increasing $A_1$ if necessary, there exist
		$c_1 \in (0,1)$ and $C_3, C_4 \in (1,\infty)$ such that 
		\begin{equation} \label{e:uhm5}
			\gren{\unifdom}(y,c_1r) \ge \gren{B(y,r)}(y,c_1r) \ge C_3^{-1} \Capa_{B(\xi,2r)}(B(\xi,r))^{-1} \ge C_4^{-1}\Capa_{B(\xi,2r)}(B(\xi,r))^{-1}
		\end{equation}
		and
		\begin{equation} \label{e:uhm4}
			\lawdiff_{x} \bigl( \sigma_{\ol{B(y,c_1r)}} < \sigma_{\unifdom^c} \bigr) \le C_3 
			\frac{\gren{\unifdom}(y,x_0)}{\gren{\unifdom}(y,c_1r)}
		\end{equation}
		for all $y \in \unifdom \cap S(\xi,3r)$ and $x_0 \in \unifdom \setminus B(\xi,4r)$.
		By Lemma \ref{l:chain}-\eqref{it:chain-unifdom}, the proof of Lemma \ref{l:hchain}
		and by increasing $A_0,A_1$ if needed, there exist $A_0, C_5 \in (1,\infty)$ such that
		for all $y \in \unifdom \cap S(\xi,3r)$ and $x_0 \in \unifdom \setminus B(\xi,A_0 r)$ we have 
		\begin{equation} \label{e:uhm6}
			\gren{\unifdom}(y,x_0) \le C_5 \gren{\unifdom}(\xi_r,x_0). 
		\end{equation}
		Choosing a maximal $c_1 r$-separated subset $\{y_i \mid 1 \le i \le N \}$ of
		$\unifdom \cap S(\xi,3r)$ on the basis of \hyperlink{MD}{\textup{MD}}, we have
		$\unifdom \cap S(\xi, 3r) \subset \bigcup_{i=1}^{N} B(y_i,c_1 r)$, where $N \in \mathbb{N}$
		has an upper bound that depends only on \hyperlink{MD}{\textup{MD}} and $c_1$.
		Therefore by \eqref{e:uhm1}, we obtain
		\begin{align} \label{e:uhm2}
			\hmeas{\unifdom}{x_0}(B(\xi,r) \cap \partial \unifdom)
			&\le \lawdiff_{x_0}\bigl( \sigma_{\bigcup_{i=1}^{N} B(y_i,c_1 r)}<\sigma_{\unifdom^c} \bigr)
			\le \sum_{i=1}^{N}\lawdiff_{x_0}\bigl( \sigma_{\ol{B(y_i, c_1r)}}<\sigma_{\unifdom^c} \bigr) \nonumber \\
			&\overset{\eqref{e:uhm4}}{\le} \sum_{i=1}^{N} C_3 \frac{\gren{\unifdom}(y,x_0)}{\gren{\unifdom}(y,c_1r)}
			\overset{\eqref{e:uhm6}}{\le} N C_3 \frac{\gren{\unifdom}(\xi_r,x_0)}{\gren{\unifdom}(y,c_1r)}
			\overset{\eqref{e:uhm5}}{\le} N C_3 C_4 \frac{\gren{\unifdom}(\xi_r,x_0)}{\gren{\unifdom}(y,c_1r)}
		\end{align}
		for all  $x_0 \in \unifdom \setminus B(\xi,A_0r)$.
		The desired upper bound in this case follows from \eqref{e:uhm2} and \eqref{e:uhm5}.

		\noindent\emph{Case 2: $(B(\xi,4r) \setminus B(\xi,2r)) \cap \partial \unifdom  \neq \emptyset$.}
		Let $\eta \in (B(\xi,4r) \setminus B(\xi,2r)) \cap \partial \unifdom$ and
		set $V \coloneqq \overline{\unifdom} \setminus (\partial \unifdom \setminus B(\xi,3r/2))$
		(note that $V$ is an open subset of $\overline{\unifdom}$). Recall from Assumption \ref{a:hkecdcbhp} that
		$\diffref=\bigl(\Omega^{\on{ref}},\events^{\on{ref}},\{\diffref_{t}\}_{t\in[0,\infty]},\{\lawref_{x}\}_{x\in\oneptcpt{\overline{\unifdom}}}\bigr)$
		is a diffusion on $\overline{\unifdom}$ as in Assumption \ref{a:feller} for the MMD space
		$(\overline{\unifdom},d,\refmeas|_{\overline{\unifdom}},\formref,\domain(\unifdom))$.
		Since $B(\eta,r/2) \cap \partial \unifdom$ is a subset of $\overline{\unifdom} \setminus V$
		and not $\formref$-polar by Lemmas \ref{l:Delta}-\eqref{it:Delta-hmeas} and
		\ref{l:harmonicm}-\eqref{it:hmeas-quasi-support},\eqref{it:hmeas-prob1},\eqref{it:hmeas-smooth-support},
		the part Dirichlet form of $(\formref,\domain(\unifdom))$ on $V$ is transient
		by the irreducibility of $(\overline{\unifdom},\refmeas|_{\overline{\unifdom}},\formref,\domain(\unifdom))$
		from Proposition \ref{p:feller}-\eqref{it:HKE-conn-irr-cons} and \cite[Proposition 2.1]{BCM}.
		Hence we have the Green function $\grenref{V}$ of $(\formref,\domainref)$
		on $V$ as given in Proposition \ref{p:goodgreen} and Lemma \ref{l:green}.
		
		By Lemma \ref{l:hit} and arguing similarly as \eqref{e:uhm5} and \eqref{e:uhm4}, along with increasing
		$A_1$ if needed, there exist $C_1, C_2 \in (1,\infty)$ and $c_1 \in (0,c_{\unifdom}/4)$ such that
		\begin{gather}
			\label{e:uhm3} \sup_{y \in S(\xi_r,c_1r)}\grenref{V}(\xi_r,y) \le C_1 \grenref{V}(\xi_r,c_1r),\\
			\label{e:uhm7} C_1^{-1} \frac{\grenref{V}(x,\xi_r)}{\grenref{V}(\xi_r,c_1r)} \le \lawref_{x}\bigl(\sigma_{\ol{B(\xi_r,c_1r)}} < \sigma_{\partial \unifdom \setminus B(\xi,3r/2)} \bigr) \le C_1 \frac{\grenref{V}(x,\xi_r)}{\grenref{V}(\xi_r,c_1r)},\\
			\label{e:uhm8} \grenref{V}(\xi_r,c_1r) \ge \gren{\unifdom}(\xi_r,c_1r) \ge C_2^{-1}\Capa_{B(\xi,2r)}(B(\xi,r))^{-1} 
		\end{gather}
		for all $x \in \overline{\unifdom} \setminus B(\xi_r,c_1r)$.
		By Harnack chaining in $V$ and \eqref{e:uhm3}, there exists $C_3 \in (1,\infty)$
		such that for all $z \in B(\xi,r)\cap \partial \unifdom$, we have
		\begin{equation} \label{e:uhm9}
			\grenref{V}(\xi_r,z) \ge C_3^{-1}\grenref{V}(\xi_r,c_1r). 
		\end{equation}
		Setting $B \coloneqq \overline{B(\xi_r,c_1r)}$, by the strong Markov property of $\diffref$ we have
		\begin{align} \label{e:hm1}
			\lawref_{x}(\sigma_{B} < \sigma_{\partial \unifdom \setminus B(\xi,3r/2)} ) &\ge \lawref_{x}\bigl( \diffref_{\sigma_{\partial \unifdom}} \in \partial \unifdom \cap B(\xi,r), \sigma_B \circ \theta_{\sigma_{\partial \unifdom}} < \sigma_{\partial \unifdom \setminus B(\xi,3r/2)} \circ \theta_{\sigma_{\partial \unifdom}} \bigr) \nonumber \\
			&\ge \hmeas{\unifdom}{x}(\partial \unifdom \cap B(\xi,r)) \inf_{z \in B(\xi,r) \cap \partial \unifdom} \lawref_z(\sigma_B < \sigma_{\partial \unifdom \setminus B(\xi,3r/2)})\nonumber \\
			&\ge C_1^{-1}C_3^{-1} \hmeas{\unifdom}{x}(\partial \unifdom \cap B(\xi,r)) \quad \mbox{(by \eqref{e:uhm9} and \eqref{e:uhm7})}
		\end{align}
		for all $x \in \overline{\unifdom} \setminus \overline{B(\xi_r,c_1r)}$. Now the proof of the desired upper bound in \eqref{e:harmonicm}
		is reduced to showing that there exists $C_{4} \in (1,\infty)$ such that, for suitably chosen $A_0 \in (1,\infty)$,
		\begin{equation} \label{eq:hm-proof-comp-green}
			\grenref{V}(x_0,\xi_r) \leq C_{4} \gren{\unifdom}(x_0,\xi_r) \qquad \textrm{for all $x_0 \in \unifdom \setminus B(\xi,A_0 r)$;}
		\end{equation}
		indeed, by combining \eqref{e:hm1}, \eqref{e:uhm7}, \eqref{e:uhm8} and \eqref{eq:hm-proof-comp-green} we obtain
		\begin{align*}
			\hmeas{\unifdom}{x_0}(\partial \unifdom \cap B(\xi,r))
			&\overset{\eqref{e:hm1}}{\le} C_1 C_3 \lawref_{x_0}(\sigma_{B} < \sigma_{\partial \unifdom \setminus B(\xi,3r/2)} )
			\overset{\eqref{e:uhm7}}{\le} C_1^2 C_3 \frac{\grenref{V}(x_0,\xi_r)}{\grenref{V}(\xi_r,c_1r)} \nonumber \\
			&\overset{\eqref{e:uhm8}}{\le} C_1^2 C_2 C_3 \grenref{V}(x_0,\xi_r) \Capa_{B(\xi,2r)}(B(\xi,r)) \nonumber \\
			&\overset{\eqref{eq:hm-proof-comp-green}}{\le} C_1^2 C_2 C_3 C_4 \gren{\unifdom}(x_0,\xi_r) \Capa_{B(\xi,2r)}(B(\xi,r)).
		\end{align*}
		%\quad \textrm{(by \eqref{e:hm8} and \eqref{e:hm9}),}
		
		To see \eqref{eq:hm-proof-comp-green}, recall the Dynkin--Hunt formula (Lemma \ref{l:dhformula}) that 
		\begin{align} \label{e:hm4}
			\grenref{V}(y,z) &= \gren{\unifdom}(y,z)+ \expref_{y}\bigl[\one_{\{\tau_{\unifdom}<\infty,\diffref_{\tau_{\unifdom}} \in V\}}  \grenref{V}(\diffref_{\tau_{\unifdom}},z)\bigr] \quad \textrm{for all $y \in \unifdom$, $z \in \unifdom \setminus \{y\}$.}
		\end{align}
		By Lemma \ref{l:green}, for any $x_0 \in \unifdom \setminus B(\xi,4r)$ and any $z \in V \cap \overline{B(\xi,d(\xi,\eta))}$ we have 
		\begin{equation} \label{e:hm2}
			\grenref{V}(z,x_0)= \expref_{z}\bigl[ \grenref{V}((\diffref)^V_{\tau_{V \cap B(\xi,d(\xi,\eta))}},x_0) \bigr] \le \sup_{ \unifdom \cap S(\xi,d(\xi,\eta))} \grenref{V}(\cdot,x_0).
		\end{equation}
		Therefore, we obtain for all $x_0 \in \unifdom \setminus B(\xi,4r)$ and all $y \in \unifdom \cap \overline{B(\xi,d(\xi,\eta))}$,
		\begin{align} \label{e:hm5}
			\grenref{V}(y,x_0) &\overset{\eqref{e:hm4}}{\le} \gren{\unifdom}(y,x_0)+ \lawref_{y}\bigl( \tau_{\unifdom}<\infty,\diffref_{\tau_{\unifdom}} \in V \bigr) \sup_{z \in V \setminus \unifdom} \grenref{V}(z,x_0) \nonumber \\
			&\overset{\eqref{e:hm2}}{\le} \gren{\unifdom}(y,x_0)+ \lawref_{y}\bigl( \tau_{\unifdom}<\infty,\diffref_{\tau_{\unifdom}} \in V \bigr) \sup_{z \in \unifdom \cap S(\xi,d(\xi,\eta))} \grenref{V}(z,x_0).
		\end{align}
		Next, we show that there exists $\delta \in (0,1)$ such that for all $y \in \unifdom \cap S(\xi,d(\xi,\eta))$,
		\begin{equation} \label{e:hm6}
			\lawref_{y} \bigl( \tau_{\unifdom}<\infty,\diffref_{\tau_{\unifdom}} \in V \bigr) \le 1- \delta.
		\end{equation}
		Indeed, by Lemma \ref{l:harmonicm}-\eqref{it:hmeas-continuous}, the function
		$h(y) \coloneqq \lawref_{y}\bigl( \tau_{\unifdom} < \infty, \diffref_{\tau_{\unifdom}} \in V \bigr)$
		is continuous and $\formref$-harmonic on $\unifdom$. Then by Lemma \ref{l:Delta}-\eqref{it:Delta-hmeas},
		there exists $c_2 \in (0,1/4)$ such that
		\begin{equation} \label{e:hm7}
			h(y) \le \frac{1}{2} \qquad \textrm{for all $y \in \unifdom \setminus B(\xi,7r/4)$ with $\delta_{\unifdom}(y)<c_2 r$,}
		\end{equation}
		whereas by \eqref{e:hm7} and Harnack chaining for the $\formref$-harmonic function
		$1-h$ on $\unifdom$ using Lemma \ref{l:chain}-\eqref{it:chain-unifdom}, 
		there exists $\delta \in (0,1)$ such that $h(y) \le 1 - \delta$ for all
		$y \in \unifdom \cap S(\xi,d(\xi,\eta))$ with $\delta_{\unifdom}(y)\geq c_2 r$, proving \eqref{e:hm6}.
		In particular, taking supremum over $y \in \unifdom \cap S(\xi,d(\xi,\eta))$ in \eqref{e:hm5} and
		using \eqref{e:hm6}, for all $x_0 \in \unifdom \setminus B(\xi,4r)$ we obtain
		\begin{equation*}
			\sup_{y \in \unifdom \cap S(\xi,d(\xi,\eta))} \grenref{V}(y,x_0)
			\le \sup_{y \in \unifdom \cap S(\xi,d(\xi,\eta))} \gren{\unifdom}(y,x_0) + (1-\delta) \sup_{y \in \unifdom \cap S(\xi,d(\xi,\eta))} \grenref{V}(y,x_0),
		\end{equation*}
		which, together with $\sup_{y \in \unifdom \cap S(\xi,d(\xi,\eta))} \grenref{V}(y,x_0) < \infty$
		implied by the maximum principle (the latter of \eqref{e:max}), yields
		\begin{equation} \label{e:hm8}
			\sup_{y \in \unifdom \cap S(\xi,d(\xi,\eta))} \grenref{V}(y,x_0) \le \delta^{-1}\sup_{y \in \unifdom \cap S(\xi,d(\xi,\eta))} \gren{\unifdom}(y,x_0).
		\end{equation}
		On the other hand, by Carleson's estimate (Proposition \ref{p:carleson}), Harnack chaining using Lemma \ref{l:chain}-\eqref{it:chain-unifdom}
		and increasing $A_0,A_1$ if needed, there exists $C_5 \in (1,\infty)$ such that for all $x_0 \in \unifdom \setminus B(\xi,A_0 r)$,
		\begin{equation} \label{e:hm9}
			\sup_{y \in \unifdom \cap S(\xi,d(\xi,\eta))} \grenref{V}(y,x_0) \ge C_{5}^{-1}\grenref{V}(\xi_r,x_0),
			\quad \sup_{y \in \unifdom \cap S(\xi,d(\xi,\eta))} \gren{\unifdom}(y,x_0) \le C_{5} \gren{\unifdom}(\xi_r,x_0). 
		\end{equation}
		Combining \eqref{e:hm8} and \eqref{e:hm9}, we obtain \eqref{eq:hm-proof-comp-green}
		and thereby complete the proof.
	\end{proof}
	
	Under an additional assumption which for instance is satisfied for the Brownian motion on $\bR^n$ with $n \ge 2$,
	the capacity density condition \ref{eq:CDC} for a domain $\unifdom$ implies the uniform perfectness of its boundary $\partial \unifdom$,
	which is relevant to the stable-like heat kernel estimates for the boundary trace process in Theorem \ref{thm:shk-trace} below.
	
	\begin{definition} \label{d:cnd}
		Let $(\ambient,d,\refmeas,\form,\domain)$ be an MMD space satisfying \hyperlink{MD}{\textup{MD}}
		and \hyperlink{ehi}{$\on{EHI}$}. We say that $(\ambient,d,\refmeas,\form,\domain)$ satisfies
		the \textbf{capacity non-decreasing condition} if there exist $C,A \in (1,\infty)$ such that
		\begin{equation} \label{e:cnd1}
			\Capa_{B(x,2r)}(B(x,r)) \le C \Capa_{B(x,2R)}(B(x,R)) \quad 
			\textrm{for all $x \in \ambient$, $0<r< R <\diam(\ambient)/A$.}
		\end{equation}
	\end{definition}
	
	We remark that the number $2$ in \eqref{e:cnd1} can be replaced with any constant larger than $1$ due to \cite[Lemma 5.22]{BCM}.
	If $(\ambient,d,\refmeas,\form,\domain)$ satisfies the stronger \hyperlink{VD}{\textup{VD}}
	and \hyperlink{hke}{$\on{HKE(\scdiff)}$} for some scale function $\scdiff$, then
	by \cite[Theorem 1.2]{GHL15}, \eqref{e:cnd1} is equivalent to the following estimate: there exist $C,A \in (1,\infty)$ such that 
	\begin{equation} \label{e:cnd2}
		\frac{\scdiff(R)}{\refmeas(B(x,R))} \le C \frac{\scdiff(r)}{\refmeas(B(x,r))} \quad 
		\textrm{for all $x \in \ambient$ and all $0<r< R <\diam(\ambient)/A$.}
	\end{equation}
	The condition \eqref{e:cnd2} was called \emph{fast volume growth} in \cite[Definition 1.5]{JM}.
	The following lemma follows from Theorem \ref{t:hmeas} and Lemma \ref{l:Delta}-\eqref{it:Delta-hfunc}.
	We omit its proof as it is just a straightforward modification of the argument in \cite[Remark 2.17]{AHMT}.
	
	\begin{lem}[Cf.\ {\cite[Remark 2.17]{AHMT}}] \label{l:up}
		Let an MMD space $(\ambient,d,\refmeas,\form,\domain)$ and a uniform domain $\unifdom$
		in $(\ambient,d)$ satisfy Assumption \ref{a:hkecdcbhp},  and assume further that
		$(\ambient,d,\refmeas,\form,\domain)$ satisfies the capacity non-decreasing condition.
		Then $(\partial \unifdom,d)$ is uniformly perfect.
	\end{lem}
	
	It is easy to see that Lemma \ref{l:up} is not true without the capacity non-decreasing condition.
	For instance, this can be seen by considering the unit interval $\unifdom=(0,1)$ for the Brownian motion on $\bR$.
	
	We provide some sufficient conditions for the capacity density condition below. 
	
	\begin{remark} \label{r:cdc}
		\begin{enumerate}[\rm(a)]\setlength{\itemsep}{0pt}
			\item\label{it:cdc-ecc} Let $(\ambient,d,\refmeas,\form,\domain)$ be an MMD space
			satisfying \hyperlink{MD}{\textup{MD}} and \hyperlink{ehi}{$\on{EHI}$},
			and let $\unifdom$ be an open subset of $\ambient$ satisfying the exterior corkscrew condition
			(see \cite[(3.2)]{JK} for the definition). Then the capacity estimates in \cite[Section 5]{BCM}
			imply the capacity density condition for $\unifdom$. In particular, non-tangentially
			accessible domains (see \cite[p.~93]{JK}) satisfy the capacity density condition.
			\item\label{it:cdc-AR} Let $(\ambient,d,\refmeas,\form,\domain)$ be an MMD space
			satisfying the heat kernel estimates \hyperlink{hke}{$\on{HKE}(\scdiff)$}
			with $\scdiff(r)=r^{d_{w}}$ for all $r \in [0,\infty)$ for some $d_{w} \in [2,\infty)$.
			Assume that $\refmeas$ is a $d_{f}$-Ahlfors regular measure for some
			$d_{f} \in (0,\infty)$, i.e., there exists $C \in (1,\infty)$ such that
			\begin{equation}\label{eq:ARmeas}
				C^{-1} r^{d_{f}} \le \refmeas(B(x,r)) \le C r^{d_{f}} \qquad \textrm{for all $x \in \ambient$ and all $r \in (0,2\diam(\ambient))$.}
			\end{equation}
			If $\unifdom$ is an open subset of $\ambient$ and its boundary $\partial \unifdom$ in $\ambient$
			admits a $p$-Ahlfors regular Borel measure for some $p \in (d_{f} - d_{w},\infty) \cap [0,\infty)$,
			then $\unifdom$ satisfies \ref{eq:CDC}; indeed, the desired lower bound on the capacity
			can be obtained by adapting the arguments in \cite[Proof of Theorem 5.9]{HeiK}.
			In particular, this shows that the uniform domains obtained by removing the bottom line or
			the outer square boundary of the Sierpi\'{n}ski carpet satisfy \ref{eq:CDC} with respect
			to the MMD space corresponding to the Brownian motion on the Sierpi\'{n}ski carpet.
			More generally, a similar statement holds also for any generalized Sierpi\'{n}ski carpet
			due to \cite[Lemma 3.7]{CC26+}.
		\end{enumerate}
	\end{remark}
	
	We recall a simple consequence of Lebesgue's differentiation theorem.
	We note that the condition \eqref{e:asympdoub} is satisfied by harmonic measure on $\partial \unifdom$
	due to Corollary \ref{c:asdouble}.
	
	\begin{lem}[Lebesgue's differentiation theorem] \label{l:differentiation}
		Let $(\ambient,d,\refmeas)$ be a metric measure space such that $(\ambient,d)$ is separable,
		$\refmeas(B(x,r))<\infty$ for some $r \in (0,\infty)$ for each $x\in \ambient$, and
		\begin{equation}\label{e:asympdoub}
			\limsup_{r \downarrow 0} \frac{\refmeas(B(x,2r))}{\refmeas(B(x,r))} < \infty \quad \textrm{for $\refmeas$-a.e.\ $x \in \ambient$.}
		\end{equation}
		Then for any locally integrable function $f\colon \ambient \to \mathbb{R}$ almost every point is a Lebesgue point of $f$; that is, 
		\begin{equation} \label{e:leb}
			\lim_{r \downarrow 0} \fint_{B(x,r)} \abs{f(y)-f(x)} \, d\refmeas(y) =0
		\end{equation}
		for $\refmeas$-a.e.\ $x \in \ambient$. In particular, for any $x \in \supp_{\ambient}[\refmeas]$ satisfying \eqref{e:leb},
		if $\varepsilon \in (0,\infty)$ and $\psi_{r} \colon \ambient \to \mathbb{R}$ is a Borel measurable function
		satisfying $\one_{B(x,r)} \le \psi_r \le \one_{B(x,2r)}$ for each $r \in (0,\varepsilon)$, then
		\begin{equation} \label{e:ratiodiff}
			\lim_{r \downarrow 0} \frac{\int_{\ambient} \psi_r f \,d\refmeas}{\int_{\ambient} \psi_r \,d\refmeas} = f(x).
		\end{equation}
	\end{lem}
	
	\begin{proof}
		The assertion given in \eqref{e:leb} follows from \cite[Theorem 3.4.3 and (3.4.10)]{HKST}.
		If $x \in \supp_{\ambient}[\refmeas]$ satisfies \eqref{e:leb}, then
		\begin{align} \label{e:ld1}
			0 &\le \limsup_{r \downarrow 0} \frac{\int_{\ambient} \abs{\psi_r(y) f(y)- \psi_r(y) f(x)} \,\refmeas(dy)}{\int_{\ambient} \psi_r \,dm} \nonumber \\
			&\le \limsup_{r \downarrow 0} \frac{\int_{B(x,2r)} \abs{f(y)-  f(x)} \,\refmeas(dy)}{\refmeas(B(x,r))} \nonumber  \quad \textrm{(since $\one_{B(x,r)} \le \psi_r \le \one_{B(x,2r)}$)}\\
			&\le \biggl(\limsup_{r \downarrow 0} \frac{\refmeas(B(x,2r))}{\refmeas(B(x,r))}\biggr)\limsup_{r \downarrow 0} \fint_{B(x,2r)}   \abs{  f(y)-  f(x)} \,\refmeas(dy) \overset{\eqref{e:leb},\eqref{e:asympdoub}}{=}0.
		\end{align}
		The desired conclusion follows from \eqref{e:ld1} and the estimate
		\begin{equation*}
			\abs{\frac{\int_{\ambient} \psi_r f \,d\refmeas}{\int_{\ambient} \psi_r \,d\refmeas} - f(x)}
			\le \frac{\int_{\ambient} \abs{\psi_r(y) f(y)- \psi_r(y) f(x)} \,\refmeas(dy)}{\int_{\ambient} \psi_r \,d\refmeas}.
			\qedhere\end{equation*}
	\end{proof}
	
	The following proposition shows that the $\form$-harmonic measure $\hmeas{\unifdom}{x}$
	of a uniform domain $\unifdom$ in $(\ambient,d)$ is the distributional Laplacian of the
	Green function $\gren{\unifdom}(x,\cdot)$. In the proof, we use the following notation
	for the ($0$-order) hitting distribution with respect to a diffusion
	$\diffref=\bigl(\Omega^{\on{ref}},\events^{\on{ref}},\{\diffref_{t}\}_{t\in[0,\infty]},\{\lawref_{x}\}_{x\in\overline{\unifdom}\cup\{\cemetery\}}\bigr)$
	on $\overline{\unifdom}$ as in Assumption \ref{a:feller} for the MMD space
	$(\overline{\unifdom},d,\refmeas|_{\overline{\unifdom}},\formref,\domain(\unifdom))$,
	where $\formref \coloneqq \formrefgen{\unifdom}$: we define
	$H^{\on{ref}}_{\partial \unifdom}\widetilde{u}\in\domainref_{e}$ by
	\begin{equation} \label{e:defhref}
		H^{\on{ref}}_{\partial \unifdom}\widetilde{u}(x)
		 \coloneqq \expref_{x}\bigl[\widetilde{u}(\diffref_{\sigma_{\partial \unifdom}}) \one_{\{\sigma_{\partial \unifdom}<\infty\}} \bigr]
		\mspace{25mu} \textrm{for $\formref$-q.e.\ $x \in \overline{\unifdom}$ for each $u \in \domainref_{e}$.}
	\end{equation}
	We will also use the fact that the strongly local part $\form^{(c)}$
	(recall \eqref{e:Beurling-Deny}) of any regular Dirichlet space
	$(\ambient,\refmeas,\form,\domain)$ satisfies the following strengthened strong locality:
	\begin{equation}\label{eq:strong-locality-strengthened}
		\form^{(c)}(u,v)=0 \quad \textrm{for any $u,v\in\domain_{e}$ with $(u-a)(v-b)=0$ $\refmeas$-a.e.\ on $\ambient$ for some $a,b\in\mathbb{R}$;}
	\end{equation}
	indeed, extending \cite[Corollary 3.2.1]{FOT} from $u\in\domain$ to $u\in\domain_{e}$
	by using \cite[Exercise 1.4.1, Lemma 2.1.4 and Theorem 2.3.3-(i)]{FOT}, and applying it together with
	\cite[Exercise 1.4.1 and Corollary 1.6.3]{FOT} and $\domain_{e}\cap L^{2}(\ambient,\refmeas)=\domain$,
	we can easily extend \cite[Theorem 4.3.8]{CF} from $u\in\domain \cap L^{\infty}(\ambient,\refmeas)$
	to $u\in\domain_{e}$, and then combining it with \cite[Lemmas 2.1.4, 3.2.3 and 3.2.4]{FOT}
	yields \eqref{eq:strong-locality-strengthened}.
	
	\begin{prop} \label{p:laplacian-green}
		Let $(\ambient,d,\refmeas,\form,\domain)$ be an MMD space satisfying Assumption \ref{a:feller},
		and let $\unifdom$ be a uniform domain in $(\ambient,d)$ such that the part Dirichlet form
		$(\form^{\unifdom},\domain^{0}(\unifdom))$ on $\unifdom$ is transient. Then for all $x \in \unifdom$
		and all $u \in \domainref\cap L^{\infty}(\overline{\unifdom},\refmeas|_{\overline{\unifdom}})$ such that
		$x \not\in \supp_{\refmeas|_{\overline{\unifdom}}}[u]$ and $\supp_{\refmeas|_{\overline{\unifdom}}}[u]$ is compact,
		\begin{equation}\label{eq:laplacian-green}
			\formref(\gren{\unifdom}(x,\cdot),u)= - \int_{\partial \unifdom} \widetilde{u} \, d\hmeas{\unifdom}{x}.
		\end{equation}
	\end{prop}
	
	\begin{proof}
		Note that the Green function $\gren{\unifdom}$ of $(\form,\domain)$ on $\unifdom$
		is also that of $(\formref,\domainref)$ on $\unifdom$, since the part Dirichlet form of
		$(\formref,\domain(\unifdom))$ on $\unifdom$ coincides with $(\form^{\unifdom},\domain^{0}(\unifdom))$
		as observed in the proof of Lemma \ref{l:harmonicm}-\eqref{it:hmeas-quasi-support}.
		Let $x \in \unifdom$, let $u \in \domainref\cap L^{\infty}(\overline{\unifdom},\refmeas|_{\overline{\unifdom}})$
		be such that $x \not\in \supp_{\refmeas|_{\overline{\unifdom}}}[u]$ and
		$\supp_{\refmeas|_{\overline{\unifdom}}}[u]$ is compact, and use \cite[Exercise 1.4.1]{FOT}
		to choose $\phi \in \domainref \cap \contfunc_{\mathrm{c}}(\overline{\unifdom})$ so that $\phi$
		is $[0,1]$-valued, $\phi=1$ on a neighborhood of $\supp_{\refmeas|_{\overline{\unifdom}}}[u]$
		and $x \notin \supp_{\overline{\unifdom}}[\phi]$. By the proof of Lemma \ref{l:dbdy}
		with $(\ambient,d,\refmeas,\form,\domain),D,y_{0}$ replaced by
		$(\overline{\unifdom},d,\refmeas|_{\overline{\unifdom}},\formref,\domain(\unifdom)),\unifdom,x$,
		under the convention of setting $\gren{\unifdom}(x,\cdot) \coloneqq 0$ on $\partial \unifdom$ we have
		$\phi \gren{\unifdom}(x,\cdot) \in \domainref$ and hence $\formref(\gren{\unifdom}(x,\cdot),u)$
		is canonically defined as $\formref(\phi \gren{\unifdom}(x,\cdot),u)$.
		
		Also by the proof of Lemma \ref{l:dbdy}, for some $\delta \in \bigl(0,\dist(x,\supp_{\overline{\unifdom}}[\phi])\bigr)$
		with $B(x,2\delta) \subset \unifdom$ we can construct a family $\{f_{r}\}_{r\in(0,\delta)}$
		of $[0,\infty)$-valued Borel measurable functions on $\unifdom$ such that
		$f_{r}^{-1}((0,\infty))=B(x,r)$, $\int_{\unifdom}f_{r}\,d\refmeas=1$,
		$\int_{\unifdom}f_{r}G_{\unifdom}f_{r}\,d\refmeas<\infty$ and
		$\phi G_{\unifdom}f_{r}\in \domain^{0}(\unifdom)$ for any $r \in (0,\delta)$
		and $\lim_{r\downarrow 0}\phi G_{\unifdom}f_{r}=\phi \gren{\unifdom}(x,\cdot)$
		in norm in $(\domainref,\formref_{1})$. Then for any $r \in (0,\delta)$,
		by \cite[Theorem 4.2.6 and (1.5.9)]{FOT} (see also \cite[Theorem 2.1.12-(i)]{CF}) we have
		\begin{equation}\label{eq:0-order-resolvent-properties}
			G_{\unifdom}f_{r} \in \domain^{0}(\unifdom)_{e}, \qquad
			f_{r}v \in L^{1}(\unifdom,\refmeas|_{\unifdom}) \quad \textrm{and} \quad
			\formref(G_{\unifdom}f_{r},v)=\int_{\unifdom}f_{r}v\,d\refmeas
		\end{equation}
		for any $v \in \domain^{0}(\unifdom)_{e}$, which in combination with \eqref{eq:strong-locality-strengthened},
		\eqref{e:extdirpart}, \eqref{eq:hit-dist-harm-test} and \eqref{eq:hit-dist-harm-var} yields
		\begin{align}
			&\formref(\phi G_{\unifdom}f_{r},u) = \formref(G_{\unifdom}f_{r},u) \quad \textrm{(by \eqref{eq:strong-locality-strengthened})} \nonumber \\
			&= \formref(G_{\unifdom}f_{r}, u - H^{\on{ref}}_{\partial \unifdom} \widetilde{u}) \quad \textrm{(by $G_{\unifdom}f_{r} \in \domain^{0}(\unifdom)_{e}$, \eqref{e:extdirpart} and \eqref{eq:hit-dist-harm-test})} \nonumber \\
			&= \int_{\unifdom} f_{r} (u - H^{\on{ref}}_{\partial \unifdom} \widetilde{u})\,d\refmeas \quad \textrm{(by \eqref{eq:0-order-resolvent-properties}, since $u - H^{\on{ref}}_{\partial \unifdom} \widetilde{u}\in \domain^{0}(\unifdom)_{e}$ by \eqref{eq:hit-dist-harm-var} and \eqref{e:extdirpart})} \nonumber \\
			&=-\int_{\unifdom} f_{r} H^{\on{ref}}_{\partial \unifdom} \widetilde{u}\,d\refmeas \quad \textrm{(by $\supp_{\refmeas|_{\overline{\unifdom}}}[u] \subset \supp_{\overline{\unifdom}}[\phi] \subset \overline{\unifdom} \setminus f_{r}^{-1}((0,\infty))$).}
			\label{e:gl1}
		\end{align}
		Now since $\abs{\widetilde{u}}\leq\norm{u}_{L^{\infty}(\overline{\unifdom},\refmeas|_{\overline{\unifdom}})}$
		$\formref$-q.e.\ on $\overline{\unifdom}$ by \cite[Lemma 2.1.4]{FOT}, we have
		$H^{\on{ref}}_{\partial \unifdom} \widetilde{u}(y)=\int_{\partial \unifdom}\widetilde{u}\,d\hmeas{\unifdom}{y}$
		for any $y\in \unifdom$ by Lemma \ref{l:harmonicm}-\eqref{it:hmeas-smooth-support},\eqref{it:hmeas-quasi-support},
		it is an $\mathbb{R}$-valued continuous function of $y\in \unifdom$ by Lemma \ref{l:harmonicm}-\eqref{it:hmeas-continuous},
		and therefore by letting $r\downarrow 0$ in \eqref{e:gl1} we obtain \eqref{eq:laplacian-green}.
	\end{proof}
	
	The Martin kernel can be viewed as the Radon--Nikodym derivative of the $\form$-harmonic measures at different starting points.
	A similar statement on non-tangentially accessible (NTA) domains in the Euclidean space was observed in \cite[Theorem 3.1]{KT}
	which is an easy consequence of the results in \cite{JK}. Jerison and Kenig \emph{defined}
	the Martin kernel as such a Radon--Nikodym derivative in \cite[Definition 1.3]{JK}. 
	For NTA domains in the Euclidean space the equivalence of our definition with
	\cite[Definition 1.3]{JK} follows from the uniqueness theorem in \cite[Theorem 5.5]{JK}.
	Our next result is a generalization of \cite[Theorem 3.1]{KT}.
	
	\begin{prop} \label{p:harmonic-mk}
		Let an MMD space $(\ambient,d,\refmeas,\form,\domain)$ and a uniform domain $\unifdom$ in $(\ambient,d)$
		satisfy Assumption \ref{a:hkecdcbhp}. Then for all $x,x_0 \in \unifdom$,
		\begin{equation} \label{e:harmonic-mk}
			\frac{d\hmeas{\unifdom}{x}}{d \hmeas{\unifdom}{x_0}}(\cdot) = \martinker^{\unifdom}_{x_0} (x,\cdot).
		\end{equation}
	\end{prop}
	
	\begin{proof}
		Let $\xi \in \partial \unifdom$, $r \in (0,\diam(\unifdom)/4)$, set
		$A \coloneqq A_{\xi,r} \coloneqq B(\xi,r)\cap \partial \unifdom$,
		$B \coloneqq B_{\xi,r} \coloneqq B(\xi,2r)^c \cap \overline{\unifdom}$
		and let $e_{A,B}$ denote the equilibrium potential for $A$ with respect to
		$(\overline{\unifdom},\refmeas|_{\overline{\unifdom}},\formref,\domainref)$
		with Dirichlet boundary condition on $B$. By Proposition \ref{p:laplacian-green}
		and Lemma \ref{l:eqmeas} there exist measures $\lambda^1_{A,B}, \lambda^0_{A,B}$
		supported on $\overline{A}$ and $\overline{\unifdom} \cap S(\xi,2r)$ respectively such that
		\begin{align} \label{e:mk1}
			0 < \int_{\partial \unifdom} \wt{e}_{A,B} \, d\hmeas{\unifdom}{x} &= - \formref(\gren{\unifdom}(x,\cdot), e_{A,B}) \nonumber \\
			&= - \biggl( \int_{\overline{A}} \gren{\unifdom}(x,y)\,d\lambda^1_{A,B}(y) - \int_{\overline{\unifdom} \cap S(\xi,2r)} \gren{\unifdom}(x,y)\,d\lambda^0_{A,B}(y) \biggr) \nonumber \\
			&= \int_{\overline{\unifdom} \cap S(\xi,2r)} \gren{\unifdom}(x,y)\,d\lambda^0_{A,B}(y).
		\end{align}
		Taking ratio of \eqref{e:mk1} for $x$ and for $x_0$ in place of $x$, we obtain 
		\begin{align} \label{e:mk2}
			\abs{\frac{\int_{\partial \unifdom} \wt{e}_{A,B} \, d\hmeas{\unifdom}{x}}{\int_{\partial \unifdom} \wt{e}_{A,B} \, d\hmeas{\unifdom}{x_0}}-\martinker^\unifdom_{x_0}(x,\xi)}
			&= \abs{ \frac{\int_{\overline{\unifdom} \cap S(\xi,2r)} \gren{\unifdom}(x,y)\,d\lambda^0_{A,B}(y)}{\int_{\overline{\unifdom} \cap S(\xi,2r)}  \gren{\unifdom}(x_0,y)\,d\lambda^0_{A,B}(y)}-\martinker^\unifdom_{x_0}(x,\xi)} \nonumber \\
			&\le \abs{ \frac{\int_{\overline{\unifdom} \cap S(\xi,2r)} \gren{\unifdom}(x_0,y) (\martinker^\unifdom_{x_0}(x,y)-\martinker^\unifdom_{x_0}(x,\xi))\,d\lambda^0_{A,B}(y)}{\int_{\overline{\unifdom} \cap S(\xi,2r)} \gren{\unifdom}(x_0,y)\,d\lambda^0_{A,B}(y)}} \nonumber \\
			&\le \frac{\int_{\overline{\unifdom} \cap S(\xi,2r)} \gren{\unifdom}(x_0,y) \abs{\martinker^\unifdom_{x_0}(x,y)-\martinker^\unifdom_{x_0}(x,\xi)}\,d\lambda^0_{A,B}(y)}{\int_{\overline{\unifdom} \cap S(\xi,2r)} \gren{\unifdom}(x_0,y)\,d\lambda^0_{A,B}(y)}.
		\end{align}
		By the boundary H\"older regularity of the Martin kernel implied by \ref{eq:BHP} and Lemma \ref{l:bhpholder},
		there exist $C_1, A_1 \in (1,\infty)$ and $\gamma \in (0,\infty)$ such that for all $x_0,x \in \unifdom$,
		all $\xi \in \partial \unifdom$, all $0< r < A_1^{-1}(d(x_0,\xi)\wedge d(x,\xi))$
		and all $y \in \overline{\unifdom}\cap B(\xi,r)$, we have
		\begin{equation} \label{e:mk3}
			\abs{\martinker^\unifdom_{x_0}(x,y)-\martinker^\unifdom_{x_0}(x,\xi)} \le C_1 \martinker^\unifdom_{x_0}(x,\xi) \biggl(\frac{r}{d(x_0,\xi)\wedge d(x,\xi)} \biggr)^\gamma.
		\end{equation}
		On the other hand, since $\hmeas{\unifdom}{x} \ll \hmeas{\unifdom}{x_{0}}$ by Lemma \ref{l:harmonicm}-\eqref{it:hmeas-AC}
		and $\hmeas{\unifdom}{x_{0}}$ satisfies \eqref{e:asympdoub} by Corollary \ref{c:asdouble}, it follows
		from Lemma \ref{l:differentiation} that for $\hmeas{\unifdom}{x_{0}}$-a.e.\ $\xi \in \partial \unifdom$,
		\begin{equation}\label{eq:differentiation-hmeas-apply}
			\lim_{r\downarrow 0}\frac{\int_{\partial \unifdom} \widetilde{e}_{A_{\xi,r},B_{\xi,r}} \, d\hmeas{\unifdom}{x}}{\int_{\partial \unifdom} \widetilde{e}_{A_{\xi,r},B_{\xi,r}} \, d\hmeas{\unifdom}{x_0}}
			=\frac{d\hmeas{\unifdom}{x}}{d\hmeas{\unifdom}{x_{0}}}(\xi).
		\end{equation}
		By \eqref{e:mk3} and \eqref{eq:differentiation-hmeas-apply}, for $\hmeas{\unifdom}{x_{0}}$-a.e.\ $\xi \in \partial \unifdom$
		we can let $r \downarrow 0$ in \eqref{e:mk2} to get
		$(d\hmeas{\unifdom}{x}/d\hmeas{\unifdom}{x_{0}})(\xi)=\martinker^\unifdom_{x_0}(x,\xi)$,
		completing the proof.
	\end{proof}
	
	% In particular, the harmonic measure is doubling at \emph{many scales}.
	% The lower bound is easier prove but relies on capacity density condition. The approach is obtain lower bound is to first hit the ball $B_U(\xi_{c_1r}, c_2r)$ whose hitting probability follows from bounds on capacity, then there is a good chance of exiting from $B_U(\xi,r)$ using the $\Delta$-regularity estimate (this uses Capacity density condition). The upper bound seems to be more general (does not rely on CDC) but  its proof is involved and is based on the box argument \cite[Lemma 3.6]{AH}.
	
	\subsection{The elliptic measure at infinity on unbounded domains} \label{ssec:elliptic-meas}
	
	On unbounded uniform domain the harmonic measure need not be doubling.
	For instance if $\partial U$ is unbounded and connected, then due to \cite[Exercise 13.1]{Hei} every doubling measure on $\partial U$ must necessarily be an infinite measure,
	and in particular there is no doubling \emph{probability} measure on $\partial U$.
	Nevertheless, as we will see there is a canonical \emph{doubling} measure on $\partial U$ obtained as a limit of scaled harmonic measures $\hmeas{\unifdom}{x}$ as $x \to \infty$.
	Propositions \ref{p:hprofile} and \ref{p:laplacian-green} suggest considering the limit of scaled harmonic measures $\gren{U}(x_0,x)^{-1} \hmeas{\unifdom}{x}\big|_{\overline{\unifdom}}$ as $x \to \infty$.
	Following \cite[Lemma 3.5]{BTZ}, we call this limit, denoted as $\nu^\unifdom_{x_0}$ below, the \textbf{$\form$-elliptic measure at infinity} of $\unifdom$ with base point $x_{0}$.
	Alternatively, the distributional Laplacian of the harmonic profile defines the elliptic measure at infinity on the boundary $\partial \unifdom$  as shown below. 
	
	\begin{prop}[Elliptic measure at infinity] \label{p:emeas}
		Let an MMD space $(\ambient,d,\refmeas,\form,\domain)$ and a uniform domain $\unifdom$ in $(\ambient,d)$
		satisfy Assumption \ref{a:hkecdcbhp}, and assume that $\unifdom$ is unbounded.
		Let $x_0 \in \unifdom$, and let $\{ x_n \}_{n \in \mathbb{N}} \subset U \setminus \{x_0\}$
		be any sequence satisfying $\lim_{n \to \infty} d(x_0,x_n) =\infty$.
		Let $\hprof{x_0}(\cdot)=\lim_{n \to \infty} \martinker^\unifdom_{x_0}(\cdot,x_n)$ denote the $\form$-harmonic profile of $\unifdom$ with $\hprof{x_{0}}(x_{0})=1$.
		Then the sequence of the measures $\nu_{n} \coloneqq \gren{\unifdom}(x_{0},x_{n})^{-1}\hmeas{\unifdom}{x_n}\big|_{\overline{\unifdom}}$
		converges in total variation on any compact subset of $\overline{\unifdom}$ to an $\formref$-smooth
		Radon measure $\nu^{\unifdom}_{x_0}$ on $\overline{\unifdom}$ with $\nu^\unifdom_{x_0}(\unifdom)=0$ and
		\begin{equation} \label{e:em0}
			\formref(\hprof{x_0}, u) = - \int_{\partial \unifdom} \wt{u} \, d\nu^\unifdom_{x_0}
		\end{equation}
		for all $u \in \domainref\cap L^{\infty}(\overline{\unifdom},\refmeas|_{\overline{\unifdom}})$
		such that $\supp_{\refmeas|_{\overline{\unifdom}}}[u]$ is compact. In particular,
		the measure $\nu^\unifdom_{x_0}$ does not depend on the choice of the sequence $(x_n)_{n \ge 1}$,
		and $\nu^{\unifdom}_{y} = (\hprof{x_0}(y))^{-1} \nu^{\unifdom}_{x_0}$ for any $y \in \unifdom$.
		Moreover, the following hold:
		\begin{enumerate}[\rm(a)]\setlength{\itemsep}{0pt}\vspace{-5pt}
			\item\label{it:emeas-AC} The measures $\nu^{\unifdom}_{x_0}$ and $\hmeas{\unifdom}{x_0}\big|_{\overline{\unifdom}}$ are mutually absolutely continuous.
			Furthermore, the Radon--Nikodym derivative $\frac{d\nu^\unifdom_{x_0}}{d \hmeas{\unifdom}{x_0}}\colon \partial \unifdom \to (0,\infty)$
			can be chosen to be a strictly positive continuous function satisfying the following estimates: there exist $C,A \in (1,\infty)$ such that
			for all $\xi \in \partial \unifdom$, all $R \in (0,A^{-1}d(x_{0},\xi))$ and all $\eta \in \partial \unifdom \cap B(\xi,R)$,
			\begin{equation}
				\label{e:embnd}
				C^{-1} \frac{\hprof{x_0}(\xi_{R})}{\gren{\unifdom}(x_0,\xi_{R})} \le	\frac{d\nu^\unifdom_{x_0}}{d\hmeas{\unifdom}{x_0}}(\eta) \le C  \frac{\hprof{x_0}(\xi_{R})}{\gren{\unifdom}(x_0,\xi_{R})}.
			\end{equation}
			\item\label{it:emeas-quasi-support} $\partial U$ is an $\formref$-quasi-support of $\nu^\unifdom_{x_0}$.
			\item\label{it:emeas-estimate} There exists $C \in (1,\infty)$ such that for all $\xi \in \partial \unifdom$ and all $R \in (0,\infty)$,
			\begin{equation} \label{e:em1}
				C^{-1} \hprof{x_0}(\xi_{R}) \Capa_{B(\xi,2R)}(B(\xi,R)) \le \nu^\unifdom_{x_0}(B(\xi,R) \cap \partial \unifdom ) \le C \hprof{x_0}(\xi_{R}) \Capa_{B(\xi,2R)}(B(\xi,R)).
			\end{equation}
			In particular, $\supp_{\overline{\unifdom}}[\nu^\unifdom_{x_0}] = \partial \unifdom$
			and $(\partial \unifdom, d, \nu^\unifdom_{x_0})$ satisfies \hyperlink{VD}{\textup{VD}}.
		\end{enumerate}
	\end{prop}
	
	\begin{proof} 
		Let   $u \in \domainref\cap L^{\infty}(\overline{\unifdom},\refmeas|_{\overline{\unifdom}})$
		be such that $\supp_{\refmeas|_{\overline{\unifdom}}}[u]$ is compact, and let
		$\{ x_n \}_{n \in \mathbb{N}} \subset \unifdom \setminus \{x_0\}$ be any sequence
		satisfying $\lim_{n \to \infty} d(x_0,x_n) =\infty$.
		Then there exist $\phi \in \domain \cap \contfunc_{\mathrm{c}}(\ambient)$ and $N \in \mathbb{N}$
		such that $x_n \notin \supp_{\ambient}[\phi]$ for all $n \ge N$, $\supp_{\refmeas}[u] \subset \supp_{\ambient}[\phi]$
		and $\phi \equiv 1$ on a neighborhood of $\supp_{\refmeas}[u]$.
		By Proposition \ref{p:laplacian-green} and the strong locality of $(\formref,\domainref)$, we have 
		\begin{equation} \label{e:em2}
			\int_{\partial \unifdom} \widetilde{u} \, d\nu_n=-\formref\biggl( \frac{\gren{\unifdom}(\cdot,x_n)}{\gren{\unifdom}(x_0,x_n)}, u \biggr) = -\formref\biggl( \phi(\cdot)\frac{\gren{\unifdom}(\cdot,x_n)}{\gren{\unifdom}(x_0,x_n)}, u \biggr),
		\end{equation}
		where we adopt the convention of extending $\gren{U}(\cdot,x_n)$ by $0$ on $U^c$.
		Similarly extending $\hprof{x_0}$ as $0$ on $U^c$, we see from Proposition \ref{p:hprofile}
		and Remark \ref{rmk:harm-conv}-\eqref{it:rmk:harm-conv-dbdry} that
		\begin{equation}\label{e:em3}
			\lim_{N\leq n \to \infty} \formref_1 \biggl( \phi(\cdot)\frac{\gren{\unifdom}(\cdot,x_n)}{\gren{\unifdom}(x_0,x_n)}-\phi \hprof{x_0}, \phi(\cdot)\frac{\gren{\unifdom}(\cdot,x_n)}{\gren{\unifdom}(x_0,x_n)}- \phi \hprof{x_0} \biggr)=0.
		\end{equation}
		Combining \eqref{e:em2}, \eqref{e:em3} and by the strong locality of $(\formref,\domainref)$, we obtain
		\begin{equation} \label{e:em2a}
			\lim_{n \to \infty}\int_{\partial \unifdom} \widetilde{u} \, d\nu_n
			= -\formref( \phi(\cdot) \hprof{x_0}, u ) = -\formref( \hprof{x_0}, u )
		\end{equation}
		for all $u \in \domainref\cap L^{\infty}(\overline{\unifdom},\refmeas|_{\overline{\unifdom}})$ such that
		$\supp_{\refmeas|_{\overline{\unifdom}}}[u]$ is compact. By Proposition \ref{p:harmonic-mk} and \eqref{e:naimb},
		\begin{equation} \label{e:em4}
			\frac{d\nu_n}{d\hmeas{\unifdom}{x_0}}(\cdot)
			= \frac{1}{\gren{\unifdom}(x_0,x_n)} \frac{d\hmeas{\unifdom}{x_n}}{d\hmeas{\unifdom}{x_0}}(\cdot)
			\overset{\eqref{e:harmonic-mk}}{=} \frac{\martinker^\unifdom_{x_0}(x_n,\cdot)}{\gren{\unifdom}(x_0,x_n)}
			\overset{\eqref{e:naimb}}{=} \naimker^\unifdom_{x_0}(x_n,\cdot).
		\end{equation}
		By \eqref{e:naimbnd} and the joint continuity of $\naimker^{\unifdom}_{x_0}$,
		the sequence $\naimker^{\unifdom}_{x_0}(x_n,\cdot)$ is uniformly bounded on every compact subset of $\partial \unifdom$.
		Similarly by \eqref{e:naimequi} and the joint continuity of $\naimker^{\unifdom}_{x_0}$,
		the sequence $\naimker^{\unifdom}_{x_0}(x_n,\cdot)$ is equicontinuous on every compact subset of $\partial \unifdom$.
		Therefore by the Arzel\`{a}--Ascoli theorem, we can choose a subsequence $\{x_{n_{k}}\}_{k\in\mathbb{N}}$
		so that $\naimker^\unifdom_{x_0}(x_{n_{k}},\cdot)$ converges uniformly on any compact subset of $\partial \unifdom$
		to a continuous function $\naimker^{\unifdom}_{x_0}(\infty,\cdot)\colon \partial \unifdom \to [0,\infty)$.
		Recalling that $\hmeas{\unifdom}{x_0}\big|_{\overline{\unifdom}}$ is $\formref$-smooth and
		$\hmeas{\unifdom}{x_0}(\unifdom)=0$ by Lemma \ref{l:harmonicm}-\eqref{it:hmeas-quasi-support},\eqref{it:hmeas-smooth-support},
		we can thus define an $\formref$-smooth Radon measure $\nu^{\unifdom}_{x_0}$ on $\overline{\unifdom}$ by
		\begin{equation}\label{e:em5}
			\nu^{\unifdom}_{x_0}(d\xi) \coloneqq \naimker^{\unifdom}_{x_0}(\infty,\xi)\,\hmeas{\unifdom}{x_0}\big|_{\overline{\unifdom}}(d\xi),
		\end{equation}
		so that $\nu^{\unifdom}_{x_0}(\unifdom)=0$, the measures
		$\nu_{n_{k}}=\gren{\unifdom}(x_{0},x_{n_{k}})^{-1}\hmeas{\unifdom}{x_{n_{k}}}\big|_{\overline{\unifdom}}$ converge to
		$\nu^{\unifdom}_{x_0}$ in total variation on any compact subset of $\overline{\unifdom}$,
		and from \eqref{e:em2a} we obtain \eqref{e:em0} for all $u \in \domainref\cap L^{\infty}(\overline{\unifdom},\refmeas|_{\overline{\unifdom}})$
		such that $\supp_{\refmeas|_{\overline{\unifdom}}}[u]$ is compact. Combining \eqref{e:em0}
		with the uniqueness of $\hprof{x_{0}}$ in Proposition \ref{p:hprofile}, \cite[Exercise 1.4.1]{FOT}
		applied to the regular Dirichlet space $(\overline{\unifdom},\refmeas|_{\overline{\unifdom}},\formref,\domainref)$,
		and the outer and inner regularity of $\nu^{\unifdom}_{x_0}$ from \cite[Theorem 2.18]{Rud},
		we conclude that $\nu^{\unifdom}_{x_0}$ is independent of particular choices of $\{x_{n}\}_{n\in\mathbb{N}}$
		and its subsequence $\{x_{n_{k}}\}_{k\in\mathbb{N}}$ in the above argument, and so is
		$\naimker^{\unifdom}_{x_0}(\infty,\cdot)$ by \eqref{e:em5}, its continuity on $\partial \unifdom$ and
		$\supp_{\overline{\unifdom}}[\hmeas{\unifdom}{x_{0}}\big|_{\overline{\unifdom}}]=\partial \unifdom$ from Corollary \ref{c:asdouble}.
		Since the sequence $\{x_{n}\}_{n\in\mathbb{N}}$ in these results can be replaced with
		\emph{any} subsequence of $\{x_{n}\}_{n\in\mathbb{N}}$, it follows that
		$\{\naimker^{\unifdom}_{x_0}(x_{n},\cdot)\}_{n\in\mathbb{N}}$ converges to
		$\naimker^{\unifdom}_{x_0}(\infty,\cdot)$ uniformly on any compact subset of $\partial \unifdom$
		and hence that $\{\nu_{n}\}_{n \in \mathbb{N}}$ converges to $\nu^{\unifdom}_{x_0}$
		in total variation on any compact subset of $\overline{\unifdom}$ (without passing to a subsequence).
		
		Moreover, by Lemma \ref{l:uniqueprofile} and \eqref{e:em0}, we have 
		\begin{equation} \label{e:em6}
			\hprof{y} = (\hprof{x_0}(y))^{-1} \hprof{x_0}
			\quad \textrm{and} \quad
			\nu^{\unifdom}_{y} = (\hprof{x_0}(y))^{-1} \nu^{\unifdom}_{x_0}
			\quad \textrm{for all $y \in \unifdom$.}
		\end{equation}
		\begin{enumerate}[\rm(a)]\setlength{\itemsep}{0pt}
			\item Letting $A \in (1,\infty)$ be as in Lemma \ref{l:naimholder},
			by \eqref{e:em4}, \eqref{e:naimbnd} and the joint continuity of $\naimker^\unifdom_{x_0}(x_{n},\cdot)$,
			for all $\xi \in \partial \unifdom$, all $R \in (0,(2A)^{-1}d(\xi,x_{0}))$ and all $\eta \in \partial \unifdom \cap B(\xi,R)$ we have
			\begin{equation*}
				\frac{d\nu_n}{d\hmeas{\unifdom}{x_0}}(\eta)
				\overset{\eqref{e:em4}}{=} \naimker^\unifdom_{x_0}(x_n,\eta)
				\overset{\eqref{e:naimbnd}}{\asymp}\naimker^\unifdom_{x_0}(x_n,\xi_R)
				= \frac{\gren{U}(\xi_{R},x_n)}{\gren{U}(x_0,x_n)} \frac{1}{\gren{U}(x_0,\xi_{R})}
			\end{equation*}
			for all $n$ sufficiently large.
			Letting $n \to \infty$ and using Proposition \ref{p:hprofile}, we obtain the estimate \eqref{e:embnd}. 
			Since $\naimker^\unifdom_{x_0}(\infty,\cdot)$ is strictly positive on $\partial \unifdom$,
			we conclude from \eqref{e:em5} and $\hmeas{\unifdom}{x_0}(\unifdom)=0$ that $\nu^\unifdom_{x_0}$
			and $\hmeas{\unifdom}{x_0}\big|_{\overline{\unifdom}}$ are mutually absolutely continuous.
			\item By the mutual absolute continuity of $\nu^{\unifdom}_{x_0}$ and $\hmeas{\unifdom}{x_0}\big|_{\overline{\unifdom}}$, they have
			the same $\formref$-quasi-supports. Hence the desired conclusion follows from Lemma \ref{l:harmonicm}-\eqref{it:hmeas-quasi-support}.
			\item For $\xi \in \partial \unifdom$ and $R \in (0,\infty)$, we choose $y \in \unifdom \setminus B(\xi,2AR)$ and estimate
			\begin{equation} \label{e:em7}
				\nu^{\unifdom}_{y}(B(\xi,R)) \overset{\eqref{e:embnd}}{\asymp} \hmeas{\unifdom}{y}(B(\xi,R)) \frac{\hprof{y}(\xi_R)}{\gren{\unifdom}(y,\xi_R)}
				\overset{\eqref{e:harmonicm}}{\asymp} \hprof{y}(\xi_{R}) \Capa_{B(\xi,2R)}(B(\xi,R)).
			\end{equation}
			The estimate \eqref{e:em1} follows from \eqref{e:em6} and \eqref{e:em7}.
			
			The volume doubling property of $\nu^\unifdom_{x_0}$ follows from \eqref{e:em1} along with Proposition \ref{p:hprofile}, Lemma \ref{l:hchain}, and \cite[Lemma 5.23]{BCM}.
			\qedhere\end{enumerate}
	\end{proof} 
	
	\begin{remark} \label{r:emeas}
		The above proof of Proposition \ref{p:emeas} implies that for any $\xi \in \partial \unifdom$ the limit
		\begin{equation*}
			\lim_{\unifdom \times \unifdom \ni (z,x) \to (\infty,\xi)} \naimker^\unifdom_{x_0}(z,x)
		\end{equation*}
		exists, and therefore this limit was suggestively denoted as $\naimker^\unifdom_{x_0}(\infty,\xi)$ in the proof.
	\end{remark}
	
	It is natural to ask whether unbounded uniform domains satisfying \ref{eq:CDC} have unbounded boundaries.
	This is not true in general since the positive half-line $U=(0,\infty)$ is a uniform domain in $\bR$ satisfying \ref{eq:CDC}
	for the Brownian motion on $\bR$. On the other hand, such examples do not occur for the Brownian motion on $\bR^{\dimeuc}$ with $\dimeuc \ge 2$.
	More generally, unbounded uniform domains satisfying \ref{eq:CDC} have unbounded boundaries under the additional assumption of
	the capacity non-decreasing condition (recall Definition \ref{d:cnd}), as follows.
	
	\begin{lem}  \label{l:unbdd}
		Let an MMD space $(\ambient,d,\refmeas,\form,\domain)$ and a uniform domain $\unifdom$ in $(\ambient,d)$
		satisfy Assumption \ref{a:hkecdcbhp}, and assume that $\unifdom$ is unbounded and that
		$(\ambient,d,\refmeas,\form,\domain)$ satisfies the capacity non-decreasing condition.
		Then $\partial \unifdom$ is unbounded, i.e., $\diam(\partial \unifdom)=\infty$.
	\end{lem}
	
	\begin{proof}
		Let $x_0 \in \unifdom$ be fixed and let $\nu^\unifdom_{x_0}$ denote the $\form$-elliptic measure
		at infinity of $\unifdom$ with base point $x_{0}$ as given in Proposition \ref{p:emeas}.
		By Lemma \ref{l:Delta}-\eqref{it:Delta-hfunc}, \eqref{e:cnd1} and
		Proposition \ref{p:emeas}-\eqref{it:emeas-estimate}, there exists $A \in (1,\infty)$ such that
		\begin{equation*}
			\nu^\unifdom_{x_0}(B(\xi,R)) < \frac{1}{2} 	\nu^\unifdom_{x_0}(B(\xi,AR))
			\quad \textrm{for all $\xi \in \partial \unifdom$ and all $R \in (0,\infty)$.}
		\end{equation*}
		This implies that $\partial \unifdom \cap ( B(\xi,AR) \setminus B(\xi,R) ) \not= \emptyset$
		for all $\xi \in \partial \unifdom$ and all $R \in (0,\infty)$,
		which in turn implies that $\partial \unifdom$ is unbounded.
	\end{proof}
	
	\section{The boundary trace process} \label{sec:boundary-trace}
	
	Throughout this section, we always assume that a scale function $\scdiff$,
	an MMD space $(\ambient,d,\refmeas,\form,\domain)$, a diffusion
	$\diff = (\Omega, \events, \{\diff_{t}\}_{t\in[0,\infty]},\{\lawdiff_{x}\}_{x \in \oneptcpt{\ambient}})$
	on $\ambient$, a uniform domain $\unifdom$ in $(\ambient,d)$, and a diffusion
	$\diffref=\bigl(\Omega^{\on{ref}},\events^{\on{ref}},\{\diffref_{t}\}_{t\in[0,\infty]},\{\lawref_{x}\}_{x\in\oneptcpt{\overline{\unifdom}}}\bigr)$
	on $\overline{\unifdom}$ satisfy Assumption \ref{a:hkecdcbhp}.
	
	\subsection{The boundary measure and the corresponding PCAF} \label{ssec:bdry-meas-pcaf}
	
	To define the boundary trace process, we choose a reference measure on the boundary $\partial \unifdom$ as given in the following definition.
	
	\begin{definition} \label{d:bdrymeas}
		If $\unifdom$ is bounded, we choose $x_{0}=\wh{\xi}_{\diam(\unifdom)/5}$ using
		Lemma \ref{l:xir}, where $\wh{\xi} \in \partial\unifdom$ is chosen arbitrarily.
		If $\unifdom$ is unbounded, we choose an arbitrary point $x_{0} \in \unifdom$.
		We define a Radon measure $\bdrymeas$ on $\overline{\unifdom}$ with
		$\supp_{\overline{\unifdom}}[\mu] \subset \partial \unifdom$ by
		\begin{equation} \label{e:defmu}
			\bdrymeas \coloneqq \begin{cases}
				\hmeas{\unifdom}{x_{0}}\big|_{\overline{\unifdom}} & \textrm{if $\unifdom$ is bounded,}\\
				\nu^{\unifdom}_{x_{0}} & \textrm{if $\unifdom$ is unbounded,}
			\end{cases}
		\end{equation}
		where $\hmeas{\unifdom}{x_{0}},\nu^{\unifdom}_{x_{0}}$ denote the $\form$-harmonic measure
		(Definition \ref{d:hmeas}, Lemma \ref{l:harmonicm}) and
		the $\formref$-elliptic measure at infinity (Proposition \ref{p:emeas}),
		respectively, of $\unifdom$ with base point $x_{0}$.
		
		In order to describe properties of $\bdrymeas$, we define
		$\wt{\scjump} \colon \partial \unifdom \times (0,\diam(\unifdom)/6)\to (0,\infty)$ by 
		\begin{equation} \label{e:defscale}
			\wt\scjump(\xi,r)=\begin{cases}
				\gren{\unifdom}(x_0,\xi_r) & \textrm{if $\unifdom$ is bounded,}\\
				\hprof{x_0}(\xi_r) & \textrm{if $\unifdom$ is unbounded},
			\end{cases}
		\end{equation}
		where $\xi_r$ is chosen as in Lemma \ref{l:xir}.
	\end{definition}
	
	Note that by \cite[Theorem 1.2]{GHL15} and \cite[Lemma 5.22]{BCM}, there exist $C,A \in (1,\infty)$ such that
	\begin{equation} \label{eq:bdry-cap-estimate}
		C^{-1} \frac{\refmeas(B(x,R))}{\scdiff(R)} \le \Capa_{B(x,2R)}(B(x,R)) \le C \frac{\refmeas(B(x,R))}{\scdiff(R)}
	\end{equation} 
	for all $x \in \ambient$ and all $R \in (0,\diam(\ambient)/A)$.
	Let us recall that the function $\wt\scjump$ is useful to estimate the measure $\bdrymeas$.
	Indeed, by Theorem \ref{t:hmeas}, Proposition \ref{p:emeas}-\eqref{it:emeas-estimate} and \eqref{eq:bdry-cap-estimate}, there exist $C,A \in (1,\infty)$ such that
	\begin{equation} \label{eq:bdry-scale-estimate}
		C^{-1} \frac{\refmeas(B(\xi,R))}{\scdiff(R)} \le \frac{\bdrymeas(B(\xi,R))}{\wt\scjump(\xi,R)} \le C \frac{\refmeas(B(\xi,R))}{\scdiff(R)}
		\quad \textrm{for all $\xi \in \partial \unifdom$ and $R \in (0,\diam(\unifdom)/A)$.}
	\end{equation}
	We record some basic estimates on  $\wt{\scjump}$ and show that $\wt{\scjump}$
	is comparable to a function $\scjump$ that has better continuity properties.
	
	\begin{lem} \label{l:scale}
		There exist $C_1,A_1 \in (1,\infty)$ and a regular scale function
		$\scjump \colon \partial \unifdom \times [0,\infty) \to [0,\infty)$ on $(\partial \unifdom,d)$
		with threshold $\diam(\unifdom)$ in the sense of Definition \ref{d:regscale} such that
		\begin{equation} \label{e:reg2}
			C_1^{-1} \wt{\scjump}(\xi,r) \le \scjump(\xi,r) \le C_1 \wt{\scjump}(\xi,r)
			\quad \textrm{for all $\xi \in \partial \unifdom$ and all $r \in (0,\diam(\unifdom)/A_1)$.}
		\end{equation}
	\end{lem}
	
	\begin{proof}
		First, we show that there exist $C,\beta_1,\beta_2 \in (0,\infty)$ and $A \in (4,\infty)$ such that
		for all $\eta,\xi \in \partial\unifdom$ and all $0<r \le R$ with $R \vee d(\xi,\eta)<\diam(\unifdom)/A$,
		\begin{equation} \label{e:Phireg}
			C^{-1} \Bigl( \frac{R}{d(\xi,\eta) \vee R} \Bigr)^{\beta_2} \Bigl( \frac{d(\xi,\eta) \vee R}{r} \Bigr)^{\beta_1}
			\le \frac{ \wt{\scjump}(\xi,R)}{\wt{\scjump}(\eta,r)} 
			\le C \Bigl( \frac{R}{d(\xi,\eta) \vee R} \Bigr)^{\beta_1} \Bigl( \frac{d(\xi,\eta) \vee R}{r} \Bigr)^{\beta_2}.
		\end{equation}
		Indeed, by Lemmas \ref{l:hchain} and \ref{l:Delta} and by the harmonicity and Dirichlet boundary conditions of
		$\gren{\unifdom}(x_0,\cdot)$ and $\hprof{x_0}$ in Propositions \ref{p:goodgreen}-\eqref{it:goodgreen-harm}, \ref{p:hprofile} and Lemma \ref{l:dbdy},
		there exist $C_1, C_2, A \in(1,\infty)$ and $\beta_1,\beta_2 \in(0,\infty)$ such that 
		\begin{equation} \label{e:rg1}
			C_1^{-1} \Bigl(\frac{R}{r}\Bigr)^{\beta_1}\le \frac{\wt{\scjump}(\xi,R)}{\wt{\scjump}(\xi,r)} \le C_1 \Bigl(\frac{R}{r}\Bigr)^{\beta_2} \quad \textrm{for all $\xi\in\partial\unifdom$ and $0<r \leq R<\diam(\unifdom)/A$,}
		\end{equation}
		and
		\begin{equation} \label{e:rg2}
			C_2^{-1} \le \frac{\wt\scjump(\xi,R\vee d(\xi,\eta))}{\wt\scjump(\eta,R\vee d(\xi,\eta))} \le C_2,  
		\end{equation}
		for all $\eta,\xi \in \partial\unifdom$ and $0<r \le R$ with $R \vee d(\xi,\eta)<\diam(\unifdom)/A$.
		The conclusion \eqref{e:Phireg} follows from \eqref{e:rg1} and \eqref{e:rg2} by using the expression
		\begin{equation*}
			\frac{\wt\scjump(\xi,R)}{\wt\scjump(\eta,r)}=\frac{\wt\scjump(\xi,R)}{\wt\scjump(\xi,R\vee d(\xi,\eta))}	\cdot\frac{\wt\scjump(\xi,R\vee d(\xi,\eta))}{\wt\scjump(\eta,R\vee d(\xi,\eta))}\cdot\frac{\wt\scjump(\eta,R\vee d(\xi,\eta))}{\wt\scjump(\eta,r)}.
		\end{equation*}
		By \eqref{e:rg1},  
		there exists $A_2 \in (1,\infty)$ such that for all $\xi \in \partial \unifdom$ and all $R \in (0,\diam(\unifdom)/A)$, 
		\begin{equation} \label{e:rg3}
			\wt{\scjump}(\xi,A_2^{-1}R) \le \frac{1}{2}\wt{\scjump}(\xi,R).
		\end{equation}
		
		Using \eqref{e:rg3}, we define $\scjump \colon \partial \unifdom \times [0,\infty) \to [0,\infty)$
		as follows: if $\unifdom$ is unbounded, we define 
		\begin{equation*}
			\scjump(\xi,A_2^k) \coloneqq \wt{\scjump}(\xi,A_2^k) \quad \mbox{for $\xi \in \partial \unifdom$ and $k \in \mathbb{Z}$,}
		\end{equation*}
		and extend $\scjump(\xi,\cdot)$ by piecewise linear interpolation to $[0,\infty)$ for each $\xi \in \partial \unifdom$.
		Using \eqref{e:Phireg} and \eqref{e:reg2}, we get the estimate \eqref{e:Psireg1} in Definition \ref{d:regscale}.
		The fact that $\scjump(\xi,\cdot)$ is a homeomorphism follows from \eqref{e:rg3}.
		This concludes the proof when $\unifdom$ is unbounded.
		
		If $\unifdom$ is bounded, we define
		\begin{equation*}
			\scjump(\xi,A_{2}^{k} (2A)^{-1} \diam(\unifdom)) \coloneqq 
			\begin{cases}
				\wt{\scjump}(\xi,A_{2}^{k} (2A)^{-1} \diam(\unifdom)) & \textrm{if $k\leq 0$,}\\
				A_{2}^{k\beta_{1}} \wt{\scjump}(\xi,(2A)^{-1} \diam(\unifdom)) &  \textrm{if $k > 0$} 
			\end{cases}
		\end{equation*}
		for $\xi \in \partial \unifdom$ and $k \in \mathbb{Z}$, and extend $\scjump(\xi,\cdot)$
		by piecewise linear interpolation to $[0,\infty)$ for each $\xi \in \partial \unifdom$.
		The conclusion follows from the same reasoning as the unbounded case.
	\end{proof}
	
	It will be convenient to use $\scjump$ in Lemma \ref{l:scale} instead of $\wt{\scjump}$ due to its better continuity property.
	So we set $\scjump$ to denote the function in Lemma \ref{l:scale} in the rest of this section.
	We apply the results in Subsection \ref{ss:capgood} to the measure $\mu$ to obtain the following proposition. 
	
	\begin{prop} \label{p:bdrymeas}
		%Let $(\ambient,d,\refmeas,\form,\domain),\unifdom,\formref,\diffref$ be as in Assumption \ref{a:hkecdcbhp}, and
		Let $\bdrymeas$ be the Radon measure on $\overline{\unifdom}$ defined in \eqref{e:defmu}. Then the following hold:
		\begin{enumerate}[\rm(a)]\setlength{\itemsep}{0pt}\vspace{-5pt}
			\item\label{it:bdrymeas-smooth} $\supp_{\overline{\unifdom}}[\bdrymeas] = \partial \unifdom$,
			and there exist $C_0,A_0,A_1 \in (1,\infty)$ such that
			\begin{equation} \label{eq:bdrymeas-cgood}
				\mspace{-15mu} C_{0}^{-1} \scjump(x,r)
				\leq \frac{\bdrymeas(B(x,r))}{\Capa_{B(x,A_{0}r)}(B(x,r))}
				\leq C_{0} \scjump(x,r)
				\quad \textrm{for all $(x,r) \in \partial \unifdom \times (0,\diam(\unifdom)/A_1)$.}
			\end{equation}
			In particular, $\bdrymeas$ is $\formref$-capacity good and $\formref$-smooth in the strict sense.
			\item\label{it:bdrymeas-pcaf-support} Let $\bdrypcaf=\{\bdrypcaf_{t}\}_{t\in[0,\infty)}$
			be a PCAF in the strict sense of $\diffref$ with Revuz measure $\bdrymeas$\textup{(,
				which exists by \eqref{it:bdrymeas-smooth} and \cite[Theorem 5.1.7]{FOT})}.
			Then the support of $\bdrypcaf$ is $\partial \unifdom$, i.e.,
			\begin{equation} \label{eq:bdrymeas-pcaf-support}
				\partial \unifdom = \bigl\{x \in \ol{\unifdom} \bigm| \lawref_{x}[ \textrm{$\bdrypcaf_{t} > 0$ for any $t \in (0,\infty)$} ] = 1 \bigr\}.
			\end{equation}
			In particular, $\partial \unifdom$ is an $\formref$-quasi-support of $\bdrymeas$.
		\end{enumerate}
	\end{prop}
	
	\begin{proof}
		By \eqref{eq:bdry-cap-estimate}, \eqref{eq:bdry-scale-estimate} and Lemma \ref{l:scale} we obtain
		$\supp_{\overline{\unifdom}}[\bdrymeas] = \partial \unifdom$ and \eqref{eq:bdrymeas-cgood}.
		In particular, $\bdrymeas$ is an $\formref$-capacity good Borel measure on $\overline{\unifdom}$
		in view of Theorem \ref{thm:hkeunif}-\eqref{it:hkeunif} and Definition \ref{d:goodmeas}
		and is therefore $\formref$-smooth in the strict sense by Lemma \ref{l:strict}, and
		\eqref{it:bdrymeas-pcaf-support} follows from Proposition \ref{prop:goodpcafsupp}.
	\end{proof}
	
	\begin{remark} \label{r:nonpolar}
		By the estimate in \eqref{e:sm5} along with \cite[Theorems 2.1.6 and 4.4.3-(ii)]{FOT},
		for the MMD space $(\overline{\unifdom}, d, \refmeas|_{\overline{\unifdom}},\formref, \domainref)$ we have
		\begin{equation} \label{eq:bdry-nonpolar}
			\Capa_1^{\on{ref}}(B(\xi,r)\cap \partial \unifdom)> 0 \quad \mbox{for all $\xi \in \partial \unifdom$ and all $r \in (0,\infty)$,}
		\end{equation}
		where $\Capa_1^{\on{ref}}(\cdot)$ denotes the $1$-capacity with respect to $(\overline{\unifdom},\refmeas|_{\overline{\unifdom}},\formref,\domainref)$.
	\end{remark}
	
	\subsection{The Doob--Na\"im formula} \label{ssec:doob-naim}
	
	Now we define the trace process and Dirichlet form on the boundary
	$\partial \unifdom$ as follows. Recall from Assumption \ref{a:hkecdcbhp} that
	$\diffref=\bigl(\Omega^{\on{ref}},\events^{\on{ref}},\{\diffref_{t}\}_{t\in[0,\infty]},\{\lawref_{x}\}_{x\in\oneptcpt{\overline{\unifdom}}}\bigr)$
	is a diffusion on $\overline{\unifdom}$ as in Assumption \ref{a:feller}
	for the MMD space $(\overline{\unifdom},d,\refmeas|_{\overline{\unifdom}},\formref,\domain(\unifdom))$.
	
	\begin{definition}[The boundary trace process and Dirichlet form] \label{d:bdry-trace}
		Set $\oneptcptp{\partial \unifdom} \coloneqq \partial \unifdom \cup \{\cemetery\}$, and
		let $\bdrymeas$ be the Radon measure on $\overline{\unifdom}$ defined in \eqref{e:defmu}.
		\begin{enumerate}[\rm(a)]\setlength{\itemsep}{0pt}\vspace{-5pt}
			\item\label{it:bdry-trace-process} Let
			$\minaugfilt^{\on{ref}}_{*}=\{\minaugfilt^{\on{ref}}_{t}\}_{t\in[0,\infty]}$
			denote the minimum augmented admissible filtration of $\diffref$,
			$\zeta^{\on{ref}}$ the life time of $\diffref$, and
			$\{ \shiftdiff^{\on{ref}}_{t} \}_{ t \in [0,\infty] }$ the shift operators of $\diffref$.
			Let $\bdrypcaf=\{\bdrypcaf_{t}\}_{t\in[0,\infty)}$ be a PCAF in the strict sense of $\diffref$
			with Revuz measure $\bdrymeas$ as considered in Proposition \ref{p:bdrymeas}-\eqref{it:bdrymeas-pcaf-support},
			with a defining set $\Lambda \in \minaugfilt^{\on{ref}}_{0}$ such that
			$\bdrypcaf_{t}(\omega)=0$ for any $(t,\omega)\in[0,\infty)\times(\Omega^{\on{ref}}\setminus\Lambda)$
			and $\{\zeta^{\on{ref}}=0\} \subset \Lambda$. Recalling \eqref{eq:time-changed-process} and
			\eqref{eq:time-changed-process-Hunt-sample-space}, we define the \textbf{boundary trace process}
			$\diffreftr=\bigl(\widecheck{\Omega}^{\on{ref}},\widecheck{\events}^{\on{ref}},\{\diffreftr_{t}\}_{t\in[0,\infty]},\{\lawref_{\xi}\}_{\xi \in \oneptcptp{\partial \unifdom}}\bigr)$
			of $\diffref$ on $\partial \unifdom$ as the time-changed process of $\diffref$ by $\bdrypcaf$,
			given for $(t,\omega)\in[0,\infty]\times\Omega^{\on{ref}}$ by
			\begin{align}
				\tau_{t}(\omega) & \coloneqq \inf\{ s \in (0,\infty) \mid \bdrypcaf_{s}(\omega)>t\}, \mspace{32mu}
				\diffreftr_{t}(\omega) \coloneqq \diffref_{\tau_t(\omega)}(\omega), \mspace{32mu}
				\widecheck{\zeta}(\omega) \coloneqq \bdrypcaf_{\infty}(\omega), \nonumber \\
				\widecheck{\Omega}^{\on{ref}}& \coloneqq \Lambda \cap \bigl\{ \textrm{$\diffreftr_{s} \in \oneptcptp{\partial \unifdom}$ for any $s\in[0,\infty)$} \bigr\} \cap \Bigl(\bigl\{\widecheck{\zeta}\in\{0,\infty\}\bigr\}\cup\Bigl\{\lim_{s\to\infty}\diffref_{s}=\cemetery\Bigr\}\Bigr), \nonumber \\
				\widecheck{\events}^{\on{ref}}& \coloneqq \minaugfilt^{\on{ref}}_{\infty}\big|_{\widecheck{\Omega}^{\on{ref}}}, \qquad
				\widecheck{\shiftdiff}^{\on{ref}}_{t}(\omega) \coloneqq \shiftdiff^{\on{ref}}_{\tau_{t}(\omega)}(\omega),
				\label{eq:bdry-trace-process}
			\end{align}
			where $\bdrypcaf_{\infty}(\omega) \coloneqq \lim_{s\to\infty}\bdrypcaf_{s}(\omega)$, so that by
			Propositions \ref{p:bdrymeas}-\eqref{it:bdrymeas-smooth} and \ref{prop:greeninv}-\eqref{it:time-change-Hunt-AC},
			$\diffreftr$ is a $\bdrymeas$-symmetric Hunt process on $\partial \unifdom$
			with life time $\widecheck{\zeta}$ and shift operators $\bigl\{\widecheck{\shiftdiff}^{\on{ref}}_{t}\bigr\}_{t\in[0,\infty]}$.
			\item\label{it:bdry-trace-Dirichlet-form} Recalling Definition \ref{d:traceDF}, we define
			\begin{equation} \label{eq:bdry-trace-domain}
				\domaintr \coloneqq \biggl\{ \widetilde{u}|_{\partial \unifdom} \biggm| \textrm{$u \in \domainref_{e}$, $\int_{\partial \unifdom} \widetilde{u}^{2} \, d\bdrymeas < \infty$} \biggr\},
			\end{equation}
			where $\widetilde{u}$ denotes any $\formref$-quasi-continuous
			$\refmeas|_{\overline{\unifdom}}$-version of $u$ and we identify functions
			that coincide $\formref$-q.e.\ on $F$; since, for each $u,v \in \domainref_{e}$,
			$\widetilde{u}=\widetilde{v}$ $\formref$-q.e.\ on $\partial \unifdom$ if and only if
			$\widetilde{u}=\widetilde{v}$ $\bdrymeas$-a.e.\ on $\overline{\unifdom}$ by \cite[Theorem 3.3.5]{CF},
			and since $\supp_{\overline{\unifdom}}[\bdrymeas]=\partial \unifdom$
			by Proposition \ref{p:bdrymeas}-\eqref{it:bdrymeas-smooth}, we can canonically
			consider $\domaintr$ as a linear subspace of $L^{2}(\partial \unifdom,\bdrymeas)$.
			Then we further define a non-negative definite symmetric bilinear form
			$\formtr \colon \domaintr \times \domaintr \to \mathbb{R}$ by
			\begin{equation} \label{eq:bdry-trace-form}
				\formtr(\widetilde{u}|_{\partial \unifdom},\widetilde{v}|_{\partial \unifdom}) \coloneqq \formref\bigl( H^{\on{ref}}_{\partial \unifdom} \widetilde{u}, H^{\on{ref}}_{\partial \unifdom} \widetilde{v} \bigr)
				\quad \textrm{for $u,v \in \domainref_{e}$ with $\widetilde{u}|_{\partial \unifdom},\widetilde{v}|_{\partial \unifdom} \in \domaintr$,}
			\end{equation}
			where $H^{\on{ref}}_{\partial \unifdom}\widetilde{u} \in \domainref_{e}$
			is defined for $\formref$-q.e.\ $x \in \overline{\unifdom}$ by
			\begin{equation} \label{eq:bdry-trace-hitdist}
				H^{\on{ref}}_{\partial \unifdom}\widetilde{u}(x)
				 \coloneqq \expref_{x}\bigl[\widetilde{u}(\diffref_{\sigma_{\partial \unifdom}}) \one_{\{ \sigma_{\partial \unifdom} < \infty \}}\bigr],
				\quad \sigma_{\partial \unifdom}=\inf\{ t \in (0,\infty) \mid \diffref_{t} \in \partial \unifdom \}
			\end{equation}
			(recall Definition \ref{d:hmeas}), and call $(\formtr,\domaintr)$ the
			\textbf{boundary trace Dirichlet form} of $(\formref,\domainref)$ on $L^{2}(\partial \unifdom,\bdrymeas)$.
		\end{enumerate}
	\end{definition}
	
	As mentioned after Definition \ref{d:traceDF}, $(\formtr,\domaintr)$ is
	a regular symmetric Dirichlet form on $L^{2}(\partial \unifdom,\bdrymeas)$,
	a subset $\mathcal{N}$ of $\partial \unifdom$ is $\formtr$-polar
	if and only if $\mathcal{N}$ is $\formref$-polar, and
	$f|_{\partial \unifdom \setminus \mathcal{N}}$ is $\formtr$-quasi-continuous on $\partial \unifdom$
	for any $\formref$-quasi-continuous function $f \colon \overline{\unifdom} \setminus \mathcal{N} \to [-\infty,\infty]$
	defined $\formref$-q.e.\ on $\overline{\unifdom}$ for some $\formref$-polar $\mathcal{N} \subset \overline{\unifdom}$.
	Moreover, the extended Dirichlet space
	$\domaintr_{e}$ of $(\partial \unifdom,\bdrymeas,\formtr,\domaintr)$ and
	the values of $\formtr$ on $\domaintr_{e}\times\domaintr_{e}$ are identified as
	\begin{equation} \label{eq:bdrytrace-ExtDiriSp}
		\domaintr_{e} = \{ \widetilde{u}|_{\partial \unifdom} \mid u \in \domainref_{e} \}, \mspace{14mu}
		\formtr(\widetilde{u}|_{\partial \unifdom},\widetilde{v}|_{\partial \unifdom})
		= \formref\bigl(H^{\on{ref}}_{\partial \unifdom}\widetilde{u},H^{\on{ref}}_{\partial \unifdom}\widetilde{v}\bigr)
		\mspace{14mu} \textrm{for any $u,v \in \domainref_{e}$,}
	\end{equation}
	and the Dirichlet form of the boundary trace process $\diffreftr$ is $(\formtr,\domaintr)$.
	
	The goal of this subsection is to compute the \emph{Beurling--Deny decomposition}
	(recall \eqref{e:Beurling-Deny}) of the boundary trace Dirichlet form $(\formtr,\domaintr)$
	defined in \eqref{eq:bdry-trace-domain} and \eqref{eq:bdry-trace-form}.
	Let $\formtrsl,\jumpmeastr,\widecheck{\kappa}$ denote the strongly local part,
	the jumping measure and the killing measure, respectively, of
	$(\partial \unifdom,\bdrymeas,\formtr,\domaintr)$, so that we have
	$\jumpmeastr((\partial \unifdom \times \mathcal{N}) \cap \offdiagp{\partial \unifdom}) = 0 = \widecheck{\kappa}(\mathcal{N})$
	for any $\formtr$-polar $\mathcal{N} \in \Borel(\partial \unifdom)$ and
	\begin{equation} \label{e:Beurling-Deny-bdry-trace}
		\formtr(u,v) = \formtrsl(u,v)+ \frac{1}{2} \int_{\offdiagp{\partial \unifdom}} (\widetilde{u}(x)-\widetilde{u}(y))(\widetilde{v}(x)-\widetilde{v}(y))\, \jumpmeastr(dx\,dy)
		+ \int_{\partial \unifdom} \widetilde{u}(x)\widetilde{v}(x)\,\widecheck{\kappa}(dx)
	\end{equation}
	for any $u,v \in \domaintr_{e}$, where $\widetilde{u},\widetilde{v}$ denote
	$\formtr$-quasi-continuous $\bdrymeas$-versions of $u,v$ respectively.
	
	The following lemma, which is an easy consequence of the $\Delta$-regularity estimate
	shown in Lemma \ref{l:Delta}-\eqref{it:Delta-hmeas}, is the main ingredient to show that the killing
	measure $\widecheck{\kappa}$ is zero.
	
	\begin{lem} \label{l:hitbdy}
		It holds that
		\begin{equation} \label{eq:hitbdy}
			\lawref_{x}(\sigma_{\partial \unifdom} < \infty) = 1 \quad \textrm{for any $x \in \overline{\unifdom}$.}
		\end{equation}
	\end{lem}
	
	\begin{proof}
		First, since the reflected diffusion $\diffref$ has the property that
		\begin{equation*}
			\lawref_{x}( \diffref_{t} \in \unifdom ) = 1 \quad \textrm{for any $(t,x) \in (0,\infty) \times \overline{\unifdom}$}
		\end{equation*}
		by $\refmeas(\partial \unifdom) = 0$ from \eqref{eq:volume-doubling-unif-bdry-zero}
		and \ref{eq:AC} and the conservativeness of $\diffref$,
		it suffices to show the claim for $x \in \unifdom$.
		Then by the Markov property at any time $t>0$, $\diffref$ hits $\partial \unifdom$
		after time $t$ $\lawref_{x}$-a.s.\ for any $x \in \partial \unifdom$.
		In particular, we can work with the original diffusion $\diff$ on the
		ambient space $\ambient$ rather than the reflected diffusion $\diffref$ on $\ol{\unifdom}$.
		% 	Note also the following fact. We set
		%	\[
		%	\tau_{A} \coloneqq \inf\{ t \geq 0 \mid X_{t} \notin A \}
		%	\]
		%	(note that here "$t \geq 0$", not "$t > 0$").
		
		%	We claim that any relatively compact open subset $D \subset \ambient$ with $\ambient \setminus D$ non-$\form$-polar,
		%	\begin{equation} \label{e:hb0}
			%	\lawdiff_{x}( \tau_{D} < \infty ) = 1 \quad \textrm{for any $x \in D$.}
			%	\end{equation}
		%	This follows by \cite[Proposition 3.2]{BCM} and \ref{eq:AC} for the part process $\diff^{D}$ on $D$.
		If $\unifdom$ is bounded, then \eqref{eq:hitbdy} follows by \ref{eq:CDC},
		Remark \ref{r:transient} and Lemma \ref{l:harmonicm}-\eqref{it:hmeas-prob1},\eqref{it:hmeas-quasi-support}.
		Assume that $\unifdom$ is unbounded, let $x \in \unifdom$, and choose $\xi \in \partial \unifdom$ and $R \in (d(x,\xi),\infty)$.
		Then by Lemmas \ref{l:harmonicm}-\eqref{it:hmeas-quasi-support},\eqref{it:hmeas-smooth-support},\eqref{it:hmeas-prob1} and \ref{l:Delta}-\eqref{it:Delta-hmeas},
		there exist $C_1, \delta \in (0,\infty)$ such that for all $K \in (1,\infty)$,
		\begin{align*}
			\lawref_{x}(\sigma_{\partial \unifdom} < \infty) &= \hmeas{\unifdom}{x}(\partial \unifdom)
			= \lawdiff_{x}(\sigma_{\ambient \setminus \unifdom}<\infty) \quad \textrm{(by Lemma \ref{l:harmonicm}-\eqref{it:hmeas-quasi-support},\eqref{it:hmeas-smooth-support})} \\
			&\ge \lawdiff_{x}( \sigma_{\ambient \setminus \unifdom} \le \tau_{B(\xi,KR)} )
			= 1 - \lawdiff_{x}( \tau_{B(\xi,KR) } < \tau_{\unifdom} ) \quad \textrm{(by Lemma \ref{l:harmonicm}-\eqref{it:hmeas-prob1})} \\
			&\ge 1 - \hmeas{\unifdom \cap B(\xi,KR)}{x}( \unifdom \cap S(\xi,KR) ) \quad \textrm{(by Lemma \ref{l:harmonicm}-\eqref{it:hmeas-smooth-support})} \\
			&\ge 1- C_1 K^{-\delta} \quad \textrm{(by Lemma \ref{l:Delta}-\eqref{it:Delta-hmeas}),}
		\end{align*}
		and we obtain $\lawref_{x}(\sigma_{\partial \unifdom} < \infty) = 1$ by letting $K \to \infty$.
	\end{proof}
	
	Our next result shows that the only non-vanishing term in the Beurling--Deny
	decomposition \eqref{e:Beurling-Deny-bdry-trace} is the jump part.
	Our main tools are Propositions \ref{prop:trace-slocal} and \ref{p:killing}.
	
	\begin{prop} \label{prop:bdry-trace-pure-jump}
		The boundary trace Dirichlet form $(\formtr,\domaintr)$ on $L^{2}(\partial \unifdom,\bdrymeas)$
		is of pure jump type, that is, $\formtrsl$ and $\widecheck{\kappa}$
		in \eqref{e:Beurling-Deny-bdry-trace} are identically zero.
	\end{prop}
	
	\begin{proof}
		The vanishing of the killing measure $\widecheck{\kappa}$ follows from Lemma \ref{l:hitbdy} and Proposition \ref{p:killing}.
		Alternatively, by \cite[Theorem 5.6.3]{CF} the killing measure is the supplementary Feller measure $V$ as defined in \cite[(5.5.7)]{CF},
		which in turn vanishes due to Lemma \ref{l:hitbdy}. 
		
		By \cite[Theorem 5.6.2]{CF}, for which we have given a new elementary proof in Proposition \ref{prop:trace-slocal}
		above, and Proposition \ref{p:bdrymeas}-\eqref{it:bdrymeas-pcaf-support},
		the strongly local part $\formtrsl$ of $(\formtr,\domaintr)$ is identified as the values
		of the $\formref$-energy measures on $\partial \unifdom$, and they are seen to vanish
		by applying \cite[Theorem 2.9]{Mur24} to the MMD space
		$(\ol{\unifdom},d,\refmeas|_{\ol{\unifdom}},\formref,\domainref)$, which satisfies
		\hyperlink{VD}{\textup{VD}} and \hyperlink{hke}{$\on{HKE(\scdiff)}$} by
		Theorem \ref{thm:hkeunif}-\eqref{it:hkeunif}, and the uniform domain $\unifdom$ in $(\overline{\unifdom},d)$.
		This concludes the proof that $(\formtr,\domaintr)$ is of pure jump type.
		
		The vanishing of the $\formref$-energy measures on $\partial \unifdom$
		of any $u \in \domainref_{e}$ can be seen more directly as follows.
		Let $u \in \domainref \cap L^{\infty}(\overline{\unifdom},\refmeas|_{\overline{\unifdom}})$.
		Then for any $f \in \domainref \cap \contfunc_{\mathrm{c}}(\overline{\unifdom})$,
		we easily see from the Leibniz rule \cite[Lemma 3.2.5]{FOT} for $\form$-energy measures that
		\begin{equation}\label{eq:energy-meas-unifdom-calc}
			d\Gamma_{\unifdom}(u,uf)-\frac{1}{2}d\Gamma_{\unifdom}(u^{2},f) = f\,d\Gamma_{\unifdom}(u,u),
		\end{equation}
		which together with \eqref{e:neumann} shows that
		\begin{equation}\label{eq:energy-meas-unifdom}
			\formref(u,uf)-\frac{1}{2}\formref(u^{2},f) = \int_{\unifdom}f\,d\Gamma_{\unifdom}(u,u).
		\end{equation}
		It follows from \eqref{eq:energy-meas-unifdom} and \eqref{e:EnergyMeas} that the $\formref$-energy measure
		of $u$ is given by $\Gamma_{\unifdom}(u,u)(\cdot \cap \unifdom)$ and hence equal to
		$\Gamma_{\unifdom}(u,u)(\partial \unifdom \cap \unifdom) = \Gamma_{\unifdom}(u,u)(\emptyset) = 0$
		on $\partial \unifdom$, and the same holds also for any $u \in \domainref_{e}$ by the definition of
		the $\formref$-energy measure of general $u \in \domainref_{e}$ presented in Definition \ref{d:EnergyMeas}.
	\end{proof}
	
	The goal of this section is the Doob--Na\"im formula stated in Theorem \ref{t:dnformula}.
	We discuss relevant previous works and approaches of proving the Doob--Na\"im formula.
	As mentioned in the introduction, this was first shown by Doob \cite{Doo} in the setting of Green spaces
	introduced in \cite{BrCh}. They are locally Euclidean and hence the result does not apply to diffusions on fractals.
	Doob's work relies on existence of fine limits to define the Na\"im kernel and existence of `fine normal derivatives' \cite[\textsection 8]{Doo} shown by Na\"im \cite{Nai}.
	It is unclear to the authors whether these results of Na\"im can be extended to our setting and we leave it as an interesting direction for future work.
	M.~Silverstein \cite[Theorem 1.3]{Sil} showed the Doob--Na\"im formula for Markov chains on countable spaces using an excursion measure.
	While it is possible to construct similar excursions in our setting as discussed in \cite[Section 5.7]{CF},
	we choose a direct approach starting from the definition \eqref{eq:bdry-trace-form} of
	the boundary trace Dirichlet form $(\formtr,\domaintr)$ and performing a fairly simple computation.
	The joint continuity of the Na\"im kernel established by using \ref{eq:BHP} in Proposition \ref{p:naim}
	and the description of the Martin kernel as the Radon--Nikodym derivative of the harmonic measure
	in Proposition \ref{p:harmonic-mk} are important ingredients of our proof.
	
	For random walks on certain trees, the trace Dirichlet form on the boundary is amenable to explicit computations.
	This was first done by Kigami \cite[Theorem 5.6]{Kig} and was later shown to coincide with the Doob--Na\"im formula in \cite[Theorem 6.4]{BGPW}.
	Kigami \cite[Theorem 7.6]{Kig} also obtained stable-like heat kernel estimates for the trace process on the boundary.
	
	By extending the results of \cite{Doo,Fuk,Sil}, we show that the Na\"im kernel $\naimker^\unifdom_{x_0}$
	is the jump kernel of the boundary trace Dirichlet form $(\formtr,\domaintr)$ with respect to
	$\hmeas{\unifdom}{x_0} \times \hmeas{\unifdom}{x_0}$.
	
	\begin{theorem}[Doob--Na\"im formula] \label{t:dnformula}
		The jumping measure $\jumpmeastr$ in the Beurling--Deny decomposition
		\eqref{e:Beurling-Deny-bdry-trace} of the trace Dirichlet form
		$(\formtr,\domaintr)$ on $L^{2}(\partial \unifdom,\bdrymeas)$ is given by
		\begin{equation} \label{eq:dnformula-jumpmeas}
			d\jumpmeastr(\xi,\eta)= \naimker^\unifdom_{x_0}(\xi,\eta) \,d\hmeas{\unifdom}{x_0}(\xi) \,d\hmeas{\unifdom}{x_0}(\eta).
		\end{equation}
		Equivalently, 
		\begin{equation} \label{eq:dnformula}
			\formtr(u,v) = \frac{1}{2} \int_{\offdiagp{\partial \unifdom}} (\wt{u}(\xi)-\wt{u}(\eta))(\wt{v}(\xi)-\wt{v}(\eta))\, \naimker^\unifdom_{x_0}(\xi,\eta) \,d\hmeas{\unifdom}{x_0}(\xi) \,d\hmeas{\unifdom}{x_0}(\eta)  
		\end{equation}
		for all $u,v \in \domaintr_{e}$, where $\wt{u},\wt{v}$ denote $\formtr$-quasi-continuous $\bdrymeas$-versions of $u,v$ respectively.
	\end{theorem}
	
	\begin{proof} 
		Let $\xi,\eta \in \partial \unifdom$ be distinct and $r<d(\xi,\eta)/4$.
		Let $A=B(\xi,r)\cap \partial \unifdom$, $B= B(\xi,2r)^c \cap \overline{\unifdom}$ and
		$e_{A,B} \in \domainref$ denote the equilibrium potential for $\Capa^{\on{ref}}_{B}(A)$
		for the Dirichlet form $(\formref,\domainref)$ as given in Lemma \ref{l:eqmeas} such that
		\begin{equation*}
			\Capa^{\on{ref}}_{B}(A)=\formref(e_{A,B},e_{A,B}),\quad
			\wt{e}_{A,B} = 1 \textrm{ $\formref$-q.e.\ on $A$,} \quad
			\wt{e}_{A,B} = 0 \textrm{ $\formref$-q.e.\ on $\ol{\unifdom} \setminus B$,}
		\end{equation*}
		where $\wt{e}_{A,B}$ is a $\formref$-quasi-continuous $\refmeas|_{\overline{\unifdom}}$-version of $e_{A,B}$.
		Let $\lambda^1_{A,B}, \lambda^0_{A,B}$ denote the associated measures as given in Lemma \ref{l:eqmeas} supported in $\ol{A}$ and $\ol{\unifdom} \cap S(\xi,2r)$ respectively.
		By \eqref{e:mk1}, we have
		\begin{equation} \label{e:nj1}
			0 < \int_{\partial \unifdom} \wt{e}_{A,B}\, d \hmeas{\unifdom}{x_0} = \int_{\overline \unifdom \cap \partial B} \gren{\unifdom}(x_0,y) \, d\lambda^0_{A,B}(y).
		\end{equation}
		
		Let $u \in \domainref \cap \contfunc_{\mathrm{c}}(\overline{\unifdom})$ be such that
		$\one_{B(\eta,r)} \le u \le \one_{B(\eta,2r)}$.	Since $H^{\on{ref}}_{\partial \unifdom} u$
		is $\formref$-harmonic on $\unifdom$ by \eqref{eq:hit-dist-harm-test} and
		$H^{\on{ref}}_{\partial \unifdom}\widetilde{e}_{A,B}=\widetilde{e}_{A,B}$
		$\formref$-q.e.\ on $\partial \unifdom$ by \eqref{eq:hit-dist-harm-var}, we have
		\begin{align} \label{e:nj2}
			\formref\bigl(H^{\on{ref}}_{\partial \unifdom} u, H^{\on{ref}}_{\partial \unifdom}\wt{e}_{A,B} \bigr)
			&= \formref\bigl(H^{\on{ref}}_{\partial \unifdom}u, \wt{e}_{A,B}\bigr) \quad \textrm{(by \eqref{eq:hit-dist-harm-test} and \eqref{eq:hit-dist-harm-var}})\nonumber \\
			&= -\int_{\ol{\unifdom} \cap S(\xi,2r)} H^{\on{ref}}_{\partial \unifdom} u \,d\lambda^0_{A,B} \quad \textrm{(by \eqref{e:eqmeas} in Lemma \ref{l:eqmeas}-\eqref{it:eqmeas-outer})} \nonumber \\
			&=  -\int_{\ol{\unifdom} \cap S(\xi,2r)} \biggl( \int_{\partial \unifdom} u(z)\,d\hmeas{\unifdom}{y}(z) \biggr) \,d\lambda^0_{A,B}(y) \nonumber \\
			&\overset{\eqref{e:harmonic-mk}}{=} -\int_{\ol{\unifdom} \cap S(\xi,2r)} \biggl( \int_{\partial \unifdom} u(z) \martinker^\unifdom_{x_0}(y,z)\,d\hmeas{\unifdom}{x_0}(z) \biggr) \,d\lambda^0_{A,B}(y).
		\end{align}
		Note that by \cite[Theorem 5.2.8]{CF}, 
		\begin{equation} \label{e:nj3}
			\textrm{$\widetilde{e}_{A,B}|_{\partial \unifdom} \in \domaintr$ and $\widetilde{e}_{A,B}|_{\partial \unifdom}$ is $\formtr$-quasi-continuous.}
		\end{equation}
		Therefore by the Beurling--Deny decomposition \eqref{e:Beurling-Deny-bdry-trace},
		\eqref{e:nj3} and Proposition \ref{prop:bdry-trace-pure-jump}, we obtain
		\begin{align} \label{e:nj4}
			\MoveEqLeft{\formref\bigl(H^{\on{ref}}_{\partial \unifdom} u, H^{\on{ref}}_{\partial \unifdom}\wt{e}_{A,B} \bigr)} \nonumber \\
			&= \formtr(u|_{\partial \unifdom},\wt{e}_{A,B}|_{\partial \unifdom}) \quad \textrm{(by \eqref{eq:bdry-trace-form})} \nonumber \\
			&= \frac{1}{2} \int_{\offdiagp{\partial \unifdom}} ( {u}(x)- {u}(y))(\wt{e}_{A,B}(x)-\wt{e}_{A,B}(y))\, \jumpmeastr(dx\,dy) \quad \mbox{(by \eqref{e:nj3} and \eqref{e:Beurling-Deny-bdry-trace})} \nonumber \\
			&= -\int_{\offdiagp{\partial \unifdom}} u(x)\wt{e}_{A,B}(y)\, \jumpmeastr(dx\,dy),
		\end{align}
		where the equality in the last line above holds since $u,\wt{e}_{A,B}$ have disjoint supports
		(note that $r<d(\xi,\eta)/4$) and $\jumpmeastr$ is symmetric. We thus obtain
		\begin{align} \label{e:nj5}
			\MoveEqLeft{\frac{\int_{\offdiagp{\partial \unifdom}} u(x) \wt{e}_{A,B}(y)\, \jumpmeastr(dx\,dy)}{\int_{\partial \unifdom} u\, d\hmeas{\unifdom}{x_0} \int_{\partial \unifdom} \wt{e}_{A,B} \, d \hmeas{\unifdom}{x_0}} } \nonumber \\&= \frac{-\formref(H^{\on{ref}}_{\partial \unifdom} u, H^{\on{ref}}_{\partial \unifdom}(\wt{e}_{A,B}) )}{\int_{\partial \unifdom} u\, d\hmeas{\unifdom}{x_0} \int_{\partial \unifdom} \wt{e}_{A,B} \, d \hmeas{\unifdom}{x_0}} \nonumber \quad \mbox{(by \eqref{e:nj4})}\\
			&= \frac{\int_{\ol{\unifdom} \cap S(\xi,2r)} \left( \int_{\partial \unifdom} u(z) \martinker^\unifdom_{x_0}(y,z)\,d\hmeas{\unifdom}{x_0}(z) \right) \,d\lambda^0_{A,B}(y)}{\int_{\partial \unifdom} u\, d\hmeas{\unifdom}{x_0} \int_{\ol{\unifdom} \cap S(\xi,2r)} \gren{\unifdom}(x_0,y) \, d\lambda^0_{A,B}(y)} \quad \mbox{(by  \eqref{e:nj1} and \eqref{e:nj2})} \nonumber \\
			&\overset{\eqref{e:naimb}}{=} \int_{\ol{\unifdom} \cap S(\xi,2r)} \int_{\partial \unifdom} \naimker^\unifdom_{x_0}(y,z) \frac{u(z)}{\int_{\partial \unifdom} u\,d\hmeas{\unifdom}{x_0}} \, d\hmeas{\unifdom}{x_0}(z)  \frac{\gren{\unifdom}(x_0,y)}{\int_{\ol{\unifdom} \cap S(\xi,2r)} \gren{\unifdom}(x_0,\cdot) \, d\lambda^0_{A,B}} \,d\lambda^0_{A,B}(y).
		\end{align}
		Let $\rho$ be the metric on $\partial \unifdom \times \partial \unifdom$ defined by $\rho((x_1,y_1),(x_2,y_2)) \coloneqq \max\{ d(x_1,x_2), d(y_1,y_2)\}$. 
		For $(x_1,x_2)\in \partial \unifdom \times \partial \unifdom$, let $B_\rho((x_1,x_2),r)$ denote the open ball of radius $r$ in the metric $\rho$ centered at $(x_1,x_2)$.
		By \cite[Lemma 4.5.4-(i)]{FOT} and using $\wt{e}_{A,B} = 1$ $\formref$-q.e.~on $A$, we have 
		\begin{equation} \nonumber
			u(x)\wt{e}_{A,B}(y) =1 \quad \mbox{for $\jumpmeastr$-a.e.\ $(x,y)\in (B(\eta,r)\times B(\xi,r)) \cap (\partial \unifdom \times \partial \unifdom)$.}
		\end{equation}
		Hence 
		\begin{equation}\label{e:nj6}
			\int_{\partial \unifdom} \int_{\partial \unifdom} u(x) \wt{e}_{A,B}(y)\, \jumpmeastr(dx\,dy) \ge \jumpmeastr(B_\rho((\eta,\xi),r)).
		\end{equation}
		By Corollary \ref{c:asdouble}, there exist $C_1 \in (1,\infty)$ and $A_1 \in (6,\infty)$
		such that for all $(\xi,\eta) \in \partial \unifdom \times \partial \unifdom$ and all $r \in \bigl(0,A_1^{-1}(d(x_0,\xi) \wedge d(x_0,\eta))\bigr)$,
		\begin{equation} \label{e:nj7}
			(\hmeas{\unifdom}{x_0} \times \hmeas{\unifdom}{x_0})(B_\rho((\eta,\xi),2r))
			\le C_1 (\hmeas{\unifdom}{x_0} \times \hmeas{\unifdom}{x_0})(B_\rho((\eta,\xi),r)).
		\end{equation}
		Since $\hmeas{\unifdom}{x_0}|_{\overline{\unifdom}}$ is $\formref$-smooth by
		Lemma \ref{l:harmonicm}-\eqref{it:hmeas-quasi-support},\eqref{it:hmeas-smooth-support},
		$\wt{e}_{A,B} \le \one_{B(\xi,2r)}$ $\formref$-q.e.~implies $\wt{e}_{A,B} \le \one_{B(\xi,2r)}$  $\hmeas{\unifdom}{x_0}$-a.e.~and hence
		\begin{equation} \label{e:nj8}
			\int_{\partial \unifdom} u\, d\hmeas{\unifdom}{x_0} \int_{\partial \unifdom} \wt{e}_{A,B} \, d \hmeas{\unifdom}{x_0}
			\le \int_{\partial \unifdom} \one_{B(\eta,2r)}\, d\hmeas{\unifdom}{x_0} \int_{\partial \unifdom} \one_{B(\xi,2r)} \, d \hmeas{\unifdom}{x_0}
			= (\hmeas{\unifdom}{x_0} \times \hmeas{\unifdom}{x_0})( B_\rho((\eta,\xi),2r)).
		\end{equation}
		Combining  \eqref{e:nj8}, \eqref{e:nj6} and \eqref{e:nj7}, we obtain
		\begin{equation} \label{e:nj9}
			{\frac{\jumpmeastr(B_\rho((\eta,\xi),r))}{(\hmeas{\unifdom}{x_0} \times \hmeas{\unifdom}{x_0})(B_\rho((\eta,\xi),r))}}
			\le C_1 \frac{\int_{\offdiagp{\partial \unifdom}} u(x) \wt{e}_{A,B}(y)\, \jumpmeastr(dx\,dy)}{\int_{\partial \unifdom} u\, d\hmeas{\unifdom}{x_0} \int_{\partial \unifdom} \wt{e}_{A,B} \, d \hmeas{\unifdom}{x_0}}
		\end{equation}
		for all $(\xi,\eta) \in \offdiagp{\partial \unifdom}$ and all
		$r \in \bigl(0,A_1^{-1}(d(x_0,\xi) \wedge d(x_0,\eta) \wedge d(\xi,\eta))\bigr)$.
		
		By using \eqref{e:naimholder} in Proposition \ref{p:naim} and increasing $A_1$ if necessary,
		there exist $C_2 \in (1,\infty)$ and $\gamma \in (0,\infty)$ such that 
		\begin{equation} \label{e:dn1}
			\abs{\frac{\int_{\offdiagp{\partial \unifdom}}  u(x) \wt{e}_{A,B}(y)\, \jumpmeastr(dx\,dy)}{\int_{\partial \unifdom} u\, d\hmeas{\unifdom}{x_0} \int_{\partial \unifdom} \wt{e}_{A,B} \, d \hmeas{\unifdom}{x_0}} - \naimker^\unifdom_{x_0}(\eta,\xi)}
			\le C_2 \naimker^\unifdom_{x_0}(\eta,\xi) \biggl( \frac{r}{d(x_0,\xi) \wedge d(x_0,\eta) \wedge d(\xi,\eta)} \biggr)^\gamma
		\end{equation}
		for all $(\eta,\xi) \in \offdiagp{\partial \unifdom}$ and all
		$r \in \bigl(0,A_1^{-1}(d(x_0,\xi) \wedge d(x_0,\eta) \wedge d(\xi,\eta))\bigr)$.
		By \eqref{e:nj9} and \eqref{e:dn1}, there exists $c_0 \in (0,A_1^{-1})$ such that
		for all $(\eta,\xi) \in \offdiagp{\partial \unifdom}$ and all
		$r \in \bigl(0,c_0 (d(x_0,\xi) \wedge d(x_0,\eta) \wedge d(\xi,\eta))\bigr]$, we have 
		\begin{equation} \label{e:dn2}
			{\frac{\jumpmeastr(B_\rho((\eta,\xi),r))}{(\hmeas{\unifdom}{x_0} \times \hmeas{\unifdom}{x_0})(B_\rho((\eta,\xi),r))}}
			\le  2C_1 \naimker^\unifdom_{x_0}(\eta,\xi).
		\end{equation}
		Using \eqref{e:dn2}, we will show the absolute continuity of $\jumpmeastr$ with respect to $\hmeas{\unifdom}{x_0} \times \hmeas{\unifdom}{x_0}$; that is
		\begin{equation} \label{e:dn3}
			\jumpmeastr \ll {\hmeas{\unifdom}{x_0} \times \hmeas{\unifdom}{x_0}}.
		\end{equation}
		By the inner regularity of $\jumpmeastr$ it suffices to prove that
		if $K \subset \offdiagp{\partial \unifdom}$ is compact and
		$(\hmeas{\unifdom}{x_0} \times \hmeas{\unifdom}{x_0})(K)=0$, then 
		\begin{equation} \label{e:dn4}
			\jumpmeastr(K)=0.
		\end{equation}
		If  $K \subset \offdiagp{\partial \unifdom}$ is compact and $(\hmeas{\unifdom}{x_0} \times \hmeas{\unifdom}{x_0})(K)=0$,
		then by the outer regularity of $\hmeas{\unifdom}{x_0} \times \hmeas{\unifdom}{x_0}$, for any $\varepsilon \in (0,\infty)$,
		there exists an open set $K_{\varepsilon} \subset \offdiagp{\partial \unifdom}$ such that
		$(\hmeas{\unifdom}{x_0} \times \hmeas{\unifdom}{x_0}) (K_{\varepsilon})<\varepsilon$.
		By the 5B-covering lemma \cite[Theorem 1.2]{Hei}, there exist balls
		$B_\rho((y_i,z_i), r_i) \subset K_{\varepsilon}$, $i \in I$ such that
		$(y_i,z_i) \in K$ and $0<r_i \le  c_0 (d(x_0,y_i) \wedge d(x_0,z_i) \wedge d(y_i,z_i))$
		for all $i \in I$, $\bigcup_{i \in I} B_\rho((y_i,z_i), r_i) \supset K$ and
		$B_\rho((y_i,z_i), r_i)/5)$, $i \in I$ are pairwise disjoint. Hence, we have
		\begin{align*}
			\jumpmeastr(K) & \le \sum_{i \in I} \jumpmeastr(B_\rho((y_i,z_i), r_i) ) \overset{\eqref{e:dn2}}{\le} \sum_{i\in I} 2C_1 \naimker^\unifdom_{x_0}(y_i,z_i) (\hmeas{\unifdom}{x_0} \times \hmeas{\unifdom}{x_0})(B_\rho((y_i,z_i),r_i)) \\
			&\le 2C_1^4 \sup_{K} \naimker^\unifdom_{x_0}(\cdot,\cdot) \sum_{i \in I} (\hmeas{\unifdom}{x_0} \times \hmeas{\unifdom}{x_0})\left(B_\rho((y_i,z_i),r_i/5)\right) \quad \mbox{(by \eqref{e:nj7})} \\
			&\le  2C_1^4 \sup_{K} \naimker^\unifdom_{x_0}(\cdot,\cdot) (\hmeas{\unifdom}{x_0} \times \hmeas{\unifdom}{x_0})(K_\varepsilon) \nonumber \\
			& \quad \textrm{(since $\bigcup_{i \in I} B_\rho((y_i,z_i), r_i) \subset K_\varepsilon$ and $B_\rho((y_i,z_i),  r_i)/5), i \in I$ are pairwise disjoint)} \\
			&\le 2C_1^4 \sup_{K} \naimker^\unifdom_{x_0}(\cdot,\cdot) \varepsilon.
		\end{align*}
		By letting $\varepsilon \downarrow 0$, we obtain \eqref{e:dn4} since
		$\sup_{K} \naimker^\unifdom_{x_0}(\cdot,\cdot)<\infty$ due to the continuity of
		$\naimker^\unifdom_{x_0}$ (Proposition \ref{p:naim}) and the compactness of $K$.
		This concludes the proof of \eqref{e:dn3}.
		
		By letting $r \downarrow 0$ in the H\"older continuity estimate \eqref{e:dn1} and
		using the asymptotic doubling property \eqref{e:nj7}  and the absolute continuity
		\eqref{e:dn3} of harmonic measures along with the Lebesgue differentiation theorem
		(\eqref{e:ratiodiff} in Lemma \ref{l:differentiation}), we obtain the desired conclusion.
	\end{proof}
	\begin{remark} \label{r:ac}
		The absolute continuity \eqref{e:dn3} can alternatively be obtained by using the identification of the jumping measure as the Feller measure in \cite[Theorem 5.6.3]{CF} along with \cite[p.~3143, equation before Example 2.1]{FHY}.
		However, we have chosen the more elementary approach using \eqref{e:nj5} because the proof of this identification presented in \cite[Sections 5.4--5.6]{CF} is quite involved.
	\end{remark}
	
	The following corollary of the Doob--Na\"im formula relates the jump density to
	the boundary reference measure $\bdrymeas$ and the function $\scjump(\cdot,\cdot)$.
	
	\begin{cor} \label{c:jpsi}
		%Let an MMD space $(\ambient,d,\refmeas,\form,\domain)$ and a uniform domain $\unifdom$ in $(\ambient,d)$ satisfy Assumption \ref{a:hkecdcbhp}.
		Define $\jumpkertr_{\bdrymeas} \colon \offdiagp{\partial \unifdom} \to (0,\infty)$ by
		\begin{equation} \label{e:defJmu}
			\jumpkertr_{\bdrymeas}(\xi,\eta) \coloneqq 
			\begin{cases}
				\naimker^\unifdom_{x_0}(\xi,\eta) & \textrm{if $\unifdom$ is bounded,} \\
				\displaystyle \naimker^\unifdom_{x_0}(\xi,\eta) \biggl( \frac{d\nu^{\unifdom}_{x_0}}{d\hmeas{\unifdom}{x_0}}(\xi)\frac{d\nu^{\unifdom}_{x_0}}{d\hmeas{\unifdom}{x_0}}(\eta)\biggr)^{-1} & \textrm{if $\unifdom$ is unbounded.}
			\end{cases}
		\end{equation}
		Then the jumping measure $\jumpmeastr$ of the trace Dirichlet form $(\formtr,\domaintr)$
		on $L^{2}(\partial \unifdom,\bdrymeas)$ is given by
		$\jumpmeastr(d\xi\,d\eta) = \jumpkertr_{\bdrymeas}(\xi,\eta) \,\bdrymeas(d\xi)\,\bdrymeas(d\eta)$,
		and there exist $C,A \in (1,\infty)$ such that for all $(\xi,\eta) \in \offdiagp{\partial \unifdom}$,
		\begin{equation} \label{e:jphi}
			\frac{C^{-1}}{\bdrymeas \bigl( B(\xi,d(\xi,\eta))\bigr) \scjump(\xi,d(\xi,\eta))}
			\le \jumpkertr_\bdrymeas(\xi,\eta)
			\le \frac{C}{\bdrymeas \bigl( B(\xi,d(\xi,\eta))\bigr) \scjump(\xi,d(\xi,\eta))}.
		\end{equation}
		%	where $r=d(\xi,\eta)/10$. 
		%	In particular, the jump kernel $J^{x_0}(\cdot,\cdot)$ with respect to the elliptic measure $\nu$ at infinity defined in Proposition \ref{p:emeas} satisfies the following estimate: there exists $C_1 \in (1,\infty)$ such that for all $\xi,\eta \in \partial \unifdom$ with $\xi \neq \eta$ we have
		%	\[
		%	C^{-1} \frac{1}{\nu(B(\xi,d(\xi,\eta))) h_\unifdom^{x_0}(\xi_r)} \le \jumpkertr(\eta,\xi) \le C \frac{1}{\nu(B(\xi,d(\xi,\eta))) h_\unifdom^{x_0}(\xi_r)},
		%	\]
		%	where $r= d(\xi,\eta)/10$.
	\end{cor}
	
	\begin{proof} 
		The jump kernel formula \eqref{e:defJmu} is a direct consequence of the Doob--Na\"im formula (Theorem \ref{t:dnformula})
		along with the mutual absolute continuity in Proposition \ref{p:emeas}-\eqref{it:emeas-AC}.
		
		The proof of \eqref{e:jphi} is based on the following two estimates:
		by \eqref{eq:bdry-scale-estimate} and Lemma \ref{l:scale}, there exist $C_1,A_1 \in (1,\infty)$ such that
		\begin{equation} \label{e:jp1}
			C_1^{-1} \frac{\scjump(\xi,R)}{\scdiff(R)} \refmeas(B(\xi,R)) \le  \bdrymeas(B(\xi,R)) \le C_1 \frac{\scjump(\xi,R)}{\scdiff(R)} \refmeas(B(\xi,R))
		\end{equation}
		for all $\xi \in \partial \unifdom$ and all $R \in (0,\diam(\unifdom)/A_1)$,
		and by \eqref{e:naimest} in Proposition \ref{p:naim} there exist $c_1 \in (0,1/4)$ and $C_2 \in (1,\infty)$
		such that for all $c_0 \in (0,c_1]$ and all $(\xi,\eta) \in \offdiagp{\partial \unifdom}$ with $c_{0}d(\xi,\eta) \leq c_{1}(d(x_0,\xi) \wedge d(x_0,\eta))$,
		\begin{equation} \label{e:jp3}
			C_2^{-1} \frac{\gren{\unifdom}(\xi_{c_{0}d(\xi,\eta)},\eta_{c_{0}d(\xi,\eta)})}{\gren{\unifdom}(x_0,\xi_{c_{0}d(\xi,\eta)}) \gren{\unifdom}(x_0,\eta_{c_{0}d(\xi,\eta)})}
				\le \naimker^\unifdom_{x_0}(\xi,\eta)
				\le C_2 \frac{\gren{\unifdom}(\xi_{c_{0}d(\xi,\eta)},\eta_{c_{0}d(\xi,\eta)})}{\gren{\unifdom}(x_0,\xi_{c_{0}d(\xi,\eta)}) \gren{\unifdom}(x_0,\eta_{c_{0}d(\xi,\eta)})}.
		\end{equation}
		
		We first estimate the factor $\gren{\unifdom}(\xi_{c_{0}d(\xi,\eta)},\eta_{c_{0}d(\xi,\eta)})$ in \eqref{e:jp3}.
		Let $c_0 \in (0,c_1]$ and $(\xi,\eta) \in \offdiagp{\partial \unifdom}$.
		By reducing $c_{0}$ if necessary and by Lemma \ref{l:gbnd}, \eqref{eq:bdry-cap-estimate}, \hyperlink{VD}{\textup{VD}} of $(\ambient,d,\refmeas)$ and \eqref{e:reg}, there exists $c_{2} \in (0,c_{0})$ independent of $(\xi,\eta)$ such that
		\begin{equation} \label{e:jp4}
			\gren{\unifdom}(\xi_{c_{0}d(\xi,\eta)},c_{2}d(\xi,\eta)) \asymp \Capa_{B(\xi,2c_{0}d(\xi,\eta))}\bigl(B(\xi,c_{0}d(\xi,\eta))\bigr)^{-1} \asymp \frac{\Psi(d(\xi,\eta))}{\refmeas\bigl(B(\xi,d(\xi,\eta))\bigr)}.
		\end{equation}
		Reducing $c_{2} \in (0,c_{0})$ if necessary, by \eqref{e:gradial} we have
		\begin{equation}\label{e:jp5}
			\sup_{z \in S(\xi_{c_{0}d(\xi,\eta)},c_{2}d(\xi,\eta))} \gren{\unifdom}(\xi_{c_{0}d(\xi,\eta)},z)
			\asymp \inf_{z \in S(\xi_{c_{0}d(\xi,\eta)},c_{2}d(\xi,\eta))} \gren{\unifdom}(\xi_{c_{0}d(\xi,\eta)},z).
		\end{equation}
		Take a uniform curve from $\xi_{c_{0}d(\xi,\eta)}$ to $\eta_{c_{0}d(\xi,\eta)}$ and choose a point $z_{\xi,\eta} \in S(\xi_{c_{0}d(\xi,\eta)},c_{2}d(\xi,\eta))$
		of it such that the subcurve from $z_{\xi,\eta}$ to $\eta_{c_{0}d(\xi,\eta)}$ is outside $B(\xi_{c_{0}d(\xi,\eta)},c_{2}d(\xi,\eta))$.
		Using a Harnack chain similar to (and simpler than) the proof of Lemma \ref{l:chain}-\eqref{it:chain-unifdom}, we obtain
		\begin{equation} \label{e:jp6}
		\begin{split}
		\gren{\unifdom}(\xi_{c_{0}d(\xi,\eta)},\eta_{c_{0}d(\xi,\eta)})
		\overset{\eqref{e:hchain}}{\asymp} \gren{\unifdom}(\xi_{c_{0}d(\xi,\eta)},z_{\xi,\eta})
		&\overset{\eqref{e:jp5}}{\asymp} \gren{\unifdom}(\xi_{c_{0}d(\xi,\eta)},c_{2}d(\xi,\eta)) \\
		&\overset{\eqref{e:jp4}}{\asymp}\frac{\Psi(d(\xi,\eta))}{\refmeas\bigl(B(\xi,d(\xi,\eta))\bigr)}.
		\end{split}
		\end{equation}
		
		Now we can deduce \eqref{e:jphi} from the above estimates.
		We first consider the case where $\unifdom$ is bounded.
		Covering $\partial \unifdom$ with balls of radii $c_3 \diam(\unifdom)$ for some sufficiently small $c_3 \in (0,1)$,
		using Lemma \ref{l:scale}, \hyperlink{VD}{\textup{VD}} of $(\ambient,d,\refmeas)$ and \eqref{e:reg}, and increasing $C_1$ if necessary, we can extend \eqref{e:jp1} to $R \in (0,\diam(\unifdom)]$, i.e.,
		\begin{equation} \label{e:jp2}
		\bdrymeas(B(\xi,R)) \asymp \frac{\scjump(\xi,R)}{\scdiff(R)}\refmeas(B(\xi,R))
		\quad \textrm{for all $\xi \in \partial \unifdom$ and all $R \in (0,\diam(\unifdom)]$.}
		\end{equation}
		Recalling our choice of $x_{0}=\widehat{\xi}_{\diam(\unifdom)/5}$ from Definition \ref{d:bdrymeas}, by reducing $c_0 \in (0,c_1]$ if necessary,
		for all $(\xi,\eta) \in \offdiagp{\partial \unifdom}$ we have $c_{0}d(\xi,\eta) \leq c_{1}(d(x_0,\xi) \wedge d(x_0,\eta))$, hence \eqref{e:jp3}, and also
		by Theorem \ref{t:hmeas}, Corollary \ref{c:asdouble}, \eqref{eq:bdry-cap-estimate}, \hyperlink{VD}{\textup{VD}} of $(\ambient,d,\refmeas)$ and \eqref{e:reg},
		\begin{equation} \label{e:jp7}
		\gren{\unifdom}(x_0,\xi_{c_{0}d(\xi,\eta)}) \asymp \hmeas{\unifdom}{x_0}\bigl(B(\xi,d(\xi,\eta))\bigr) \frac{\Psi(d(\xi,\eta))}{\refmeas\bigl(B(\xi,d(\xi,\eta))\bigr)}
		\end{equation}
		(and similarly for $\gren{\unifdom}(x_0,\eta_{c_{0}d(\xi,\eta)})$ by replacing $(\xi,\eta)$ with $(\eta,\xi)$).
		Combining \eqref{e:defJmu}, \eqref{e:jp3}, \eqref{e:defmu}, \eqref{e:jp7}, \eqref{e:jp6}, \eqref{e:jp2}, and \eqref{e:Psireg1} from Lemma \ref{l:scale},
		we obtain \eqref{e:jphi} in the case where $\unifdom$ is bounded.
		
		Lastly, assume that $\unifdom$ is unbounded.
		In this case, by Proposition \ref{p:emeas}-\eqref{it:emeas-estimate},
		\eqref{eq:bdry-cap-estimate}, \hyperlink{VD}{\textup{VD}} of $(\ambient,d,\refmeas)$
		and \eqref{e:reg}, for all $(\xi,\eta) \in \offdiagp{\partial \unifdom}$ we have
		\begin{equation} \label{e:jp8}
		\hprof{x_0}(\xi_{c_{0}d(\xi,\eta)}) \asymp \nu^{\unifdom}_{x_{0}}\bigl(B(\xi,d(\xi,\eta))\bigr) \frac{\Psi(d(\xi,\eta))}{\refmeas\bigl(B(\xi,d(\xi,\eta))\bigr)}.
		\end{equation}
		(and similarly for $\hprof{x_0}(\eta_{c_{0}d(\xi,\eta)})$ by replacing $(\xi,\eta)$ with $(\eta,\xi)$).
		Moreover, since $\jumpmeastr(d\xi\,d\eta) = \jumpkertr_{\bdrymeas}(\xi,\eta) \,\bdrymeas(d\xi)\,\bdrymeas(d\eta)$ is independent of $x_0$ and
		$\jumpkertr_{\bdrymeas}$ is continuous by Propositions \ref{p:naim} and \ref{p:emeas}-\eqref{it:emeas-AC}, we see by \eqref{e:defJmu}, \eqref{e:defmu}, \eqref{e:em6},
		and $\supp_{\overline{\unifdom}}[\nu^\unifdom_{x_0}] = \partial \unifdom$ from Proposition \ref{p:emeas}-\eqref{it:emeas-estimate} that for all $y \in \unifdom$,
		\begin{equation} \label{e:jp9}
		\jumpkertr_{y}(\xi,\eta)
			 \coloneqq \naimker^\unifdom_{y}(\xi,\eta) \biggl( \frac{d\nu^{\unifdom}_{y}}{d\hmeas{\unifdom}{y}}(\xi)\frac{d\nu^{\unifdom}_{y}}{d\hmeas{\unifdom}{y}}(\eta)\biggr)^{-1}
			= (\hprof{x_0}(y))^{2} \jumpkertr_{\bdrymeas}(\xi,\eta)
			\quad \textrm{for all $(\xi,\eta) \in \offdiagp{\partial \unifdom}$.}
		\end{equation}
		Now for each $(\xi,\eta) \in \offdiagp{\partial \unifdom}$, choosing $y \in \unifdom$
		so that $d(\xi,\eta) < d(y,\xi)/A$ for some large enough $A \in (1,\infty)$, and
		recalling \eqref{e:em6}, \eqref{e:defmu}, \eqref{e:defscale} and Lemma \ref{l:scale},
		we see from \eqref{e:embnd}, \eqref{e:jp3}, \eqref{e:jp8}, \eqref{e:jp6},
		\eqref{e:jp1}, and \eqref{e:Psireg1} from Lemma \ref{l:scale}, all with $x_0$ replaced by $y$,
		that \eqref{e:jphi} with $\jumpkertr_{y},(\hprof{x_0}(y))^{-1}\bdrymeas,(\hprof{x_0}(y))^{-1}\scjump$
		in place of $\jumpkertr_\bdrymeas,\bdrymeas,\scjump$ holds, which together with
		\eqref{e:jp9} shows \eqref{e:jphi} in the case where $\unifdom$ is unbounded.
	\end{proof}
	
	\begin{remark} \label{rmk:bdry-meas-doubling}
		\begin{enumerate}[\rm(a)]\setlength{\itemsep}{0pt}
			\item\label{it:bdry-meas-doubling} The estimates \eqref{e:jp2} and \eqref{e:jp1} along with \hyperlink{VD}{\textup{VD}} of $(\ambient,d,\refmeas)$,
			Lemma \ref{l:scale} and \eqref{e:reg} show that $(\partial \unifdom,d,\bdrymeas)$ is \hyperlink{VD}{\textup{VD}}.
			\item\label{it:bdry-trace-jumpker-hprofile} If $\unifdom$ is unbounded, we can use \eqref{e:em5} and \eqref{e:defJmu} to derive another formula for
			$\jumpkertr_\bdrymeas(\xi,\eta)$ in terms of the Green function $\gren{\unifdom}(\cdot,\cdot)$ and the harmonic profile $\hprof{x_0}$ as follows:
			\begin{align} \label{e:jkformula}
				\jumpkertr_\bdrymeas(\xi,\eta)& \overset{\eqref{e:defJmu}}{=} \naimker^\unifdom_{x_0}(\xi,\eta) \biggl( \frac{d\nu^{\unifdom}_{x_0}}{d\hmeas{\unifdom}{x_0}}(\xi)\frac{d\nu^{\unifdom}_{x_0}}{d\hmeas{\unifdom}{x_0}}(\eta)\biggr)^{-1} \nonumber \\
				& \overset{\eqref{e:em5}}{=} \frac{\naimker^\unifdom_{x_0}(\xi,\eta)}{\naimker^\unifdom_{x_0}(\infty, \xi ) \naimker^\unifdom_{x_0}(\infty,\eta)}  \nonumber \\
				&= \lim_{\substack{(z,x,y) \to (\infty,\xi,\eta),\\ z,x,y \in \unifdom}} \frac{\naimker^\unifdom_{x_0}(x,y)}{\naimker^\unifdom_{x_0}(z, x ) \naimker^\unifdom_{x_0}(z,y)} \quad \textrm{(by Remark \ref{r:emeas} and Proposition \ref{p:naim})} \nonumber \\
				&= \lim_{\substack{(z,x,y) \to (\infty,\xi,\eta),\\ z,x,y \in \unifdom}} \frac{\gren{\unifdom}(x,y)}{\martinker^\unifdom_{x_0}(x,z) \martinker^\unifdom_{x_0}(y,z)} \quad \textrm{(by \eqref{e:defnaim} and \eqref{e:defMartin})} \nonumber \\
				&= \lim_{\substack{(x,y) \to (\xi,\eta), \\ x,y \in \unifdom}} \frac{\gren{\unifdom}(x,y)}{ \hprof{x_0}(x)\hprof{x_0}(y)} \quad \textrm{(by Proposition \ref{p:hprofile} and \eqref{e:hp1}).}
			\end{align}
			We note that the existence of the limit in \eqref{e:jkformula} follows from
			\ref{eq:BHP} by using arguments similar to the proof of Proposition \ref{p:naim}.
		\end{enumerate}
	\end{remark} 
	
	\subsection{Stable-like heat kernel estimates for the trace process} \label{ssec:shk-trace}
	
	The following exit time lower estimate is a key ingredient in the proof of
	the stable-like heat kernel estimates for the boundary trace process.
	It is deduced from \hyperlink{hke}{$\on{HKE(\scdiff)}$} for
	$(\overline{\unifdom},d,\refmeas|_{\overline{\unifdom}},\formref,\domainref)$
	obtained in \cite{Mur24} (Theorem \ref{thm:hkeunif}-\eqref{it:hkeunif}) and
	the invariance of the Green functions under the operation of taking trace
	Dirichlet forms (Proposition \ref{prop:greeninv}).
	
	\begin{prop} \label{p:bdry-trace-exit}
		There exist $C_1,A_1 \in (1,\infty)$ such that all $\xi \in \partial \unifdom$
		and all $r \in (0,\diam(\partial U)/2)$,
		\begin{equation} \label{e:bdry-trace-exit}
			\expref_{\xi}\bigl[\widecheck{\tau}_{B(\xi,r)}\bigr]
			\ge C_1^{-1} \scjump(\xi,r),
		\end{equation}
		where $\widecheck{\tau}_{B(\xi,r)} \coloneqq \inf\bigl\{ t \in [0,\infty) \bigm| \diffreftr_{t} \not \in B(\xi,r) \bigr\}$.
	\end{prop}
	
	\begin{proof}
		Recall that $(\overline{\unifdom},\refmeas|_{\overline{\unifdom}},\formref,\domainref)$
		is irreducible by Theorem \ref{thm:hkeunif}-\eqref{it:hkeunif} and Proposition \ref{p:feller}-\eqref{it:HKE-conn-irr-cons}.
		By Remark \ref{r:nonpolar}, this irreducibility and \cite[Proposition 2.1]{BCM}, for any $\xi \in \partial \unifdom$ and any $r \in (0,\diam(\partial \unifdom)/2)$
		the part Dirichlet form of $(\formref,\domainref)$ on $\overline{\unifdom} \setminus (\partial \unifdom \cap B(\xi,r)^c)$ is transient.
		By Theorem \ref{thm:hkeunif}-\eqref{it:hkeunif}, \cite[Theorem 1.2]{GHL15},
		\hyperlink{VD}{\textup{VD}} of $(\ambient,d,\refmeas)$, \eqref{e:reg}
		and the domain monotonicity of the Green functions, there exist
		$A_0,C_2 \in (1,\infty)$ such that for all $x \in \overline{\unifdom}$
		and all $r \in (0,\diam(\unifdom)/2)$, we have 
		\begin{equation} \label{e:ext1}
			\grenref{\overline{\unifdom} \cap B(x,r)}(x,y) \ge C_2^{-1} \frac{\scdiff(r)}{\refmeas(B(x,r))} \quad \textrm{for all $y \in B(x,A_0^{-1}r)$.}
		\end{equation}
		The domain monotonicity of the Green functions also yields
		\begin{equation} \label{e:ext2}
			\grenref{\ol \unifdom \setminus \left( \partial \unifdom \cap B(\xi,r)^c\right)}(\cdot,\cdot) 	 \ge \grenref{\overline{\unifdom} \cap B(\xi,r)}(\cdot,\cdot) 
		\end{equation}
		for all $\xi \in \partial \unifdom$ and all $r \in (0,\diam(\partial \unifdom)/2)$. 
		Therefore, noting that Proposition \ref{prop:greeninv} is applicable by
		Proposition \ref{p:bdrymeas}-\eqref{it:bdrymeas-smooth} and applying \eqref{e:odtr}
		in Proposition \ref{prop:greeninv} with $f \equiv 1$, for all
		$\xi \in \partial \unifdom$ and all $r \in (0,\diam(\partial \unifdom)/2)$, we have
		\begin{align} \label{e:ext3}
			\expref_{\xi}[\widecheck{\tau}_{B(\xi,r)}] &= \int_{\partial \unifdom \cap B(\xi,r)} \grenref{\ol \unifdom \setminus ( \partial \unifdom \cap B(\xi,r)^c)} (\xi,\eta) \, \bdrymeas(d\eta)
			\overset{\eqref{e:ext2}}{\ge} \int_{\partial \unifdom \cap B(\xi,r)} \grenref{\overline{\unifdom} \cap B(\xi,r)} (\xi,\eta) \, \bdrymeas(d\eta) \nonumber \\
			&\ge C_2^{-1} \frac{\scdiff(r)}{\refmeas(B(\xi,r))} \bdrymeas(\partial \unifdom \cap B(\xi,r)) \quad \textrm{(by \eqref{e:ext1}).}
		\end{align} 
		The exit time lower estimate \eqref{e:bdry-trace-exit} follows by combining \eqref{e:ext3} with \eqref{e:jp1} or \eqref{e:jp2}.
	\end{proof}
	
	Given the jump kernel estimate (Corollary \ref{c:jpsi}) and the exit time lower estimate
	(Proposition \ref{p:bdry-trace-exit}) for the boundary trace process, by Theorem \ref{t:shkchar} we obtain
	the stable-like heat kernel estimates for it, as stated in the following main theorem of this subsection.
	
	\begin{theorem}\label{thm:shk-trace}
		Let a scale function $\scdiff$, an MMD space $(\ambient,d,\refmeas,\form,\domain)$,
		a uniform domain $\unifdom$ in $(\ambient,d)$, and a diffusion
		$\diffref=\bigl(\Omega^{\on{ref}},\events^{\on{ref}},\{\diffref_{t}\}_{t\in[0,\infty]},\{\lawref_{x}\}_{x\in\oneptcpt{\overline{\unifdom}}}\bigr)$
		on $\overline{\unifdom}$ satisfy Assumption \ref{a:hkecdcbhp}, and assume that
		$(\partial \unifdom,d)$ is uniformly perfect. Let $\bdrymeas$ be the Radon measure
		on $\overline{\unifdom}$ defined in \eqref{e:defmu}, and let $\scjump$ be the
		regular scale function on $(\partial \unifdom,d)$ given by \eqref{e:defscale} and
		Lemma \ref{l:scale}. Then the NLMMD space $(\partial \unifdom,d,\bdrymeas,\formtr,\domaintr)$
		is of pure jump type and satisfies \hyperlink{VD}{\textup{VD}} and
		\hyperlink{shk}{$\on{SHK}(\scjump)$}, and consequently the following hold:
		\begin{enumerate}[\rm(a)]\setlength{\itemsep}{0pt}\vspace{-5pt}
			\item\label{it:shk-trace-irr-conserv} $(\partial \unifdom,\bdrymeas,\formtr,\domaintr)$
			is irreducible and conservative.
			\item\label{it:shk-trace-CHK} A (unique) continuous heat kernel
			$\hkreftr=\hkreftr_{t}(\xi,\eta)\colon (0,\infty) \times \partial \unifdom \times \partial \unifdom \to [0,\infty)$
			of $(\partial \unifdom,\bdrymeas,\formtr,\domaintr)$ exists and satisfies
			\eqref{eq:shk-trace-upper-intro} and \eqref{eq:shk-trace-lower-intro}
			for any $(t,\xi,\eta) \in (0,\infty) \times \partial \unifdom \times \partial \unifdom$
			for some $C_{1} \in (1,\infty)$.
			\item\label{it:shk-trace-Feller} The boundary trace process
			$\diffreftr=\bigl(\widecheck{\Omega}^{\on{ref}},\widecheck{\events}^{\on{ref}},\{\diffreftr_{t}\}_{t\in[0,\infty]},\{\lawref_{\xi}\}_{\xi \in \oneptcptp{\partial \unifdom}}\bigr)$
			of $\diffref$ on $\partial \unifdom$ as defined in Definition \ref{d:bdry-trace}-\eqref{it:bdry-trace-process}
			is a conservative Hunt process on $\partial \unifdom$, and its Markovian transition function is given by
			$\lawref_{\xi}(\diffreftr_{t} \in d\eta)= \hkreftr_{t}(\xi,\eta) \, \bdrymeas(d\eta)$
			for any $(t,\xi) \in (0,\infty) \times \partial \unifdom$ and
			has the Feller property and the strong Feller property.
			\item\label{it:shk-trace-domain} Let $\jumpkertr_{\bdrymeas} \colon \offdiagp{\partial \unifdom} \to (0,\infty)$
			be as given in \eqref{e:defJmu}. Then $\domaintr$ considered as a linear subspace of
			$L^{2}(\partial \unifdom,\bdrymeas)$ is identified as
			\begin{equation}\label{eq:shk-trace-domain}
				\domaintr=\biggl\{u\in L^{2}(\partial \unifdom,\bdrymeas) \biggm| \int_{\partial \unifdom}\int_{\partial \unifdom} (u(x)-u(y))^{2} \jumpkertr_{\bdrymeas}(x,y)\, \bdrymeas(dx)\,\bdrymeas(dy) < \infty \biggr\}.
			\end{equation}
		\end{enumerate}
	\end{theorem}
	
	\begin{proof}
		$\diffreftr$ is a $\bdrymeas$-symmetric Hunt process on $\partial \unifdom$ whose
		Dirichlet form is $(\formtr,\domaintr)$ as noted after \eqref{eq:bdry-trace-process} and
		after \eqref{eq:bdrytrace-ExtDiriSp}, and $(\partial \unifdom,d,\bdrymeas,\formtr,\domaintr)$
		is an NLMMD space of pure jump type by \cite[Theorem 5.2.13-(i)]{CF} and Proposition \ref{prop:bdry-trace-pure-jump}
		and satisfies \hyperlink{VD}{\textup{VD}} by Remark \ref{rmk:bdry-meas-doubling}-\eqref{it:bdry-meas-doubling},
		the jump kernel estimate \hyperlink{jphi}{$\on{J}(\scjump)$} by Corollary \ref{c:jpsi},
		and the exit time lower estimate \hyperlink{exit}{$\on{E}(\scjump)_{\geq}$}
		by Proposition \ref{p:bdry-trace-exit}. Thus by Theorem \ref{t:shkchar},
		$(\partial \unifdom,d,\bdrymeas,\formtr,\domaintr)$ satisfies \hyperlink{shk}{$\on{SHK}(\scjump)$} and
		the claims \eqref{it:shk-trace-irr-conserv}, \eqref{it:shk-trace-CHK} and \eqref{it:shk-trace-domain} hold.
		
		It thus remains to prove \eqref{it:shk-trace-Feller}. Let $(\widecheck{P}_{t})_{t>0}$
		denote the Markovian transition function of $\diffreftr$, which satisfies 
		$\widecheck{P}_{t}(\xi,\cdot) \ll \bdrymeas$ for any $(t,\xi) \in (0,\infty) \times \partial \unifdom$
		by Propositions \ref{p:bdrymeas}-\eqref{it:bdrymeas-smooth} and
		\ref{prop:greeninv}-\eqref{it:time-change-Hunt-AC}, and define a Markovian
		transition function $(\widecheck{Q}_{t})_{t>0}$ on $\partial \unifdom$ by
		$\widecheck{Q}_{t}(\xi,d\eta) \coloneqq \hkreftr_{t}(\xi,\eta)\,\bdrymeas(d\eta)$,
		$(t,\xi)\times(0,\infty)\times \partial \unifdom$, so that by Theorem \ref{t:shkchar}-\eqref{it:shkchar-Feller},
		$(\widecheck{Q}_{t})_{t>0}$ has the Feller property and the strong Feller property and satisfies
		$\widecheck{Q}_{t}(\xi,\partial \unifdom)=1$ for any $(t,\xi) \in (0,\infty)\times \partial \unifdom$.
		We show $(\widecheck{P}_{t})_{t>0} = (\widecheck{Q}_{t})_{t>0}$ by applying the argument
		in the proof of Proposition \ref{p:feller}-\eqref{it:HKE-part-CHK-sFeller}.
		Since the Dirichlet form of $\diffreftr$ is $(\formtr,\domaintr)$,
		we have $\widecheck{P}_{t}f = \widecheck{Q}_{t}f$ $\bdrymeas$-a.e.\ on $\partial \unifdom$
		for any $f \in L^{2}(\partial \unifdom,\bdrymeas)$ and any $t \in (0,\infty)$.
		Now let $f \in \contfunc_{\mathrm{c}}(\partial \unifdom)$. Then for any $s,t \in (0,\infty)$
		and any $\xi \in \partial \unifdom$, by the Markov property of $\diffreftr$,
		$\widecheck{P}_{t} f = \widecheck{Q}_{t} f$ $\bdrymeas$-a.e.\ on $\partial \unifdom$
		and $\widecheck{P}_{s}(\xi,\cdot) \ll \bdrymeas$ we obtain
		\begin{equation*}
			\widecheck{P}_{t}( \widecheck{P}_{s} f )(\xi)
			= (\widecheck{P}_{t+s} f)(\xi)
			= \widecheck{P}_{s}( \widecheck{P}_{t} f )(\xi)
			= \widecheck{P}_{s}( \widecheck{Q}_{t} f )(\xi),
		\end{equation*}
		and letting $s \downarrow 0$ yields
		\begin{equation} \label{eq:diffreftr-transition-function-equal}
			( \widecheck{P}_{t} f )(\xi) = ( \widecheck{Q}_{t} f )(\xi)
		\end{equation}
		by the dominated convergence theorem since
		$\lim_{s \downarrow 0}(\widecheck{P}_{s} f)(\eta) = f(\eta)$ for any $\eta \in \partial \unifdom$
		and $\lim_{s \downarrow 0}\widecheck{P}_{s}( \widecheck{Q}_{t} f )(\xi) = (\widecheck{Q}_{t} f)(\xi)$
		by the sample-path right-continuity of $\diffreftr$, $f \in \contfunc_{\mathrm{c}}(\partial \unifdom)$,
		and $\widecheck{Q}_{t}f \in \contfunc(\partial \unifdom)$ implied by the strong Feller property of
		$\widecheck{Q}_{t}$. We thus conclude from the validity of \eqref{eq:diffreftr-transition-function-equal}
		for any $f \in \contfunc_{\mathrm{c}}(\partial \unifdom)$ that
		$\widecheck{P}_{t}( \xi, \cdot ) = \widecheck{Q}_{t}( \xi, \cdot )$
		for any $(t,\xi) \in (0,\infty) \times \partial \unifdom$, proving \eqref{it:shk-trace-Feller}.
	\end{proof}
	
	\begin{remark} \label{r:scale}
		Let an MMD space $(\ambient,d,\refmeas,\form,\domain)$ and a uniform domain $\unifdom$ in $(\ambient,d)$ satisfy the assumptions of Theorem \ref{thm:shk-trace}.
		Let $\Capa,\Capa^{\on{ref}},\Capa^{\on{tr}}$ denote the capacities for the   spaces $(\ambient,d,m,\form,\domain)$, $(\ol{\unifdom},d,\restr{m}{\ol{\unifdom}},\formref,\domainref)$, and $(\partial \unifdom, d, \bdrymeas, \formtr,\domaintr)$  as defined in \eqref{e:defCapD} respectively.  
		Using the Poincar\'e inequality in \cite[Definition 7.5]{CKW} for lower bound on capacity across annuli and \cite[Proposition 2.3-(5)]{CKW} for a matching upper bound we obtain the following estimate: there exist $C,A \in (1,\infty)$ such that for all $\xi \in \partial \unifdom, 0<r<\diam(\partial \unifdom,d)/A$, we obtain
		\begin{equation} \label{e:captr}
			C^{-1}	 \frac{\bdrymeas(B(\xi,r))}{\scjump(\xi,r)} \le \Capa^{\on{tr}}_{B(\xi,2r)\cap \partial \unifdom} (B(\xi,r)\cap \partial \unifdom) \le C \frac{\bdrymeas(B(\xi,r))}{\scjump(\xi,r)}
		\end{equation}
		On the other hand, by \cite[Theorem 1.2]{GHL15}, Theorem \ref{thm:hkeunif}-\eqref{it:hkeunif} and \cite[Lemma 5.22]{BCM},
		there exist $C,A \in (1,\infty)$ such that 
		\begin{equation} \label{e:capref}
			C^{-1} \frac{\refmeas(B(\xi,r))}{\scdiff(r)} \le \Capa^{\on{ref}}_{B(x,2r)\cap \ol \unifdom} (B(x,r)\cap \ol \unifdom) \le C \frac{\refmeas(B(x,r))}{\scdiff(r)}
		\end{equation}
		for all $x \in \ol{\unifdom}$ and all $r \in (0,\diam(\unifdom)/A)$, and 
		\begin{equation} \label{e:capamb}
			C^{-1} \frac{\refmeas(B(x,r))}{\scdiff(r)} \le \Capa_{B(x,2r)} (B(x,r)) \le C \frac{\refmeas(B(x,r))}{\scdiff(r)}
		\end{equation}
		for all $x \in \ambient$ and all $r \in (0,\diam(\ambient)/A)$.
		By combining \eqref{e:captr}, \eqref{e:capref}, \eqref{e:capamb} and \eqref{eq:bdry-scale-estimate},
		there exists $A \in (1,\infty)$ such that 
		\begin{equation} \label{e:capamb-capref-captr-comp}
			\Capa_{B(\xi,2r)} (B(\xi,r)) \asymp \Capa^{\on{ref}}_{B(\xi,2r)\cap \ol \unifdom} (B(\xi,r)\cap \ol \unifdom) \asymp \Capa^{\on{tr}}_{B(\xi,2r)\cap \partial \unifdom} (B(\xi,r)\cap \partial \unifdom) 
		\end{equation}
		for all $\xi \in \partial \unifdom$ and all $r \in (0,\diam(\partial \unifdom)/A)$.
	\end{remark}
	
	By Lemma \ref{l:up} and Remark \ref{r:cdc}-\eqref{it:cdc-ecc}, Theorem \ref{thm:shk-trace}
	applies to the reflected Brownian motion on any non-tangentially accessible domain in $\mathbb{R}^{\dimeuc}$ with $\dimeuc \ge 2$.
	Theorem \ref{thm:shk-trace} applies also to the Brownian motion on the Sierpi\'{n}ski carpet
	and the uniform domain $\unifdom$ in it formed by removing either the bottom line or the outer square boundary
	(by \cite[Proposition 4.4]{Lie22}, \cite[Proposition 2.4]{CQ} and Remark \ref{r:cdc}-\eqref{it:cdc-AR});
	note that in this case the reflected Dirichlet form on $\unifdom$ coincides with
	the Dirichlet form of the original Brownian motion on the Sierpi\'{n}ski carpet
	by Theorem \ref{thm:hkeunif}-\eqref{it:formref-energy-meas-zero}.
	
	Another related direction of research is the Calder\'on's inverse problem.
	In our setting, we can phrase it as follows: Does the Dirichlet form of the
	boundary trace process determine the Dirichlet form of the underlying reflected diffusion?
	We refer to \cite{SU} for further context, background, and a solution to this problem
	for a class of Dirichlet forms in $\mathbb{R}^{\dimeuc}$. 
	
	\subsection{Extension to the case with weak capacity density condition} \label{ss:wCDC}
	
	All our results, except those in this subsection, were available prior to the arXiv submission of
	the recent preprint \cite{CC26} by Cao and Chen, and they consider there,
	as already described in Remark \ref{rmk:differences-CaoChan-intro}, a slightly
	more general framework than ours specified in Assumption \ref{a:hkecdcbhp}.
	The purpose of this subsection is to show that our methods apply with minor modifications to a more general setting and to compare our results with those in \cite{CC26}.
	To this end, we introduce a weaker variant of the capacity density condition in Definition \ref{d:cdc} with the condition imposed on a smaller range of radii.
	
	\begin{definition}[Weak capacity density condition (wCDC)] \label{dfn:wCDC}
		Let $(\ambient,d,\refmeas,\form,\domain)$ be an MMD space satisfying \hyperlink{MD}{\textup{MD}}
		and \hyperlink{ehi}{$\on{EHI}$}. Recalling Lemma \ref{l:chain}-\eqref{it:EHI-MD-RBC}, let
		$K \in (1,\infty)$ be such that $(\ambient,d)$ is $K$-relatively ball connected. 
		We say that a uniform domain $\unifdom$ in $(\ambient,d)$ satisfies the
		\textbf{weak capacity density condition}, abbreviated as \ref{eq:wCDC},
		if $\#(\partial \unifdom)\geq 2$ and there exist $A_0 \in (8K,\infty)$ and $A_1, C \in (1,\infty)$
		such that for all $\xi \in \partial \unifdom$ and all $R \in (0,\diam(\partial \unifdom)/A_1)$,
		\begin{equation} \tag*{\textup{wCDC}} \label{eq:wCDC}
			\Capa_{B(\xi,A_0 R)}(B(\xi,R)) \le C \Capa_{B(\xi,A_0 R)}(B(\xi,R)\setminus \unifdom).
		\end{equation}
	\end{definition}

\begin{remark} \label{rmk:wCDC-vacuous}
The condition $\#(\partial \unifdom)\geq 2$, which is equivalent to $\diam(\partial \unifdom) \in (0,\infty]$,
is needed in Definition \ref{dfn:wCDC} to prevent \ref{eq:wCDC} from being vacuous.
It is not strictly necessary, though; indeed, all the discussions of this subsection
on replacing \ref{eq:CDC} by \ref{eq:wCDC} remain applicable to the case where
$\#(\partial \unifdom) = 1$ and \ref{eq:wCDC} with ``$R \in (0,\cdcthres{\unifdom}/A_1)$''
in place of ``$R \in (0,\diam(\partial \unifdom)/A_1)$'' holds for some $\cdcthres{\unifdom} \in (0,\diam(\unifdom)]$,
as long as $\cdcthres{\unifdom}$ is used instead of $\diam(\partial \unifdom)$.
\end{remark}

The only difference of \ref{eq:wCDC} from \ref{eq:CDC} is that the range of radii $R$
is $(0,\diam(\partial\unifdom)/A_1)$ instead of $(0,\diam(\unifdom)/A_1)$; note that
$(0,\diam(\partial\unifdom)/A_1)\subset(0,\diam(\unifdom)/A_1)$ by the trivial inequality
$\diam(\partial\unifdom)\leq\diam(\unifdom)$. By obvious minor modifications of our arguments,
\begin{enumerate}[\rm(a)]\setlength{\itemsep}{0pt}\vspace{-1pt}
\item Lemma \ref{l:cdc1} with $\diam(\partial D)$ in place of $\diam(D)$ holds,
\end{enumerate}
and we obtain the following slightly weaker versions of the results in Section \ref{sec:harm-ell-meas} and Subsection \ref{ssec:bdry-meas-pcaf}
under the setting of Assumption \ref{a:hkecdcbhp} with \ref{eq:CDC} replaced by \ref{eq:wCDC}:
\begin{enumerate}[\rm(a)]\addtocounter{enumi}{1}\setlength{\itemsep}{0pt}\vspace{-2pt}
\item Lemmas \ref{l:Delta} and \ref{l:gbnd} with $\diam(\partial \unifdom)$ in place of $\diam(\unifdom)$ hold.
\item Theorem \ref{t:hmeas} and Corollary \ref{c:asdouble} with ``$r \in \bigl(0,(\diam(\partial \unifdom) \wedge d(\xi,x_0))/A\bigr)$''
	in place of ``$r \in (0,d(\xi,x_0)/A)$'' hold.
\item Lemma \ref{l:up} holds, with the capacity non-decreasing condition \eqref{e:cnd1}
	required only for $x \in \partial \unifdom$ and $0<r<R<\diam(\partial \unifdom)/A$.
\item\label{it:wCDC-emeas} Propositions \ref{p:harmonic-mk}, \ref{p:emeas} and Remark \ref{r:emeas}, with ``$R \in (0,\infty)$'' in
	Proposition \ref{p:emeas}-\eqref{it:emeas-estimate} replaced by ``$R \in (0,2\diam(\partial \unifdom))$'', hold.
\item \eqref{eq:bdry-scale-estimate}, Lemma \ref{l:scale}, Proposition \ref{p:bdrymeas} and \eqref{eq:bdry-nonpolar}
	with $\diam(\partial \unifdom)$ in place of $\diam(\unifdom)$ hold.
\end{enumerate}

Next, we discuss how the results of Subsection \ref{ssec:doob-naim} are affected by replacing \ref{eq:CDC} with \ref{eq:wCDC}.
Lemma \ref{l:hitbdy} need not be true anymore and hence the killing measure $\widecheck{\kappa}$ of the boundary trace Dirichlet form
need not be zero in Proposition \ref{prop:bdry-trace-pure-jump}, whereas the vanishing of the strongly local part
in Proposition \ref{prop:bdry-trace-pure-jump} holds with the same proof. The expression \eqref{eq:dnformula-jumpmeas}
for the jumping measure of the boundary trace Dirichlet form in Theorem \ref{t:dnformula} holds with the same proof under \ref{eq:wCDC},
and Corollary \ref{c:jpsi} and Remark \ref{rmk:bdry-meas-doubling} also hold under \ref{eq:wCDC},
with $\diam(\unifdom)$ needed to be replaced by $\diam(\partial \unifdom)$ and
\eqref{e:jp2} needed to be shown for $R \in (0,\diam(\partial \unifdom)]$ also when
$\diam(\partial \unifdom) < \infty = \diam(\unifdom)$ in the proof of Corollary \ref{c:jpsi}.

We note that Lemma \ref{l:hitbdy} and the proof of the vanishing of the killing measure $\widecheck{\kappa}$
in Proposition \ref{prop:bdry-trace-pure-jump} hold if the uniform domain $\unifdom$ is bounded.
Therefore it remains only to determine $\widecheck{\kappa}$ when $\unifdom$ is unbounded.
In fact, it follows from the extension of Proposition \ref{p:emeas} to the case with \ref{eq:wCDC}
that \emph{the killing measure $\widecheck{\kappa}$ is a constant multiple of
the $\form$-elliptic measure at infinity}, as we prove in the following proposition.
Although the notion of elliptic measure at infinity dates back to \cite[Corollary 3.2]{KT},
its role as the killing measure of the boundary trace Dirichlet form seems new to the best of our knowledge.

\begin{prop} \label{prop:killingmeas}
Assume the setting of Assumption \ref{a:hkecdcbhp} with \ref{eq:CDC} replaced by \ref{eq:wCDC},
and that $\unifdom$ is unbounded. Let $x_{0} \in \unifdom$. Then the killing measure $\widecheck{\kappa}$
in the Beurling--Deny decomposition \eqref{e:Beurling-Deny-bdry-trace} of the trace Dirichlet form
$(\formtr,\domaintr)$ on $L^{2}(\partial \unifdom,\bdrymeas)$ is given by
\begin{equation} \label{eq:killingmeas}
	\widecheck{\kappa} = \lawref_{x_{0}}(\sigma_{\partial \unifdom} = \infty) \nu^{\unifdom}_{x_{0}},
\end{equation}
where $\nu^{\unifdom}_{x_{0}}$ denotes the $\form$-elliptic measure at infinity of $\unifdom$ with
base point $x_{0}$ as obtained in the extension of Proposition \ref{p:emeas} to the present setting.
\end{prop}

\begin{proof}
If $\partial \unifdom$ is unbounded, then both sides of \eqref{eq:killingmeas} are zero by Lemma \ref{l:hitbdy} and Proposition \ref{prop:bdry-trace-pure-jump}. 
Hence it suffices to consider the case where $\diam(\partial \unifdom)<\infty$.
Then we have $\nu^{\unifdom}_{x_{0}}(\partial \unifdom)<\infty$ by the extension of
Proposition \ref{p:emeas} to the present situation mentioned in \eqref{it:wCDC-emeas} above,
and we can choose $v \in \domainref \cap \contfunc_{\mathrm{c}}(\overline{\unifdom})$ so that
$v$ is identically one on a neighborhood of $\partial \unifdom$. Recalling \eqref{eq:bdry-trace-hitdist},
we define an $\formref$-quasi-continuous function $q \colon \overline{\unifdom} \to [0,1]$ by
$q(x) \coloneqq 1-H^{\on{ref}}_{\partial \unifdom}v(x)=\lawref_{x}(\sigma_{\partial \unifdom} = \infty)$,
so that $q=0$ $\formref$-q.e.\ on $\partial \unifdom$ by \eqref{eq:hit-dist-harm-var},
$q\vert_{\unifdom} \in \domain_{\loc}(\unifdom)$ and $q$ is continuous on $\unifdom$ and $\form$-harmonic on $\unifdom$
by Lemma \ref{l:harmonicm}-\eqref{it:hmeas-quasi-support},\eqref{it:hmeas-continuous}.
Moreover, for each open subset $A$ of $\unifdom$ with $\overline{A}$ compact, we can choose
$\varphi \in \domainref \cap \contfunc_{\mathrm{c}}(\overline{\unifdom})$ so that
$\varphi\vert_{A}=\one_{A}$, and then we have $\varphi q=0$ $\formref$-q.e.\ on $\partial \unifdom$,
$\varphi q \in \domainref$ by $H^{\on{ref}}_{\partial \unifdom}v \in \domainref_{e}$, \cite[Exercise 1.1.10]{CF} and
$\domainref_{e}\cap L^{2}(\overline{\unifdom},\refmeas\vert_{\overline{\unifdom}})=\domainref$,
and thus the $\refmeas\vert_{\unifdom}$-equivalence class of $\varphi q\vert_{\unifdom}$ belongs to
$\domain^{0}(\unifdom)$ since the part Dirichlet form of $(\formref,\domainref)$ on $\unifdom$ coincides with
$(\form^{\unifdom},\domain^{0}(\unifdom))$ as observed in the proof of Lemma \ref{l:harmonicm}-\eqref{it:hmeas-quasi-support}.
It follows in view of \eqref{eq:dbdry} that $q\vert_{\unifdom} \in \domain^{0}_{\loc}(\unifdom,\unifdom)$,
and therefore from Lemma \ref{l:uniqueprofile} that
\begin{equation} \label{e:kmeas1}
	q\vert_{\unifdom}(\cdot) = \lawref_{x_{0}}(\sigma_{\partial \unifdom} = \infty) \hprof{x_{0}}(\cdot).
\end{equation}
Now, recalling Remark \ref{rmk:harm-conv}-\eqref{it:rmk:harm-conv-dbdry},
for any $u \in \domainref \cap L^{\infty}(\overline{\unifdom},\refmeas|_{\overline{\unifdom}})$
such that $\supp_{\refmeas\vert_{\overline{\unifdom}}}[u]$ is compact, we have
\begin{align} \label{e:kmeas2}
\formref\bigl( H^{\on{ref}}_{\partial \unifdom} \widetilde{u}, H^{\on{ref}}_{\partial \unifdom} v \bigr)
	&= \formref\bigl( \widetilde{u}, H^{\on{ref}}_{\partial \unifdom} v \bigr) \quad
	\textrm{(by \eqref{eq:hit-dist-harm-var} and \eqref{eq:hit-dist-harm-test})} \nonumber \\
&= \formref\bigl( \widetilde{u}, - \lawref_{x_{0}}(\sigma_{\partial \unifdom} = \infty) \hprof{x_{0}} \bigr) \quad
	\textrm{(by $H^{\on{ref}}_{\partial \unifdom}v=\one_{\overline{\unifdom}}-q$ and \eqref{e:kmeas1})} \nonumber \\
&= \lawref_{x_{0}}(\sigma_{\partial \unifdom} = \infty) \int_{\partial \unifdom} \widetilde{u}\,d\nu^{\unifdom}_{x_{0}} \quad
	\textrm{(by \eqref{e:em0}),}
\end{align}
where for the second equality above we also used the strong locality of $\formref$.
By \eqref{e:kmeas2}, \eqref{eq:bdry-trace-form}, the Beurling--Deny decomposition \eqref{e:Beurling-Deny-bdry-trace},
and the strong locality of the strongly local part $\formtrsl$ of the boundary trace
Dirichlet form $(\formtr,\domaintr)$, we obtain
\begin{equation*}
\lawref_{x_{0}}(\sigma_{\partial \unifdom} = \infty) \int_{\partial \unifdom} \widetilde{u}\,d\nu^\unifdom_{x_{0}}
	= \formref\bigl( H^{\on{ref}}_{\partial \unifdom} \widetilde{u}, H^{\on{ref}}_{\partial \unifdom} v \bigr)
	= \formtr(\widetilde{u}\vert_{\partial \unifdom},v\vert_{\partial \unifdom})
	= \int_{\partial \unifdom} \widetilde{u}\,d\widecheck{\kappa}
\end{equation*}
for all $u \in \domainref \cap L^{\infty}(\overline{\unifdom},\refmeas\vert_{\overline{\unifdom}})$
such that $\supp_{\refmeas\vert_{\overline{\unifdom}}}[u]$ is compact, which proves \eqref{eq:killingmeas}.
\end{proof}

Finally, we describe the changes required in Subsection \ref{ssec:shk-trace} if we replace \ref{eq:CDC} with \ref{eq:wCDC}.
Proposition \ref{p:bdry-trace-exit} holds with the same proof under \ref{eq:wCDC}.
By Propositions \ref{prop:bdry-trace-pure-jump}, \ref{prop:killingmeas} and
the discussion in the paragraph before Proposition \ref{prop:killingmeas},
if $\diam(\partial \unifdom)=\infty$ or $\diam(\unifdom)<\infty$ or
$\lawref_{x_{0}}(\sigma_{\partial \unifdom} = \infty) = 0$, then
the killing measure $\widecheck{\kappa}$ vanishes and therefore
Theorem \ref{thm:shk-trace} still holds with the same proof in this case.
On the other hand, if $\diam(\partial \unifdom) < \infty = \diam(\unifdom)$ and
$\lawref_{x_{0}}(\sigma_{\partial \unifdom} = \infty) > 0$, then
$\widecheck{\kappa} = \lawref_{x_{0}}(\sigma_{\partial \unifdom} = \infty) \nu^{\unifdom}_{x_{0}}$
is not zero by Proposition \ref{prop:killingmeas}, so that the boundary trace Dirichlet form
$(\formtr,\domaintr)$ on $L^{2}(\partial \unifdom,\bdrymeas)$ is neither of pure jump type nor conservative.
In this case, we can still prove a slight variant of Theorem \ref{thm:shk-trace}
as in Theorem \ref{thm:shk-trace-wCDC} below. Note also the following lemma
characterizing precisely when $\lawref_{x_{0}}(\sigma_{\partial \unifdom} = \infty) > 0$
under the assumption that $\diam(\partial \unifdom) < \infty = \diam(\unifdom)$.

\begin{lem}
	Assume the setting of Assumption \ref{a:hkecdcbhp} with \ref{eq:CDC} replaced by \ref{eq:wCDC},
	and that $\diam(\partial \unifdom) < \infty = \diam(\unifdom)$. Let $x_{0} \in \unifdom$.
	Then $\lawref_{x_{0}}(\sigma_{\partial \unifdom} = \infty) > 0$ if and only if
	$(\overline{\unifdom},\refmeas|_{\overline{\unifdom}},\formref,\domainref)$ is transient.
\end{lem}

\begin{proof}
	Recall that Proposition \ref{p:feller} is applicable to
	$(\overline{\unifdom},d,\refmeas|_{\overline{\unifdom}},\formrefgen{\unifdom},\domain(\unifdom))$
	by Theorem \ref{thm:hkeunif}-\eqref{it:hkeunif}.
	If $(\overline{\unifdom},\refmeas|_{\overline{\unifdom}},\formref,\domainref)$ is not transient,
	then it is irreducible and recurrent by Proposition \ref{p:feller}-\eqref{it:HKE-conn-irr-cons}
	and \cite[Proposition 2.1.3-(iii)]{CF}, which together with \eqref{eq:bdry-nonpolar} and
	\ref{eq:AC} of $\diffref$ guarantees that \cite[Theorem 4.7.1-(iii) and Exercise 4.7.1]{FOT}
	can be applied to $\diffref$ and $\partial \unifdom$ and yield
	$\lawref_{x_{0}}(\sigma_{\partial \unifdom} = \infty) = 0$,
	proving the ``only if'' part.
	
	Conversely, assume that $(\overline{\unifdom},\refmeas|_{\overline{\unifdom}},\formref,\domainref)$ is transient,
	so that by combining \cite[Theorem 3.5.2]{CF} with the Markov property of $\diffref$ and \ref{eq:AC}
	and the conservativeness of $\diffref$ from Proposition \ref{p:feller}-\eqref{it:HKE-Feller} we obtain
	\begin{equation} \label{eq:transient-conserv}
	\lawref_{x_{0}}\Bigl( \textrm{$\zeta^{\on{ref}} = \infty$ and $\lim_{t \to \infty} \diffref_{t}  = \cemetery$} \Bigr)
		= \lawref_{x_{0}}( \zeta^{\on{ref}} = \infty ) = 1.
	\end{equation}
	Now suppose that $\lawref_{x_{0}}(\sigma_{\partial \unifdom} = \infty) = 0$. Then by \eqref{e:kmeas1}
	we would have $\lawref_{x}(\sigma_{\partial \unifdom} = \infty) = 0$ for any $x \in \unifdom$,
	but since $\lawref_{x_{0}}(\diffref_{n} \in \unifdom) = 1$ for any $n \in \mathbb{N}$
	by $\refmeas(\partial \unifdom) = 0$ from \eqref{eq:volume-doubling-unif-bdry-zero}
	and \ref{eq:AC} and the conservativeness of $\diffref$,
	we would see from the Markov property of $\diffref$ that
	\begin{equation*}
	\lawref_{x_{0}}(\sigma_{\partial \unifdom} \circ \shiftdiff^{\on{ref}}_{n} < \infty)
		= \expref_{x_{0}}\bigl[ \lawref_{\diffref_{n}}(\sigma_{\partial \unifdom} < \infty) \bigr]
		= 1 \quad \textrm{for any $n \in \mathbb{N}$}
	\end{equation*}
	and hence that
	$\lawref_{x_{0}}\bigl( \bigcap_{n \in \mathbb{N}} \{ \sigma_{\partial \unifdom} \circ \shiftdiff^{\on{ref}}_{n} < \infty \} \bigr) = 1$.
	This contradicts \eqref{eq:transient-conserv} by the assumed compactness of $\partial \unifdom$
	and thereby proves that $\lawref_{x_{0}}(\sigma_{\partial \unifdom} = \infty) > 0$.
\end{proof}
\begin{theorem}\label{thm:shk-trace-wCDC}
Assume the setting of Assumption \ref{a:hkecdcbhp} with \ref{eq:CDC} replaced by \ref{eq:wCDC},
that $\diam(\partial \unifdom) < \infty = \diam(\unifdom)$, and that $(\partial \unifdom,d)$ is uniformly perfect.
Let $x_{0} \in \unifdom$ and assume that $\lawref_{x_{0}}(\sigma_{\partial \unifdom} = \infty) > 0$.
Let $\bdrymeas$ be the Radon measure on $\overline{\unifdom}$ defined in \eqref{e:defmu},
and let $\scjump$ be the regular scale function on $(\partial \unifdom,d)$
given by \eqref{e:defscale} and Lemma \ref{l:scale}. Then
Theorem \ref{thm:shk-trace}-\eqref{it:shk-trace-irr-conserv},\eqref{it:shk-trace-CHK},\eqref{it:shk-trace-Feller},\eqref{it:shk-trace-domain}
with ``conservative'' in \eqref{it:shk-trace-irr-conserv},\eqref{it:shk-trace-Feller}
removed, $C_{1}$ in \eqref{eq:shk-trace-upper-intro} replaced by $C_{1}e^{-\lambda t}$ and
$C_{1}^{-1}$ in \eqref{eq:shk-trace-lower-intro} replaced by $C_{1}^{-1}e^{-\lambda t}$ hold,
where $\lambda \coloneqq \lawref_{x_{0}}(\sigma_{\partial \unifdom} = \infty)$.
\end{theorem}

\begin{proof}
$\diffreftr$ is a $\bdrymeas$-symmetric Hunt process on $\partial \unifdom$ whose
Dirichlet form is $(\formtr,\domaintr)$ as noted after \eqref{eq:bdry-trace-process}
and after \eqref{eq:bdrytrace-ExtDiriSp}. We see by \eqref{e:Beurling-Deny-bdry-trace},
$\formtrsl \equiv 0$ from Proposition \ref{prop:bdry-trace-pure-jump},
$\widecheck{\kappa} = \lambda \bdrymeas$, $\lambda > 0$ and \cite[Theorem 5.2.17]{CF} that
$(\partial \unifdom,d,\bdrymeas,\formtr-\lambda \langle \cdot,\cdot \rangle_{L^{2}(\ambient,\bdrymeas)},\domaintr)$
is an NLMMD space of pure jump type, and it satisfies \hyperlink{VD}{\textup{VD}}
by Remark \ref{rmk:bdry-meas-doubling}-\eqref{it:bdry-meas-doubling} and
\hyperlink{jphi}{$\on{J}(\scjump)$} by Corollary \ref{c:jpsi}. It also follows
from \cite[Theorem 5.2.17]{CF} and \cite[Theorems A.2.11 and 4.2.8]{FOT} that
any $\bdrymeas$-symmetric Hunt process on $\partial \unifdom$ whose Dirichlet form is 
$(\formtr-\lambda \langle \cdot,\cdot \rangle_{L^{2}(\ambient,\bdrymeas)},\domaintr)$
has the Markovian transition function $e^{\lambda t}\lawref_{\xi}(\diffreftr_{t} \in d\eta)$
for all $t \in (0,\infty)$ for $\formtr$-q.e.\ $\xi \in \partial \unifdom$ and has the
expectations of the exit times from all open balls in $(\partial \unifdom,d)$
no less than those for $\diffreftr$ for $\formtr$-q.e.\ starting point $\xi \in \partial \unifdom$.
In particular, Proposition \ref{p:bdry-trace-exit} implies that
$(\partial \unifdom,d,\bdrymeas,\formtr-\lambda \langle \cdot,\cdot \rangle_{L^{2}(\ambient,\bdrymeas)},\domaintr)$
satisfies \hyperlink{exit}{$\on{E}(\scjump)_{\geq}$}. Thus Theorem \ref{t:shkchar} applies to
$(\partial \unifdom,d,\bdrymeas,\formtr-\lambda \langle \cdot,\cdot \rangle_{L^{2}(\ambient,\bdrymeas)},\domaintr)$,
its parts \eqref{it:shkchar-irr-conserv},\eqref{it:shkchar-CHK},\eqref{it:shkchar-domain} respectively yield
Theorem \ref{thm:shk-trace}-\eqref{it:shk-trace-irr-conserv},\eqref{it:shk-trace-CHK},\eqref{it:shk-trace-domain}
with the stated changes, and its part \eqref{it:shkchar-Feller} implies the
Feller property and the strong Feller property of the Markovian transition function
$\hkreftr_{t}(\xi,\eta) \, \bdrymeas(d\eta)$ on $\partial \unifdom$.
The remaining assertion in Theorem \ref{thm:shk-trace}-\eqref{it:shk-trace-Feller} is proved
in exactly the same way as the second paragraph of the proof of Theorem \ref{thm:shk-trace}.
\end{proof}

We conclude this subsection with the following remark summarizing
some advantages of our results in comparison to those of \cite{CC26}.

\begin{remark}\label{r:compare}
	\begin{enumerate}[(a)]\setlength{\itemsep}{0pt}\vspace{-5pt}
		\item\label{it:compare-jumpker} We prove an exact formula for the jump kernel (Theorem \ref{t:dnformula})
			as opposed to estimates in \cite[Theorem 7.1-(a)]{CC26}. An advantage of this
			exact formula is that we are able to obtain the H\"older continuity of
			the jump kernel (see \eqref{e:naimholder}). Although we do not pursue
			this direction in this paper, it is known that continuity estimates
			for the jump kernel have applications to the regularity theory for the
			Poisson-type equations for the corresponding non-local operator;
			see, e.g., \cite[Definition 2.1.22, Lemmas 2.2.6, 2.2.10 and Theorem 2.4.1]{FR}.
			We note that the proof of the H\"older continuity estimate \eqref{e:naimholder}
			of the jump kernel $\jumpkertr_{\bdrymeas} = \naimker^\unifdom_{x_0}$
			for bounded domains easily extends to the jump kernel
			$\jumpkertr_\bdrymeas(\xi,\eta) = \naimker^\unifdom_{x_0}(\xi,\eta) \bigl( \frac{d\nu^{\unifdom}_{x_0}}{d\hmeas{\unifdom}{x_0}}(\xi)\frac{d\nu^{\unifdom}_{x_0}}{d\hmeas{\unifdom}{x_0}}(\eta) \bigr)^{-1}$
			for unbounded domains by using its expression \eqref{e:jkformula}.
		\item\label{it:compare-killingmeas} In the case where the uniform domain $\unifdom$ is unbounded, we identify
			the killing measure $\widecheck{\kappa}$ of the boundary trace Dirichlet form
			as the escape probability $\lawref_{x_{0}}(\sigma_{\partial \unifdom} = \infty)$
			to infinity of the reflected diffusion $\diffref$ times the $\form$-elliptic
			measure $\nu^{\unifdom}_{x_0}$ at infinity (Proposition \ref{prop:killingmeas}),
			in contrast to the corresponding result \cite[Theorem 7.1-(c)]{CC26}
			proving only two-sided estimates on $\widecheck{\kappa}$.
		\item\label{it:compare-em-unique} In our construction of the rescaled limit of harmonic measures in
			Proposition \ref{p:emeas}, we show the convergence as the base point
			tends to infinity and provide an independent characterization \eqref{e:em0}
			of the limit. This is in contrast with the subsequential limit obtained in
			\cite[Proof of Theorem 5.9]{CC26}, where it is not addressed
			whether or not the limit depends on the subsequence. 
		\item\label{it:compare-estimates} In both \eqref{it:compare-jumpker} and
			\eqref{it:compare-killingmeas} above, our exact formulas
			easily imply the estimates obtained in \cite[Theorem 7.1-(a),(c)]{CC26}.
			These estimates do not seem to follow easily from the established identification of the jumping
			and killing measures of trace Dirichlet forms as the Feller and supplementary
			Feller measures in \cite[Theorem 5.6.3]{CF}.
		\item\label{it:compare-str-local-part} We provide a simple proof of the
			identification \eqref{eq:trace-slocal} of the strongly local part of trace
			Dirichlet forms in \cite[Theorem 5.6.2]{CF} (Proposition \ref{prop:trace-slocal}),
			whereas \cite[(1.8)]{CC26} just applies it to show that the strongly local part
			of the boundary trace Dirichlet form is identically zero.
		\item\label{it:compare-support-PCAF} We show that the $\form$-harmonic measure
			$\hmeas{\unifdom}{x_0}$ and the $\form$-elliptic measure $\nu^{\unifdom}_{x_0}$
			at infinity are $\formref$-smooth in the strict sense and that the support
			of the PCAF in the strict sense of $\diffref$ with Revuz measure
			$\hmeas{\unifdom}{x_0}$ or $\nu^{\unifdom}_{x_0}$ coincides with
			$\partial \unifdom$(, which is also the topological supports of
			$\hmeas{\unifdom}{x_0}$ and $\nu^{\unifdom}_{x_0}$) (Proposition \ref{p:bdrymeas};
			see also Lemma \ref{l:strict} and Proposition \ref{prop:goodpcafsupp}).
			Moreover, we prove that the boundary trace process of $\diffref$ is
			a Hunt process on $\partial \unifdom$ whose transition density is
			the continuous heat kernel $\hkreftr=\hkreftr_{t}(\xi,\eta)$
			of the boundary trace Dirichlet form for \emph{any} starting point
			$\xi \in \partial \unifdom$ (Theorem \ref{thm:shk-trace}-\eqref{it:shk-trace-Feller};
			see also Proposition \ref{prop:greeninv}-\eqref{it:time-change-Hunt-AC}).
			The validity of these properties \emph{without removing properly exceptional
			sets from the set of starting points of the associated Markov process} does not
			follow directly from the general theory of regular symmetric Dirichlet forms
			presented in \cite{FOT,CF}, and is not discussed in \cite{CC26}.
	\end{enumerate}
\end{remark}

\subsection{Examples} \label{ssec:examples}
	
	\begin{example}[Molchanov--Ostrowski diffusion on the upper half space] \label{x:mo-cs}
		Let $\dimeuc \in \mathbb{N}$. We consider the Molchanov--Ostrowski diffusion \cite{MO} on the closed upper half-space
		$\ambient = \{(x,y) \mid x\in \bR^{\dimeuc},\,y \in [0,\infty)\}=\bR^{\dimeuc} \times [0,\infty)$, induced by the strongly local
		regular symmetric Dirichlet form $(\form,\domain)$ on $L^{2}(\mathbb{R}^{\dimeuc}\times[0,\infty),\abs{y}^{1-\alpha}\,dy \,dx ) \eqqcolon L^{2}_{\alpha}$ given by
		\begin{equation*}
			\form(u,u) \coloneqq \int_{\mathbb{R}^{\dimeuc}} \int_{0}^{\infty} \abs{\nabla u}^2(x,y) \abs{y}^{1-\alpha}\,dy \,dx
			\quad \textrm{and} \quad \domain \coloneqq \bigl\{ u \in W^{1,2}_{\loc} \cap L^{2}_{\alpha} \bigm\vert \form(u,u) < \infty \bigr\},
		\end{equation*}
		where $\alpha \in (0,2)$ and $W^{1,2}_{\loc} \coloneqq (W^{1,2}(\mathbb{R}^{\dimeuc + 1}))_{\loc}(\mathbb{R}^{\dimeuc}\times(0,\infty))$
		(recall Definition \ref{d:domain-local}).
		The function $w \colon \bR^{\dimeuc+1} \to [0,\infty)$ given by $w(x,y)= \abs{y}^{1-\alpha}$ for $x \in \bR^{\dimeuc}$ and $y \in \bR$ is a Muckenhoupt $A_2$ weight. The weighted Lebesgue measure in this case is known to satisfy the volume doubling property and Poincar\'e inequality \cite[Theorem 1.5]{FKS}. By the characterization of Gaussian heat kernel estimates due to Grigor'yan \cite{Gri91} and Saloff-Coste \cite{Sal} in terms of the volume doubling property and Poincar\'e inequality, we have Gaussian heat kernel estimates in this example.
		Then the open upper half-space $\unifdom= \bR^{\dimeuc} \times (0,\infty)$ is a uniform domain in $\ambient$ (with respect to the Euclidean metric). Using chain rule and scaling property of Lebesgue measure, it is easy to see that if $u \in \domain$ and $r \in (0,\infty)$ then $u_r(x,y) \coloneqq u(rx,ry) \in \domain$ and $\form(u_r,u_r)= r^{\alpha-\dimeuc} \form(u,u)$, which in turn implies that the corresponding Green function satisfies 
		\begin{equation}\label{e:gscale}
			\gren{\unifdom}((rx_1,ry_1),(rx_2,ry_2)) = 	r^{\alpha-\dimeuc}\gren{\unifdom}((x_1,y_1),(x_2,y_2))  
		\end{equation}
		for all $(x_1,y_1),(x_2,y_2) \in \mathbb{R}^{\dimeuc} \times (0,\infty)$ and all $r \in (0,\infty)$.
		Similarly, it is easy to see that the Dirichlet energy is invariant under
		the Euclidean isometries in the $\mathbb{R}^{\dimeuc}$-direction.
		This implies that the Green function inherits these properties; that is,
		\begin{equation} \label{e:gtran}
			\gren{\unifdom}((x+x_1,y_1),(x+x_2,y_2))=	\gren{\unifdom}((x_1,y_1),(x_2,y_2))
		\end{equation}
		and
		\begin{equation} \label{e:grot}
			\gren{\unifdom}((x_1,y_1),(x_2,y_2))=\gren{\unifdom}((Ax_1,y_1),(Ax_2,y_2))
		\end{equation}
		for any $(x_1,y_1),(x_2,y_2) \in \mathbb{R}^{\dimeuc} \times (0,\infty)$,
		any $x \in \mathbb{R}^{\dimeuc}$ and any orthogonal matrix $A \in O(\dimeuc)$.
		Let us fix a base point $x_0=(0,\ldots,0,1) \in \unifdom$.
		Since $L_\alpha (y^\alpha) \equiv 0$ where $L_\alpha$ is given by \eqref{e:cs-generator},
		the harmonic profile is given by
		\begin{equation} \label{e:mohprof}
			\hprof{x_0}(x,y)= y^\alpha \quad \textrm{for all $(x,y) \in \mathbb{R}^{\dimeuc} \times (0,\infty)$.}
		\end{equation}
		Let $\abs{\cdot}$ denote the Euclidean norm on $\mathbb{R}^{\dimeuc} \equiv \partial \unifdom$.
		By \eqref{e:jkformula}, the corresponding jump kernel $\jumpkertr_{\bdrymeas}(\xi,\eta)$ can be computed
		for all pairs of distinct points $\xi, \eta \in \partial \unifdom \equiv \mathbb{R}^{\dimeuc}$ as
		\begin{align} \label{e:mojkernel}
			\jumpkertr_\bdrymeas(\xi,\eta)&= \lim_{r \downarrow 0} \frac{\gren{\unifdom}((\xi,r),(\eta,r))}{\hprof{x_0}((\xi,r)) \hprof{x_0}((\eta,r))}
			\overset{\eqref{e:mohprof}}{=} \lim_{r \downarrow 0} r^{-2\alpha} \gren{\unifdom}((\xi,r),(\eta,r))  \nonumber\\
			&\overset{\eqref{e:gtran}}{=} \lim_{r \downarrow 0} r^{-2\alpha} \gren{\unifdom}((\xi-\eta,r),(\mathbf{0},r))
			\overset{\eqref{e:gscale}}{=} \lim_{r \downarrow 0} r^{-\dimeuc-\alpha} \gren{\unifdom}((r^{-1}(\xi-\eta),1),(\mathbf{0},1)) \nonumber \\
			&= \abs{\xi-\eta}^{-\dimeuc- \alpha} \lim_{s \downarrow 0}  s^{-\dimeuc-\alpha} \gren{\unifdom}((s^{-1}\abs{\xi-\eta}^{-1}(\xi-\eta),1),(\mathbf{0},1))
			= c_1 \abs{\xi-\eta}^{-\dimeuc- \alpha}, 
		\end{align}
		where $c_1 \in (0,\infty)$ does not depend on the choice $\xi, \eta$ due to
		the rotation invariance of the Green function in \eqref{e:grot}.
		Since the Dirichlet form is invariant under translations in $\mathbb{R}^{\dimeuc}$-direction,
		by Proposition \ref{p:emeas} the elliptic measure $\mu=\nu^\unifdom_{x_0}$
		is a constant multiple of the Lebesgue measure $\lambda$ on $\mathbb{R}^{\dimeuc} \equiv \partial \unifdom$
		say $\mu=c_2 \lambda$, where $c_2 \in (0,\infty)$. This along with \eqref{e:mojkernel} implies that
		the jumping measure of the boundary trace Dirichlet form is given by 
		\begin{equation*}
			\jumpkertr_\bdrymeas(\xi,\eta) \, \bdrymeas(d\xi) \bdrymeas(d\eta)= c_1 c_2^2 \abs{\xi-\eta}^{-\dimeuc- \alpha} \lambda(d\xi) \lambda(d\eta),
		\end{equation*}
		which allows us to recover the extension theorem of Caffarelli and Silvestre \cite{CS}
		up to identifying the multiplicative constant as a special case of our Doob--Na\"im formula.
	\end{example}
	
	\begin{example}[Reflected Brownian motion on the orthant] \label{x:orthant}
		Let $\unifdom \coloneqq (0,\infty)^{\dimeuc}$ denote the open orthant in $\mathbb{R}^{\dimeuc}$, where $\dimeuc \geq 2$.
		We consider the reflected Brownian motion on $\unifdom$.
		We choose $x_0=(1,\ldots,1) \in \unifdom$ as the base point.
		One can check that the harmonic profile is
		\begin{equation} \label{e:hporth}
			\hprof{x_0}(y)= \prod_{i=1}^\dimeuc y_i \quad \textrm{for all $y=(y_1,\ldots,y_\dimeuc) \in \unifdom$.}
		\end{equation}
		In this case the space-time scaling function
		$\scjump \colon \partial \unifdom \times [0,\infty) \to [0,\infty)$
		for the boundary trace Dirichlet form can be chosen as
		\begin{equation*}
			\scjump(\xi,r)= \hprof{x_0}((\xi_1+r,\xi_2+r,\ldots,\xi_\dimeuc+r))=\prod_{i=1}^\dimeuc (\xi_i+r)
			\quad \textrm{for all $\xi=(\xi_1,\ldots,\xi_\dimeuc) \in \partial \unifdom$.}
		\end{equation*}
		Next, we describe our two-sided estimates on the corresponding elliptic measure
		at infinity $\bdrymeas=\nu^{\unifdom}_{x_0}$ and on the jump kernel $\jumpkertr_\bdrymeas$
		of the boundary trace Dirichlet form with respect to $\bdrymeas$.
		By \eqref{eq:bdry-scale-estimate}, there exists $C \in (1,\infty)$ such that
		for all $(\xi_1,\ldots,\xi_\dimeuc) \in \partial \unifdom$ and all $r \in (0,\infty)$,
		\begin{equation} \label{e:EllipMeas-orthant}
			C^{-1} r^{\dimeuc-2} \prod_{i=1}^\dimeuc (\xi_i+r)
			\le \bdrymeas(B(\xi,r))
			\le C r^{\dimeuc-2} \prod_{i=1}^\dimeuc (\xi_i+r).
		\end{equation}
		Similarly, by Corollary \ref{c:jpsi} and \eqref{e:EllipMeas-orthant}, there exists $C \in (1,\infty)$ such that
		for any pair of distinct points $\xi=(\xi_1,\ldots,\xi_\dimeuc), \eta=(\eta_1,\ldots,\eta_\dimeuc) \in \partial \unifdom$,
		\begin{equation*}
			C^{-1} d(\xi,\eta)^{2-\dimeuc} \prod_{i=1}^{\dimeuc} (\xi_i+d(\xi,\eta))^{-2}
			\le \jumpkertr_\bdrymeas(\xi,\eta)
			\le C d(\xi,\eta)^{2-\dimeuc} \prod_{i=1}^{\dimeuc} (\xi_i+d(\xi,\eta))^{-2},
		\end{equation*}
		where $d(\xi,\eta) \coloneqq \abs{\xi - \eta}$ denotes the Euclidean distance between $\xi,\eta$.
		
		An interesting feature of this example is that expected exit time of the boundary trace process
		started at the origin from a ball of radius $r$ centered at origin grows like $r^{\dimeuc}$.
		In particular, this provides examples of jump process with (anchored) exit time exponent arbitrary large.
		Such examples are known to exist on fractals but this example shows that such behavior
		can happen also in smooth settings, which seems to be a new observation.
		More generally for any $\xi=(\xi_1,\ldots,\xi_\dimeuc) \in \partial \unifdom$,
		setting $I_{\xi} \coloneqq \{i \in \{1,\ldots,\dimeuc\} \mid \xi_i =0\}$, we have
		\begin{equation*}
			\lim_{r \downarrow 0} \frac{\scjump(\xi,r)}{r^{\#I_{\xi}}} = \prod_{i \in \{1,\ldots,\dimeuc\} \setminus I_{\xi}} \xi_i \in (0,\infty),
			\qquad \lim_{r \to \infty} \frac{\scjump(\xi,r)}{r^\dimeuc} =1.
		\end{equation*}
		The space-time scaling exponent for the boundary trace process starting at $\xi \in \partial \unifdom$
		is $\#I_{\xi}$ at very small scales and $\dimeuc$ at large scale.
		We note that $\#I_{\xi}$ can be any integer between $1$ and $\dimeuc$ depending on $\xi$.
	\end{example}
	
	\begin{example}[Ahlfors--Beurling example: quasi-conformal image of Brownian motion] \label{x:ba-qc}
		%	\strut\\\noindent% This line should be removed when the file is sent to the publisher.
		This example is essentially due to Ahlfors and Beurling \cite{BA} and was later revisited in \cite{CFK}. 
		We consider the reflected Brownian motion on the closed two-dimension upper half-space
		$\ambient=\mathbb{R} \times [0,\infty)$ and let $\refmeas$ denote the restriction of the Lebesgue measure on $\ambient$.
		We consider the domain $\unifdom = \mathbb{R} \times (0,\infty)$.
		Let $\lambda$ denote the one-dimensional Lebesgue measure on $\partial \unifdom \equiv \mathbb{R}$.
		By \cite[Theorem 3]{BA}, there exists a homeomorphism $F \colon \ambient \to \ambient$ with the following properties:
		\begin{enumerate}[\rm(a)]\setlength{\itemsep}{0pt}\vspace{-5pt}
			\item The boundary correspondence  $F|_{\partial \unifdom}$ is singular in the sense that the measures $\lambda$ and the push-forward measure
			$(F|_{\partial \unifdom})_*( \lambda)$ are singular.
			\item The function $F|_{\unifdom}\colon \unifdom \to \unifdom$ is a $\contfunc^{1}$-bijection. Writing $F(x,y)=(F_1(x,y),F_2(x,y))$ in coordinates,
			where $F_1 \colon \ambient \to \bR$ and $F_2 \colon \ambient \to [0,\infty)$, let $DF$ denote the differential on $\unifdom$, that is
			\begin{equation*}
				DF \coloneqq \begin{bmatrix}
					\partial_{x} F_1 & \partial_{y} F_1 \\
					\partial_{x} F_2 & \partial_{y} F_2 \\
				\end{bmatrix}.
			\end{equation*}
			There exists $C \in(0,\infty)$ such that the map $F$ satisfies the quasiconformality condition
			\begin{equation} \label{e:qc}
				0<	\on{Tr}(DF^T(z) DF(z)) \le C \det(DF(z)) \quad \textrm{for all $z \in \unifdom$,}
			\end{equation}
			where $\on{Tr}, \det$ denote the trace and determinant of a matrix.
		\end{enumerate}
		We define the positive definite matrix valued function $\mathcal{A}\colon \unifdom \to \bR^{2 \times 2}$ given by
		\[
		\mathcal{A}(z)= \det(DF(w))^{-1}  DF(w)^{T}  DF(w), \quad \mbox{where $w=F^{-1}(z)$}.
		\]
		We note immediately from \eqref{e:qc}  that $\det(\mathcal{A}(z)) \equiv 1$ on $\unifdom$ and the eigenvalues of $\mathcal{A}(z)$ are bounded from above by $C$ and below by $C^{-1}$ for all $z \in \unifdom$. In particular, $\mathcal{A}(\cdot)$ defines a uniformly elliptic divergence form operator $f \mapsto \div(\mathcal{A}(\cdot) \nabla f)$.
		The Dirichlet form corresponding to the image of the reflected Brownian motion on $\ambient$ under the homeomorphism $F$ is given by $\form(f,f)= \int_\bR \int_{0}^\infty \abs{ \nabla(f \circ F)}^2(x,y) \,dy \,dx$ where $f$ varies over all functions such that $f \circ F$ belongs to the $W^{1,2}$ Sobolev space on $\ambient$.
		We identify the gradient $\nabla f$ as the column vector $[ \partial_{x} f \ \ \partial_{y} f ]^T$. The Dirichlet energy $\mathcal{E}$ can rewritten as (see also \cite[p.~919]{CFK})
		\begin{equation*}
			\form(f,f)= \int_{\bR} \int_{(0,\infty)} (\nabla f(x,y))^T \mathcal{A}(x,y) \nabla f(x,y) \,dy \,dx.
		\end{equation*}
		By a time change we can assume that the domain $\domain$ of the form is $W^{1,2}$ space on $\ambient$, so that $(\form,\domain)$ is a Dirichlet form on $L^2(\ambient,\refmeas)$.
		The Gaussian heat kernel bound for this Dirichlet form follows from the characterization in terms of the volume doubling property of $m$ and the Poincar\'e inequality due to Grigor'yan and Saloff-Coste \cite{Gri91,Sal} and the uniform ellipticity condition on $\mathcal{A}(\cdot)$ mentioned before.
		
		The elliptic measure at infinity $\bdrymeas$ can be easily seen to be a constant multiple of the measure
		$(F|_{\partial \unifdom})_*( \lambda)$ and hence singular with respect to the Lebesgue measure, and
		the harmonic profile can be identified as the function $G_2 \colon \ambient \to [0,\infty)$,
		where $G_1, G_2 \colon \ambient \to \bR$ are such that $F^{-1}(x,y)= (G_1(x,y),G_2(x,y))$. Hence the space-time scaling of the boundary trace process is given by 
		\begin{equation*}
			\scjump((x,0),r)= G_2(x,r) \quad \textrm{for all $(x,0) \in \partial \unifdom$ and all $r \in (0,\infty)$.}
		\end{equation*}
		By \eqref{eq:bdry-scale-estimate}, there exists $C \in (1,\infty)$
		such that the corresponding elliptic measure at infinity $\bdrymeas= \nu^{\unifdom}_{x_0}$ satisfies
		\begin{equation} \label{e:EllipMeas-ba-qc}
			C^{-1} G_2(x,r) \le \bdrymeas(B((x,0),r)) \le C G_2(x,r) 
		\end{equation}
		for all $(x,0) \in \partial \unifdom$ and all $r \in (0,\infty)$.
		Similarly, by Corollary \ref{c:jpsi} and \eqref{e:EllipMeas-ba-qc}, there exists $C \in (1,\infty)$ such that for any pair of distinct points $(u,0), (v,0) \in \partial \unifdom$, we have 
		\begin{equation*}
		C^{-1}G_2(u,\abs{u-v})^{-2}	\le \jumpkertr_\bdrymeas((u,0),(v,0)) \le C G_2(u,\abs{u-v})^{-2}.
		\end{equation*}
	\end{example}

Our results apply also to some \emph{inner uniform domains} (cf.\ \cite[Definition 3.6]{GS})
as they can be viewed as uniform domains by a suitable change of the metric. We illustrate this
in the following examples with the cases of the exterior domain of a parabola and of a slit domain.
	
	\begin{example}[Exterior of a parabola] \label{x:parabola}
		We consider the reflected Brownian motion on the domain
		$\unifdom = \{(x,y) \in \mathbb{R}^{2} \mid y < x^{2}\}$
		given by the sub-level set of the square function, and equip $\unifdom$
		with the inner metric $d_{\unifdom} \colon \unifdom \times \unifdom \to [0,\infty)$,
		where $d_{\unifdom}(z_{0},z_{1})$ is defined as the infimum of the lengths $\ell(\gamma)$
		(with respect to the Euclidean metric $d$; recall \eqref{eq:length-dfn})
		of the curves $\gamma$ in $\unifdom$ from $z_{0}$ to $z_{1}$.
		It is easy to see that $d_{\unifdom}$ uniquely extends to a continuous function on
		$\overline{\unifdom} \times \overline{\unifdom}$ and gives a complete metric on
		$\overline{\unifdom}$ compatible with the Euclidean topology.
		Then, being an inner uniform domain (though not a uniform domain) in
		$(\mathbb{R}^{2},d)$ by \cite[Proposition 6.16]{GS},
		the set $\unifdom$ is a uniform domain in $(\overline{\unifdom},d_{\unifdom})$.
		Set $\ambient \coloneqq \overline{\unifdom}$, and let $\refmeas$ be the Borel measure
		on $\ambient$ given by $\refmeas \coloneqq \lambda(\cdot \cap \unifdom)$,
		where $\lambda$ denotes the Lebesgue measure on $\mathbb{R}^{2}$.
		
		Let $(\form,\domain)$ denote the Dirichlet form on $L^{2}(\mathbb{R}^{2},\lambda)$ of the
		Brownian motion on $\mathbb{R}^{2}$. Even though Theorem \ref{thm:hkeunif}-\eqref{it:hkeunif}
		is \emph{not} applicable to $(\mathbb{R}^{2},d,\lambda,\form,\domain)$ and $\unifdom$,
		it still holds by the results \cite[Theorems 3.30 and 3.34]{GS} due to Gyrya and Saloff-Coste that
		$(\ambient,d_{\unifdom},\refmeas,\formrefgen{\unifdom},\domain(\unifdom))$ (recall Definition \ref{dfn:reflected-form})
		is an MMD space satisfying (\hyperlink{VD}{\textup{VD}} and) Gaussian heat kernel estimates.
		The $\formrefgen{\unifdom}$-harmonic profile of $\unifdom$ with base point $x_0=(0,-3/4)$ is given in \cite[p.~6]{GS} as
		\begin{equation}
			\hprof{x_0}(x,y)= \sqrt{2 \biggl(\sqrt{x^2 + \Bigl( \frac{1}{4}-y \Bigr)^2} + \frac{1}{4}-y \biggr)} -1
			\quad \textrm{for all $(x,y) \in \unifdom$.}
		\end{equation}
		In this case, the space-time scaling function
		$\scjump \colon \partial \unifdom \times [0,\infty) \to [0,\infty)$ for the boundary trace
		Dirichlet form is given for $(x,x^2) \in \partial \unifdom$ and $r \in (0,\infty)$ by
		\begin{align} \label{e:hpep}
			\scjump((x,x^2),r) &= \hprof{x_0}\biggl(x+ \frac{2xr}{\sqrt{1+4x^2}} , x^2-\frac{r}{\sqrt{1+4x^2} } \biggr), \nonumber \\
			&= \sqrt{2 \biggl( \sqrt{ r^2+ \Bigl( \frac{1}{4}+x^2 \Bigr)^2 + \frac{r}{2} \sqrt{1+4x^2}} + \frac{1}{4}-x^2 + \frac{r}{\sqrt{1+4x^2}} \biggr) } - 1,
		\end{align}
		which, as a function of $r \in [0,\infty)$, is easily seen to be a homeomorphism from $[0,\infty)$ to itself for any $(x,x^2) \in \partial \unifdom$.
		By \eqref{eq:bdry-scale-estimate}, there exists $C \in (1,\infty)$ such that the corresponding
		$\formrefgen{\unifdom}$-elliptic measure at infinity $\bdrymeas= \nu^{\unifdom}_{x_0}$ satisfies 
		\begin{equation} \label{e:EllipMeas-parabola}
			C^{-1} \scjump((x,x^2),r) \le \bdrymeas(B_{d_{\unifdom}}((x,x^2),r)) \le C \scjump((x,x^2),r) 
		\end{equation}
		for all $(x,x^2) \in \partial \unifdom$ and all $r \in (0,\infty)$.
		Similarly, by Corollary \ref{c:jpsi} and \eqref{e:EllipMeas-parabola},
		there exists $C \in (1,\infty)$ such that for all $(\xi,\eta) \in \offdiagp{\partial \unifdom}$,
		\begin{equation*}
			C^{-1}\scjump(\xi,d_{\unifdom}(\xi,\eta))^{-2}
			\le \jumpkertr_\bdrymeas(\xi,\eta)
			\le C \scjump(\xi,d_{\unifdom}(\xi,\eta))^{-2}.
		\end{equation*}
		
		From \eqref{e:hpep}, it follows that for any $\xi \in \partial \unifdom$,
		there exist $c_1(\xi),c_2(\xi) \in (0,\infty)$ such that 
		\begin{equation*}
			\lim_{r \downarrow 0} \frac{\scjump(\xi,r)}{r} = c_1(\xi),
			\qquad \lim_{r \to \infty} \frac{\scjump(\xi,r)}{\sqrt{r}}= c_2(\xi).
		\end{equation*}
		In other words, the boundary trace process behaves like a Cauchy process at small scales while it is similar to a $1/2$-stable process at very large scales.
	\end{example}
	
	\begin{example}[A slit domain]\label{x:slit}
		We consider the reflected Brownian motion on the slit domain
		$\unifdom = \mathbb{R}^{2} \setminus \{(x,0) \mid x \in (-\infty,0]\}$,
		and equip $\unifdom$ with the inner metric
		$d_{\unifdom} \colon \unifdom \times \unifdom \to [0,\infty)$
		defined in the same way as in Example \ref{x:parabola} above, so that
		$d_{\unifdom}(z_{0},z_{1})$ is either $\abs{z_{0}}+\abs{z_{1}}$ or $\abs{z_{0}-z_{1}}$,
		depending on whether or not the segment $\{(1-t)z_{0}+tz_{1}\mid t\in[0,1]\}$
		intersects $\{(x,0) \mid x \in (-\infty,0]\}$.
		Let $(\ambient,d_{\unifdom})$ denote the completion of $(\unifdom,d_{\unifdom})$
		and let $\refmeas$ denote the Borel measure on $\ambient$ given by
		$\refmeas = \lambda(\cdot \cap \unifdom)$, where $\lambda$ is the Lebesgue measure on $\mathbb{R}^{2}$
		($\unifdom$ is viewed as an open subset of $\ambient$ as usual).
		Since every $d_{\unifdom}$-Cauchy sequence is Cauchy with respect to the
		Euclidean metric $d$, the identity map of $\unifdom$ extends uniquely
		to a continuous map $p \colon \ambient \to \mathbb{R}^{2}$. The boundary
		$\partial \unifdom$ of $\unifdom$ in $(\ambient,d_{\unifdom})$ is then given as
		\begin{equation*}
		\partial \unifdom = p^{-1} \bigl( \{(x,0) \mid x \in (-\infty,0]\} \bigr) = \ambient \setminus \unifdom.
		\end{equation*}
		The set $\unifdom$ is not a uniform domain in $(\mathbb{R}^{2},d)$, but is an inner uniform domain
		in $(\mathbb{R}^{2},d)$ and hence a uniform domain in $(\ambient,d_{\unifdom})$.
		
		Let $(\form,\domain)$ denote the Dirichlet form on $L^{2}(\mathbb{R}^{2},\lambda)$
		of the Brownian motion on $\mathbb{R}^{2}$. Again by \cite[Theorems 3.30 and 3.34]{GS},
		the reflected Dirichlet form $(\formrefgen{\unifdom},\domain(\unifdom))$
		of $(\form,\domain)$ on $\unifdom$ (recall Definition \ref{dfn:reflected-form})
		is a strongly local regular symmetric Dirichlet form on $L^{2}(\ambient,\refmeas)$
		(but it is \emph{not} regular on $L^{2}(\mathbb{R}^{2},\lambda)$) and
		the MMD space $(\ambient,d_{\unifdom},\refmeas,\formrefgen{\unifdom},\domain(\unifdom))$
		satisfies (\hyperlink{VD}{\textup{VD}} and) Gaussian heat kernel estimates.
		The $\formrefgen{\unifdom}$-harmonic profile of $\unifdom$ with base point $z_{0} = (1,0) \in \unifdom$ is given by
		\begin{equation*}
		\hprof{z_0}(x,y)= \biggl( \frac{x+\sqrt{x^2+y^2}}{2} \biggr)^{1/2}
		\quad \textrm{for all $(x,y) \in \unifdom$.}
		\end{equation*}
		In this case, the space-time scaling function
		$\scjump \colon \partial \unifdom \times [0,\infty) \to [0,\infty)$
		for the boundary trace Dirichlet form is given by 
		\begin{equation*}
		\scjump(\xi,r)=\hprof{z_0}(x_{\xi},r)
			=\Biggl(\frac{x_{\xi}+\sqrt{x_{\xi}^2+r^2}}{2}\Biggr)^{1/2}
			\asymp \frac{r}{\sqrt{\abs{x_{\xi}}}} \wedge \sqrt{r}
		\end{equation*}
		for all $(\xi,r) \in \partial \unifdom \times (0,\infty)$,
		where $x_{\xi} \in (-\infty,0]$ is given by $p(\xi)=(x_{\xi},0)$.  
		By \eqref{eq:bdry-scale-estimate}, there exists $C \in (1,\infty)$
		such that the corresponding $\formrefgen{\unifdom}$-elliptic measure
		at infinity $\bdrymeas = \nu^{\unifdom}_{z_0}$ satisfies 
		\begin{equation} \label{e:EllipMeas-slit}
		C^{-1} \scjump(\xi,r) \le \bdrymeas(B_{d_{\unifdom}}(\xi,r)) \le C \scjump(\xi,r)
			\quad \textrm{for all $(\xi,r) \in \partial \unifdom \times (0,\infty)$.}
		\end{equation}
		Similarly, by Corollary \ref{c:jpsi} and \eqref{e:EllipMeas-slit},
		there exists $C \in (1,\infty)$ such that
		\begin{equation*}
		C^{-1}\scjump(\xi,d_{\unifdom}(\xi,\eta))^{-2}
			\le \jumpkertr_\bdrymeas(\xi,\eta)
			\le C \scjump(\xi,d_{\unifdom}(\xi,\eta))^{-2}
			\quad \textrm{for all $(\xi,\eta) \in \offdiagp{\partial \unifdom}$.}
		\end{equation*}
		By Theorem \ref{thm:shk-trace}-\eqref{it:shk-trace-CHK}, we have stable-like
		heat kernel estimates \eqref{eq:shk-trace-upper-intro} and \eqref{eq:shk-trace-lower-intro}
		with respect to the metric $d_{\unifdom}$ for the boundary trace Dirichlet form.
	\end{example}

	\noindent Research Institute for Mathematical Sciences, Kyoto University,
	Kitashirakawa-Oiwake-cho, Sakyo-ku, Kyoto 606-8502, Japan.\\
	nkajino@kurims.kyoto-u.ac.jp\smallskip\\
	
	\noindent Department of Mathematics, University of British Columbia,
	Vancouver, BC V6T 1Z2, Canada.\\
	mathav@math.ubc.ca
	
\end{document}